\RequirePackage{fix-cm}
\documentclass[smallextended,stropt,envcountsect,envcountsame]{svjour3_1}       

\smartqed

\usepackage[a4paper, margin=2.5cm]{geometry}
\usepackage{amsfonts,amssymb,amscd,textcomp,fontenc}
\usepackage{amsbsy}
\usepackage{amsxtra}
\usepackage{latexsym}
\usepackage{stmaryrd} 
\usepackage{mathrsfs}
\usepackage{upgreek}
\usepackage{wasysym}
\usepackage{eucal}
\usepackage{tkz-graph}
\usetikzlibrary{decorations.markings}
\usetikzlibrary{arrows,positioning} 
\usepackage{verbatim}
\usepackage{tikz}
\usepackage{multirow}
\usepackage{float}
\usepackage{caption}
\usepackage{subcaption}

\usepackage[utf8]{inputenc}
\usepackage{graphicx}
\usepackage[all,cmtip]{xy}
\usepackage{xcolor}
\definecolor{rouge}{rgb}{0.85,0.1,.4}
\definecolor{bleu}{rgb}{0.1,0.2,0.9}
\definecolor{violet}{rgb}{0.7,0,0.8}

\usepackage[colorlinks=true,linkcolor=bleu,urlcolor=violet,citecolor=violet]{hyperref} 

\numberwithin{equation}{section}

\def\P{\mathscr P}
\def\Z{\mathbb Z}

\def\N{\mathbb N}
\def\R{\mathbb R}
\def\C{\mathbb C}

\def\F{\mathscr{F}}
\def\E{\mathscr E}
\def\I{\underline{\mathbf{I}}}

\def\g{\mathfrak{g}}  
\def\b{\mathfrak{b}}  
\def\n{\mathfrak{n}}
\def\h{\mathfrak{h}}  
\def\B{\mathscr{B}}   
\def\sl{\mathfrak{sl}}
\def\sp{\mathfrak{sp}}

\def\eps{\varepsilon}

\def\c#1{\check{#1}}  
\def\s#1{#1^{\sharp}} 
\def\sub#1{\underline{#1}}
\def\d{{\rm d}} 

\def\ch{{\rm Ch}}
\def\hc{{\rm hc}}
\def\sym{\beta}

\def\ev{\mathrm{ev}}
\def\Cl{{\rm Cl}}
\def\m{\boldsymbol{\sub{\mu}}}

\newcommand{\ad}{\mathrm{ad}}
\newcommand{\tr}{\mathrm{tr}}

\def\he{\mathrm{ht}}
\newcommand{\wt}{{\rm wt}}

\newcommand{\Exterior}{\mathchoice{{\textstyle\bigwedge}}%
    {{\bigwedge}}%
    {{\textstyle\wedge}}%
    {{\scriptstyle\wedge}}}

\begin{document}

\title{Kostant principal filtration and paths in weight lattices
\thanks{The first named author is supported in part by
Directorate General of Resources for Science, Technology,
and Higher Education, Ministry of Research, Technology, and Higher Education of Indonesia.
The second named author is supported in part by the ANR Project GeoLie Grant number ANR-15-CE40-0012, 
and in part by the Labex CEMPI (ANR-11-LABX- 0007-01)}
}

\author{Nilamsari Kusumastuti        \and
        Anne Moreau}


\institute{Nilamsari Kusumastuti \at
              Department of Mathematics, Tanjungpura University, Pontianak 78124, Indonesia \\
              \email{nilamsari@math.untan.ac.id, nilamsari.kusumastuti@math.univ-poitiers.fr}           
           \and
           Anne Moreau \at
              Laboratoire Paul Painlevé (UMR 8524 CNRS)
Université de Lille, Lille 59655, France\\
\email{anne.moreau@univ-lille.fr}
}


\maketitle

\begin{abstract}
There are several interesting filtrations on the Cartan subalgebra of a complex simple Lie algebra coming from very different contexts: one is {\em the principal filtration} coming from the Langlands dual, one is coming from the Clifford algebra associated with a non-degenerate invariant bilinear form, one is coming from the symmetric algebra and the Chevalley projection, and two other ones are coming from the enveloping algebra and Harish-Chandra projections. It is now known that all these filtrations coincide. This results from a combination of works of several authors (Rohr, Joseph, Alekseev and the second named author), and was essentially conjectured by Kostant. 
In this paper, we establish a {\em direct} correspondence between the enveloping filtration and the symmetric filtration for a simple Lie algebra of type $A$ or $C$. 
Our proof is very different from Rohr and Joseph approaches. 
The idea is to use an explicit description of the symmetric and enveloping 
invariants in term of combinatorial objects, 
called {\em weighted paths}, in the crystal graph of the standard representation. 
\end{abstract}

\section{Introduction}
In this paper, we are interested in several increasing filtrations, 
that are described below, on the Cartan subalgebra $\h$ 
of a complex simple Lie algebra $\g$ with triangular decomposition, 
$\g = \n_- \oplus \h \oplus \n_+,$ 
that come from different contexts. 
It is now known that all these filtrations coincide. This results from 
a combination of several works  
\cite{Rohr08,AlekMor12,Jos11b,Jos11a}. 
The remarkable connexion between the principal filtration (\S\ref{sub:principal}) 
and the filtration coming from 
the Clifford algebra (cf.~\S\ref{sub:Clifford}) was essentially  
conjectured by Kostant. 

The aim of this paper is to provide another proof, in type $A$ and type $C$, of 
the existing relation between the symmetric filtration (cf.~\S\ref{sub:symmetric})
and the enveloping filtration (cf.~\S\ref{sub:enveloping}) 
using special paths in 
the weight lattice of the standard
representation. 

\subsection{The principal filtration}
\label{sub:principal}
Let $\Delta \subset \h^*$ be the root system of $(\g,\h)$, 
$\Pi = \{ \beta_1,\ldots,\beta_r \}$ the system of simple roots 
with respect to $\b = \h \oplus \n_+$ 
and $\Delta_+$ the corresponding set of positive roots. 
The root system is realized in an Euclidean space $\R^N$
with standard basis $\mathcal{E}=(\eps_{1},\ldots,\eps_{N})$.
For $\alpha \in \Delta$, we denote by $\c{\alpha}$ its coroot. 
We fix a Chevalley basis $\{e_\alpha,e_{-\alpha}, \c{\beta}_i 
\; ; \; \alpha \in \Delta_+, i=1,\ldots,r\}$
of $\g$, where $e_\alpha$ is a nonzero $\alpha$-root vector. 
Let $\varpi_1,\ldots,\varpi_r$ 
be the fundamental weights, 
and $\c{\varpi}_1,\ldots,\c{\varpi}_r$ the fundamental 
co-weights, associated with $\beta_1,\ldots,\beta_r$, respectively. 

Let $B_\g$ be an invariant non-degenerate 
bilinear form on $\g \times \g$, 
and $B_\g^\sharp \colon \h^* \to \h$ the induced 
isomorphism. 
For $x\in \h^*$ we denote by $\s{x}$ its image by $B_\g^\sharp$. 
Let $(e,h,f)$ be a principal $\sl_2$-triple of $\g$ corresponding to the 
above triangular decomposition, that is, 
$e=\sum_{i=1}^r e_{\beta_i}, 
\, h=2 \sum_{i=1}^r \c{\varpi}_i,  
\, f=\sum_{i=1}^r c_i e_{-\beta_i},$
where $c_i$ is a nonzero complex number such that $h=[e,f]$. 
The elements $e,h,f$ are {\rm regular} elements of $\g$, which means that their centralizer 
in $\g$ have minimal dimension~$r$. 

One defines an increasing filtration $(\F^{(m)}\h)_{m\ge 0}$ of $\h$ by: 
$$\F^{(m)}\h:= \{x \in \h \; | \ (\ad\, e)^{m+1}x = 0\}.$$ 
Notice that the dimension of the spaces $\F^{(m)}\h$ 
jumps at the exponents $m=m_1,\ldots,m_r$ 
of $\g$.  
We can also describe the filtration $(\F^{(m)}\h)_m$ as follows. 
The algebra $\mathfrak{s}:= \C e \oplus \C h \oplus \C f \cong \sl_2$ 
acts on $\g$ by the adjoint action. 
Let $\g=\bigoplus_{i=1}^r V_i$ 
be the decomposition of $\g$ into simple $\mathfrak{s}$-modules. 
We have $\dim V_i=2 m_i+1$  
and $\dim V_i \cap \h=1$ for any $i=1,\ldots,r$. 
Then 
$$\F^{(m)}\h= \bigoplus_{j, \, m_j \leqslant m} V_j \cap \h.$$ 
If the exponents $m_1,\ldots,m_r$ are pairwise distinct, 
we have for $i=1,\ldots,r$, 
$\F^{(m_i)}\h= \bigoplus_{j=1}^{i} V_j \cap \h.$
In most cases, the exponents $m_1,\ldots,m_r$ 
are all distinct. The exception is the case of the $D_r$ 
series, for even $r$ and $r\ge 4$, 
when there are two coincident exponents (equal to $r -1$). 

Let $\c{\g}$ be the Langlands dual of $\g$ which is the 
simple Lie algebra defined by 
the dual root system $\c{\Delta}= \{\c{\alpha} \; ; 
\; \alpha \in \Delta\}$. One may identify a 
Cartan subalgebra $\c{\h}$ of $\c{\g}$ 
with $\h^*$. 
Let $(\c{e},\c{h},\c{f})$ be the corresponding 
principal $\sl_2$-triple of $\c{\g}$, 
and let $\rho$ be the half-sum of positive roots. 
Note that 
$$ \rho = \sum_{i=1}^r \varpi_i = \frac{1}{2}\c{h}.$$
Similarly, the principal $\sl_2$-triple $(\c{e},\c{h},\c{f})$ 
 defines an increasing filtration $(\F^{(m)} \c{\h})_m$ 
of $\c{\h}$.  
Since $\c{\g}$ has the same exponents as $\g$ 
the dimension of $\F^{(m)}\c{\h}$ 
jumps at the exponents $m=m_1,\ldots,m_r$,  too.
The {\em principal filtration} of 
$\h$ is the filtration $(\c{\F}^{(m)}\h)_m$, where 
\begin{align*} 
\c{\F}^{(m)}\h:=\left(B_\g^\sharp(\F^{(m)}\c{\h})\right)_m.
\end{align*}
Since $[\c{e},[\c{e},\c{h}]]=0$, we have $\rho \in \F^{(1)}\c{\h}$ 
and $\s{\rho}=B_\g^\sharp(\rho) \in \c{\F}^{(1)}\h$. 

\subsection{The symmetric filtration}
\label{sub:symmetric}
Let $S(\g)$ be the symmetric algebra of $\g$, 
and $\ch \colon S(\g) \to S(\h)$ the {Chevalley projection map} 
with respect to decomposition $S(\g) =S(\h) \oplus (\n_- + \n_+)S(\g)$. 
For any finite-dimensional simple $\g$-module $V$, 
consider the map, 
\begin{align*} 
\ch \otimes 1 \colon S(\g) \otimes V \to  S(\h) \otimes V.
\end{align*}
We have  $(\ch \otimes 1)((S(\g) \otimes V)^\g) \subset S(\h) \otimes V_0$,   
where $(S(\g) \otimes V)^\g$ is the subspace of invariants 
under the diagonal action of $\g$, and 
$V_0=\{v \in V \; |\; h.v=0\; \text{ for all } h \in \h\}$ is 
the zero-weight space of $V$.  
In the case where $V=\g$ is the adjoint representation, the free $S(\g)^\g$-module 
$(S(\g) \otimes \g)^\g$ is generated by the differentials $\d p_1,\ldots,\d p_r$ 
of homogeneous generators $p_1,\ldots,p_r$ of $S(\g)^\g$. 
Such homogeneous generators are of degrees $m_1+1,\ldots,m_r+1$, 
respectively, if we order them by increasing degrees. 
In this case, note that $(\ch \otimes 1)((S(\g) \otimes \g)^\g) \subset S(\h) \otimes \h$,
since the zero weight subspace of the adjoint representation $\g$ is $\h$.

Let 
$\ev_\rho \colon S(\h) \to  \C$ denotes the evaluation map at $\rho$, that is, 
$\ev_\rho(x)  = \langle \rho,x \rangle$ for all $x\in S(\h) \cong \C[\h^*] $, 
where $\langle \cdot , \cdot \rangle$ is the pairing between $\h^*$ 
and $\h$.
We define the {\em symmetric filtration} of $\h$ by setting for all $m \in \Z_{\geqslant 0}$, 
\begin{align*} 
\F_S ^{(m)} \h := (\ev_\rho \otimes 1) \circ (\ch \otimes 1) ( (S^m(\g) \otimes \g)^\g),
\end{align*}
where $(S^m(\g))_m$ is the standard filtration on $S(\g)$ induced 
by the degree of elements. 
This is indeed a filtration of $\h$ since 
the elements $\d p_1 (\rho) ,\ldots,\d p_r (\rho)$ are linearly 
independent, $\rho$ being regular \cite{Kos63}.  
In fact, Rohr proved \cite{Rohr08} that 
the principal filtration and the symmetric filtration coincide\footnote{
This fact is not hard to establish, but Rohr obtained a 
more precise result \cite{Rohr08}, and explicitly described  
an orthogonal basis with respect to $B_\g$ in $\h$ from the algebra 
of invariants $(S(\g)\otimes \g)^\g$.}. 

\subsection{The enveloping filtration(s)}
\label{sub:enveloping}
Let $U(\g)$ be the universal enveloping algebra of $\g$. 
The Harish-Chandra map  $\hc \colon U(\g) \to U(\h) = S(\h)$ 
is projection 
with respect to the decomposition 
$U(\g) = U(\h) \oplus (\n_- U(\g) + U(\g)\n_+)$.
It is well-known that its restriction to 
$U(\g)^\h$ 
is a morphism of associative algebras. 
Khoroshkin, Nazarov and Vinberg established the following 
triangular decomposition \cite{KNV11}: 
$$U(\g) \otimes V = (S(\h)\otimes V) \oplus \left(\rho_L(\n_-) (U(\g) \otimes V) +  
\rho_R(\n_+) (U(\g) \otimes V) \right),$$
where $\rho_L$ and $\rho_R$ are the two commuting $\g$-actions on $U(\g) \otimes V$ 
given by 
$\rho_L(x) (a \otimes b) = x a \otimes b$ 
and  $\rho_R(x) (a\otimes b) = -a x \otimes b + a \otimes 
x.b$ for $x \in\g$, 
respectively. 
They also showed that the image by the 
{\em generalized Harish-Chandra projection map} 
$\widetilde{\hc} \colon 
U(\g) \otimes V  \to S(\h) \otimes V$ 
with respect to the above triangular decomposition
of the invariant part $(U(\g) \otimes V)^\g$ 
for the diagonal action $\rho = \rho_L +\rho_R$ 
is the space of invariant under all the {\em Zhelobenko operators} (cf.~\cite[Theorem 1]{KNV11}). 

We now consider again the special case where $V=\g$ is the adjoint representation. 
Then we get that $(\hc \otimes 1) )( (U^m(\g) \otimes \g)^\g) 
\subset (S(\h) \otimes \g)^\g \subset S(\h) \otimes \h$ since, as noted before,
$\h$ is the zero weight subspace of the adjoint representation $\g$.
We define  the {\em enveloping filtration} 
and the {\em generalized enveloping filtration} of $\h$, respectively, by setting 
for all $m \in \Z_{\geqslant 0}$:
\begin{align*} 
&\F_U ^{(m)} \h  :=( (\ev_\rho \otimes 1) \circ (\hc \otimes 1) )( (U^m(\g) \otimes \g)^\g),\\ \nonumber
&\widetilde{\F}_{U} ^{(m)} \h  := ((\ev_\rho \otimes 1) \circ \widetilde{\hc}) 
( (U^m(\g) \otimes \g)^\g),
\end{align*}
where $(U^{m}(\g))_m$ is the standard filtration of $U(\g)$.    
It is not a priori clear that the sets $\F_U^{(m)} \h$, $m \geqslant  0$, exhaust $\h$  
since we do not know a priori whether 
$( (\ev_\rho \otimes 1) \circ (\hc \otimes 1)) ( (U(\g) \otimes \g)^\g)$ 
is equal to $\h$. Similarly, 
we do not know a priori whether 
$ ((\ev_\rho \otimes 1) \circ \widetilde{\hc}) ( (U(\g) \otimes \g)^\g)$ 
is equal to $\h$. 
This is indeed the case.  
A stronger result is in fact true. 

\begin{theorem}[{\cite{Jos11b,Jos11a}}] 
\label{theorem:Joseph2}
For any $m\in \Z_{\geqslant 0}$, we have:
$\F_U ^{(m)} \h 
=\c{\F}^{(m)} \h$ and $\widetilde{\F}_U ^{(m)} \h 
=\c{\F}^{(m)} \h.$ 
\end{theorem}

Joseph's theorem is a highly non-trivial result, 
especially in the non-simply laced cases. Its proof is deeply based on the 
description by Khoroshkin, Nazarov and Vinberg 
of the invariant spaces in term of Zhelobenko operators. 

\subsection{The Clifford filtration}
\label{sub:Clifford} 
Let $\Cl(\g)$ be the Clifford algebra over $\g$ associated with 
the bilinear form $B_\g$. 
Consider the decomposition of the Clifford algebra
$\Cl(\g) = \Cl(\h) \oplus (\n_-  \Cl(\g) + \Cl(\g) \n_+),$ 
 where $\g$ is viewed as a subalgebra of $\Cl(\g)$ using the 
canonical injection $\g \hookrightarrow \Cl(\g)$.  
The corresponding projection 
$\hc_{\rm odd}
\colon \Cl(\g) \to \Cl(\h)$ is called the 
{\em odd Harish-Chandra projection}. 
By a non-trivial result of Bazlov \cite[\S5.6]{Baz09} and Kostant 
(private communication), 
the odd Harish-Chandra projection $\hc_{\rm odd}$ 
maps the invariant algebra $\Cl(\g)^\g \cong \Cl(P(\g))$ 
onto $\h \subset \Cl(\h)$. Here $P(\g)$ is the space of {primitive invariants}. 
The {\em Clifford filtration} of $\h$ is defined by setting for all $m \in \Z_{\geqslant 0}$:
\begin{align*}
\F_{\Cl} ^{(m)} \h :=  \hc_{\rm odd} (q(P^{(2m+1)}(\g))),
\end{align*}
where $(P^{2m+1}(\g))_m$ is the natural filtration on $P(\g)$ induced 
from the degrees of generators, and $q : \Exterior \g \to \Cl(\g)$ is the quantisation map. 
It is a well-defined filtration on $\h$ thanks to the above quoted result of 
Bazlov--Kostant.  
Moreover, we have the following fact. 

\begin{theorem}[{\cite{AlekMor12}}]  \label{theorem:Clifford2}
For any $m \in \Z_{\geqslant 0}$, we have:
$\F_{U} ^{(m)} \h= \F_{\Cl} ^{(m)} \h$ and $\widetilde{\F}_{U} ^{(m)} \h= \F_{\Cl} ^{(m)} \h.$
\end{theorem}
Kostant observed that $\hc_{\rm odd} (q(\phi)) = B_\g^\sharp(\rho)$, with $\phi \colon (x,y,z) \mapsto B_\g (x , [y,z]) \in \Exterior^3\g$ the invariant differential  Cartan 3-form of $\g$, and 
conjectured 
that the images of the higher 
generators $q(\phi_i)$, for $i=2,\ldots,r$, by the odd Harish-Chandra projection 
can be described using the principal 
filtration. 
Next theorem positively answers his conjecture:
\begin{theorem}[{\cite{Baz03,Jos11b,Jos11a,AlekMor12}}] \label{theorem:Kostant}  
For any $m\in \Z_{\geqslant 0}$, we have 
$ \F_{\Cl}^{(m)}\h=  \c{\F}^{(m)}\h. $ 
\end{theorem}
The conjecture was proved in type $A$  
by Bazlov \cite{Baz03} in 2003, and then in full generality 
combining the works of Joseph 
\cite{Jos11b,Jos11a}, together with the work of Alekseev and the second 
author \cite{AlekMor12} (see Theorems \ref{theorem:Joseph2} and~\ref{theorem:Clifford2}).

\subsection{Main results and strategy}
In this paper we establish a {\em direct} connection between 
the symmetric filtration and the enveloping filtration
for $\g$ of type $A$ or $C$. 
\begin{theorem} \label{theorem:main} 
Assume that $\g$ is $\sl_{r+1}$ or $\sp_{2r}$. 
Then for any $m\in \Z_{\geqslant 0}$, 
$\F_S^{(m)} \h  = \F_U^{(m)} \h .$ 
\end{theorem}

Together with Rohr's result (cf.~\ref{sub:symmetric})   
and Theorem~\ref{theorem:Clifford2},  
Theorem \ref{theorem:main} gives another proof of Joseph's theorem, 
and so another proof of Kostant's conjecture for the types $A$ 
and~$C$.
Since the symmetric filtration and 
the enveloping filtration are similarly defined, 
our approach is quite natural. 
It is surprising that their equality was first established only indirectly.
We hope that our method can be adapted to other contexts, 
for instance to affine Lie algebras in types $A$ and $C$ or to 
other particular representations different from the adjoint representation.  
Remember that the symmetric filtration and 
the enveloping filtration are both defined for the special case 
where $V=\g$ is the adjoint representation. 
One may ask if an analogue result 
can be extended to some other 
finite-dimensional simple $\g$-module $V$.   
It is known that the equality does not hold for all 
simple $\g$-module $V$ (a counter-example was found by Anton 
Alekseev), but it would be interesting 
to know which are the representations for which it does;  
see also \cite[\S3.3]{Jos11a} for related topics. 
Alekseev's counter-example 
suggests that it is a hard problem, and general arguments 
cannot be applied.

We summarize in Figure \ref{fig:connection} how the connections between 
all filtrations described above were established by different researchers.

\begin{figure}[h]
\begin{center}
\tikzset{
    >=stealth',
    punkt/.style={
           rectangle,
           rounded corners,
           draw=black, very thick,
           text width=6.5em,
           minimum height=2em,
           text centered},
    pil/.style={
           ->,
           thick,
           shorten <=2pt,
           shorten >=2pt,}
}
\begin{tikzpicture}[node distance=1cm, auto,]
 \node[punkt] (principal) {$\c{\F}^{(m)}\h$};
 \node[punkt, inner sep=5pt,below=0.5cm of principal]
 (clifford) {${\F}_{\Cl}^{(m)}(\h)$};
 \node[above=of principal] (dummy) {};
 \node[punkt,right=of dummy] (enveloping){${\F}_U^{(m)}(\h)$}
   edge[pil,draw=red,<->,bend left=45]  node[left=0.3cm,above] {\small {\color{red}\cite{Jos11b,Jos11a}}}(principal.east)
  edge[pil,draw=blue, <->,bend left=45] node[right=1mm] {} (clifford.east);
  \node[punkt,right=0.5mm of enveloping] (gen_enveloping){$\widetilde{\F}_U^{(m)}(\h)$}
   edge[pil,draw=red, <->,bend left=45]  node[left=0.5cm,above] {}(principal.east)
   edge[pil,draw=blue, <->,bend left=45] node[right=2mm] {\small {\color{blue}\cite{AlekMor12}}} (clifford.east);
 \node[punkt, left=of dummy] (symmetric) {${\F}_S^{(m)}(\h)$}
   edge[pil, draw=brown, <->,bend right=45] node[left=1mm] {\small {\color{brown}\cite{Rohr08}}} (principal.west)
   edge[pil,<->, bend left=45] node[auto] {Theorem \ref{theorem:main}, $\g=\sl_{r+1}$ or $\g=\sp_{2r}$} (enveloping);
 \node[below=1.5cm of symmetric] (gbelow){Kostant conjecture};
 \node[above=0.75mm of clifford] (dummy2) {};
 \path[<->,thick] (principal) edge (clifford);
 \path[->,thick] (gbelow) edge (dummy2);
\end{tikzpicture}
 \caption{Connections between several filtrations on the Cartan subalgebra $\h$.}
 \label{fig:connection}
\end{center}
\end{figure}

We outline below our strategy to prove Theorem \ref{theorem:main}. 
Assume that $\g=\sl_{r+1}$ or $\g=\sp_{2r}$. 
Then the exponents are pairwise distincts in both cases. 
One can choose homogeneous generators of $S(\g)^\g$ as follows. 
Set $n=r+1$ if $\g=\sl_{r+1}$ and $n=2r$ if $\g=\sp_{2r}$, 
and 
let $(\pi,\C^{n})$ be the standard representation of $\g$. It is 
the finite-dimensional irreducible representation with highest weight  
$\varpi_1$. Set $p_m := \dfrac{1}{(m+1)!} \tr \circ \pi^{m+1}$ for $m \in \Z_{\geqslant 0}$. 
Identifying $\g$ with $\g^*$ through the 
inner product $B_\g$, 
the elements $p_{m_1},\ldots,p_{m_r}$ are homogeneous generators of 
$\C[\g]^\g \cong S(\g)^\g$ of degree $m_1+1,\ldots,m_r+1$, respectively. 
Their differentials, $\d p_1,\ldots,\d p_m$, 
are homogeneous free generators of the free module $(S(\g)\otimes \g)^\g$ 
over $S(\g)^\g$. 
Moreover $(\sym \otimes 1)(\d p_{m_1}),
\ldots,(\sym \otimes 1)(\d p_{m_r})$ 
are homogeneous free generators of the free module $(U(\g)\otimes \g)^\g$ 
over $Z(\g)\cong U(\g)^\g$ (cf.~Proposition~\ref{Pro:gen}), 
where 
$\sym \colon S(\g) \to U(\g)$
is the symmetrization map. 

Fix $m \in \Z_{> 0}$. 
Since ${\d p}_{m}$ 
and $(\beta \otimes 1)({\d p}_m)$ are $\g$-invariant, 
the elements $\overline{\d p}_m := (\ch \otimes 1) (\d p_m)$ and  
$\widehat{\d p}_m := ((\hc \otimes 1) \circ (\beta \otimes 1))(\d p_m)$ 
lie in $S(\h) \otimes \h$. 
Define the elements $\overline{\d p}_{m,k} $ and 
$\widehat{\d p}_{m,k}$ of $S(\h)$ for $k\in\{1,\ldots,r\}$ by writing 
$\overline{\d p}_{m}$ and $\widehat{\d p}_{m} $ as: 
\begin{align*} \overline{\d p}_{m}  
= \dfrac{1}{m!}  \sum_{k=1}^{r}  \overline{\d p}_{m,k}
\otimes  \s{\varpi_k},  \qquad   
\widehat{\d p}_{m} = \dfrac{1}{m!}  \sum_{k=1}^{r}  
\widehat{\d p}_{m,k} \otimes \s{\varpi_k} .& 
\end{align*}

Our main results are the following: 

\begin{theorem} \label{theorem:main1} 
Assume that $\g$ is $\sl_{r+1}$ or $\sp_{2r}$, and  
let $m \in \Z_{> 0}$.  
For some polynomial $\bar{Q}_m \in \C[X]$ of degree $m-1$ 
we have $\ev_\rho(\overline{{\rm d} p}_{m,k})= \bar{Q}_m(k)$ 
for $k=1,\ldots,r$ so that 
$$\ev_\rho (\overline{{\rm d} p}_{m}  ) 
= \dfrac{1}{m!} \sum_{k=1}^{r}  \bar{Q}_m(k) 
 \s{\varpi_k}.$$ 
Moreover $\bar{Q}_1=1$ if $\g=\sl_{r+1}$,  
and $\bar{Q}_1=2$ if $\g=\mathfrak{sp}_{2r}$. 
\end{theorem}

\begin{remark}
\label{rem:Qi-sym}
For $m>1$, the polynomial $\bar{Q}_m$ is explicitly 
described by formula \eqref{eq:Qi-sl} for $\sl_{r+1}$, 
and by formula \eqref{eq:Qi-sp} for $\sp_{2r}$.  
\end{remark}

\begin{theorem} \label{theorem:main2} 
Assume that $\g$ is $\sl_{r+1}$ or $\sp_{2r}$, 
and let $m \in \Z_{> 0}$.  
For some polynomial $\hat{Q}_m \in \C[X]$ of degree 
at most $m-1$ 
we have $\ev_\rho(\widehat{{\rm d}  p}_{m,k})= \hat{Q}_m(k)$ 
for $k=1,\ldots,r$ so that 
$$\ev_\rho (\widehat{{\rm d} p}_{m}  ) 
= \dfrac{1}{m!} \sum_{k=1}^{r}   \hat{Q}_m(k) 
 \s{\varpi_k}.$$ 
Moreover $\hat{Q}_1=1$ if $\g=\sl_{r+1}$,  
and $\hat{Q}_1=2$ if $\g=\mathfrak{sp}_{2r}$.  
\end{theorem} 

\begin{remark}
For $m>1$, there is no nice general description of the polynomial $\hat{Q}_m$ 
in term of $r$ and $m$ 
as for the symmetric case. 
The polynomials $\hat{Q}_m$'s are defined inductively by formula \eqref{eq:hatQi-sl} 
for $\sl_{r+1}$, and by formula \eqref{eq:hatQi-sp} for $\sp_{2r}$.  
\end{remark}

We claim that Theorem \ref{theorem:main1} and Theorem \ref{theorem:main2} 
are sufficient to prove the main theorem: 

\begin{proposition}
Theorem \ref{theorem:main1} and Theorem \ref{theorem:main2} 
imply Theorem \ref{theorem:main}.  
\end{proposition}

\begin{proof}
Assume that $\g$ is $\sl_{r+1}$ or $\sp_{2r}$. 
First, the elements $\ev_\rho (\overline{{\rm d} p}_{m_1}  ), 
\ldots,\ev_\rho (\overline{{\rm d} p}_{m_r}  )$ are linearly independent 
in $\h$ since $\rho$ is regular \cite{Kos63}. 
Thus, for any $j \in \{1,\ldots,r\}$, 
the vector space generated by  $\ev_\rho (\overline{{\rm d} p}_{m_1}  ), 
\ldots,\ev_\rho (\overline{{\rm d} p}_{m_j} )$ has dimension $j$. 
We denote it by $ \bar{V}_j$. 
Set for $j \in \Z_{\geqslant 0}$, 
$w_j= \sum_{k=1}^{r}   k^{j}
\s{\varpi}_k$, and $W_j := {\rm span}_{\C}\left\lbrace w_0,\ldots,w_{j-1} \right\rbrace.$
We have $\dim W_j \leqslant  j$. 

We first show that for all $j \in \{1,\ldots,r\}$, $\bar{V}_j = W_{m_j}$.
By Theorem \ref{theorem:main1}, the inclusion $ \bar{V}_j \subset W_{m_j}$ holds. 
Assume that there is $w \in W_{m_j} \setminus  \bar{V}_j$. 
Write $w$ in the basis $\ev_\rho (\overline{{\rm d} p}_{m_1}),\ldots, 
\ev_\rho (\overline{{\rm d} p}_{m_r})$: 
$$w = \sum_{i=1}^r a_i \, \ev_\rho (\overline{{\rm d} p}_{m_i}),$$
and let $q$ be the maximal integer $i \in \{1,\ldots,r\}$ such that $a_i \not=0$. 
Since $w \not \in  \bar{V}_j$, $q > j$. But then $w$ is in $W_{m_q}$ and not in $W_{m_{q}-1}$ 
by Theorem \ref{theorem:main1}, which contradicts the fact that $w \in W_{m_j}$ 
because $m_j \leqslant m_q -1$. We have shown the expected equality of vector spaces. 

Let us now denote by $\hat{V}_j$ the vector space generated
 by  $\ev_\rho (\widehat{{\rm d} p}_{m_1}  ), 
\ldots,\ev_\rho (\widehat{{\rm d} p}_{m_j} )$ for $j \in\{1,\ldots,r\}$. 
Our aim is to show that $\hat{V}_j =  \bar{V}_j$  for all $j \in\{1,\ldots,r\}$. 
By the first step, it suffices to establish 
that $\hat{V}_j = W_{m_j}$ for all $j \in\{1,\ldots,r\}$. 
According to Theorem \ref{theorem:main2}, we have the inclusion 
$\hat{V}_j \subset W_{m_j}$ for all $j \in\{1,\ldots,r\}$, 
and $\dim W_{m_j}$ has dimension $j$ by the first step. 
Theorem \ref{theorem:Clifford2} implies that 
$ \hat{V}_r =\h$, and so the elements $\ev_\rho (\widehat{{\rm d} p}_{m_1}  ), 
\ldots,\ev_\rho (\widehat{{\rm d} p}_{m_r} )$ are linearly independent. 
In particular, each $\hat{V}_j$ has dimension $j$, hence the expected 
equality $\hat{V}_j = W_{m_j}$  for all $j \in\{1,\ldots,r\}$. 
This finishes the proof.  \qed
\end{proof}

The rest of the paper is organized as follows. 
Section~\ref{sec:gen} is about generalities on invariants coming from representations 
of simple Lie algebras. 
We introduce in this section our central notion of weighted paths. 
Section~\ref{sec:main2-A} is devoted to the proofs 
of Theorem \ref{theorem:main1} and Theorem \ref{theorem:main2}
in the case where $\g=\sl_{r+1}$.  The proofs of Theorem \ref{theorem:main1} 
and Theorem \ref{theorem:main2} for $\g=\sp_{2r}$ are performed in 
Section~\ref{sec:main2-C}. 

\section{General setting and weighted paths in crystal graph}
\label{sec:gen}

We first set up notation and 
standard facts on representations of simple Lie algebras and 
related invariant polynomials. 
Let
$Q = \sum\limits_{i=1}^r \Z\beta_i$ 
be the root lattice
 and set 
$Q_+ :=  \sum\limits_{i=1}^r \Z_{\geqslant 0}\beta_i.$ 
We define a partial order on $\h^*$ by:
$ \mu \succcurlyeq \lambda \; \Longleftrightarrow \; \mu - \lambda \in Q_+.$
For $(\lambda,\mu) \in (\h^*)^2$, 
we denote by $[\![\lambda,\mu ]\!]$ 
the set of $\nu \in \h^*$ such that 
$\lambda \preccurlyeq \nu \preccurlyeq \mu$. 
Let also 
$P = \sum\limits_{i=1}^r \Z \varpi_i$ and $P_+ : = \sum\limits_{i=1}^r \Z_{\geqslant 0} \varpi_i$
be the weight lattice and the set of integral dominant weights, respectively. 
For $\lambda \in P_+$, we write $(\pi_\lambda,V(\lambda))$ 
for the finite-dimensional 
irreducible representation of $\g$ with highest weight 
$\lambda$ and for $\mu \in \h^*$, $V(\lambda)_\mu$ stands 
for 
the $\mu$-weight space of $V(\lambda)$. 
The set of nonzero weights of $V(\lambda)$ 
will be denoted by $P(\lambda)$.
Fix $m \in \Z_{> 0}$, 
and let 
$$p_m^{(\lambda)}= \frac{1}{(m+1)!} \tr \circ \pi_\lambda^{m+1}.$$
Thus $p_m^{(\lambda)}$  
is a $\g$-invariant element 
of $\C[\g] \cong S(\g^*)$ 
of degree $m+1$ (see, for example, \cite[Lemma~31.2.3]{TauvelYu}). 

Let $\B=(b_1,\ldots,b_d)$ be a basis of $\g$, and  
$\B^*$
its dual basis so that
\begin{align}\label{eq:pm}
p_{m}^{(\lambda)}
=\frac{1}{(m+1)!} \sum\limits_{1 \leqslant i_1,\ldots,i_{m+1} \leqslant d } 
\tr ( \pi_\lambda(b_{i_1}) 
\circ \cdots 
\circ \pi_\lambda(b_{i_{m+1}}) )  b_{i_1}^*  \ldots b_{i_{m+1}}^*.
\end{align}
Identifying $\g$ with $\g^*$ through $B_\g$, 
$p_{m}^{(\lambda)}$ becomes an element of 
$\C[\g^*]^\g \cong S(\g)^\g$ of degree $m+1$. 
Moreover, its differential ${\rm d} p_m^{(\lambda)}$, 
defined by 
$  {\rm d} p_m^{(\lambda)} = \sum_{k=1}^d \dfrac{\partial p_m^{(\lambda)}}{\partial b_k^*} \otimes b_k^*,$ 
is an element of $(S(\g) \otimes \g)^\g$. 

Recall that  $\sym \colon S(\g) \to U(\g)$ is the symmetrization map. 
Let $Z(\g)\cong U(\g)^\g$ be the center of the enveloping algebra. 
The first part of the following proposition is due to Kostant \cite{Kos63}. 
The second part is probably well-known. For the convenience of the reader, 
we give a proof. 

\begin{proposition} \label{Pro:gen} 
Assume that $p_{m_1}^{(\lambda)},\ldots,p_{m_r}^{(\lambda)}$ 
are homogeneous generators of $S(\g)^\g$.
Then ${\rm d} p_{m_1}^{(\lambda)},\ldots,{\rm d} p_{m_r}^{(\lambda)}$ 
are free homogeneous generators of the free module $(S(\g)\otimes \g)^\g$ 
over $S(\g)^\g$, and 
$(\sym \otimes 1)({\rm d} p_{m_1}^{(\lambda)}),
\ldots,(\sym \otimes 1)({\rm d} p_{m_r}^{(\lambda)})$ 
are free homogeneous generators of the free module $(U(\g)\otimes \g)^\g$ 
over $Z(\g)$. 
\end{proposition}

\begin{proof} 
It suffices to prove the second part. 
Our arguments is adapted from Diximer's \cite{Dixmier}. 
Denoting $\hat{\beta}=\sym \otimes 1$, we first observe that 
 $\hat{\beta}((S(\g)\otimes \g)^{\g})$ is equal to $(U(\g)\otimes \g)^{\g}$,
since $\hat{\beta}$ is an isomorphism which commutes
with the diagonal action of $\g$. Then it sends invariants to invariants.
Let $M$ be the sub-$Z(\g)$-module of $(U(\g)\otimes \g)^{\g}$ generated by 
$\hat{\beta}(\d p_{m_1}),\ldots,\hat{\beta}(\d p _{m_r})$. 
Then the graded module associated with the filtration induced 
on $M$ is a sub-$S(\g)^{\g}$-module of $(S(\g)\otimes \g)^{\g}$,
since $S(\g)^{\g}$ is the graded space associated with 
the filtration induced on $Z(\g)$. 
It contains $\d p_{m_1}\ldots, \d p _{m_r}$, and so it is equal to 
$(S(\g)\otimes \g)^{\g}$. 
On the other side, the graded space of $(U(\g)\otimes \g)^{\g}$ is contained in $(S(\g)\otimes \g)^{\g}$.
Hence we get that 
$M = (U(\g)\otimes \g)^{\g}$. 

For the freeness, let $K$ be the module of relations between 
$\hat{\beta}(\d p_{m_1}),\ldots,\hat{\beta}(\d p _{m_r})$ on $Z(\g)$, so that 
$K \subset Z(\g)^r.$
The filtration on $Z(\g)$ induces a filtration on $K$ and the graded 
module associated with this 
filtration is contained in the module of relations on $S(\g)^{\g}$ between 
$\d p_{m_1}\ldots, \d p _{m_r}$, which is zero. Hence $K=0$.\qed
\end{proof}

\begin{example} \label{ex:generators}
In Table \ref{Tab:Mehta}, we give examples 
where $p_{m_1}^{(\lambda)},\ldots,
p_{m_r}^{(\lambda)}$ are homogeneous generators of $S(\g)^\g$ 
following  
\cite{Mehta88}. 
 {\begin{table}[h]
\caption{Examples of weights $\lambda$ 
for which the conditions of Proposition 
\ref{Pro:gen} hold.} \label{Tab:Mehta}  
 \small
 \begin{center}
\begin{tabular}{|l|l|l|l|l|}
\hline
Type & $\lambda$ & $\dim V(\lambda)$ & $N$ & decomposition in 
the basis $\mathcal{E}$ \\
\hline 
$A_r$ & $\varpi_1$ & $r+1$ & $r+1$ & 
$\varpi_1=\eps_1-\frac{1}{r+1} (\eps_1+\cdots+\eps_{r+1})$ \\
$B_r$ & $\varpi_1$ & $2r+1$ & $r$ & 
$\varpi_1=\eps_1$ \\
$C_r$ & $\varpi_1$ & $2r$ & $r$ & 
$\varpi_1=\eps_1$ \\
$D_r$, odd $r$ & $\varpi_1$ & $2r$ & $r$ & 
$\varpi_r=\eps_1$ \\
$G_2$ & $\varpi_1$ & $7$ & $3$ & $\varpi_1=-\eps_2+\eps_3$\\
$F_4$ & $\varpi_4$ & $26$ &  $4$ & $\varpi_4=\eps_1$\\
$E_6$ & $\varpi_1$ & $27$ & $8$ & 
$\varpi_1=\frac{2}{3}(\eps_8-\eps_7 -\eps_6)$\\
 \hline
\end{tabular}
 \end{center}
 \end{table}}
\end{example}
\vspace{-2em}
We fix from now 
on $\delta \in P_+$ 
such that each nonzero weight space $V(\delta)_\mu$, for 
$\mu \in P(\delta)$, has dimension one. 
This happens for example if $\lambda$ is a {\em minuscule} weight, that is,  
if $V(\lambda)$ is not trivial and if $P(\lambda) = W.\lambda$, where $W.\lambda$
denotes the orbit of the highest weight $\lambda$ 
under the action of the Weyl group $W=W(\g,\h)$. 
The minuscule weights are given in Table \ref{Tab:minus}. 
 {\begin{table}[h]\small
\caption{Minuscule weights}  \label{Tab:minus}
 \begin{center}
\begin{tabular}{|l|l|l|l|}
\hline
Type & minuscule weights & $N$ & decomposition in 
the basis $\mathcal{E}$ \\
\hline 
$A_r$ & $\varpi_1,\ldots,\varpi_r$ & $r+1$ & 
$\varpi_i = \eps_1+\cdots +\eps_i - \frac{i}{r+1}(\eps_1+\cdots +\eps_{r+1})$ \\
$B_r$ & $\varpi_r$ & $r$ & 
$\varpi_r=\frac{1}{2}(\eps_1+\cdots +\eps_r)$ \\
$C_r$ & $\varpi_1$ & $r$ & 
$\varpi_1 =\eps_1$ \\
$D_r$ & $\varpi_1,\varpi_{r-1},\varpi_r$ & $r$ & 
$\varpi_1=\eps_1, \varpi_{n+t} = \frac{1}{2}(\eps_1+\cdots + \eps_n) 
+ t\eps_t$, $t \in \{-1,0\}$\\
$E_6$ & $\varpi_1,\varpi_{6}$ & $8$ & 
$\varpi_1=\frac{2}{3}(\eps_8-\eps_7 -\eps_6), 
\varpi_6=\frac{1}{3}(\eps_8-\eps_7 -\eps_6) +\eps_5$ \\
 $E_7$ & $\varpi_7$ & $8$ & 
 $\varpi_7=\eps_6 +\frac{1}{2}(\eps_8-\eps_7)$\\
 \hline
\end{tabular}
 \end{center}
 \end{table}}

Choose for any $\mu \in P(\delta)$ a nonzero vector $v_\mu \in V(\delta)_\mu$. 
By our assumption, 
the set $\{v_\mu \; | \; \mu \in P(\delta)\}$ forms a basis of 
$V(\delta)$.
For $(\lambda,\mu) \in P(\delta)^2$ and $b \in \B$, 
define the scalar $a_{\lambda,\mu}^{(b)}$ by: 
\begin{align} \label{eq:tr}
\pi_\delta(b)v_\mu = 
\sum\limits_{\lambda \in P(\delta) } a_{\lambda,\mu}^{(b)} v_\lambda.
\end{align}
Next lemma immediately follows from \eqref{eq:pm}.

\begin{lemma} \label{lem:dec2}
For any $m >0$, we have: 
\begin{align*}
p_{m}^{(\delta)} 
=  \frac{1}{(m+1)!}\sum\limits_{1 \leqslant 
i_1, \ldots , i_{m+1} \leqslant d}  
\sum\limits_{( \mu_{j_1}, \ldots , \mu_{j_{m+1}}) \atop \in\, P(\delta)^{m+1} } 
 a_{\mu_{j_{1}},\mu_{j_2}}^{(b_{i_1})} \ldots a_{\mu_{j_{m+1}},\mu_{j_{1}}}^{(b_{i_{m+1}})} 
 b_{i_1}^* \ldots b_{i_{m+1}}^*. 
\end{align*}
\end{lemma}

Assume now that $\B=\{e_\alpha, \, \c{\beta}_i 
\; | \; \alpha \in \Delta , i=1,\ldots,r\}$ 
is the Chevalley basis of $\g$ (in a fixed order). 
Identifying $\g$ with $\g^*$ through the 
inner product $B_\g$, we get that  
$\B^*=\{c_\alpha e_{-\alpha}, \, \s{\varpi_i}
\; | \; \alpha \in \Delta, i=1,\ldots,r\},$
with $c_\alpha \not=0$ for $\alpha \in \Delta$.   
We denote by $\overline{\d p}_m^{(\delta)}$  
and $\widehat{\d p}_m^{(\delta)}$ the images of 
$\d p_m^{(\delta)}$ by $\ch \otimes 1$ and $(\hc \otimes 1)\circ \hat{\beta}$, 
respectively, where $\hat{\beta}$ is $\beta \otimes 1$ as in the proof of Proposition 
\ref{Pro:gen}.
When there will be no ambiguity about $\delta$, we will simply denote 
by $p_m$, $\d p_m$, $\overline{\d p}_m$, $\widehat{\d p}_m$ 
the corresponding elements. 

Let $m>0$. 
Since ${\d p}_m^{(\delta)}$ 
and $ \hat{\beta}({\d p}_m^{(\delta)})$ 
are $\g$-invariant, 
$\overline{\d p}_m^{(\delta)}$  and 
$\widehat{\d p}_m^{(\delta)}$ lie in $S(\h) \otimes \h$. 
Define the elements $\overline{\d p}_{m,k}^{(\delta)} $ and 
$\widehat{\d p}_{m,k}^{(\delta)}$ of $S(\h)$, for $k\in\{1,\ldots,r\}$, by: 
\begin{align*} \overline{\d p}_{m}^{(\delta)}  
&= \frac{1}{m!}  \sum_{k=1}^{r}  \overline{\d p}_{m,k}^{(\delta)}
\otimes  \s{\varpi_k},  &&  
\widehat{\d p}_{m}^{(\delta)} = \frac{1}{m!}  \sum_{k=1}^{r}  
\widehat{\d p}_{m,k}^{(\delta)} \otimes \s{\varpi_k} .& 
\end{align*}
For $k\in\{1,\ldots,r\}$, 
we set 
$$
P(\delta)_k := \{\mu \in P(\delta) \; |\; \langle \mu, \c{\beta}_k \rangle 
\not=0\}.
$$

\begin{lemma} \label{lem:dec3}
We have  
\begin{enumerate}
\item 
$ 
\overline{{\rm d} p}_{m,k}^{(\delta)} 
=
\sum\limits_{\mu \in P(\delta)_k } 
\sum\limits_{1\le i_1, \ldots , i_{m} \leqslant r } 
\langle \mu, \c{\beta}_{i_1} \rangle   \ldots  \langle \mu, \c{\beta}_{i_m} \rangle  \langle \mu, \c{\beta}_k \rangle \, 
\s{\varpi_{i_1}}   \ldots \s{\varpi_{i_m}},
$
\item 
$
\widehat{{\rm d} p}_{m,k}^{(\delta)} = 
 \sum\limits_{\mu_{j_1}, \ldots , \mu_{j_{m}} \in P(\delta) \atop  \mu_{j_1}\in \, P(\delta)_k } 
\sum\limits_{1\le i_1, \ldots , i_{m} \leqslant r }
  a_{\mu_{j_1},\mu_{j_2}}^{(b_{i_1})} \ldots 
   a_{\mu_{j_{m}},\mu_{j_{1}}}^{(b_{i_m})}  \langle \mu_{j_1}, \c{\beta}_k \rangle
   \hc(b_{i_1}^* \ldots b_{i_{m}}^*). 
$
\end{enumerate}
\end{lemma}

\begin{proof}
First of all, by Lemma \ref{lem:dec2}, we have  
\begin{align*}
{\d p}_{m}^{(\delta)} 
= \frac{1}{m!}
\sum\limits_{k=1}^d  
 \sum\limits_{\mu_{j_1}, \ldots , \mu_{j_{m+1}}\in P (\delta) \atop \mu_{j_1} \in\, P(\delta)_k } 
\sum\limits_{1\le i_1, \ldots , i_{m} \leqslant d  } 
a_{\mu_{j_1},\mu_{j_2}}^{(b_{i_1})} \ldots a_{\mu_{j_{m}},\mu_{j_{m+1}}}^{(b_{i_m})} a_{\mu_{j_{m+1}}, \mu_{j_{1}}}^{(b_{k})} 
b_{i_1}^* \ldots b_{i_{m}}^* \otimes b_{k}^*.
\end{align*}

(1) For $\lambda,\mu \in P(\delta)$ and 
$i \in \{ 1,\ldots,r \}$, 
\begin{align*} 
a_{\lambda,\mu}^{(\c{\beta}_i)} = \begin{cases} 
\langle \mu , \c{\beta}_i \rangle & \textrm{ if } \lambda = \mu,\\ 
0 & \textrm{ if }\lambda \neq \mu.
\end{cases}
\end{align*}
Since ${\d p}_m^{(\delta)} $ is $\g$-invariant, 
its image by $\ch \otimes 1$ 
belongs to $S(\h) \otimes \h$, and we have:   
\begin{align*} 
&\overline{\d p}_m^{(\delta)}=(\ch \otimes 1) (\d p_m^{(\delta)}) 
= \frac{1}{m!} \sum\limits_{k=1}^r  
\sum\limits_{\mu \in P(\delta) } 
\sum\limits_{1\le i_1, \ldots , i_{m} \leqslant r } 
\langle \mu, \c{\beta}_{i_1} \rangle   \ldots  \langle \mu, \c{\beta}_{i_m} \rangle  \langle \mu, \c{\beta}_k \rangle \, 
\s{\varpi_{i_1}}   \ldots \s{\varpi_{i_m}}  
 \otimes \s{\varpi_{k}}, &
\end{align*}
whence the statement. 

(2) The image of $\d p_m^{(\delta)}$ 
by the map $\hat{\sym}$ is the following 
element of $(U(\g)\otimes \g)^\g$: 
\begin{align*}
&\hat{\sym}({\d p}_m^{(\delta)}) &\\ 
& \quad = \frac{1}{m!}
\sum\limits_{k=1}^d  
\sum\limits_{1\le i_1, \ldots , i_{m} \leqslant d  } 
 \sum\limits_{\mu_{j_1}, \ldots , \mu_{j_{m+1}}\in P(\delta) 
 \atop \mu_{j_1}\in\, P(\delta)_k } 
a_{\mu_{j_1},\mu_{j_2}}^{(b_{i_1})} \ldots a_{\mu_{j_{m}},\mu_{j_{m+1}}}^{(b_{i_m})} a_{\mu_{j_{m+1}}, \mu_{j_{1}}}^{(b_{k})} 
b_{i_1}^* \ldots b_{i_{m}}^* \otimes b_{k}^*.
\end{align*}
Since $(\sym \otimes 1) ({\d p}_m^{(\delta)}) $ is $\g$-invariant, 
its image by $\hc \otimes 1$ 
belongs to $S(\h) \otimes \h$ and we have:  
\begin{align*} 
&\widehat{\d p}_m^{(\delta)}=
\frac{1}{m!}\sum_{k=1}^r  
\sum\limits_{1\le i_1, \ldots ,i_{m} \leqslant d} 
 \sum\limits_{\mu_{j_1}, \ldots , \mu_{j_{m}}\in P(\delta) 
 \atop \mu_{j_1} \in \, P(\delta)_k} 
  a_{\mu_{j_1},\mu_{j_2}}^{(b_{i_1})} \ldots 
   a_{\mu_{j_{m}},\mu_{j_{1}}}^{(b_{i_m})}  \langle \mu_{j_1}, \c{\beta}_k \rangle 
   \hc(b_{i_1}^* \ldots b_{i_{m}}^*) \otimes  \s{\varpi_{k}},&
\end{align*}
whence the statement. \qed
\end{proof}

Recall that the {\em crystal graph} $\mathscr{C}(\delta)$ of the integral dominant weight $\delta$ 
is a graph 
containing $\# P(\delta)$ vertices and whose arrows are labeled by the simple roots $\beta_i, i \in {1,\dots, r}$. Moreover an arrow \xymatrix{\delta_i \ar[r]^{\beta_j} & \delta_k} exists when $\delta_k = \delta_i - \beta_j$.

\begin{example} 
\label{ex:crystal_classical}
Following the notation of Appendices \ref{sec:rootTypeA} and~\ref{sec:rootTypeC}, 
we represent the crystal graph $\mathscr{C}(\delta)$ of $\delta=\varpi_1$ for $\g=\sl_{r+1}$ 
and $\g=\sp_{2r}$ in Figure~\ref{fig:crystal}.

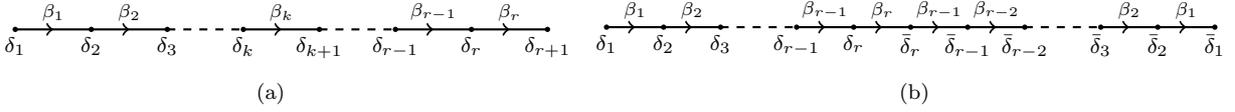
\begin{figure}[h]
\begin{subfigure}{.45\textwidth}
\begin{center}
\begin{tikzpicture}[scale=0.5]
 \draw[ thick,
        decoration={markings, mark=at position 0.5 with {\arrow{>}}},
        postaction={decorate}
        ]
        (0,0) -- (2,0)node(xline)[midway,above] {\scriptsize{$\beta_1$}} ;
     \draw[ thick,
        decoration={markings, mark=at position 0.5 with {\arrow{>}}},
        postaction={decorate}
        ]
        (2,0) -- (4,0)node(xline)[midway,above] {\scriptsize{$\beta_2$}};   
     \draw[ thick,
        decoration={markings, mark=at position 0.5 with {\arrow{>}}},
        postaction={decorate}
        ]
        (6,0) -- (8,0)node(xline)[midway,above] {\scriptsize{$\beta_k$}}; 
     \draw[ thick,
        decoration={markings, mark=at position 0.5 with {\arrow{>}}},
        postaction={decorate}
        ]
        (10,0) -- (12,0)node(xline)[midway,above] {\scriptsize{$\beta_{r-1}$}};
      \draw[ thick,
        decoration={markings, mark=at position 0.5 with {\arrow{>}}},
        postaction={decorate}
        ]
        (12,0) -- (14,0)node(xline)[midway,above] {\scriptsize{$\beta_r$}};
     \draw[thick, dashed] (4,0) -- (6,0);
      \draw[thick, dashed] (8,0) -- (10,0);
    \fill (0,0) circle (0.07)node(xline)[below] {$\delta_1$}; 
    \fill (2,0) circle (0.07)node(xline)[below] {$\delta_2$};
    \fill (4,0) circle (0.07)node(xline)[below] {$\delta_3$};
    \fill (6,0) circle (0.07)node(xline)[below] {$\delta_k$}; 
    \fill (8,0) circle (0.07)node(xline)[below] {$\delta_{k+1}$};
    \fill (10,0) circle (0.07)node(xline)[below] {$\delta_{r-1}$};
    \fill (12,0) circle (0.07)node(xline)[below] {$\delta_r$}; 
    \fill (14,0) circle (0.07)node(xline)[below] {$\delta_{r+1}$};
\end{tikzpicture}
 \caption{} \label{fig:crystalA} 
\end{center}
\end{subfigure}
\begin{subfigure}{.6\textwidth}
\begin{center}
\begin{tikzpicture}[scale=0.5]
 \draw[ thick,
        decoration={markings, mark=at position 0.5 with {\arrow{>}}},
        postaction={decorate}
        ]
        (0,0) -- (1.5,0)node(xline)[midway,above] {\scriptsize{$\beta_1$}} ;
     \draw[ thick,
        decoration={markings, mark=at position 0.5 with {\arrow{>}}},
        postaction={decorate}
        ]
        (1.5,0) -- (3,0)node(xline)[midway,above] {\scriptsize{$\beta_2$}};   
     \draw[ thick,
        decoration={markings, mark=at position 0.5 with {\arrow{>}}},
        postaction={decorate}
        ]
        (5,0) -- (6.5,0)node(xline)[midway,above] {\scriptsize{$\beta_{r-1}$}}; 
     \draw[ thick,
        decoration={markings, mark=at position 0.5 with {\arrow{>}}},
        postaction={decorate}
        ]
        (6.5,0) -- (8,0)node(xline)[midway,above] {\scriptsize{$\beta_{r}$}};
      \draw[ thick,
        decoration={markings, mark=at position 0.5 with {\arrow{>}}},
        postaction={decorate}
        ]
        (8,0) -- (9.5,0)node(xline)[midway,above] {\scriptsize{$\beta_{r-1}$}};
        \draw[ thick,
        decoration={markings, mark=at position 0.5 with {\arrow{>}}},
        postaction={decorate}
        ]
        (9.5,0) -- (11,0)node(xline)[midway,above] {\scriptsize{$\beta_{r-2}$}};
      \draw[ thick,
        decoration={markings, mark=at position 0.5 with {\arrow{>}}},
        postaction={decorate}
        ]
        (13,0) -- (14.5,0)node(xline)[midway,above] {\scriptsize{$\beta_{2}$}};
       \draw[ thick,
        decoration={markings, mark=at position 0.5 with {\arrow{>}}},
        postaction={decorate}
        ]
        (14.5,0) -- (16,0)node(xline)[midway,above] {\scriptsize{$\beta_{1}$}};
     \draw[thick, dashed] (3,0) -- (5,0);
      \draw[thick, dashed] (11,0) -- (13,0);
    \fill (0,0) circle (0.07)node(xline)[below] {\footnotesize{$\delta_1$}}; 
    \fill (1.5,0) circle (0.07)node(xline)[below] {\footnotesize{$\delta_2$}};
    \fill (3,0) circle (0.07)node(xline)[below] {\footnotesize{$\delta_3$}};
    \fill (5,0) circle (0.07)node(xline)[below] {\footnotesize{$\delta_{r-1}$}}; 
    \fill (6.5,0) circle (0.07)node(xline)[below] {\footnotesize{$\delta_{r}$}};
    \fill (8,0) circle (0.07)node(xline)[below] {\footnotesize{$\bar{\delta}_{r}$}};
    \fill (9.5,0) circle (0.07)node(xline)[below] {\footnotesize{$\bar{\delta}_{r-1}$}}; 
    \fill (11,0) circle (0.07)node(xline)[below] {\footnotesize{$\bar{\delta}_{r-2}$}};
    \fill (13,0) circle (0.07)node(xline)[below] {\footnotesize{$\bar{\delta}_{3}$}};
    \fill (14.5,0) circle (0.07)node(xline)[below] {\footnotesize{$\bar{\delta}_{2}$}}; 
    \fill (16,0) circle (0.07)node(xline)[below] {\footnotesize{$\bar{\delta}_{1}$}};
\end{tikzpicture}
 \label{fig:cryst-C}  \caption{}
\end{center}
\end{subfigure}
\caption{the Crystal graph $\mathscr{C}(\delta)$ of $\delta=\varpi_1$ for $\g=\mathfrak{sl}_{r+1}$(a) and $\g=\mathfrak{sp}_{2r}$(b).}
\label{fig:crystal}
\end{figure}
\end{example} 

Let $m \in \Z_{> 0}$ and  $(\mu,\nu) \in P(\delta)^2$. 
We denote by $\P_{m}(\mu,\nu)$ the set of $\sub{\mu}=(\mu^{(1)},\ldots,\mu^{(m+1)})$ in $P(\delta)$ 
such that $\mu^{(1)} = \mu$, $\mu^{(m+1)}= \nu$ 
and for all $i=1,\ldots,m$, $\mu^{(i)}-\mu^{(i+1)} \in \Delta \cup\{0\}$. 
The elements of $\P_{m}(\mu,\nu)$ are called the 
{\em paths of length $m$ starting at $\mu$ and ending at $\nu$}. 
When $\mu = \nu$, we write ${\P}_{m}(\mu)$ for ${\P}_{m}(\mu,\mu)$.

\begin{remark}  
\label{Rem:difference}
For $\g=\sl_{r+1}$ or $\g=\mathfrak{sp}_{2r}$ with $\delta=\varpi_1$ 
(cf.~Example \ref{ex:crystal_classical}), 
the difference of two different weights is always a root. 
So the condition $\mu^{(i)}-\mu^{(i+1)} \in \Delta \cup\{0\}$ 
is automatically satisfied. 
Note that, furthermore, in these cases each root $\alpha \in \Delta$ 
can be written as a difference of two weights. 
\end{remark}

Such paths have been considered in a more general situation in \cite{LLP10} 
in connection with Kashiwara's crystal basis theory \cite{Kas95}.
Recall that the {\em height} of a positive roots 
$\alpha$ is  $\he(\alpha) := \sum_{i=1}^r n_i,$
if $\alpha = \sum_{i=1}^r n_i \, \beta_i \in~\Delta_+$. 
For $\alpha \in - \Delta_+$, we define its height by 
$\he(\alpha):=-\he(-\alpha)$. 
We adopt the convention that $\he(0) :=0$.  
The height of a path $\sub{\mu}$ is defined by $\he(\sub{\mu})  :=  ( \he(\sub{\mu})_1, \ldots,  \he(\sub{\mu})_m) \in \Z^m$,
where $\he(\sub{\mu})_i := \he(\mu^{(i)}-\mu^{(i+1)})$.
For $\sub{\mu} \in \P_m(\mu)$,  
we have $\sum_{i=1}^m \he(\sub{\mu})_i =0$.

\begin{definition}[weighted path]   
\label{definition:weighted_path}
Let $\hat{\P}_m(\mu,\nu)$ be the set of pairs $(\sub{\mu},\sub{\alpha})$ 
where $\sub{\mu} \in \P_m(\mu,\nu)$ 
and $\sub{\alpha} := (\alpha^{(1)},\ldots,\alpha^{(m)})$ is a sequence of roots 
satisfying for any $ j \in \{1,\ldots,m \}$ the following conditions:   
\begin{enumerate} 
\item if $\he(\sub{\mu})_j \not= 0$ then $\alpha^{(j)} = \mu^{(j)} - \mu^{(j+1)}$,  
\item if $\he(\sub{\mu})_j = 0$ then $\alpha^{(j)} \in \Pi$ 
and $\langle \mu^{(j)}, \c{\alpha}^{(j)} \rangle \not= 0$. 
\end{enumerate} 
We call the elements of $\hat{\P}_{m}(\mu,\nu)$ the 
{\em weighted paths of length $m$ starting at $\mu$ and ending at $\nu$}. 
When $\mu = \nu$, we write $\hat{\P}_{m}(\mu)$ for $\hat{\P}_{m}(\mu,\mu)$.  
We denote by ${\bf 1}_\mu$ the trivial path $(\mu,\varnothing)$. 
It has by convention length $0$. 
\end{definition}

We represent a weighted path $(\sub{\mu},\sub{\alpha}) \in  \hat{\P}_m(\mu,\nu)$ 
by a colored and oriented graph as follows.   
The vertices are the weights $\mu^{(1)},\ldots,\mu^{(m+1)}$ and the 
oriented arrow 
from $\mu^{(j)}$ to $\mu^{(j+1)}$ for $j \in \{1,\ldots, m\}$ is labeled by the root $\alpha^{(j)}$.    
\begin{example} 
\noindent \label{Ex:weighted-path}
\begin{enumerate}
\item 
 Assume that $\g = \sl_6$ and $\delta = \varpi_1$. We represent in Figure \ref{fig:weighted-path} the weighted path 
 $(\sub{\mu},\sub{\alpha})$ of length 6 starting and ending at $\delta_3$
with $\sub{\mu} 
 = (\delta_3, \delta_4, \delta_6, \delta_5,\delta_5,\delta_5,\delta_3)$
and $\sub{\alpha} = (\beta_3, \beta_4 + \beta_5, -\beta_5, \beta_5, \beta_4, -\beta_3 - \beta_4)$. 
Then $ \he(\sub{\mu}) = (1, 2, -1, 0, 0, -2)$ and  $\sum_{i=1}^6 \he(\sub{\mu})_i =0.$  
\begin{figure}[h]
\begin{subfigure}{.4\textwidth}
\begin{center}
\begin{tikzpicture}[scale=0.6]
 \draw[thick, draw=blue, decoration={markings, mark=at position 0.5 with {\arrow{>}}},
        postaction={decorate}  ]
        (6,-1.5) to (9,-1.5); \node[text=blue] at (7.5,-1.1){\scriptsize{$\alpha^{(1)}$}} ;
 \draw[thick, draw=blue, decoration={markings, mark=at position 0.5 with {\arrow{>}}},
        postaction={decorate}  ]
        (9,-1.5) to (15,-1.5); \node[text=blue] at (12,-1.1){\scriptsize{$\alpha^{(2)}$}} ;
 \draw[thick, draw=blue, decoration={markings, mark=at position 0.5 with {\arrow{>}}},
        postaction={decorate}  ]
        (15,-1.65) to (12,-1.65); \node[text=blue] at (13.5,-2.1){\scriptsize{$\alpha^{(3)}$}} ;
 \draw[thick, draw=blue, decoration={markings, mark=at position 0.5 with {\arrow{>}}},
        postaction={decorate}  ]
        (12,-1.65) to (6,-1.65); \node[text=blue] at (9,-2.1){\scriptsize{$\alpha^{(6)}$}} ;
the loops!
 \draw[thick, draw=blue] (12,-1.65) to [out=350,in=50] (12.8,-2.45);
   \draw[thick, draw=blue, decoration={markings, mark=at position 0.625 with {\arrow{>}}},
        postaction={decorate}] (12.8,-2.45) to [out=220,in=280] (12,-1.65);
         \node[text=blue] at (13.1,-2.65){\footnotesize{$\beta_5$}} ;
  \draw[thick, draw=blue, decoration={markings, mark=at position 0.625 with {\arrow{<}}},
        postaction={decorate}] (12,-1.65) to [out=190,in=130] (11.2,-2.45);
   \draw[thick, draw=blue] (11.2,-2.45) to [out=320,in=260] (12,-1.65);
         \node[text=blue] at (11,-2.65){\footnotesize{$\beta_4$}} ;
  \fill[draw=blue] (6,-1.5) circle (0.05);
  \fill[draw=blue] (9,-1.5) circle (0.05);
  \fill[draw=blue] (15,-1.5) circle (0.05);
  \fill[draw=blue] (12,-1.65) circle (0.05);
  
      \draw[ thick,
        decoration={markings, mark=at position 0.5 with {\arrow{>}}},
        postaction={decorate}
        ]
        (6,0) -- (9,0);
        \draw[ thick,
        decoration={markings, mark=at position 0.5 with {\arrow{>}}},
        postaction={decorate}
        ]
        (9,0) -- (12,0);
      \draw[ thick,
        decoration={markings, mark=at position 0.5 with {\arrow{>}}},
        postaction={decorate}
        ]
        (12,0) -- (15,0);
    \fill (6,0) circle (0.07)node(xline)[below] {$\delta_3$};
    \fill (9,0) circle (0.07)node(xline)[below] {$\delta_{4}$}; 
    \fill (12,0) circle (0.07)node(xline)[below] {${\delta}_{5}$};
    \fill (15,0) circle (0.07)node(xline)[below] {${\delta}_{6}$}; 
\end{tikzpicture}
 \caption{An example of weighted path in $\mathfrak{sl}_{6}$. }
  \label{fig:weighted-path}
\end{center}
\end{subfigure}
\begin{subfigure}{.6\textwidth}
\begin{center}
\begin{tikzpicture} [scale=0.6]
 \draw[thick, draw=blue, decoration={markings, mark=at position 0.5 with {\arrow{>}}},
        postaction={decorate}  ]
        (0,-1.5) to (6,-1.5); \node[text=blue] at (3,-1.15){\scriptsize{$\alpha^{(1)}$}} ;
 \draw[thick, draw=blue, decoration={markings, mark=at position 0.5 with {\arrow{>}}},
        postaction={decorate}  ]
        (6,-1.5) to (8,-1.5); \node[text=blue] at (7,-1.15){\scriptsize{$\alpha^{(2)}$}} ;
 \draw[thick, draw=blue, decoration={markings, mark=at position 0.5 with {\arrow{>}}},
        postaction={decorate}  ]
        (8,-1.5) to (14,-1.5); \node[text=blue] at (11,-1.15){\scriptsize{$\alpha^{(4)}$}} ;
 \draw[thick, draw=blue, decoration={markings, mark=at position 0.5 with {\arrow{>}}},
        postaction={decorate}  ]
        (14,-1.65) to (10,-1.65); \node[text=blue] at (12,-2.1){\scriptsize{$\alpha^{(6)}$}} ;
 \draw[thick, draw=blue, decoration={markings, mark=at position 0.5 with {\arrow{>}}},
        postaction={decorate}  ]
        (10,-1.65) to (0,-1.65); \node[text=blue] at (5,-2.1){\scriptsize{$\alpha^{(7)}$}} ;
\draw[ thick, draw=blue,
        decoration={markings, mark=at position 0.625 with {\arrow{<}}},
        postaction={decorate}
        ]
        (14.5,-1.5) ellipse (0.5cm and 0.3cm);\node[text=blue] at (14.5,-0.9){\footnotesize{$\beta_1$}} ;
\draw[thick, draw=blue] (8,-1.5) to [out=155,in=180] (8,-0.7);
\draw[thick, draw=blue, decoration={markings, mark=at position 0.625 with {\arrow{>}}},
        postaction={decorate}] (8,-0.7) to [out=360,in=25] (8,-1.5);
        \node[text=blue] at (8.6,-0.9){\footnotesize{$\beta_3$}} ;
  \fill[draw=blue] (0,-1.5) circle (0.05);
  \fill[draw=blue] (6,-1.5) circle (0.05);
  \fill[draw=blue] (8,-1.5) circle (0.05);
  \fill[draw=blue] (14,-1.5) circle (0.05);
  \fill[draw=blue] (10,-1.65) circle (0.05);
  
   \draw[ thick,
        decoration={markings, mark=at position 0.5 with {\arrow{>}}},
        postaction={decorate}
        ]
        (0,0) -- (2,0); 
     \draw[ thick,
        decoration={markings, mark=at position 0.5 with {\arrow{>}}},
        postaction={decorate}
        ]
        (2,0) -- (4,0);
      \draw[ thick,
        decoration={markings, mark=at position 0.5 with {\arrow{>}}},
        postaction={decorate}
        ]
        (4,0) -- (6,0);
        \draw[ thick,
        decoration={markings, mark=at position 0.5 with {\arrow{>}}},
        postaction={decorate}
        ]
        (6,0) -- (8,0);
      \draw[ thick,
        decoration={markings, mark=at position 0.5 with {\arrow{>}}},
        postaction={decorate}
        ]
        (8,0) -- (10,0);
       \draw[ thick,
        decoration={markings, mark=at position 0.5 with {\arrow{>}}},
        postaction={decorate}
        ]
        (10,0) -- (12,0);
      \draw[ thick,
        decoration={markings, mark=at position 0.5 with {\arrow{>}}},
        postaction={decorate}
        ]
        (12,0) -- (14,0);
    \fill (0,0) circle (0.07)node(xline)[below] {$\delta_1$}; 
    \fill (2,0) circle (0.07)node(xline)[below] {$\delta_2$};
    \fill (4,0) circle (0.07)node(xline)[below] {$\delta_3$};
    \fill (6,0) circle (0.07)node(xline)[below] {$\delta_{4}$}; 
    \fill (8,0) circle (0.07)node(xline)[below] {$\bar{\delta}_{4}$};
    \fill (10,0) circle (0.07)node(xline)[below] {$\bar{\delta}_{3}$}; 
    \fill (12,0) circle (0.07)node(xline)[below] {$\bar{\delta}_{2}$};
    \fill (14,0) circle (0.07)node(xline)[below] {$\bar{\delta}_{1}$};
\end{tikzpicture}
 \caption{An example of weighted path in $\mathfrak{sp}_{8}$. }
  \label{fig:weighted-pathC}
\end{center}
\end{subfigure}
\caption{Examples of weighted paths}
\end{figure}
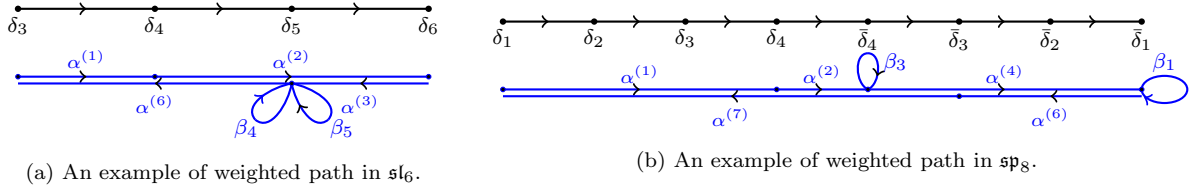

\item Assume that $\g = \sp_{8}$ and $\delta=\varpi_1$.
We represent in Figure \ref{fig:weighted-pathC} 
the weighted path $(\sub{\mu},\sub{\alpha})$ of length $7$ 
starting and ending at $\delta_1$ 
with $\sub{\mu} = (\delta_1,\delta_4,\overline{\delta}_4,\overline{\delta}_4,\overline{\delta}_1, \overline{\delta}_1,
\overline{\delta}_3, \delta_1)$ and $\sub{\alpha} = (\eps_1-\eps_4,2\eps_4, \beta_3, \eps_1-\eps_4, \beta_1, \eps_3-\eps_1, -\eps_1-\eps_3)$.
Then $\he(\sub{\mu})=(3,1,0,3,0,-2,-5)$.
\end{enumerate}
\end{example}
For $\beta \in \Pi$, we write $\varpi_\beta$ 
the fundamental weight corresponding to $\beta$. 
Thus $\varpi_{\beta_i} = \varpi_i$, $i=1,\ldots,r$,  
but it will be convenient to have both notations.  
Pick $\mu,\nu \in P(\delta)$. 
Let $m\in\Z_{> 0}$ and 
$(\sub{\mu},\sub{\alpha}) \in \hat{\P}_m(\mu,\nu)$.  
For $j \in \{1,\ldots,m \}$, the element 
$b_{(\sub{\mu},\sub{\alpha}),j}$ 
of $\B$ is defined by: 
\begin{enumerate}  
\item if $\he(\sub{\mu})_j \not= 0$, 
set $b_{(\sub{\mu},\sub{\alpha}),j} := c_{\alpha^{(j)}} e_{-\alpha^{(j)}}$ so that 
$b_{(\sub{\mu},\sub{\alpha}),j}^* := e_{\alpha^{(j)}}$, 
\item if $\he(\sub{\mu})_j = 0$, set 
$b_{(\sub{\mu},\sub{\alpha}),j} := \c{\alpha}^{(j)}$ 
so that $b_{(\sub{\mu},\sub{\alpha}),j}^*=\s{\varpi_{\alpha^{(j)}}}$. 
\end{enumerate}
The element $b_{\sub{\mu},\sub{\alpha}}^*$ of $U(\g)$ is defined by 
\begin{align}\label{eq:vector}
b_{\sub{\mu},\sub{\alpha}}^* := 
a_{\sub{\mu},\sub{\alpha}}
b_{(\sub{\mu},\sub{\alpha}),1}^* \ldots 
b_{(\sub{\mu},\sub{\alpha}),m}^*,
\end{align} 
where 
$a_{\sub{\mu},\sub{\alpha}} := 
a^{(b_{(\sub{\mu},\sub{\alpha}),m})}_{\mu^{(m+1)},\mu^{(m)}}
\ldots 
a^{(b_{(\sub{\mu},\sub{\alpha}),1})}_{\mu^{(2)},\mu^{(1)}}$. 
In the case where $\mu=\nu$, we observe that 
$b_{\sub{\mu},\sub{\alpha}}^*$ belongs to $U(\g)^\h$. In this case,
we define the {\em weight} of $(\sub{\mu},\sub{\alpha})$ 
to be the complex number, 
\begin{align} \label{eq:weight}
\wt(\sub{\mu},\sub{\alpha})
:=(\ev_\rho \circ \hc)( {b}_{\sub{\mu},\sub{\alpha}}^*).
\end{align}
We adopt the convention that $\wt({\bf 1}_\mu) = 1$. 

For example, for the weighted path of Example~\ref{Ex:weighted-path}(1) we have  
$$ 
 a_{\sub{\mu},\sub{\alpha}} =
 \langle \delta_5,\c{\beta}_4\rangle \langle \delta_5,\c{\beta}_5 \rangle = -1 \, \text{ and } \,
 b_{\sub{\mu},\sub{\alpha}}^* = - e_{\eps_3-\eps_4}e_{\eps_4-\eps_6}
 e_{\eps_6-\eps_5}\c{\varpi_5} \c{\varpi_4}e_{\eps_5-\eps_3}, 
$$
and for the weighted path of Example~\ref{Ex:weighted-path}(2) we have 
$$ a_{\sub{\mu},\sub{\alpha}} = \langle \overline{\delta}_1, \c{\beta}_1 \rangle \langle \overline{\delta}_4, \c{\beta}_3 \rangle = -2
\, \text{ and } \, 
b_{\sub{\mu},\sub{\alpha}}^* = -2 e_{\eps_1-\eps_4} e_{2\eps_4}\c{\varpi_3}e_{\eps_1-\eps_4} \c{\varpi_1}e_{\eps_3-\eps_1}e_{-\eps_1-\eps_3}. $$

\begin{lemma} \label{Lem:formulas} 
Let $m\in \Z_{> 0}$ 
and $k \in \{1,\ldots, r\}$. 
We have: 
$$
\overline{{\rm d} p}_{m,k} =  \sum_{\mu \in P(\delta)_k} \, \sum_{(\sub{\mu},\sub{\alpha}) 
\in \hat{\P}_m(\mu)\atop \he(\sub{\mu}) = \sub{0} }  \langle \mu, \c{\beta}_k \rangle 
b_{\sub{\mu},\sub{\alpha}}^* \quad \text{ and } \quad 
\widehat{{\rm d}  p}_{m,k} 
=   \sum_{\mu \in P(\delta)_k} \, \sum_{(\sub{\mu},\sub{\alpha}) \in \hat{\P}_m(\mu)}  \langle \mu, \c{\beta}_k \rangle 
\hc({b}_{\sub{\mu},\sub{\alpha}}^*).$$ 	
Therefore, we have: 
\begin{align*}
\ev_\rho(\overline{{\rm d}  p}_{m,k} ) 
& =  \sum_{\mu \in P(\delta)_k} \, \sum_{(\sub{\mu},\sub{\alpha}) 
\in \hat{\P}_m(\mu)\atop \he(\sub{\mu}) = \sub{0} }  
\wt(\sub{\mu},\sub{\alpha}) \langle \mu, \c{\beta}_k \rangle, \\
\ev_\rho( \widehat{{\rm d}  p}_{m,k}) 
&=  \sum_{\mu \in P(\delta)_k} \, \sum_{(\sub{\mu},\sub{\alpha}) \in \hat{\P}_m(\mu)}  
\wt(\sub{\mu},\sub{\alpha}) \langle \mu, \c{\beta}_k \rangle. 
\end{align*}	
\end{lemma}

\begin{proof} 
Recall that 
for $\alpha \in \Delta$ and $\mu \in P(\delta)$, 
$$\pi_\delta(e_\alpha) v_\mu \in V(\delta)_{\mu+\alpha} = \C
v_{\mu+\alpha},$$
with the convention that $v_{\mu+\alpha}=0$ if 
$\mu +\alpha \not\in P(\delta)$. 
The equality  $V(\delta)_{\mu+\alpha} = \C
v_{\mu+\alpha}$ holds because of our assumption 
that all weight spaces have dimension one. 
Therefore, 
$$a_{\mu,\nu}^{(e_{\alpha})} \not= 0 
\; \Longrightarrow \; \big(\mu =\nu+\alpha \quad \text{and} 
\quad \nu+\alpha \in P(\delta)\big).$$ 
For $i\in\{1,\ldots,r\}$ and $\mu\in P(\delta)$, 
$$\pi_\delta(\c{\beta}_i) v_ \mu = \langle \mu,\c{\beta}_i \rangle 
v_\mu.$$
As a result,  
$$a_{\mu,\nu}^{(\c{\beta}_i)} \not= 0 
\iff \big( \nu =\mu  \quad \text{and} 
\quad \langle \mu,\c{\beta}_i \rangle \not= 0 \big).$$
\\
According to Lemma \ref{lem:dec3}, 
we have: 
\begin{align*} 
\widehat{\d p}_{m,k} = 
 \sum\limits_{(\mu_{j_1}, \ldots , \mu_{j_{m}}) \atop  \in \, P(\delta)^{m} } 
\sum\limits_{1\le i_1, \ldots , i_{m} \leqslant r }
  a_{\mu_{j_{1}},\mu_{j_2}}^{(b_{i_1})} \ldots 
   a_{\mu_{j_{m}},\mu_{j_{1}}}^{(b_{i_m})} \langle \mu_{j_1}, \c{\beta}_k \rangle
   \hc(b_{i_1}^* \ldots b_{i_{m}}^*), 
\end{align*} 
and it is enough to sum over 
$\mu_{\sub{j}}:=(\mu_{j_1}, \ldots , \mu_{j_{m}})   \in P(\delta)^{m}$ 
and $\sub{i}:=(i_1,\ldots,i_m) \in \{1,\ldots,r\}^m$ 
such that 
$$a_{\mu_{j_1},\mu_{j_2}}^{(b_{i_1})} \ldots 
 a_{\mu_{j_{m}},\mu_{j_{1}}}^{(b_{i_m})} \not=0 \quad 
 \text{ and }
 \quad \langle \mu_{j_1}, \c{\beta}_k \rangle \not= 0.$$
Fix such $(\sub{i},\mu_{\sub{j}})$ and set for $ j \in \{1,\ldots,m\}$, 
\begin{align*} 
&
\mu^{(1)}  := \mu_{j_1} , \, \mu^{(2)}  := \mu_{j_m} , \, \mu^{(3)}  := \mu_{j_{m-1}} \quad \ldots, \quad \mu^{(m)} := \mu_{j_2}, 
\quad 
\mu^{(m+1)}:=\mu_{j_1}, & \\
& \alpha^{(j)} := 
\begin{cases} \alpha & \text{if} \quad b_{i_j}= c_{\alpha}e_{-\alpha}, 
\quad \alpha \in \Delta, \\ 
 \beta & \text{if}\quad b_{i_j} = \c{\beta},\quad \beta \in \Pi, 
 \end{cases}& 
\end{align*}
and 
$\mu := \mu_{j_1}.$
Then $\mu \in P(\delta)_k$ and 
$(\sub{\mu},\sub{\alpha}) \in \hat{\P}_m(\mu)$. 
Moreover, following~\eqref{eq:vector}, 
we get: 
\begin{align*}  
b_{(\sub{\mu},\sub{\alpha}),j} = b_{i_{ m-j+1}} \quad \text{ for }\quad  j=1,\ldots,m,\\
a_{\sub{\mu},\sub{\alpha}} = a_{\mu^{(m+1)},\mu^{(m)}}^{(b_{(\sub{\mu}, \sub{\alpha}),m})} \ldots 
  a_{\mu^{(3)},\mu^{(2)}}^{(b_{(\sub{\mu}, \sub{\alpha}),2})}  a_{\mu^{(2)},\mu^{(1)}}^{(b_{(\sub{\mu}, \sub{\alpha}),1})} \quad \text{ and }\quad 
b_{\sub{\mu},\sub{\alpha}}^* & = a_{\sub{\mu},\sub{\alpha}} b_{(\sub{\mu}, \sub{\alpha}),1}^* \ldots b_{(\sub{\mu}, \sub{\alpha}),m-1}^* 
b_{(\sub{\mu}, \sub{\alpha}),m}^*, 
\end{align*}
whence the expected formula for $\widehat{\d p}_{m,k}$:
$$\widehat{{\rm d}  p}_{m,k}  
=   \sum_{\mu \in P(\delta)_k} \, \sum_{(\sub{\mu},\sub{\alpha}) \in \hat{\P}_m(\mu)}  \langle \mu, \c{\beta}_k \rangle 
\hc({b}_{\sub{\mu},\sub{\alpha}}^*).$$
Then the formula for $\ev_\rho(\widehat{\d p}_{m,k} )$ is obvious 
by~\eqref{eq:weight}.

Let us now turn to $\overline{\d p}_{m,k}^{(\delta)} $ 
and $\ev_\rho(\overline{\d p}_{m,k} )$. 
By Lemma \ref{lem:dec3}, we have:
\begin{align*} 
\overline{\d p}_{m,k}^{(\delta)} 
=
\sum\limits_{\mu \in P(\delta)_k } 
\sum\limits_{1\le i_1, \ldots , i_{m} \leqslant r } 
\langle \mu, \c{\beta}_{i_1} \rangle   \ldots  \langle \mu, \c{\beta}_{i_m} \rangle  \langle \mu, \c{\beta}_k \rangle\, 
\s{\varpi_{i_1}}   \ldots \s{\varpi_{i_m}}  . 
\end{align*} 
To each $\sub{i}:=(i_1,\ldots,i_m) \in \{1,\ldots,r\}^m$ 
such that 
$\langle \mu, \c{\beta}_{i_1} \rangle   \ldots  \langle \mu, \c{\beta}_{i_m} \rangle  \not=0$
we attach the weighted paths 
$(\sub{\mu},\sub{\alpha}) \in \hat{\P}_m(\mu)$ 
with 
$$\mu^{(1)}=\cdots=\mu^{(m)}=\mu 
\quad \text{ and }
\quad \alpha^{(j)} = \c{\beta}_{i_j} \quad \text{for} 
\quad  j=1,\ldots,m.$$ 
Since all weighted paths $(\sub{\mu},\sub{\alpha}) \in \hat{\P}_m(\mu)$ 
with $\he(\sub{\mu})=\sub{0}$ are of this form, we get 
the desired statement. Indeed, for such paths, 
$b_{\sub{\mu},\sub{\alpha}}^* \in U(\h)=S(\h)$ 
thus $\hc (b_{\sub{\mu},\sub{\alpha}}^*)=b_{\sub{\mu},\sub{\alpha}}^*$.

Note that for paths $(\sub{\mu},\sub{\alpha}) \in \hat{\P}_m(\mu)$ 
such that $\he(\sub{\mu}) = \sub{0}$, we have 
$$\wt(\sub{\mu},\sub{\alpha}) = \langle \mu, \c{\alpha}^{(1)} \rangle 
\ldots \langle \mu, \c{\alpha}^{(m)} \rangle 
\langle \rho, \s{\varpi_{\alpha^{(1)}}} \rangle 
\ldots \langle \rho,\s{\varpi_{\alpha^{(m)}}} \rangle,$$
and the $\alpha^{(j)}$'s run through the set 
$\Pi_\mu:=\{\beta \in \Pi \; |\; \langle \mu ,\c{\beta} \rangle \not=0 \}.$  
Hence we get, 
\begin{align} \label{eq:formulas}
\ev_\rho(\overline{\d p}_{m,k} )
=   \sum_{\mu \in P(\delta)_k} 
\sum_{ \sub{\alpha}\in (\Pi_\mu)^m }  
\langle \mu, \c{\alpha}^{(1)} \rangle 
\ldots \langle \mu, \c{\alpha}^{(m)} \rangle 
\langle \rho,  \s{\varpi_{\alpha^{(1)}}} \rangle 
\ldots \langle \rho, \s{\varpi_{\alpha^{(m)}}} \rangle \langle \mu, \c{\beta}_k \rangle .
\end{align} \qed
\end{proof}

We now explore some useful operations on the set of weighted paths. 
Let $(m,n) \in (\Z_{> 0})^2$, $(\lambda,\mu,\nu) \in P(\delta)^3$,   
$(\sub{\mu} ,\sub{\alpha}) \in \hat{\P}_{m}(\lambda,\mu)$, and 
$(\sub{\mu}',\sub{\alpha}') \in \hat{\P}_{n}(\mu,\nu)$. 
We define the sequences 
$\sub{\mu} \star \sub{\mu}'$ and $\sub{\alpha} \star \sub{\alpha}'$ by: 
\begin{align*}
\sub{\mu} \star \sub{\mu}' : = (\mu^{(1)},\ldots,\mu^{(m+1)} = 
\mu'^{(1)},\ldots,\mu'^{(n+1)}) , 
\quad \sub{\alpha} \star \sub{\alpha}' : = (\alpha^{(1)},\ldots,\alpha^{(m)} , \alpha'^{(1)},\ldots,\alpha'^{(n)}).
\end{align*}
The pair $(\sub{\mu} ,\sub{\alpha}) \star (\sub{\mu}',\sub{\alpha}') 
:=(\sub{\mu} \star \sub{\mu}',\sub{\alpha} \star \sub{\alpha}')$ 
defines a weighted path of $\hat{\P}_{m+n}(\lambda,\nu)$ that  
we call the {\em concatenation of 
the paths $(\sub{\mu},\sub{\alpha})$ and $(\sub{\mu}',\sub{\alpha}')$}. 

\begin{lemma} 
Let $(\sub{\mu} ,\sub{\alpha}) \in \hat{\P}_{m}(\lambda,\mu)$ 
and $(\sub{\mu}',\sub{\alpha}')\in \hat{\P}_{n}(\mu,\nu)$. 
\begin{enumerate}
\item We have $b_{(\sub{\mu} ,\sub{\alpha}) \star (\sub{\mu}',\sub{\alpha}')}^* 
= b_{\sub{\mu} ,\sub{\alpha}}^* \, b_{\sub{\mu}' ,\sub{\alpha}'}^*.$
\item If $\lambda = \mu =\nu$ then 
$\hc({b}_{(\sub{\mu} ,\sub{\alpha}) \star (\sub{\mu}',\sub{\alpha}')}^*)
= \hc({b}_{\sub{\mu} ,\sub{\alpha}}^*) \hc({b}_{\sub{\mu}' ,\sub{\alpha}'}^*)$ 
and $\wt(\sub{\mu} \star \sub{\mu}',\sub{\alpha} \star \sub{\alpha}') 
= \wt (\sub{\mu} ,\sub{\alpha}) \, \wt (\sub{\mu}' ,\sub{\alpha}').$
\end{enumerate}
\end{lemma}
\begin{proof} 
Part (1) is clear. The restrictions to $U(\g)^{\h}$ of $\hc$ and of $\ev_\rho \otimes \hc$ are morphisms of algebras,
and $b_{\sub{\mu},\sub{\alpha}}^*, b_{\sub{\mu}',\sub{\alpha}'}^*$ belongs to $U(\g)^\h$ if $\lambda=\mu=\nu$. Hence Part(2) follows. \qed
\end{proof}

Let $m \in \Z_{> 1}$, $(\mu,\nu) \in P(\delta)^2$ 
and $i \in \{2,\ldots,m\}$. 
Let $(\sub{\mu},\sub{\alpha}) \in \hat{\P}_m(\lambda,\mu)$ and assume that 
$(\sub{\mu},\sub{\alpha})$ is a weighted path with no ramification {\protect\footnotemark}. We 
define the weighted path $(\sub{\mu},\sub{\alpha})^{\# i}
:=(\sub{\mu}^{\# i},\sub{\alpha}^{\# i})$ which obtained from $(\sub{\mu},\sub{\alpha})$ by 
^^ ^^ cutting the vertex $\mu^{(i)}$'' as follows:

\footnotetext{By path with no ramification, it means that any vertex $\mu^{(i)}$ is related by an arrow to at most two other vertices.}

\begin{enumerate} 
\item 
If $\he(\sub{\mu})_{i-1}\not=0$, 
$\he(\sub{\mu})_i \not=0$ and $\alpha^{(i-1)} 
+\alpha^{(i)} \not=0$, then we set: 
\begin{align*}
\sub{\mu}^{\# i} := (\mu^{(1)},\ldots, \mu^{(i-1)},
\mu^{(i+1)},\ldots,\mu^{(m+1)}), \quad  
\sub{\alpha}^{\# i}  :=  (\alpha^{(1)},\ldots,\alpha^{(i-1)} 
+\alpha^{(i)},\ldots, \alpha^{(m)}).
\end{align*} 
In this case, $(\sub{\mu},\sub{\alpha})^{\# i} \in \hat{\P}_{m-1}(\lambda,\mu)$.  	
\item If $\he(\sub{\mu})_{i-1}\not=0$, 
$\he(\sub{\mu})_i \not=0$ 
and $\alpha^{(i-1)}+\alpha^{(i)}=0$, then we set: 
\begin{align*}\sub{\mu}^{\# i} :=  (\mu^{(1)},\ldots, 
\mu^{(i-1)}=\mu^{(i+1)},\ldots,\mu^{(m+1)}), \quad 
\sub{\alpha}^{\# i}:=  (\alpha^{(1)},\ldots,\alpha^{(i-2)} 
, \alpha^{(i+1)},\ldots, \alpha^{(m)}).
\end{align*} 
We have $(\sub{\mu},\sub{\alpha})^{\# i} \in \hat{\P}_{m-2}(\lambda,\mu)$. 
\item If $\he(\sub{\mu})_i =0$, then we set: 
\begin{align*}
\sub{\mu}^{\# i} :=  (\mu^{(1)},\ldots, \mu^{(i-1)},
\mu^{(i+1)},\ldots,\mu^{(m+1)}), \quad 
\sub{\alpha}^{\# i} :=  (\alpha^{(1)},\ldots,\alpha^{(i-1)}, 
\alpha^{(i+1)},\ldots, \alpha^{(m)}).
\end{align*} 
Thus $(\sub{\mu},\sub{\alpha})^{\# i} \in
\hat{\P}_{m-1}(\lambda,\mu)$.  
\end{enumerate}
Our definition cannot be applied for a vertex $\mu^{(i)}$ 
such that $\he(\sub{\mu})_{i-1}=0$ and $\he(\sub{\mu})_i \not=0$. 
So we cannot cut such a  
vertex $\mu^{(i)}$. 
In such situation $\mu^{(i-1)}=\mu^{(i)}$, so we can 
cut the vertex $\mu^{(i-1)}$ instead. 
We illustrate in Figure \ref{fig:delate} the operation 
of ^^ ^^ cutting the vertex $\mu^{(3)}$'' 
in the three above situations.  
\begin{figure}[h]
\begin{center}
\begin{tikzpicture} [scale=0.8]
\node at (-1,0) {(1)};
 \draw[thick, decoration={markings, mark=at position 0.5 with {\arrow{>}}},
        postaction={decorate}  ]
        (0,0) to (2,0); \node at (1,0.4){{\footnotesize $\alpha^{(1)}$}} ;
 \draw[thick, decoration={markings, mark=at position 0.5 with {\arrow{>}}},
        postaction={decorate}  ]
        (2,0) to (6,0); \node at (4,0.4){{\footnotesize $\alpha^{(2)}$}} ;
 \draw[thick, decoration={markings, mark=at position 0.5 with {\arrow{>}}},
        postaction={decorate}  ]
        (6,-0.15) to (4,-0.15); \node at (5,-0.55){{\footnotesize $\alpha^{(3)}$}} ;
\draw[thick, decoration={markings, mark=at position 0.5 with {\arrow{>}}},
        postaction={decorate}  ]
        (4,-0.15) to (0,-0.15); \node at (2,-0.55){{\footnotesize $\alpha^{(4)}$}} ;
\fill (0,0) circle (0.07)node(xline)[above] {{\small $\mu^{(1)}$}};     
\fill (2,0) circle (0.07)node(xline)[above] {{\small $\mu^{(2)}$}};
\fill (6,0) circle (0.07)node(xline)[above] {{\small $\mu^{(3)}$}};
\fill (4,-0.15) circle (0.07)node(xline)[below] {{\small $\mu^{(4)}$}}; 

\draw[ ->, dashed, draw=blue]
        (7,-0.15) to (8,-0.15);\node[text=blue] at (7.5,0.25){$\# 3$} ;

 \draw[thick, decoration={markings, mark=at position 0.5 with {\arrow{>}}},
        postaction={decorate}  ]
        (9,0) to (11,0); \node at (10,0.4){{\footnotesize $\alpha^{(1)}$}} ;
 \draw[thick, decoration={markings, mark=at position 0.5 with {\arrow{>}}},
        postaction={decorate}  ]
        (11,0) to (13,0); \node at (12.3,0.4){{\footnotesize $\alpha^{(2)}+\alpha^{(3)}$}} ;
\draw[thick, decoration={markings, mark=at position 0.5 with {\arrow{>}}},
        postaction={decorate}  ]
        (13,-0.15) to (9,-0.15); \node at (11,-0.55){{\footnotesize $\alpha^{(4)}$}} ;
\fill (9,0) circle (0.07)node(xline)[above] {{\small $\mu^{(1)}$}};     
\fill (11,0) circle (0.07)node(xline)[above] {{\small $\mu^{(2)}$}};
\fill (13,0) circle (0.07); \node at (13.5,0) {{\small $\mu^{(4)}$}};
 \draw[thick, decoration={markings, mark=at position 0.5 with {\arrow{>}}},
        postaction={decorate}  ]
        (0,-1.85) to (4,-1.85); \node at (2,-1.55){{\footnotesize $\alpha^{(1)}$}} ;
 \draw[thick, decoration={markings, mark=at position 0.5 with {\arrow{>}}},
        postaction={decorate}  ]
        (4,-1.85) to (6,-1.85); \node at (5,-1.55){{\footnotesize $\alpha^{(2)}$}} ;
 \draw[thick, decoration={markings, mark=at position 0.5 with {\arrow{>}}},
        postaction={decorate}  ]
        (6,-2) to (2,-2); \node at (4,-2.3){{\footnotesize $\alpha^{(3)}$}} ;
 \draw[thick, decoration={markings, mark=at position 0.5 with {\arrow{>}}},
        postaction={decorate}  ]
        (2,-2) to (0,-2); \node at (1,-2.3){{\footnotesize $\alpha^{(4)}$}} ;
\fill (0,-1.85) circle (0.07)node(xline)[above] {{\small $\mu^{(1)}$}};     
\fill (4,-1.85) circle (0.07)node(xline)[above] {{\small $\mu^{(2)}$}};
\fill (6,-1.85) circle (0.07)node(xline)[above] {{\small $\mu^{(3)}$}};
\fill (2,-2) circle (0.07)node(xline)[below] {{\small $\mu^{(4)}$}}; 

\draw[ ->, dashed, draw=blue]
        (7,-2) to (8,-2);\node[text=blue] at (7.5,-1.6){$\# 3$} ;

 \draw[thick, decoration={markings, mark=at position 0.5 with {\arrow{>}}},
        postaction={decorate}  ]
        (9,-1.85) to (13,-1.85); \node at (11,-1.55){{\footnotesize $\alpha^{(1)}$}} ;
 \draw[thick, decoration={markings, mark=at position 0.5 with {\arrow{>}}},
        postaction={decorate}  ]
        (13,-2) to (11,-2); \node at (12.3,-2.3){{\footnotesize $\alpha^{(2)}+\alpha^{(3)}$}} ;
\draw[thick, decoration={markings, mark=at position 0.5 with {\arrow{>}}},
        postaction={decorate}  ]
        (11,-2) to (9,-2); \node at (10,-2.3){{\footnotesize $\alpha^{(4)}$}} ;
\fill (9,-1.85) circle (0.07)node(xline)[above] {{\small $\mu^{(1)}$}};     
\fill (13,-1.85) circle (0.07); \node at (13.5,-1.85) {{\small $\mu^{(2)}$}};
\fill (11,-2) circle (0.07)node(xline)[below] {{\small $\mu^{(4)}$}};

\node at (-1,-3.7) {(2)};
 \draw[thick, decoration={markings, mark=at position 0.5 with {\arrow{>}}},
        postaction={decorate}  ]
        (0,-3.7) to (2,-3.7); \node at (1,-3.4){{\footnotesize $\alpha^{(1)}$}} ;
 \draw[thick, decoration={markings, mark=at position 0.5 with {\arrow{>}}},
        postaction={decorate}  ]
        (2,-3.7) to (6,-3.7); \node at (4,-3.4){{\footnotesize $\alpha^{(2)}$}} ;
 \draw[thick, decoration={markings, mark=at position 0.5 with {\arrow{>}}},
        postaction={decorate}  ]
        (6,-3.85) to (2,-3.85); \node at (4,-4.15){{\footnotesize $\alpha^{(3)}$}} ;
 \draw[thick, decoration={markings, mark=at position 0.5 with {\arrow{>}}},
        postaction={decorate}  ]
        (2,-3.85) to (0,-3.85); \node at (1,-4.15){{\footnotesize $\alpha^{(4)}$}} ;
\fill (0,-3.7) circle (0.07)node(xline)[above] {{\small $\mu^{(1)}$}};     
\fill (2,-3.7) circle (0.07)node(xline)[above] {{\small $\mu^{(2)}$}};
\fill (6,-3.7) circle (0.07)node(xline)[above] {{\small $\mu^{(3)}$}};
\fill (2,-3.85) circle (0.07)node(xline)[below] {{\small $\mu^{(4)}$}}; 

\draw[ ->, dashed, draw=blue]
        (7,-3.85) to (8,-3.85);\node[text=blue] at (7.5,-3.45){$\# 3$} ;

 \draw[thick, decoration={markings, mark=at position 0.5 with {\arrow{>}}},
        postaction={decorate}  ]
        (9,-3.7) to (11,-3.7); \node at (10,-3.4){{\footnotesize $\alpha^{(1)}$}} ;
 \draw[thick, decoration={markings, mark=at position 0.5 with {\arrow{>}}},
        postaction={decorate}  ]
        (11,-3.85) to (9,-3.85); \node at (10,-4.15){{\footnotesize $\alpha^{(4)}$}} ;
\fill (9,-3.7) circle (0.07)node(xline)[above] {{\small $\mu^{(1)}$}};     
\fill (11,-3.7) circle (0.07); \node at (12,-3.7) {{\small $\mu^{(2)}=\mu^{(4)}$}};

\node at (-1,-5.4) {(3)};
 \draw[thick, decoration={markings, mark=at position 0.5 with {\arrow{>}}},
        postaction={decorate}  ]
        (0,-5.4) to (2,-5.4); \node at (1,-5.1){{\footnotesize $\alpha^{(1)}$}} ;
 \draw[thick, decoration={markings, mark=at position 0.5 with {\arrow{>}}},
        postaction={decorate}  ]
        (2,-5.4) to (4,-5.4); \node at (3,-5.1){{\footnotesize $\alpha^{(2)}$}} ;
 \draw[thick, decoration={markings, mark=at position 0.325 with {\arrow{<}}},
        postaction={decorate}  ]
        (4.5,-5.4) ellipse (0.5cm and 0.3cm); \node at (5.4,-5.2){{\footnotesize $\alpha^{(3)}$}} ;
 \draw[thick, decoration={markings, mark=at position 0.5 with {\arrow{>}}},
        postaction={decorate}  ]
        (4,-5.55) to (0,-5.55); \node at (2,-5.85){{\footnotesize $\alpha^{(4)}$}} ;
\fill (0,-5.4) circle (0.07)node(xline)[above] {{\small $\mu^{(1)}$}};     
\fill (2,-5.4) circle (0.07)node(xline)[above] {{\small $\mu^{(2)}$}};
\fill (4,-5.4) circle (0.07);
\node at (3.95,-5.95) {{\small $\mu^{(3)}=\mu^{(4)}$}};

\draw[ ->, dashed, draw=blue]
        (7,-5.55) to (8,-5.55);\node[text=blue] at (7.5,-5.15){$\# 3$} ;

 \draw[thick, decoration={markings, mark=at position 0.5 with {\arrow{>}}},
        postaction={decorate}  ]
        (9,-5.4) to (11,-5.4); \node at (10,-5.1){{\footnotesize $\alpha^{(1)}$}} ;
 \draw[thick, decoration={markings, mark=at position 0.5 with {\arrow{>}}},
        postaction={decorate}  ]
        (11,-5.4) to (13,-5.4); \node at (12,-5.1){{\footnotesize $\alpha^{(2)}$}} ;
 \draw[thick, decoration={markings, mark=at position 0.5 with {\arrow{>}}},
        postaction={decorate}  ]
        (13,-5.55) to (9,-5.55); \node at (11,-5.85){{\footnotesize $\alpha^{(4)}$}} ;
\fill (9,-5.4) circle (0.07)node(xline)[above] {{\small $\mu^{(1)}$}};
\fill (11,-5.4) circle (0.07)node(xline)[above] {{\small $\mu^{(2)}$}};
\fill (13,-5.4) circle (0.07); \node at (13.5,-5.4) {{\small $\mu^{(4)}$}};
\end{tikzpicture}
 \caption{\label{fig:delate} Cutting the vertex $\mu^{(3)}$} 
\end{center}
\end{figure}
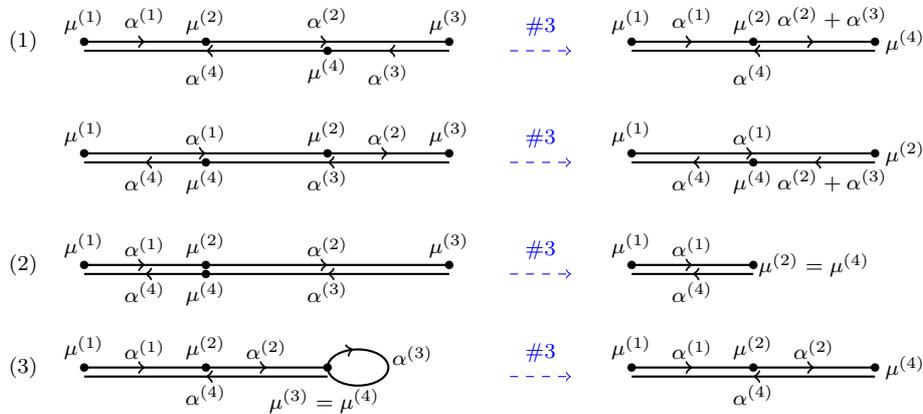

Let $m \in \Z_{> 1}$ and  $(\mu,\nu) \in P(\delta)^2$. 
For $\sub{i} := (i_1,\ldots,i_m )\in \Z^m $,  
set 
\begin{align*}
\hat{\P}_m(\mu,\nu)_{\sub{i}} 
:=\{ (\sub{\mu},\sub{\alpha}) \in \hat{\P}_m(\mu,\nu) \ | \ 
\he(\sub{\mu}) = \sub{i} \ \}.
\end{align*}
The elements of $\hat{\P}_m(\mu,\nu)_{\sub{i}}$ share 
the same vertices 
and only the labels of the ^^ ^^ loops'' (i.e.~the labels $\alpha^{(j)}$ 
such that $\mu^{(j)}=\mu^{(j+1)}$) may differ. 
We obtain the following partition of $\hat{\P}_m(\mu,\nu)$:
$$\hat{\P}_m(\mu,\nu) = \bigsqcup_{\sub{i} \in \Z^m } \hat{\P}_m(\mu,\nu)_{\sub{i}}.$$
The set $\hat{\P}_m(\mu,\nu)_{\sub{i}}$ is empty for almost all $\sub{i}$. 
We endow the set $\Z^m$ with the lexicographical order~$\preccurlyeq$. 
The zero element is $\sub{0} = (0,\ldots,0)$. 
We denote by $\Z^m_{\succcurlyeq \sub{0}}$ (respectively, $\Z^m_{\succ \sub{0}}$) the set 
of elements in $\Z^m$ greater 
(respectively, strictly greater) than $\sub{0}$ 
for the lexicographical order.  

Let $\sub{i} \in {\Z}^m_{\succ \sub{0}}$ such that $\sum_{j=1}^m i_j =0$.  
Let $q(\sub{i})$ be the smallest integer $q$ of $\{1,\ldots,m\}$ 
such that $\sub{i}_q < 0$, 
and let $p(\sub{i})$ be the largest 
integer $p$ of $\{ 1,\ldots, q(\sub{i})\}$ 
such that $\sub{i}_{p-1} >0$. 
Thus $\sub{i}$ is as follows: 
\begin{align*}
\sub{i}=(\underbrace{i_1,\ldots,i_{p(\sub{i})-2},}_{\geqslant 0}\underbrace{ i_{p(\sub{i})-1}}_{> 0}, 0,\ldots,0,
\underbrace{i_{q(\sub{i})}}_{<0},i_{q(\sub{i})+1},\ldots,i_m).
\end{align*}
For a path $\sub{\mu}$ of height $\sub{i}$, we will say that 
{\em $\mu^{(p(\sub{i}))}$ is the position of the first turning back}.
Note that $p(\sub{i})$ is always strictly greater than $1$ 
for such $\sub{i}$. 

\begin{lemma} \label{Lem:a} 
Let $m \in \Z_{> 1}$, $(\mu,\nu) \in P(\delta)^2$, 
$i \in \{2,\ldots,m\}$ 
and $(\sub{\mu},\sub{\alpha}) \in \hat{\P}_m(\lambda,\mu)$. 
\begin{enumerate}
\item If $\he(\sub{\mu})_{i-1}\not=0$, 
$\he(\sub{\mu})_i \not=0$ 
and $\alpha^{(i-1)}+\alpha^{(i)}=0$, then 
$
a_{\sub{\mu},\sub{\alpha}} = (c_{\alpha^{(i-1)}})^2 
a_{(\sub{\mu},\sub{\alpha})^{\# i}}.
$
\item If $\he(\sub{\mu})_i =0$, then
$a_{\sub{\mu},\sub{\alpha}} = \langle \mu^{(i)}, \c{\alpha}^{(i)} \rangle 
a_{(\sub{\mu},\sub{\alpha})^{\# i}}.$
\end{enumerate}
\end{lemma}

\begin{proof}
\begin{enumerate}
 \item If $\alpha^{(i-1)}+\alpha^{(i)}=0$ then $\alpha^{(i-1)} = -\alpha^{(i)}$ and $\mu^{(i-1)}=\mu^{(i+1)}.$ We have, 
\begin{align*}
a_{\sub{\mu},\sub{\alpha}} &= 
a_{\mu^{(m+1)},\mu^{(m)}}^{(b_{(\sub{\mu},\sub{\alpha}),m})} 
\ldots a_{\mu^{(i+1)},\mu^{(i)}}^{(c_{\alpha^{(i)}}e_{-\alpha^{(i)}})} 
a_{\mu^{(i)},\mu^{(i-1)}}^{(c_{\alpha^{(i-1)}} e_{-\alpha^{(i-1)}})} 
\ldots   a_{\mu^{(2)},\mu^{(1)}}^{(b_{(\sub{\mu},\sub{\alpha}),1})} \\ 
a_{(\sub{\mu},\sub{\alpha})^{\#i}} &= 
a_{\mu^{(m+1)},\mu^{(m)}}^{(b_{(\sub{\mu},\sub{\alpha}),m})} 
\ldots a_{\mu^{(i+2)},\mu^{(i+1)}}^{(b_{(\sub{\mu},\sub{\alpha}),i+1})}
a_{\mu^{(i-1)},\mu^{(i-2)}}^{(b_{(\sub{\mu},\sub{\alpha}),i-2})} 
\ldots   a_{\mu^{(2)},\mu^{(1)}}^{(b_{(\sub{\mu},\sub{\alpha}),1})}.
\end{align*}
On the other hand, 
\begin{align*} 
a_{\mu^{(i+1)},\mu^{(i)}}^{(e_{-\alpha^{(i)}})} 
a_{\mu^{(i)},\mu^{(i-1)}}^{(e_{-\alpha^{(i-1)}})}  v_{\mu^{(i+1)}} 
& = \pi_\delta(e_{-\alpha^{(i)}}) \pi_\delta(e_{-\alpha^{(i-1)}} ) v_{\mu^{(i-1)}}
 = \pi_{\delta}([e_{-\alpha^{(i)}},e_{-\alpha^{(i-1)}}]) 
v_{\mu^{(i-1)}} \\
& = \pi_\delta(\c{\alpha}^{(i-1)} ) v_{\mu^{(i-1)}} 
= \langle \mu^{(i-1)}, \c{\alpha}^{(i-1)} \rangle v_{\mu^{(i+1)}},
\end{align*}
since $\mu^{(i-1)} - \alpha^{(i)} \notin P(\delta)$.
Therefore,
\begin{align*}
a_{\sub{\mu},\sub{\alpha}} &= 
a_{\mu^{(m+1)},\mu^{(m)}}^{(b_{(\sub{\mu},\sub{\alpha}),m})} 
\ldots a_{\mu^{(i+1)},\mu^{(i)}}^{(c_{\alpha^{(i)}}e_{-\alpha^{(i)}})} 
a_{\mu^{(i)},\mu^{(i-1)}}^{(c_{\alpha^{(i-1)}}e_{-\alpha^{(i-1)}})} 
\ldots   a_{\mu^{(2)},\mu^{(1)}}^{(b_{(\sub{\mu},\sub{\alpha}),1})} \\  
&= (c_{\alpha^{(i-1)}})^2 \, \langle \mu^{(i-1)}, \c{\alpha}^{(i-1)} \rangle
a_{(\sub{\mu},\sub{\alpha})^{\#i}} 
= (c_{\alpha^{(i-1)}})^2 \, a_{(\sub{\mu},\sub{\alpha})^{\#i}},
\end{align*}
since $\langle \mu^{(i-1)}, \c{\alpha}^{(i-1)} \rangle = 1.$
\item If $\he(\sub{\mu})_i =0$, we have $\mu^{(i)}=\mu^{(i+1)}$ and, 
\begin{align*}
a_{\sub{\mu},\sub{\alpha}} &= 
a_{\mu^{(m+1)},\mu^{(m)}}^{(b_{(\sub{\mu},\sub{\alpha}),m})} 
\ldots  
a_{\mu^{(i+1)},\mu^{(i)}}^{(\c{\alpha}^{(i)})}  
\ldots   a_{\mu^{(2)},\mu^{(1)}}^{(b_{(\sub{\mu},\sub{\alpha}),1})} \\ 
a_{(\sub{\mu},\sub{\alpha})^{\#i}} &= 
a_{\mu^{(m+1)},\mu^{(m)}}^{(b_{(\sub{\mu},\sub{\alpha}),m})} 
\ldots a_{\mu^{(i+2)},\mu^{(i+1)}}^{(b_{(\sub{\mu},\sub{\alpha}),i+1})}
a_{\mu^{(i)},\mu^{(i-1)}}^{(b_{(\sub{\mu},\sub{\alpha}),i-1})} 
\ldots   a_{\mu^{(2)},\mu^{(1)}}^{(b_{(\sub{\mu},\sub{\alpha}),1})} .
\end{align*}
On the other hand, 
\begin{align*} 
a_{\mu^{(i+1)},\mu^{(i)}}^{(\c{\alpha}^{(i)})}  v_{\mu^{(i+1)}} 
= \pi_\delta(\c{\alpha}^{(i)}) v_{\mu^{(i)}}  
= \langle \mu^{(i)}, \c{\alpha}^{(i)} \rangle v_{\mu^{(i)}},
\end{align*}
whence the statement. \qed
\end{enumerate} 
\end{proof}

\section{The proofs for type $A$}\label{sec:main2-A}

This section is devoted to the proofs of Theorem~\ref{theorem:main1} and Theorem~\ref{theorem:main2} for $\g=\sl_{r+1}$.
Throughout this section, it is assumed 
that $\g=\sl_{r+1}, r \geqslant 2$ and $\delta=\varpi_1$. 
We retain all relative notation from previous section and Appendix \ref{sec:rootTypeA}. 
In particular, $P(\delta) = \{\delta_1, \cdots, \delta_{r+1} \}$. 
Moreover, 
we have (cf.~Figure \ref{fig:crystal}), 
$P(\delta)_k=\{\delta_{k},\delta_{k+1}\}$ for all $k \in\{ 1, \ldots, r\}$  
and 
$\Pi_{\delta_k}=\{\beta_{k-1},\beta_k\}$ for $k=2,\ldots,r$, 
$\Pi_{\delta_1}=\{\beta_1\}$, $\Pi_{\delta_{r+1}}=\{\beta_{r}\}.$

According to Lemma \ref{Lem:formulas} and \eqref{eq:formulas}, 
we get  
\begin{align} \label{eq:main1}
\ev_\rho(\overline{\d p}_{m,k} ) = \sum_{ \sub{\alpha} \in (\Pi_{\delta_k})^m}  
\prod_{i=1}^m \langle \delta_k, \c{\alpha}^{(i)} \rangle 
\langle \rho,\c{\varpi}_{\alpha^{(i)}} \rangle 
-  \sum_{(\sub{\alpha} \in (\Pi_{\delta_{k+1}})^m}  
\prod_{i=1}^m \langle \delta_{k+1}, \c{\alpha}^{(i)} \rangle 
\langle \rho, \c{\varpi}_{\alpha^{(i)}}  \rangle & 
\end{align}
since $\s{\varpi_i}=\c{\varpi_i}$ 
for all $i=1,\ldots,r$, $\g$ being simply laced.

\begin{lemma}
\noindent \label{Lem1:main1}
\begin{enumerate}
\item For any $j \in\{1,\ldots,r+1\}$, 
$\langle \rho, \eps_j \rangle =\frac{r}{2} - j+1.$

\item For $k \in \{1,\ldots,r+1\}$, 
$\langle \rho, \c{\varpi}_k \rangle = \frac{k}{2}(r-k+1)$ and    
$\langle \rho, \c{\varpi}_k -\c{\varpi}_{k-1} \rangle
= \frac{r}{2} -k +1,$ 
where by convention $\varpi_0=\varpi_{r+1}=0$. 
\end{enumerate}
\end{lemma}
The proof of the lemma is easy. The verifications are left to the reader.

\begin{lemma}\label{Lem2:main1}
For some polynomial $T_m \in \C[X]$ of degree $m$,  
$$
\sum_{ \sub{\alpha} \in (\Pi_{\delta_k})^m}  
\prod_{i=1}^m \langle \delta_k, \c{\alpha}^{(i)} \rangle 
\langle \rho,\c{\varpi}_{\alpha^{(i)}} \rangle =T_m(k), 
\qquad \forall \, k =1,\ldots,r+1.$$
Moreover, the leading term of $T_m$ is $(-X)^m$. 
\end{lemma}

\begin{proof} 
Assume first that $k\in \{2,\ldots,r\}$. 
Then 
by Lemma \ref{Lem1:main1}, 
\begin{align*}
\sum_{ \sub{\alpha} \in (\Pi_{\delta_k})^m}  
\prod_{i=1}^m \langle \delta_k, \c{\alpha}^{(i)} \rangle 
\langle \rho,\c{\varpi}_{\alpha^{(i)}} \rangle 
&= \sum_{i=0}^m \begin{pmatrix} m \\
i 
\end{pmatrix} (-1)^{i} 
\langle \rho,\c{\varpi}_{k-1} \rangle^{i} 
\langle \rho,\c{\varpi}_{k} \rangle^{m-i}
=(- \langle \rho,\c{\varpi}_{k-1} \rangle 
+\langle \rho,\c{\varpi}_{k} \rangle)^{m} \\
&= (\langle \rho,\c{\varpi}_{k} - \c{\varpi}_{k-1} \rangle )^{m}  
= \left(\frac{r}{2} -k +1 \right)^m .
\end{align*}

If $k=1$, then 
by Lemma \ref{Lem1:main1},  
 \begin{align*}
\sum_{ \sub{\alpha} \in (\Pi_{\delta_1})^m}  
\prod_{i=1}^m \langle \delta_1, \c{\alpha}^{(i)} \rangle 
\langle \rho,\c{\varpi}_{\alpha^{(i)}} \rangle 
&= \langle \rho,\c{\varpi}_{1} \rangle^{m} 
=\left(\frac{r}{2} \right)^m=
 \left(\frac{r}{2} -k +1\right)^m.
\end{align*}

If $k=r+1$, then 
by Lemma \ref{Lem1:main1},  
 \begin{align*}
\sum_{ \sub{\alpha} \in (\Pi_{\delta_r})^m}  
\prod_{i=1}^m \langle \delta_r, \c{\alpha}^{(i)} \rangle 
\langle \rho,\c{\varpi}_{\alpha^{(i)}} \rangle 
&= (-1)^m \langle \rho,\c{\varpi}_{r} \rangle^{m} 
=(-1)^m \left( \frac{r}{2}  \right)^m 
= \left(\frac{r}{2} - k +1\right)^m.
\end{align*}

Hence, setting $T_m (X):= (\frac{r}{2} -X +1)^m$ 
we get the statement.  \qed
\end{proof}

We are now in a position to prove Theorem \ref{theorem:main1}  for $\g=\sl_{r+1}$.
\begin{proof}
[Proof of Theorem \ref{theorem:main1} for $\g=\sl_{r+1}$] 

Let $m \in \{1,\ldots,r\}$. 
By Lemma \ref{Lem2:main1} and \eqref{eq:main1}, 
we have for any $k \in \{1,\ldots,r\}$, 
\begin{align*}
\ev_\rho(\overline{\d p}_{m,k} ) & =
T_m(k) - T_m(k+1) = 
 m(-k)^{m-1} +\sum_{i=0}^{m-2} 
\begin{pmatrix} 
m \\ i 
\end{pmatrix} (-k)^{i}  \left(\left( \frac{r}{2}+1\right)^{m-i} - \left(\frac{r}{2}\right)^{m-i}\right).
 \end{align*}
 Hence, by setting
\begin{align} \label{eq:Qi-sl}
\bar{Q}_m(X):=m(-X)^{m-1} +\sum_{i=0}^{m-2} 
\begin{pmatrix} m \\
i 
\end{pmatrix} (-X)^{i} \left(\left(\frac{r}{2}+1 \right)^{m-i} -\left(\frac{r}{2}\right)^{m-i} \right),
\end{align}
 we get 
\begin{align*}
\ev_\rho (\overline{\d p}_{m}  ) 
=\ev_\rho \left(\frac{1}{m!}\sum_{k=1}^{r} \overline{\d p}_{m,k} \otimes \s{\varpi_k} \right)
=\frac{1}{m!}\sum_{k=1}^{r}\ev_\rho (\overline{\d p}_{m,k}) \check{\varpi}_k 
=\frac{1}{m!} \sum_{k=1}^{r}  \bar{Q}_m(k) 
\c{\varpi}_k. 
\end{align*} 
Moreover, $\bar{Q}_1= 1$ and $\bar{Q}_m(X)$ is a polynomial of degree $m-1$.\qed
\end{proof}

The rest of the section is devoted to the proof of Theorem \ref{theorem:main2} for $\g=\sl_{r+1}$. 
We first establish some reduction results 
in order to show Theorem \ref{corollary:cut}. 

\begin{lemma}     \label{Lem2:cut-bis}
Let $m \in \Z_{> 0}$, $(\mu,\nu) \in P(\delta)^2$,  $\gamma \in \Delta_+$ 
and $(\sub{\mu},\sub{\alpha}) \in \hat{\P}_m(\mu,\nu)$ such that $\he(\sub{\mu})_i \geqslant 0$ for any $i$.  
Assume that $\gamma = \mu' -\nu'$, with $\mu',\nu' \in P(\delta)$, 
and that $\nu' \prec \mu^{(i)}$ for all 
$i\in\{1,\ldots,m+1\}$.
Note that $\mu' \succ \nu'$ since $\gamma \in \Delta_+$. 
\begin{enumerate}
\item Let $i\in\{1,\ldots,m+1\}$. If $\he(\sub{\mu})_i >0$ then 
either $\alpha^{(i)} - \gamma \not \in \Delta$ or 
$\alpha^{(i)} - \gamma \in - \Delta_+$. 
Moreover, if $\alpha^{(i)} - \gamma \in - \Delta_+$, then 
$(\sub{\mu},\sub{\alpha})$ and $\gamma' := \gamma -\alpha^{(i)}$ 
still satisfy the above conditions 
with $\gamma'$ in place of $\gamma$. 
\item For all $u \in U(\g)$, we have 
$\hc(b_{\sub{\mu},\sub{\alpha}}^*  e_{-\gamma} u) =0.$
\end{enumerate}
\end{lemma}

\begin{proof}
(1) Write $\gamma= \eps_j-\eps_k$, with $j < k$. 
The hypothesis says that for all $i \in \{1,\ldots,m\}$ 
such that $\he(\sub{\mu})_i>0$ then 
$\alpha^{(i)} = \eps_{j_i} -\eps_{k_{i}}$ 
with $j_i < k_{i} < k$. 
Hence 
$$\alpha^{(i)} - \gamma=  \eps_{j_i} - \eps_{k_{i}} - \eps_j + \eps_k$$
is a root if and only if $j_i = j$. 
If it is so, then 
$\alpha^{(i)} - \gamma=\eps_k - \eps_{k_{i}}$ is 
a negative root since $k > k_{i}$.  
Moreover, $\gamma':=\eps_{k_{i}} - \eps_k$ still verifies 
the condition of the lemma. 

\noindent
(2) We prove the assertion by induction on $m$. 
Set $a := a_{\mu^{(m)},\mu^{(m+1})}^{(b_{(\sub{\mu},\sub{\alpha}),m})}.$

\noindent
Assume $m=1$. 
If $\he(\sub{\mu})_1 >0$ then by Part 1 either 
$ \alpha^{(1)} - \gamma \not\in \Delta$,  and 
$$\hc(b_{\sub{\mu},\sub{\alpha}}^* e_{-\gamma} u )
=\hc(a e_{\alpha^{(1)}}  e_{-\gamma} u )
=\hc( a e_{-\gamma}e_{\alpha^{(1)}} u) =0,$$
or  $\alpha^{(1)} - \gamma= - \gamma'$, with $\gamma' \in \Delta_+$, 
and 
\begin{align*}\hc(b_{\sub{\mu},\sub{\alpha}}^* e_{-\gamma} u )
=\hc(a e_{\alpha^{(1)}}  e_{-\gamma} u )
=\hc( a e_{-\gamma}e_{\alpha^{(1)}} u) 
+ \hc( a a' e_{- \gamma'} u) =0,
\end{align*}
where $a' \in \C$. 

If $\he(\sub{\mu})_1 =0$, then 
\begin{align*}
\hc(b_{\sub{\mu},\sub{\alpha}}^* e_{-\gamma} u ) 
 = \hc(a \c{\varpi}_{\alpha^{(1)}} e_{-\gamma} u ) 
 = \hc(a e_{-\gamma} \c{\varpi}_{\alpha^{(1)}} u 
- a \langle \gamma, \c{\varpi}_{\alpha^{(1)}}  \rangle e_{-\gamma} u) =0.
\end{align*} 
In both cases, we obtain the statement. 

\noindent
Let $m \geqslant 2$ and assume the statement true 
for any $m' \in \{1,\ldots,m-1\}$. Write 
$(\sub{\mu},\sub{\alpha}) = (\sub{\mu}',\sub{\alpha}') \star ((\mu^{(m)},\mu^{(m+1)}), \alpha^{(m)})$, 
where $(\sub{\mu}',\sub{\alpha}')$ has length $m-1$. 
Note that the weighted $(\sub{\mu}',\sub{\alpha}')$ 
and $\gamma$ satisfy the 
conditions of the lemma. 

There are two cases: 

\noindent
$\ast$ $\he(\sub{\mu})_m > 0$. 
By Part 1 either $\alpha^{(m)} -\gamma \not \in\Delta$, then 
by induction hypothesis, we get  
\begin{align*}\hc(b_{\sub{\mu},\sub{\alpha}}^* e_{-\gamma} u )  
= \hc(a b_{\sub{\mu}',\sub{\alpha}'}^* e_{\alpha^{(m)}}  e_{-\gamma} u )  
= \hc(a b_{\sub{\mu}',\sub{\alpha}'}^* e_{-\gamma} e_{\alpha^{(m)}} u )
=0,
\end{align*}
or $\alpha^{(m)} -\gamma = -\gamma'$ with $\gamma' \in \Delta_+$, 
and by induction, 
\begin{align*}
& \hc(b_{\sub{\mu},\sub{\alpha}}^* e_{-\gamma} u )  
= \hc(a b_{\sub{\mu}',\sub{\alpha}'}^* e_{\alpha^{(m)}}  e_{-\gamma} u ) 
= \hc(a b_{\sub{\mu}',\sub{\alpha}'}^* e_{-\gamma} e_{\alpha^{(m)}} u ) 
+ \hc(aa' b_{\sub{\mu}',\sub{\alpha}'}^* e_{-\gamma'} u )  
=0,& 
\end{align*}
where $a' \in \C$, since the path $(\sub{\mu}',\sub{\alpha}')$ 
and $\gamma'$ still satisfy the 
conditions of the lemma by (1). 

\noindent
$\ast$ If $\he(\sub{\mu})_m = 0$, then 
by induction hypothesis, we get
\begin{align*} 
\hc(b_{\sub{\mu},\sub{\alpha}}^* e_{-\gamma} u ) 
& = \hc(a b_{\sub{\mu}',\sub{\alpha}'}^* \c{\varpi}_{\alpha^{(m)}} 
e_{-\gamma} u) = \hc(a b_{\sub{\mu}',\sub{\alpha}'}^* 
e_{-\gamma} \c{\varpi}_{\alpha^{(m)}} u   
-  a \langle \gamma, \c{\varpi}_{\alpha^{(m)}} \rangle 
b_{\sub{\mu}',\sub{\alpha}'}^* e_{-\gamma} u )=0,
 \end{align*} 
 whence the statement. \qed
\end{proof}

\begin{lemma}     \label{Lem0:cut}
Let $\mu \in P(\delta)$, $m \in \Z_{> 1}$,  
$\sub{i} \in {\Z}^m_{\succ \sub{0}}$ and 
$(\sub{\mu},\sub{\alpha}) \in \hat{\P}_m(\mu)_{\sub{i}}$.   
Set $p:=p(\sub{i})$ and $q:=q(\sub{i})$.   
\begin{enumerate}
\item Assume $p=q$ and $\alpha^{(p-1)}+\alpha^{(p)} \neq  0$. 
Then $\wt(\sub{\mu},\sub{\alpha})^{\# p}.$
\item Assume $p  = q$ and $\alpha^{(p-1)}+\alpha^{(p)} = 0$.  
\begin{enumerate} 
\item If $ i_1=\cdots=i_{p-2}=0$, or if $p=2$, then 
$\wt(\sub{\mu},\sub{\alpha})= \he(\c{\alpha}^{(p-1)}) 
\,\wt(\sub{\mu},\sub{\alpha})^{\# p}.$
\item Otherwise, 
$\wt(\sub{\mu},\sub{\alpha})=
(\he(\c{\alpha}^{(p-1)}) +1) 
\wt(\sub{\mu},\sub{\alpha})^{\# p}.$
\end{enumerate}
\item Assume $p < q$. 
Then $i_p=0$ and $\alpha^{(p)} \in \Pi_{\mu^{(p)}}
=\{\beta \in \Pi \; | \; 
\langle \mu^{(p)},\c{\beta} \rangle\not=0\}$.  
\begin{enumerate} 
\item If $\langle \mu^{(p)},\c{\alpha}^{(p)} \rangle =1$, then 
$\wt(\sub{\mu},\sub{\alpha})=   
\langle \rho, \c{\varpi}_{\alpha^{(p)}} \rangle 
\wt(\sub{\mu},\sub{\alpha})^{\# p}.$
\item If $\langle \mu^{(p)},\c{\alpha}^{(p)} \rangle =-1$, then  
$\wt(\sub{\mu},\sub{\alpha})=  
(-\langle \rho, \c{\varpi}_{\alpha^{(p)}}  \rangle 
+ 1)
\wt(\sub{\mu},\sub{\alpha})^{\# p}.$
\end{enumerate}
\end{enumerate}
\end{lemma}

\begin{proof}
(1) Note that $\alpha^{(p-1)}+\alpha^{(p)} \neq  0$ implies $i_{p-1} + i_p \neq 0$. Write 
$$(\sub{\mu},\sub{\alpha}) = (\sub{\mu}',\sub{\alpha}') \star 
((\mu^{(p-1)},\mu^{(p)}),\alpha^{(p-1)}) 
\star ((\mu^{(p)},\mu^{(p+1)}),\alpha^{(p)}) 
\star (\sub{\mu}'',\sub{\alpha}''),$$
where $(\sub{\mu}',\sub{\alpha}')$  
and $ (\sub{\mu}'',\sub{\alpha}'')$ have length $p-2$ and $m - p$, respectively.  

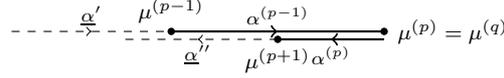
\begin{center}
\begin{tikzpicture} [scale=0.7]
 \draw[dashed,
        decoration={markings, mark=at position 0.5 with {\arrow{>}}},
        postaction={decorate}
        ]
        (0,0) -- (3,0);\node at (1.5,0.3) {{\scriptsize $\sub{\alpha}'$}}; 
  \draw[ thick,
        decoration={markings, mark=at position 0.5 with {\arrow{>}}},
        postaction={decorate}
        ]
        (3,0) -- (7,0); \node at (5,0.3) {{\scriptsize ${\alpha}^{(p-1)}$}};
  \draw[ thick,
        decoration={markings, mark=at position 0.5 with {\arrow{>}}},
        postaction={decorate}
        ]
        (7,-0.15) -- (5,-0.15); \node at (6,-0.45) {{\scriptsize ${\alpha}^{(p)}$}};
   \draw[ dashed,
        decoration={markings, mark=at position 0.5 with {\arrow{>}}},
        postaction={decorate}
        ]
        (5,-0.15) -- (2,-0.15);\node at (3.5,-0.45) {{\scriptsize $\sub{\alpha}''$}};
 \fill (3,0) circle (0.07)node(xline)[above] {{\footnotesize $\mu^{(p-1)}$}}; 
    \fill (7,0) circle (0.07); \node at (8.3,0) {{\footnotesize $\mu^{(p)} = \mu^{(q)}$}};
   \fill (5,-0.15) circle (0.07)node(xline)[below] {{\footnotesize $\mu^{(p+1)}$}}; 
\end{tikzpicture}
\end{center}
By \S\ref{sec:rootTypeA},  
$a_{\mu^{(i+1)},\mu^{(i)}}^{(e_{-\alpha^{(i)}})}$,  
$a_{\mu^{(i)},\mu^{(i-1)}}^{(e_{-\alpha^{(i-1)}})}$, 
$a_{\mu^{(i+1)},\mu^{(i-1)}}^{(e_{-\alpha^{(i-1)}-\alpha^{(i)}})}$ 
and $n_{\alpha^{(p-1)},\alpha^{(p)}}$ 
are all equal to $1$, and so 
\begin{align*}
\hc(b_{(\sub{\mu},\sub{\alpha})}^*) & = 
 \hc(b_{(\sub{\mu}',\sub{\alpha}')}^* 
 e_{\alpha^{(p-1)}} e_{\alpha^{(p)}} 
 b_{(\sub{\mu}'',\sub{\alpha}'')}^*) 
 \\
& = 
\hc(b_{(\sub{\mu}',\sub{\alpha}')}^* e_{\alpha^{(p)}} 
e_{\alpha^{(p-1)}} 
 b_{(\sub{\mu}'',\sub{\alpha}'')}^* )
+ \hc(n_{\alpha^{(p-1)},\alpha^{(p)}} 
b_{(\sub{\mu}',\sub{\alpha}')}^*  e_{\alpha^{(p-1)}+\alpha^{(p)}}
b_{(\sub{\mu}'',\sub{\alpha}'')}^* ) = \hc(b_{(\sub{\mu},\sub{\alpha})^{\# p}}^*), 
\end{align*} 
since the weighted path $(\sub{\mu}',\sub{\alpha}')$ 
 and the positive root $\gamma=-\alpha^{(p)}$ verify the conditions of 
 Lemma~\ref{Lem2:cut-bis}, and so 
 $\hc(b_{(\sub{\mu}',\sub{\alpha}')}^* e_{\alpha^{(p)}} 
 e_{\alpha^{(p-1)}} 
 b_{(\sub{\mu}'',\sub{\alpha}'')}^*)=0.$ 
 Hence $$\wt(\sub{\mu},\sub{\alpha})
= \wt(\sub{\mu},\sub{\alpha})^{\# p}.$$
 
\noindent
(2) (a) Assume first that $i_{1} = \cdots =i_{p-2}=0$ 
and $p \not = 2$ (that is, $p>2$). Write 
$$(\sub{\mu},\sub{\alpha}) = (\sub{\mu}',\sub{\alpha}') \star 
((\mu^{(p-1)},\mu^{(p)}),\alpha^{(p-1)}) \star 
((\mu^{(p)},\mu^{(p+1)}),\alpha^{(p)}) 
\star (\sub{\mu}'',\sub{\alpha}'')$$
as in Part 1. 
Here $(\sub{\mu}',\sub{\alpha}')$ is a concatenation of loops. 

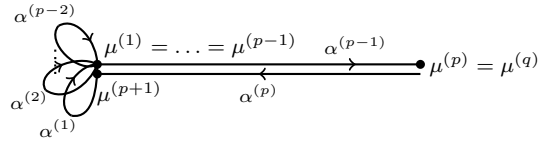
\begin{center}
\begin{tikzpicture} [scale=0.85]
 \draw[thick,
        decoration={markings, mark=at position 0.8 with {\arrow{>}}},
        postaction={decorate}
        ]
        (0,0) -- (5,0);\node at (4,0.3) {{\scriptsize ${\alpha}^{(p-1)}$}}; 
  \draw[ thick,
        decoration={markings, mark=at position 0.5 with {\arrow{>}}},
        postaction={decorate}
        ]
        (5,-0.15) -- (0,-0.15); \node at (2.5,-0.45) {{\scriptsize ${\alpha}^{(p)}$}};
 \fill (0,0) circle (0.07);\node at (1.6,0.3) {{\footnotesize $\mu^{(1)}=\ldots=\mu^{(p-1)}$}}; 
  \fill (5,0) circle (0.07);\node at (6,0) {{\footnotesize $\mu^{(p)}=\mu^{(q)}$}};
    \fill (0,-0.15) circle (0.07);\node at (0.5,-0.45) {{\footnotesize $\mu^{(p+1)}$}}; 
\draw[thick] (0,0) to [out=270,in=350] (-0.4,-0.8);
\draw[thick, decoration={markings, mark=at position 0.625 with {\arrow{>}}},
        postaction={decorate}] (-0.4,-0.8) to [out=150,in=180] (0,0);
\node at (-0.6,-1) {{\scriptsize ${\alpha}^{(1)}$}};

\draw[thick] (0,0) to [out=250,in=325] (-0.8,-0.4);
\draw[thick, decoration={markings, mark=at position 0.5 with {\arrow{>}}},
        postaction={decorate}] (-0.8,-0.4) to [out=120,in=150] (0,0);        
\node at (-1.05,-0.55) {{\scriptsize ${\alpha}^{(2)}$}};

\draw[thick] (0,0) to [out=210,in=230] (-0.6,0.6);
\draw[thick, decoration={markings, mark=at position 0.625 with {\arrow{>}}},
        postaction={decorate}] (-0.6,0.6) to [out=30,in=90] (0,0);
\node at (-0.8,0.8) {{\scriptsize ${\alpha}^{(p-2)}$}};
\draw[thick, dotted] (-0.65,-0.15) to  (-0.65,0.25);
\end{tikzpicture}
\end{center}
\noindent
In particular ${b}_{\sub{\mu}',\sub{\alpha}'}^*$ is in $S(\h)$. 
Note that $a_{\mu^{(p+1)},\mu^{(p)}}^{(e_{-\alpha^{(p)}})}  = 
a_{\mu^{(p)},\mu^{(p-1)}}^{(e_{-\alpha^{(p-1)}})}
=1$, and so
\begin{align*}  
\hc({b}_{\sub{\mu},\sub{\alpha}}^*) = 
\hc(b_{(\sub{\mu}',\sub{\alpha}')}^* 
 e_{\alpha^{(p-1)}} e_{-\alpha^{(p-1)}} 
 b_{(\sub{\mu}'',\sub{\alpha}'')}^*)
 = \hc(b_{(\sub{\mu}',\sub{\alpha}')}^* 
 \c{\alpha}^{(p-1)} b_{(\sub{\mu}'',\sub{\alpha}'')}^*) 
 = \c{\alpha}^{(p-1)} \hc(b_{(\sub{\mu},\sub{\alpha})^{\# p}}^*),
 \end{align*}  
since $\alpha^{(p)}=-\alpha^{(p-1)}.$
Since 
$$\ev_\rho(\c{\alpha}^{(p-1)}) = \langle \rho, \c{\alpha}^{(p-1)} \rangle 
= \he(\c{\alpha}^{(p-1)}),$$
we get the expected equality:  
$$\wt(\sub{\mu},\sub{\alpha})=
\he(\c{\alpha}^{(p-1)})\,\wt(\sub{\mu},\sub{\alpha})^{\# p}.$$
If $p=2$, we have 
\begin{align*}  
b_{\sub{\mu},\sub{\alpha}}^* =
e_{\alpha^{(1)}} e_{-\alpha^{(1)}} b_{\sub{\mu}'',\sub{\alpha}''}^*
=(e_{-\alpha^{(1)}} e_{\alpha^{(1)}} + \c{\alpha}^{(1)}) 
b_{\sub{\mu}'',\sub{\alpha}''}^*, 
 \end{align*} 
where $(\sub{\mu}'',\sub{\alpha}'')$ is a weighted path of length 
$m-2$. Then we conclude as in the first situation. 

\noindent
(b) Assume that we are not in one of the situations of (a). 
Write 
$$(\sub{\mu},\sub{\alpha}) = (\sub{\mu}',\sub{\alpha}') \star 
((\mu^{(p-1)},\mu^{(p)}),\alpha^{(p-1)}) \star ((\mu^{(p)},\mu^{(p+1)}),
\alpha^{(p)}) \star (\sub{\mu}'',\sub{\alpha}'')$$
as in Part 1.
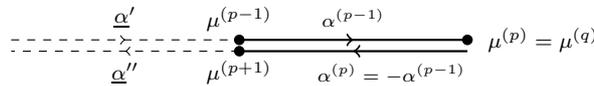
\begin{center}
\begin{tikzpicture}
 \draw[dashed,
        decoration={markings, mark=at position 0.5 with {\arrow{>}}},
        postaction={decorate}
        ]
        (0,0) -- (3,0);\node at (1.5,0.3) {{\footnotesize $\sub{\alpha}'$}}; 
  \draw[ thick,
        decoration={markings, mark=at position 0.5 with {\arrow{>}}},
        postaction={decorate}
        ]
        (3,0) -- (6,0); \node at (4.5,0.3) {{\scriptsize ${\alpha}^{(p-1)}$}};
  \draw[ thick,
        decoration={markings, mark=at position 0.5 with {\arrow{>}}},
        postaction={decorate}
        ]
        (6,-0.15) -- (3,-0.15); \node at (5,-0.45) {{\scriptsize ${\alpha}^{(p)}=-\alpha^{(p-1)}$}};
   \draw[ dashed,
        decoration={markings, mark=at position 0.5 with {\arrow{>}}},
        postaction={decorate}
        ]
        (3,-0.15) -- (0,-0.15);\node at (1.5,-0.45) {{\footnotesize $\sub{\alpha}''$}};
 \fill (3,0) circle (0.07)node(xline)[above] {{\small $\mu^{(p-1)}$}}; 
    \fill (6,0) circle (0.07); \node at (7,0) {{\small $\mu^{(p)} = \mu^{(q)}$}};
   \fill (3,-0.15) circle (0.07) node(xline)[below] {{\small $\mu^{(p+1)}$}}; 
\end{tikzpicture}
\end{center}

Note that 
$a_{\mu^{(p)},\mu^{(p-1)}}^{(-e_{\alpha^{(p-1)}})} 
a_{\mu^{(p+1)},\mu^{(p)}}^{(-e_{\alpha^{(p)}})}=1$.  
We have
\begin{align} \label{eq:b3}
\hc(b_{(\sub{\mu},\sub{\alpha})}^*) & = 
\hc( 
b_{(\sub{\mu}',\sub{\alpha}')}^* 
e_{\alpha^{(p-1)}} e_{\alpha^{(p)}}  b_{(\sub{\mu}'',\sub{\alpha}'')}^*)
= \hc(
 b_{(\sub{\mu}',\sub{\alpha}')}^* \c{\alpha}^{(p-1)} 
 b_{(\sub{\mu}'',\sub{\alpha}'')}^*), 
 \end{align}
 since $\alpha^{(p)}=-\alpha^{(p-1)}.$  
Let $\alpha^{(s)} \in \sub{\alpha}'$ such that $i_s >0$ and $\langle \alpha^{(s)}, \c\alpha^{(p-1)} \rangle \not= 0$. 
Observe that
$$\langle \alpha^{(s)}, \c{\alpha}^{(p-1)} \rangle = \langle \mu^{(s)} - \mu^{(s+1)}, \c{\alpha}^{(p-1)}  \rangle \not= 0$$ 
if and only if $\mu^{(s+1)} = \mu^{(p-1)}$, thus $\alpha^{(s)}$ is unique (see Figure~\ref{fig:beta}). 

For all other roots $\alpha^{(t)} \in \sub{\alpha}'$ with $t \not= s$, if $i_t>0$ then 
$\langle \alpha^{(t)}, \c{\alpha}^{(p-1)} \rangle =0$, and so
$$b_{(\sub{\mu},\sub{\alpha}),t}^* \ \c{\alpha}^{(p-1)} 
= \c{\alpha}^{(p-1)} \ b_{(\sub{\mu},\sub{\alpha}),t}^*.$$ 
Otherwise,  
$b_{(\sub{\mu},\sub{\alpha}),t}^*= \s\varpi_{(\alpha^{(t)})} \in \h$ thus we also get
$b_{(\sub{\mu},\sub{\alpha}),t}^* \ \c{\alpha}^{(p-1)} = \c{\alpha}^{(p-1)} \ b_{(\sub{\mu},\sub{\alpha}),t}^*$. We see that  
$\c{\alpha}^{(p-1)}$ commutes with all roots in $\sub{\alpha}'$, except with $\alpha^{(s)}$.

Write $$(\sub{\mu}',\sub{\alpha}') = (\sub{\mu}'_1,\sub{\alpha}'_1) 
\star ((\mu^{(s)},\mu^{(s+1)}),\alpha^{(s)}) 
\star (\sub{\mu}'_2,\sub{\alpha}'_2),$$ 
where $(\sub{\mu}'_2,\sub{\alpha}'_2)$ is a concatenation of loops 
and $ (\sub{\mu}'_1,\sub{\alpha}'_1)$ has length $s-1$. 
Note that the weighted path $(\sub{\mu}'_1,\sub{\alpha}'_1)$ may be trivial.  
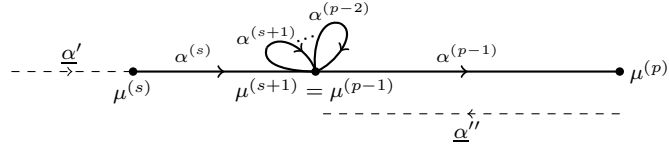
\begin{figure}[h]
\centering
  \begin{tikzpicture} [scale=0.8]

\draw[dashed,
        decoration={markings, mark=at position 0.5 with {\arrow{>}}},
        postaction={decorate}
        ]
        (0,0) -- (2,0);\node at (1,0.3) {{\footnotesize $\sub{\alpha}'$}}; 
  \draw[ thick,
        decoration={markings, mark=at position 0.5 with {\arrow{>}}},
        postaction={decorate}
        ]
        (2,0) -- (5,0); \node at (3,0.3) {{\scriptsize ${\alpha}^{(s)}$}};
   \draw[ thick,
        decoration={markings, mark=at position 0.5 with {\arrow{>}}},
        postaction={decorate}
        ]
        (5,0) to (10,0);\node at (7.5,0.3) {{\scriptsize ${\alpha}^{(p-1)}$}};
   \draw[dashed,
        decoration={markings, mark=at position 0.5 with {\arrow{>}}},
        postaction={decorate}
        ]
        (10,-0.7) -- (5,-0.7);\node at (7.5,-1.0) {{\footnotesize $\sub{\alpha}''$}};
 \fill (2,0) circle (0.07)node(xline)[below] {{\small $\mu^{(s)}$}}; 
    \fill (5,0) circle (0.07); \node at (5,-0.3) {{\small $\mu^{(s+1)} = \mu^{(p-1)}$}};
   \fill (10,0) circle (0.07); \node at (10.5,0) {{\small $\mu^{(p)}$}};
\draw[thick] (5,0) to [out=180,in=250] (4.2,0.4);
 \draw[thick, decoration={markings, mark=at position 0.625 with {\arrow{>}}},
         postaction={decorate}] (4.2,0.4) to [out=50,in=105] (5,0);
 \node at (4.2,0.6) {{\scriptsize ${\alpha}^{(s+1)}$}};
 
 \draw[thick] (5,0) to [out=100,in=165] (5.4,0.8);
 \draw[thick, decoration={markings, mark=at position 0.5 with {\arrow{>}}},
         postaction={decorate}] (5.4,0.8) to [out=335,in=20] (5,0);        
 \node at (5.4,1) {{\scriptsize ${\alpha}^{(p-2)}$}};
\draw[dotted, thick] (4.7,0.5) to (5,0.6);
\end{tikzpicture}
  \caption{Path in case (2) (b) }
  \label{fig:beta}
\end{figure}
\noindent
Note that $a_{\mu^{(s+1)},\mu^{(s)}}^{(-e_{\alpha^{(s)}})}=1$. We get 
\begin{align*} 
& b_{(\sub{\mu}',\sub{\alpha}')}^* \c{\alpha}^{(p-1)}
  b_{(\sub{\mu}'',\sub{\alpha}'')}^* =
  b_{(\sub{\mu}'_1,\sub{\alpha}'_1)}^*  e_{\alpha^{(s)}} 
  b_{(\sub{\mu}'_2,\sub{\alpha}'_2)}^* \c{\alpha}^{(p-1)}
  b_{(\sub{\mu}'',\sub{\alpha}'')}^* & \\
& \quad = \c{\alpha}^{(p-1)}
  b_{(\sub{\mu}'_1,\sub{\alpha}'_1)}^*  
  e_{\alpha^{(s)}}b_{(\sub{\mu}'_2,\sub{\alpha}'_2)}^*
  b_{(\sub{\mu}'',\sub{\alpha}'')}^*
  - \langle \alpha^{(s)}, \c\alpha^{(p-1)} \rangle
  b_{(\sub{\mu}'_1,\sub{\alpha}'_1)}^*  
  e_{\alpha^{(s)}}b_{(\sub{\mu}'_2,\sub{\alpha}'_2)}^*
  b_{(\sub{\mu}'',\sub{\alpha}'')}^*
  = (\c{\alpha}^{(p-1)}+1 )
  b_{(\sub{\mu},\sub{\alpha})^{\#p}}^*,
  \end{align*}
since $\c\alpha^{(p-1)}$ commutes with all roots of $\sub{\alpha}'_1$ and $\sub{\alpha}'_2$,
$\langle \alpha^{(s)}, \c\alpha^{(p-1)} \rangle = -1 $ and $$(\sub{\mu}'_1,\sub{\alpha}'_1) \star ((\mu^{(s)},\mu^{(s+1)}),
\alpha^{(s)}) 
\star (\sub{\mu}'_2,\sub{\alpha}'_2) \star (\sub{\mu}'',\sub{\alpha}'') = (\sub{\mu},\sub{\alpha})^{\# p}.$$ 
Hence,
$$\wt(\sub{\mu},\sub{\alpha})= 
(\he(\c{\alpha}^{(p-1)}) +1)  
\wt(\sub{\mu},\sub{\alpha})^{\# p}.$$
(3) Write 
$$(\sub{\mu},\sub{\alpha}) = (\sub{\mu}',\sub{\alpha}') 
 \star ((\mu^{(p-1)},\mu^{(p)}),\alpha^{(p-1)})  \star 
((\mu^{(p)},\mu^{(p+1)}),\alpha^{(p)}) 
\star (\sub{\mu}'',\sub{\alpha}''),$$
where $(\sub{\mu}',\sub{\alpha}')$  
and $(\sub{\mu}'',\sub{\alpha}'')$ 
have length $p - 2$ and $m - p$, respectively.  

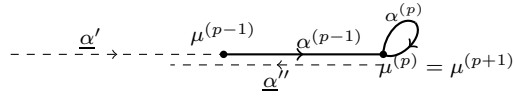
\begin{center} 
\begin{tikzpicture}[scale=0.7]
 \draw[dashed,
        decoration={markings, mark=at position 0.5 with {\arrow{>}}},
        postaction={decorate}
        ]
        (0,0) -- (4,0);\node at (1.5,0.3) {{\footnotesize $\sub{\alpha}'$}}; 
  \draw[ thick,
        decoration={markings, mark=at position 0.5 with {\arrow{>}}},
        postaction={decorate}
        ]
        (4,0) -- (7,0); \node at (6,0.3) {{\footnotesize ${\alpha}^{(p-1)}$}};
\draw[thick] (7,0) to [out=90,in=150] (7.6,0.6);
 \draw[thick, decoration={markings, mark=at position 0.5 with {\arrow{>}}},
         postaction={decorate}] (7.6,0.6) to [out=310,in=330] (7,0);        
 \node at (7.4,0.8) {{\scriptsize ${\alpha}^{(p)}$}};
   \draw[ dashed,
        decoration={markings, mark=at position 0.5 with {\arrow{>}}},
        postaction={decorate}
        ]
        (7,-0.2) -- (3,-0.2);\node at (5,-0.5) {{\footnotesize $\sub{\alpha}''$}};
 \fill (4,0) circle (0.07)node(xline)[above] {{\small $\mu^{(p-1)}$}}; 
    \fill (7,0) circle (0.07); \node at (8.2,-0.2) {{\small $\mu^{(p)} = \mu^{(p+1)}$}}; 
\end{tikzpicture}
\end{center} 
Let ${\rm supp} (\alpha)$ be the {\em support} of $\alpha \in \Delta$, that is, the set of 
$\beta \in \Pi$ such that $\langle \alpha, \c{\varpi}_\beta \rangle \not= 0$.
We have
\begin{align*}
b_{(\sub{\mu},\sub{\alpha})}^*  & =  
a_{\mu^{(p-1)},\mu^{(p)}}^{(-e_{\alpha^{(p-1)}})}
\langle \mu^{(p)}, \c{\alpha}^{(p)} \rangle 
 b_{(\sub{\mu}',\sub{\alpha}')}^*   e_{\alpha^{(p-1)}} 
\c{\varpi}_{\alpha^{(p)}} b_{(\sub{\mu}'',\sub{\alpha}'')}^* 
=\langle \mu^{(p)}, \c{\alpha}^{(p)} \rangle 
(\c{\varpi}_{\alpha^{(p)}} -  
\langle \alpha^{(p-1)} , \c{\varpi}_{\alpha^{(p)}} \rangle ) 
b_{(\sub{\mu},\sub{\alpha})^{\# p}}^*,
\end{align*}
\noindent
since
$ \c{\varpi}_{\alpha^{(p)}}$ 
commutes with all roots in $\sub{\alpha}'$. 
Indeed, the support of $\alpha'^{(j)}$, for 
$j=1,\ldots,p-2$, does not contain the simple root 
$\alpha^{(p)}$. 
  
If $\langle \mu^{(p)}, \c{\alpha}^{(p)} \rangle =1$ 
then $\langle \alpha^{(p-1)}, 
\c{\varpi}_{\alpha^{(p)}} \rangle =0,$ 
and so 
$$\wt(\sub{\mu},\sub{\alpha})=  
\langle \rho , \c{\varpi}_{\alpha^{(p)}} \rangle 
\wt (\sub{\mu},\sub{\alpha})^{\# p}.$$

If $\langle \mu^{(p)}, \c{\alpha}^{(p)}\rangle =-1$ 
then 
$\langle \alpha^{(p-1)} , \c{\varpi}_{\alpha^{(p)}} \rangle =1,$ 
and so
$$\wt(\sub{\mu},\sub{\alpha})= 
(-\langle \rho, \c{\varpi}_{\alpha^{(p)}}  \rangle +1 )
\wt(\sub{\mu},\sub{\alpha})^{\# p}.$$ \qed
\end{proof}

As a direct consequence of Lemma \ref{Lem0:cut}, 
we get the following result. 

\begin{proposition} \label{Pro0:cut}
Let $\mu \in P(\delta)$, $m \in \Z_{>1}$,  
$\sub{i} \in {\Z}^m_{\succ \sub{0}}$ and 
$(\sub{\mu},\sub{\alpha}) \in \hat{\P}_m(\mu)_{\sub{i}}$.  
In particular, $\sum_{j=1}^m i_j =0$. 
Set $p:=p(\sub{i})$ and $q:=q(\sub{i})$. 
Then for some scalar $K_{\sub{\mu},\sub{\alpha}}$, we have:  
\begin{align*}  
\wt(\sub{\mu},\sub{\alpha})
= K_{\sub{\mu},\sub{\alpha}} \wt(\sub{\mu},\sub{\alpha})^{\# p}.
\end{align*}
In particular, $\wt(\sub{\mu},\sub{\alpha})=0$ 
if $ \wt(\sub{\mu},\sub{\alpha})^{\# p}=0$. 
\end{proposition} 

Next theorem constitutes an important step towards the proof of Theorem \ref{theorem:main2}.

\begin{theorem}   \label{corollary:cut}
Let $m \in \Z_{> 0}$, $\mu \in P(\delta)$ 
and $(\sub{\mu},\sub{\alpha}) \in \hat{\P}_m(\mu)$. 
Assume that for some $i \in \{ 1,\ldots, m\}$, 
$\mu^{(i)} \succ  \mu$. 
Then $\wt(\sub{\mu},\sub{\alpha}) = 0$. 
\end{theorem}

Thanks to Theorem \ref{corollary:cut}, it will enough 
in many situations to consider only weighted paths 
$(\sub{\mu},\sub{\alpha}) \in  \hat{\P}_m(\mu)$ 
such that $\mu^{(i)} \preccurlyeq \mu$ for any $i \in \{1,\ldots,m\}$. 

\begin{proof}[Proof of Theorem \ref{corollary:cut}]  Let $(\sub{\mu},\sub{\alpha})$ be as in the theorem and set $\sub{i} := \he(\sub{\mu})$.
First of all, we observe that if $\sub{i} \in {\Z}^m_{\prec \sub{0}}$, 
then $\hc({b}_{\sub{\mu},\sub{\alpha}}^*)=0$ and so 
the statement is clear. 

We prove the statement by induction on $m$. 
Necessarily, $m \geqslant 2$. 
If $m=2$, then the hypothesis 
implies that $\sub{i} \in {\Z}^m_{\prec \sub{0}}$ and so the statement 
is true. 

Assume $m \geqslant 3$ and 
that for all weighted paths $(\sub{\mu'},\sub{\alpha'}) \in \hat{\P}_{m'}(\mu)$, 
with $m' < m$, such that for some $i' \in \{ 1,\ldots, m'\}$, 
$\mu'^{(i')} \succcurlyeq  \mu$, we have 
$\wt(\sub{\mu'},\sub{\alpha'}) = 0$. 
If $\sub{i} \in {\Z}^m_{\prec \sub{0}}$ 
the statement is true. 
So we can assume that  
$\sub{i} \in {\Z}^m_{\succcurlyeq \sub{0}}$. 
Then necessarily $\sub{i} \in {\Z}^m_{\succ \sub{0}}$. 
Let $p,q$ be as in Proposition \ref{Pro0:cut}.  
We observe that the weighted path $(\sub{\mu},\sub{\alpha})^{\# p}$ 
satisfies the hypothesis of the theorem and it is not empty.  
Hence by our induction hypothesis and Proposition \ref{Pro0:cut}, 
we get the statement.  
Notice that for $m=3$, $\sub{i}$ is necessarily of the form 
$(i_1,i_2,i_3)$ with $i_1 >0$, $i_2 <0$, $i_3>0$ and 
$i_1+i_2+i_3=0$ so $(\sub{\mu},\sub{\alpha})^{\# p}$ has length 
$2,$ and we can indeed apply the induction hypothesis. \qed
\end{proof}

For $\eps \in \{ -1,0,1\}$, denote by $\eps \Z_{\succcurlyeq \sub{0}}$ the set $\{ \eps n \ | \ n \in \Z_{\succcurlyeq \sub{0}}\}$.
Denote by $\P_m$ (respectively, $\hat{\P}_m$)
the union of all sets $\P_m(\mu)$ 
(respectively, $\hat{\P}_m(\mu)$) for $\mu$ running through $P(\delta)$.

\begin{definition}     \label{definition:equivalenceA} We define an equivalence relation $\sim$  
on  $\hat{\P}_m$ by induction on $m$ as follows. 
\begin{enumerate}
\item If $m=1$, there is only one equivalence 
class represented by the trivial path of length 0. 
\item For $m=2$, we say that two paths $(\sub{\mu},\sub{\alpha})$ and $(\sub{\mu'},\sub{\alpha'})$ in $\hat{\P}_m$,
with $\he(\sub{\mu}) = \sub{i}$ and $\he(\sub{\mu}')=\sub{i}'$,
are {\em equivalent}, 
if there is $(\eps_1, \eps_2) \in \{ 0,1 \}^2$ such that 
$\sub{i} \in \eps_1\Z_{\succcurlyeq \sub{0}} \times \eps_2\Z_{\succcurlyeq \sub{0}}  $ and 
$\sub{i}' \in \eps_1\Z_{\succcurlyeq \sub{0}} \times \eps_2\Z_{\succcurlyeq \sub{0}} $.
\item For $m \geqslant 3$, 
we say that we say that two paths $(\sub{\mu},\sub{\alpha})$ and $(\sub{\mu'},\sub{\alpha'})$ in $\hat{\P}_m$,
with $\he(\sub{\mu}) = \sub{i}$ and $\he(\sub{\mu}')=\sub{i}'$, are equivalent, 
if the following conditions are satisfied: 
\begin{enumerate}
\item 
there is $(\eps_1,\ldots,\eps_m) \in \{ -1,0,1\}^m$ such that 
$\sub{i} \in \prod_{i=1}^m \eps_i \Z_{\succcurlyeq \sub{0}}$ and 
$\sub{i}' \in \prod_{i=1}^m \eps_i \Z_{\succcurlyeq \sub{0}}$, 
\item the paths $(\sub{\mu},\sub{\alpha})^{\# p(\sub{i})}$ and 
$(\sub{\mu'},\sub{\alpha'})^{\# p(\sub{i}')}$ are equivalent.
\end{enumerate}
\end{enumerate}
For $(\sub{\mu},\sub{\alpha}) \in \hat{\P}_m$, denote 
by $[(\sub{\mu},\sub{\alpha})]$ the equivalence class 
of $(\sub{\mu},\sub{\alpha})$ in $\hat{\P}_m$ with respect to $\sim$, 
and denote by $\E_m$ the set of equivalence classes. 
\end{definition}

We observe that the equivalence class of $(\sub{\mu},\sub{\alpha}) \in \hat{\P}_m$ 
only depends on the sequence $\he(\sub{\mu})$. 
Hence, by abuse of notation we will often write $[\he(\sub{\mu})]$ 
for the class of $(\sub{\mu},\sub{\alpha})$ 
(this will be not anymore the case in type $C$). 

\begin{example} \noindent
\begin{enumerate}
\item Assume $m=2$. 
We have only two equivalence classes: 
$[0,0]$ and $[1,-1].$ 

\noindent
We represent below weighted paths whose heights are in $[0,0]$ and $[1,-1]$ 
respectively: 

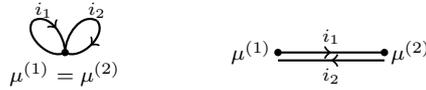
\begin{center} 
\begin{tikzpicture}[scale=0.7]
\fill (0,0) circle (0.07)node(xline)[below] {{\small $\mu^{(1)}=\mu^{(2)}$}}; 
\draw[thick] (0,0) to [out=210,in=230] (-0.6,0.6);
 \draw[thick, decoration={markings, mark=at position 0.5 with {\arrow{>}}},
         postaction={decorate}] (-0.6,0.6) to [out=30,in=90] (0,0);        
\node at (-0.4,0.8) {{\scriptsize $i_1$}};
\draw[thick] (0,0) to [out=90,in=150] (0.6,0.6);
 \draw[thick, decoration={markings, mark=at position 0.5 with {\arrow{>}}},
         postaction={decorate}] (0.6,0.6) to [out=310,in=330] (0,0);        
\node at (0.6,0.8) {{\scriptsize $i_2$}};
\draw[thick,
        decoration={markings, mark=at position 0.5 with {\arrow{>}}},
        postaction={decorate}
        ]
        (4,0) -- (6,0);\node at (5,0.3) {{\scriptsize $i_1$}};
  \draw[ thick,
        decoration={markings, mark=at position 0.5 with {\arrow{>}}},
        postaction={decorate}
        ]
        (6,-0.15) -- (4,-0.15); \node at (5,-0.45) {{\scriptsize $i_2$}};
 \node at (3.5,0) {{\small $\mu^{(1)}$}};
 \node at (6.5,0) {{\small $\mu^{(2)}$}};
 \fill (4,0) circle (0.07); 
    \fill (6,0) circle (0.07); 
\end{tikzpicture}
\end{center} 

\item Assume $m=3$. 
We have six equivalence classes: 
\begin{align*}
[0,0,0], \quad [0,1,-1], \quad [2,-1,-1], \quad [1,-1,0], \quad [1,0,-1],  \quad [1,1,-2].
\end{align*} 
We represent below weighted paths whose heights 
are in the above respective classes. 
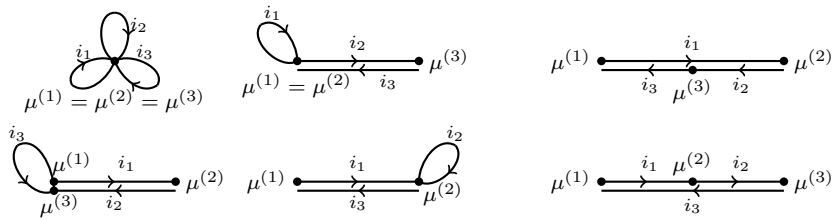
\begin{center} 
\begin{tikzpicture} [scale=0.8]
\fill (0,0) circle (0.07);\node at (0,-0.65) {{\footnotesize $\mu^{(1)}=\mu^{(2)} =\mu^{(3)}$}}; 
\draw[thick] (0,0) to [out=250,in=325] (-0.7,-0.4);
 \draw[thick, decoration={markings, mark=at position 0.5 with {\arrow{>}}},
         postaction={decorate}] (-0.7,-0.4) to [out=120,in=150] (0,0);        
\node at (-0.5,0.15) {{\scriptsize $i_1$}};
\draw[thick] (0,0) to [out=155,in=180] (0,0.8);
 \draw[thick, decoration={markings, mark=at position 0.5 with {\arrow{>}}},
         postaction={decorate}] (0,0.8) to [out=360,in=25] (0,0);        
\node at (0.4,0.6) {{\scriptsize $i_2$}};
\draw[thick] (0,0) to [out=30,in=60] (0.7,-0.4);
 \draw[thick, decoration={markings, mark=at position 0.5 with {\arrow{>}}},
         postaction={decorate}] (0.7,-0.4) to [out=215,in=290] (0,0);        
\node at (0.5,0.15) {{\scriptsize $i_3$}};
 \fill (3,0) circle (0.07);
  \node at (3,-0.4) {{\footnotesize $\mu^{(1)}=\mu^{(2)}$}};
\draw[thick] (3,0) to [out=210,in=230] (2.4,0.6);
 \draw[thick, decoration={markings, mark=at position 0.5 with {\arrow{>}}},
         postaction={decorate}] (2.4,0.6) to [out=30,in=90] (3,0);        
\node at (2.6,0.8) {{\scriptsize $i_1$}};
\draw[thick,
        decoration={markings, mark=at position 0.5 with {\arrow{>}}},
        postaction={decorate}
        ]
        (3,0) -- (5,0);\node at (4,0.3) {{\scriptsize $i_2$}};
  \draw[ thick,
        decoration={markings, mark=at position 0.5 with {\arrow{>}}},
        postaction={decorate}
        ]
        (5,-0.15) -- (3,-0.15); \node at (4.5,-0.35) {{\scriptsize $i_3$}};
 \node at (5.5,0) {{\footnotesize $\mu^{(3)}$}}; 
    \fill (5,0) circle (0.07); 
 \fill (8,0) circle (0.07);
  \node at (7.5,0) {{\footnotesize $\mu^{(1)}$}};
\draw[thick,
        decoration={markings, mark=at position 0.5 with {\arrow{>}}},
        postaction={decorate}
        ]
        (8,0) -- (11,0);\node at (9.5,0.3) {{\scriptsize $i_1$}};
\fill (11,0) circle (0.07);
 \node at (11.5,0) {{\footnotesize $\mu^{(2)}$}};
 \draw[ thick,
        decoration={markings, mark=at position 0.5 with {\arrow{>}}},
        postaction={decorate}
        ]
        (11,-0.15) -- (9.5,-0.15); \node at (10.3,-0.4) {{\scriptsize $i_2$}};
\fill (9.5,-0.15) circle (0.07);
 \node at (9.5,-0.5) {{\footnotesize $\mu^{(3)}$}};
 \draw[ thick,
        decoration={markings, mark=at position 0.5 with {\arrow{>}}},
        postaction={decorate}
        ]
        (9.5,-0.15) -- (8,-0.15); \node at (8.8,-0.4) {{\scriptsize $i_3$}};
\fill (-1,-2) circle (0.07);\node at (-0.7,-1.7) {{\footnotesize $\mu^{(1)}$}};
\draw[thick,
        decoration={markings, mark=at position 0.5 with {\arrow{>}}},
        postaction={decorate}
        ]
        (-1,-2) -- (1,-2);\node at (0.2,-1.7) {{\scriptsize $i_1$}};
\fill (1,-2) circle (0.07); \node at (1.5,-2) {{\footnotesize $\mu^{(2)}$}}; 
  \draw[ thick,
        decoration={markings, mark=at position 0.5 with {\arrow{>}}},
        postaction={decorate}
        ]
        (1,-2.15) -- (-1,-2.15); \node at (0,-2.35) {{\scriptsize $i_2$}};
 \node at (-0.9,-2.4) {{\footnotesize $\mu^{(3)}$}}; 
    \fill (-1,-2.15) circle (0.07);
\draw[thick, decoration={markings, mark=at position 0.5 with {\arrow{<}}},
         postaction={decorate}] (-1,-2.15) to [out=210,in=230] (-1.6,-1.4);
 \draw[thick] (-1.6,-1.4) to [out=30,in=90] (-1,-2.15);        
\node at (-1.6,-1.2) {{\scriptsize $i_3$}};
 \fill (3,-2) circle (0.07);
  \node at (2.5,-2) {{\footnotesize $\mu^{(1)}$}};
  \draw[thick,
        decoration={markings, mark=at position 0.5 with {\arrow{>}}},
        postaction={decorate}
        ]
        (3,-2) -- (5,-2);\node at (4,-1.7) {{\scriptsize $i_1$}};
\fill (5,-2) circle (0.07); 
 \node at (5.4,-2.2) {{\footnotesize $\mu^{(2)}$}}; 
\draw[thick] (5,-2) to [out=90,in=150] (5.6,-1.4);
 \draw[thick, decoration={markings, mark=at position 0.5 with {\arrow{>}}},
         postaction={decorate}] (5.6,-1.4) to [out=310,in=330] (5,-2);        
\node at (5.6,-1.2) {{\scriptsize $i_2$}}; 
  \draw[ thick,
        decoration={markings, mark=at position 0.5 with {\arrow{>}}},
        postaction={decorate}
        ]
        (5,-2.15) -- (3,-2.15); \node at (4,-2.35) {{\scriptsize $i_3$}};
 \fill (8,-2) circle (0.07);
  \node at (7.5,-2) {{\footnotesize $\mu^{(1)}$}};
\draw[thick,
        decoration={markings, mark=at position 0.5 with {\arrow{>}}},
        postaction={decorate}
        ]
        (8,-2) -- (9.5,-2);\node at (8.8,-1.7) {{\scriptsize $i_1$}};
\fill (9.5,-2) circle (0.07);
 \node at (9.5,-1.7) {{\footnotesize $\mu^{(2)}$}};
 \draw[ thick,
        decoration={markings, mark=at position 0.5 with {\arrow{>}}},
        postaction={decorate}
        ]
        (9.5,-2) -- (11,-2); \node at (10.3,-1.7) {{\scriptsize $i_2$}};
\fill (11,-2) circle (0.07);
 \node at (11.5,-2) {{\footnotesize $\mu^{(3)}$}};
 \draw[ thick,
        decoration={markings, mark=at position 0.5 with {\arrow{>}}},
        postaction={decorate}
        ]
        (11,-2.15) -- (8,-2.15); \node at (9.5,-2.4) {{\scriptsize $i_3$}};
 \end{tikzpicture}
\end{center} 

\item Assume $m = 4$. 
The following four weighted paths have pairwise 
not equivalent heights. 
The heights of $\I_1$ and $\I_2$ (also $\I_3$ and $\I_4$) 
satisfy condition (a) of Definition \ref{definition:equivalenceA} 
but not the condition (b), so they are not equivalent. Here, 
$\I_i$ refers as the equivalent class of the corresponding path.
\begin{center} 
\begin{tikzpicture}[scale=0.8]
\draw[thick] (0,0) to (0.8,0);
\node at (1,0) {{\scriptsize $i_1$}};
\draw[thick,->](1.2,0) to (2,0);
\draw[thick] (2,-0.2) to (1.2,-0.2);
\node at (1,-0.2) {{\scriptsize $i_2$}};
\draw[thick,->](0.8,-0.2) to (0,-0.2);       
\draw[thick] (0,-0.4) to (1.3,-0.4);
\node at (1.5,-0.4) {{\scriptsize $i_3$}};
\draw[thick,->](1.7,-0.4) to (3,-0.4);
\draw[thick] (3,-0.6) to (1.7,-0.6);
\node at (1.5,-0.6) {{\scriptsize $i_4$}};
\draw[thick,->](1.3,-0.6) to (0,-0.6);        
\node at (1.5,-1.2) {{\small $\I_1$}};
\draw[thick] (4,0) to (4.8,0);
\node at (5,0) {{\scriptsize $i_1$}};
\draw[thick,->](5.2,0) to (6,0);
\draw[thick] (6,-0.2) to (5.7,-0.2);
\node at (5.5,-0.2) {{\scriptsize $i_2$}};
\draw[thick,->](5.3,-0.2) to (5,-0.2);       
\draw[thick] (5,-0.4) to (5.8,-0.4);
\node at (6,-0.4) {{\scriptsize $i_3$}};
\draw[thick,->](6.2,-0.4) to (7,-0.4);
\draw[thick] (7,-0.6) to (5.7,-0.6);
\node at (5.5,-0.6) {{\scriptsize $i_4$}};
\draw[thick,->](5.3,-0.6) to (4,-0.6);        
\node at (5.5,-1.2) {{\small $\I_2$}};
\draw[thick] (-1,-2) to (-0.2,-2);
\node at (0,-2) {{\scriptsize $i_1$}};
\draw[thick,->](0.2,-2) to (1,-2);
\draw[thick] (1,-2) to (1.8,-2);
\node at (2,-2) {{\scriptsize $i_2$}};
\draw[thick,->](2.2,-2) to (3,-2);       
\draw[thick] (3,-2.2) to (1.7,-2.2);
\node at (1.5,-2.2) {{\scriptsize $i_3$}};
\draw[thick,->](1.3,-2.2) to (0,-2.2);
\draw[thick] (0,-2.2) to (-0.3,-2.2);
\node at (-0.5,-2.2) {{\scriptsize $i_4$}};
\draw[thick,->](-0.7,-2.2) to (-1,-2.2);        
\node at (1,-2.8) {{\small $\I_3$}};
\draw[thick] (4,-2) to (4.8,-2);
\node at (5,-2) {{\scriptsize $i_1$}};
\draw[thick,->](5.2,-2) to (6,-2);
\draw[thick] (6,-2) to (6.8,-2);
\node at (7,-2) {{\scriptsize $i_2$}};
\draw[thick,->](7.2,-2) to (8,-2);       
\draw[thick] (8,-2.2) to (7.7,-2.2);
\node at (7.5,-2.2) {{\scriptsize $i_3$}};
\draw[thick,->](7.3,-2.2) to (7,-2.2);
\draw[thick] (7,-2.2) to (5.7,-2.2);
\node at (5.5,-2.2) {{\scriptsize $i_4$}};
\draw[thick,->](5.3,-2.2) to (4,-2.2);        
\node at (6,-2.8) {{\small $\I_4$}};
 \end{tikzpicture}
\end{center}

The equivalence classes in ${\Z}^4$ are
\begin {align*}
 & [0,0,0,0],\ [0,0,1,-1],\ [0, 1, 0,-1],\ [0,1,-1,0],\ [1,-1,0,0],[1,0,-1,0],\\
 & [1,0,0,-1],\ [2,-1,-1,0],\ [2,-1,0,-1], \ [1,1,0,-2], \ [1,0,1,-2], \ [1,1,-2,0],\\
 & [0,2,-1,-1], \ [0,1,1,-2], \ [2,0,-1,-1],\ [1,1,-1,-1], \ [1,-1,1,-1], \ [1,1,1,-3],\\
 & [1,2,-1,2], \ [2,1,-2,-1], \ [2,-1,1,-2], \ [3,-1,-1,-1].
\end {align*}
They are indeed the only pairwise non-equivalent 
classes. The verifications are left to the reader. 
\end{enumerate}
\end{example}

Let $m \in \Z_{> 0}$ and $\I \in  \E_m$.  
The number $n$ of zero values of $\sub{i}:=\he(\sub{\mu})$ 
does not depend on 
$(\sub{\mu'},\sub{\alpha'})$ in $\I$.  
We will say that $\I$ {\em has $n$ zeroes}.  
The positions of the zeroes only depend on $\I$. 
If $n =0$, we will say that {\em $\I$ has no zero}.   
By definition, the position $p(\sub{i})$ of the first turning 
back does not depend on $\sub{i} \in \I$.  
Similarly, the integer $q(\sub{i})$ does not depend on $\sub{i} \in \I$. 
Denote by $p(\I)$ and $q(\I)$ these integers. 
Furthermore, the class of $(\sub{\mu},\sub{\alpha})^{\# p(\sub{i})}$ only 
depends on $\I$. 
Denote by $\I^{\# }$ this equivalence class. 
For $m' \in \Z_{>0}$, denote by $\ell(\I'):=m'$the length of $\I'$ for some equivalence class 
$\I' \in \E_{m'}$ . 
We have $\ell(\I)=m$, $\ell(\I^\#)=m-1$ if $i_p+i_{p-1}\not=0$, 
$\ell(\I^\#)=m-2$ if $i_p+i_{p-1}=0$, etc. 

We begin by studying the elements of $\E_m$ without zero. 
If $\I$ has no zero, note that, necessarily, $p(\I) = q(\I)$ 
and $\I^\# $ has no zero, too.  
\begin{remark} \label{Rem:height}
If $\I$ has no zero then for any 
weighted path $(\sub{\mu},\sub{\alpha}) \in \hat{\P}_m(\mu)$, 
$\mu \in P(\delta)$, such that $\he(\sub{\mu}) \in \I$, 
we have 
$\he(\c{\alpha}^{(j)}) =i_j$ for all $j=1,\ldots,m,$ 
since $\g$ has type $A$, where $\he(\sub{\mu}) = (i_1, \ldots, i_m)$.
\end{remark}
Remark \ref{Rem:height} will be used repeatedly in the sequel.

\begin{lemma}     \label{Lem:weight1} 
Let $\I \in {\E}_m$ without zero, and set $p:=p(\I)$. 
\begin{enumerate} 
\item There is a polynomial $A_{\I} \in  \C[X_1,\ldots,X_{m-1}]$ 
of total degree $\leqslant \lfloor \frac{m}{2}\rfloor $ 
such that 
for all $\sub{i} =(i_1,\ldots,i_m) \in \I$ and   
for all $(\sub{\mu},\sub{\alpha}) \in \hat{\P}_m$ 
with $\he(\sub{\mu}) = \sub{i}$, 
$$\wt(\sub{\mu},\sub{\alpha}) = A_{\I} (i_1, \ldots,i_{m-1}).$$
(Here, $\lfloor x\rfloor$ denote the largest integer $\leqslant x$.)
Moreover, $A_{\I}$ is a sum of monomials of the 
form $X_{j_1}\ldots X_{j_l}$, $1 \leqslant j_1 < \cdots < j_{l} < m$. 
\item The polynomial $A_{\I}$ is defined by 
induction as follows.  We have $A_{[1,-1]}(X_1) = X_1$, 
$A_{[2,-1,-1]}(X_1,X_2) = X_1+ X_2$, 
$A_{[1,1,-2]}(X_1,X_2) = X_1$, 
and for $m\ge 4$, 
\begin{enumerate}
\item 
if $i_{p-1}+i_p \not= 0$, then 
$$
A_{\I} (X_1,\ldots,X_{m-1}) 
= A_{\I^\# }(X_1,\ldots,X_{p-2},X_{p-1}+X_{p},\ldots,X_{m-1}),
$$
\item 
if $i_{p-1}+i_p = 0$ and $p=2$, then 
$$
A_{\I} (X_1,\ldots,X_{m-1}) = 
X_{p-1}A_{\I^\# }(X_1,\ldots,X_{p-2},X_{p+1},\ldots,X_{m-1}), 
$$
\item 
if $i_{p-1}+i_p = 0$ and $p >2$, then 
$$
A_{\I} (X_1,\ldots,X_{m-1}) =  
(X_{p-1}+1) A_{\I^\# }(X_1,\ldots,X_{p-2},X_{p+1},\ldots,X_{m-1}) . 
$$
\end{enumerate}
\end{enumerate}
\end{lemma}

\begin{proof} 
We prove all the statements together 
by induction on $m$. 

Assume $m=2$. 
The only equivalence class in without zero is $[1,-1]$; 
$p([1,-1])=2$ and $i_1+i_2=0$.  
Let $(\sub{\mu},\sub{\alpha}) \in \hat{\P}_2$ with $\he(\sub{\mu}) \sim (1,-1)$. 
By Lemma~\ref{Lem0:cut} (2) (a) we have 
$\wt(\sub{\mu},\sub{\alpha}) = \he(\c{\alpha}^{(1)}) = i_1.$
Hence, $A_{[1,-1]}(X_1) = X_1$ satisfies the conditions of the lemma. 

Assume $m=3$. 
There are two equivalence classes without zero: 
$[2,-1,-1]$ and $[1,1,-2]$.  
By Lemma~\ref{Lem0:cut}~(1), for any $(\sub{\mu},\sub{\alpha}) \in \hat{\P}_3$ 
with $\he(\sub{\mu}) \sim (2,-1,-1)$, we have $\wt(\sub{\mu},\sub{\alpha}) = i_1 + i_2$ 
and for any $(\sub{\mu},\sub{\alpha}) \in \hat{\P}_3$ 
with $\he(\sub{\mu}) \sim (1,1,-2)$, we have $\wt(\sub{\mu},\sub{\alpha}) = i_1$. 
Hence 
$$A_{[2,-1,-1]}(X_1,X_2) = X_1+ X_2 
\quad \text{ and }\quad A_{[1,1,-2]}(X_1,X_2) = X_1$$ 
satisfy the conditions of the lemma.  

Let $m \geqslant 4$ and assume the proposition true for any 
$m' \in \{ 2,\ldots,m-1\}$ and any $\I' \in {\E}_{m'}$. 
Let $\I \in \E_m$ and 
$(\sub{\mu},\sub{\alpha})  \in \I$. 

Assume that $i_{p-1}+i_p \not= 0$. 
Then $\he(\sub{\mu})^{\# p}$ is in $\I^\# $ and 
by Lemma~\ref{Lem0:cut}~(1), 
$\wt(\sub{\mu},\sub{\alpha}) = \wt\left((\sub{\mu},\sub{\alpha})^{\# p}
\right)$. 
By our induction hypothesis, 
there is a polynomial $A_{\I^\# } \in  \C[Y_1,\ldots,Y_{m-2}]$ 
of total degree $\leqslant \lfloor \frac{m-1}{2}\rfloor$
such that  
$$\wt(\sub{\mu},\sub{\alpha}) = \wt\left((\sub{\mu},\sub{\alpha})^{\# p(\I)}
\right)=
A_{\I^\# }(i_1,\ldots,i_{p-1} + i_{p},\ldots,i_{m-1}).$$  
Hence the polynomial 
$$A_{\I} (X_1,\ldots,X_{m-1}) 
:= A_{\I^\# }(X_1,\ldots,X_{p-2},X_{p-1}+X_{p},\ldots,X_{m-1})$$  
satisfies the conditions of the lemma.   

Assume that $i_{p-1}+i_p = 0$. 
Then $\he(\sub{\mu})^{\# p(\I)}$ is in $\I^\# $. 

$\ast$ If $p=2$, then 
by Lemma \ref{Lem0:cut} (2)(a), we have: 
$$\wt(\sub{\mu},\sub{\alpha}) = i_{p-1} \wt((\sub{\mu},\sub{\alpha})^{\# p(\I)}).$$  
By our induction hypothesis, 
there is a polynomial $A_{\I^\# } \in  \C[Y_1,\ldots,Y_{m-3}]$ 
of total degree $\leqslant \lfloor \frac{m-2}{2}\rfloor$ 
such that  
\begin{align*}
\wt(\sub{\mu},\sub{\alpha}) &=i_{p-1}  \wt\left(
(\sub{\mu},\sub{\alpha})^{\# p(\I)}\right) 
= i_{p - 1}  A_{\I^\# }(i_1,\ldots,i_{p - 2}, i_{p + 1},\ldots,i_{m-1}).& 
\end{align*} 
Hence the polynomial 
$$A_{\I} (X_1,\ldots,X_{m-1})
:= X_{p-1} A_{\I^\# }(X_1,\ldots,X_{p-2} , X_{p + 1},\ldots,X_{m-1})
$$ 
satisfies the conditions of the lemma.   

$\ast$ If $p>2$, then 
by Lemma \ref{Lem0:cut}(2)(b), we have: 
$$\wt(\sub{\mu},\sub{\alpha}) = (i_{p-1} + 1) \wt((\sub{\mu},\sub{\alpha})^{\# p(\I)}).$$  
By our induction hypothesis, 
there is a polynomial $A_{\I^\# } \in  \C[Y_1,\ldots,Y_{m-2}]$ 
of total degree $\leqslant \lfloor \frac{m-2}{2}\rfloor$ 
such that  
\begin{align*}
\wt(\sub{\mu},\sub{\alpha}) &=(i_{p-1} + 1) \wt\left(
(\sub{\mu},\sub{\alpha})^{\# p(\I)}\right) 
= (i_{p - 1} + 1) A_{\I^\# }(i_1,\ldots,i_{p - 2}, i_{p + 1},\ldots,i_{m-1}).& 
\end{align*} 
Hence the polynomial 
$$A_{\I} (X_1,\ldots,X_{m-1})
:= (X_{p-1} +1) A_{\I^\# }(X_1,\ldots,X_{p-2} , X_{p + 1},\ldots,X_{m-1})
$$ 
satisfies the conditions of the lemma.   \qed
\end{proof}

The following very classical result  
will be used in the proof of Lemma~\ref{Lem:T_m}.   

\begin{lemma}    \label{Lem:Newton}
Let $d \in \Z_{\geqslant 0}$ and $N\in\Z_{> 0}$. 
There is a polynomial $S_d \in \C[X]$ of degree $d+1$ such that 
$\sum\limits_{i = 1}^N i^{d} =  S_{d}(N).$ 
Namely, if $B_1,B_2,B_3,\ldots$ are the Bernoulli numbers, 
then $S_d$ is given by: 
$S_d (X)=\displaystyle{\frac{1}{d+1}} 
\sum_{j=0}^{d} \begin{pmatrix} d+1 \\ 
j \end{pmatrix} \tilde{B}_j X^{d+1-j},$ 
where $\tilde{B}_0=1$, $\tilde{B}_1=\frac{1}{2}$ 
and $\tilde{B}_j=B_j$ for $j\ge 2$. 
In particular, the leading term of $S_d(X)$ is 
$\displaystyle{\frac{X^{d+1}}{d+1}}$ and $S_d(0)=0$.  
\end{lemma}

\begin{lemma}    \label{Lem:T_m} 
Let $m \in \Z_{> 0}$, $d\in\Z_{\geqslant 0}$, and 
$\I \in \E_m$ without zero. Set $p:=p(\I)$. 
Let $\sub{d}=(d_1,\ldots,d_{m-1})$ with $d_1+\cdots+d_{m-1}=d$. 
Then for some polynomial $T_{\sub{d},\I} \in \C[X]$ of degree  
$\leqslant d+m - \deg A_{\I}$, we have 
$\tilde{T}_{\sub{d},{\I}}(k)= T_{\sub{d},\I}(k)$ for all $k \in \{1, \ldots, r+1\}$,
where $$
\tilde{T}_{\sub{d},{\I}}(k):= 
\sum_{(\sub{\mu},\sub{\alpha}) \in \hat{\P}_m(\delta_k),
\atop \sub{\mu} \in [\![ \delta_{r+1}, \delta_k ]\!], 
\, \he(\sub{\mu}) \in  {\I}} i_1^{d_1}\cdots i_{m-1}^{d_{m-1}},$$
if $i_j$ denotes the height $\he({\sub{\mu}})_j$ for $j = 1, \ldots, m-1$. 
In particular, if for some $k \in \{1,\ldots,r+1\}$, the set  
$\{ (\sub{\mu},\sub{\alpha}) \in \hat{\P}_m(\delta_k) \; | \; 
\sub{\mu} \in [\![ \delta_{r+1}, \delta_k]\!] ,\, \he(\sub{\mu}) \in \I \}$ 
is empty, we have $T_{\sub{d},\I}(k) =0$.  
\end{lemma}
The lemma implies that
$$ \sum_{{\tiny \substack{(\sub{\mu},\sub{\alpha}) \in \hat{\P}_m(\delta_k), 
\\ \sub{\mu} \in [\![ \delta_{r+1}, \delta_k]\!], 
\\ \he(\sub{\mu}) \in \I}}}  i_1^{d_1}\ldots i_{m-1}^{d_{m-1}} {A}_{\I}(i_1, \ldots, i_{m-1})$$
is a polynomial on $k$ of degree $\leqslant  d+m$.
\begin{proof}
We prove the lemma by induction on $m$. 
More precisely, we prove by induction on $m$ the following: 

{\em For all $\I \in \E_m$ without zero 
and all $d \in \Z_{\geqslant 0}$   
with $d_1+\cdots+d_{m-1}=d$. 
Then there exists some polynomial $T_{\sub{d},{\I}} \in \C[X]$ of degree 
 at most $d+m-\deg A_{\I}$ such that 
$T_{\sub{d},{\I}}(k) = \tilde{T}_{\sub{d},{\I}}(k),$ 
for all $k \in \{1, \ldots, r+1\}$.  
In particular, if for some $k \in \{1,\ldots,r+1\}$, the set  
$\{ (\sub{\mu},\sub{\alpha}) \in \hat{\P}_m(\delta_k) \; | \; 
\sub{\mu} \in [\![ \delta_{r+1}, \delta_k ]\!] ,\, \he(\sub{\mu}) \in {\I} \}$ 
is empty, we have $T_{\sub{d},{\I}}(k) =0$.}

The case $m=1$ is empty. 
Assume $m = 2$, 
and let $\I \in \E_2$ without zero. 
The only equivalence class in $\E_2$ without zero 
is $[1,-1]$, so $\I=[1,-1]$ and $p([1,-1])=2$. 
Let  $k \in \{1,\ldots,r\}$. Then the set 
$\{ (\sub{\mu},\sub{\alpha}) \in \hat{\P}_2(\delta_k) \ | \  
\sub{\mu} \in [\![ \delta_{r+1}, \delta_k]\!] ,\, \he(\sub{\mu}) \sim (1,-1) \}$ 
is nonempty. It is empty for $k=r+1$.  
Let $d=d_1 \in \Z_{\geqslant 0}$. 
We get  
$$ \sum_{(\sub{\mu},\sub{\alpha}) 
\in \hat{\P}_2(\delta_k) 
\atop \sub{\mu} \in [\![ \delta_{r+1}, \delta_k]\!], 
\,  \he(\sub{\mu})  \sim (1,-1)} i_1^{d_1} 
=  \sum_{i_1 =1}^{r +1- k}   i_1^{d}  = S_{d}(r+1 -  k). 
$$
By Lemma \ref{Lem:Newton}, the polynomial 
$T_{d,[1,-1]}(X) := S_{d}(r+1 -  X),$ 
has degree $d+1=d+2- \deg A_{[1,-1]}$ 
since $A_{[1,-1]}(X_1)=X_1$. 
Hence, for any $k \in \{1,\ldots,r\}$, 
$$\sum_{(\sub{\mu},\sub{\alpha}) \in \hat{\P}_2(\delta_k) 
\atop \sub{\mu} \in [\![ \delta_{r+1}, \delta_k]\!], 
\, \he(\sub{\mu})  \sim (1,-1) } i_1^{d_1} 
= T_{d,[1,-1]} (k).$$ 
Moreover, the set $\{ (\sub{\mu},\sub{\alpha}) \in \hat{\P}_2(\delta_k) \ | \  
\sub{\mu} \in [\![ \delta_{r+1}, \delta_k]\!] ,\, \he(\sub{\mu}) \sim (1,-1) \}$ is empty if and only if $k = r+1$, but $T_{d,[1,-1]} (r+1) = S_{d}(0) =0$. Therefore the equality still holds.
This proves the claim for $m=2$. 

Assume $m \geqslant 3$ and the claim proven for any $m' \in \{ 1,\ldots,m-1\}.$
Let $\I \in \E_m$ without zero, set $p:=p(\I)$, and 
let $\sub{d}=(d_1,\ldots,d_{p-1})$ with 
$d_1 + \cdots+ d_{p-1} =d$. 
Let $k \in \{1,\ldots,r+1\}$ such that 
the set $\{ (\sub{\mu},\sub{\alpha}) \in \hat{\P}_m(\delta_k), \ | \  
\sub{\mu} \in [\![ \delta_{r+1}, \delta_k]\!] \textrm{ and } 
\he(\sub{\mu})\in \I \}$ 
is nonempty. 
Then the set 
$\{ (\sub{\mu}',\sub{\alpha}') \in \hat{\P}_{m-1}(\delta_k), \ | \  
\sub{\mu}' \in [\![ \delta_{r+1}, \delta_k]\!] \textrm{ and } 
\he(\sub{\mu}') \in \I^\#  \}$ 
is nonempty, too.

$\ast$ 
Assume that $i_{p-1}+i_p >0$. 
Let $\sub{\mu} \in [\![\delta_{r+1},\delta_k ]\!]$ 
such that $\sub{i}:=\he(\sub{\mu}) \in \I$. 
Set $\sub{i}' := \he(\sub{\mu}^{\# p})$, 
\begin{align*}
 \sub{i}' 
 &= (i_1', \ldots, i'_{p-2}, i'_{p-1}, i'_p, \ldots, i'_{m-1})
\end{align*}
where $i'_{p-1} = i_{p-1}+i_p >0$. 
Then $i_{p-1} > i'_{p-1}$, and so $ i_{p-1} = i'_{p-1} + i$
with $i$ runs through $\{1,\ldots,r +1 - k -i'_1 - \cdots - i'_{p-1}\}$.
Hence, there are precisely $r+1-k-(i'_1 +\cdots + i'_{p - 1})$ 
elements $ \sub{i} \in {\Z}^m$ such that $\sub{i} \in {\I}$ 
and $\sub{i}^{\# p} = \he(\sub{\mu}^{\# p})$. The heights  $\sub{i}:=\he(\sub{\mu}) \in \I$
can be expressed in term of $\sub{i}'$ as follows:
\begin{align*}
i_1 = i'_1, \quad \ldots, \quad i_{p-2} = i'_{p-2}, \, i_{p-1} = i'_{p-1}+ i, \qquad i_p = -i, 
\, i_{p+1} = i'_{p}, \quad \ldots,  \quad i_{m-1} = i'_{m-2},
\end{align*}
where $i$ runs through $\{1,\ldots,r +1 - k -i'_1 - \cdots - i'_{p-1}\}$ 
(see Figure \ref{fig:keyA1} 
 for an illustration).

\begin{figure}[h]
\begin{center}
\begin{tikzpicture}[scale=0.8]
    \fill (0,-1) circle (0.05)node(xline)[above] {$\delta_k$};
    \fill (8,-1) circle (0.05)node(xline)[above] {$\mu^{(p)}$};
    \fill (9.5,-1) circle (0.05)node(xline)[above] {$\delta_{r+1}$}; 
\fill[blue] (0,-1) circle (0.05);
 \draw[thick, blue, decoration={markings, mark=at position 0.5 with {\arrow{>}}},
        postaction={decorate}](0,-1) -- (1,-1);
  \node[text=blue] at (0.5,-0.8) {{\scriptsize $i_1$}} ;
 \draw[thick, dotted,blue] (1,-1) -- (2,-1);
  \draw[thick, blue, decoration={markings, mark=at position 0.5 with {\arrow{>}}},
        postaction={decorate}](2,-1) -- (3,-1);
  \node[text=blue] at (2.5,-0.8) {{\scriptsize $i_{p-2}$}} ;
\fill[blue] (3,-1) circle (0.05);
  \draw[thick, blue, decoration={markings, mark=at position 0.5 with {\arrow{>}}},
        postaction={decorate}](3,-1) -- (8,-1);
  \node[text=blue] at (5.5,-0.8) {{\scriptsize $i_{p-1}$}} ;
\fill[blue] (8,-1) circle (0.05);
  \draw[thick, blue, decoration={markings, mark=at position 0.5 with {\arrow{>}}},
        postaction={decorate}](8,-1.15) -- (6,-1.15);
  \node[text=blue] at (7,-1.35) {{\scriptsize $i_{p}$}} ;
  \draw[thick, blue, decoration={markings, mark=at position 0.5 with {\arrow{>}}},
        postaction={decorate}](6,-1.15) -- (4.5,-1.15);
  \node[text=blue] at (5.3,-1.35) {{\scriptsize $i_{p+1}$}} ;
\fill[blue] (6,-1.15) circle (0.05);
 \draw[thick, dotted,blue] (4.5,-1.15) -- (2.5,-1.15);
\fill[red] (0,-2.1) circle (0.05);
 \draw[thick, red, decoration={markings, mark=at position 0.5 with {\arrow{>}}},
        postaction={decorate}](0,-2.1) -- (1,-2.1);
  \node[text=red] at (0.5,-1.85) {{\scriptsize $i'_1$}} ;
 \draw[thick, dotted,red] (1,-2.1) -- (2,-2.1);
  \draw[thick, red, decoration={markings, mark=at position 0.5 with {\arrow{>}}},
        postaction={decorate}](2,-2.1) -- (3,-2.1);
  \node[text=red] at (2.5,-1.85) {{\scriptsize $i'_{p-2}$}} ;
\fill[red] (3,-2.1) circle (0.05);
  \draw[thick, red, decoration={markings, mark=at position 0.5 with {\arrow{>}}},
        postaction={decorate}](3,-2.1) -- (6,-2.1);
  \node[text=red] at (4.6,-1.85) {{\scriptsize $i'_{p-1}=i_{p-1}+i_p$}} ;
\fill[red] (6,-2.1) circle (0.05);
  \draw[thick,dashed, red, decoration={markings, mark=at position 0.5 with {\arrow{>}}},
        postaction={decorate}](6,-2.1) -- (9.5,-2.1);
  \node[text=red] at (7.75,-1.85) {{\scriptsize $i$}} ;
  \draw[thick,dashed, red, decoration={markings, mark=at position 0.5 with {\arrow{>}}},
        postaction={decorate}](9.5,-2.25) -- (6,-2.25);
  \node[text=red] at (8,-2.5) {{\scriptsize $-i$}} ;
  \draw[thick, red, decoration={markings, mark=at position 0.5 with {\arrow{>}}},
        postaction={decorate}](6,-2.25) -- (4.5,-2.25);
  \node[text=red] at (5.3,-2.5) {{\scriptsize $i'_p = i_{p+1}$}} ;
\fill[red] (6,-2.25) circle (0.05);
 \draw[thick, dotted,red] (4.5,-2.25) -- (2.5,-2.25);
 \end{tikzpicture}
 \caption{\label{fig:keyA1}
$\sub{i}= \he(\sub{\mu})$ and $\sub{i}' := \he(\sub{\mu}^{\# p})$ 
for the case $i_{p-1}+i_p >0$} 
\end{center}
\end{figure}
By Lemma \ref{Lem:Newton} we get,   
{\small
\begin{align*}
& \tilde{T}_{\sub{d},\I}(k)  
  =\sum_{{\tiny \substack{(\sub{\mu}',\sub{\alpha}') \in \hat{\P}_{m-1}(\delta_k) 
    \\ \sub{\mu} \in [\![ \delta_{r+1}, \delta_k]\!] \\ \he(\sub{\mu}') \in  \I^\#}}}
  \sum_{{\tiny \substack{1\le i  \leqslant  r +1-k \\- i'_1 - \cdots  - i'_{p-1} }}}
  {\footnotesize (i'_{1})^{d_1} \cdots (i'_{p-2})^{d_{p-2}} (i'_{p-1}+i)^{d_{p-1}}
  ({-i})^{d_p} \cdots (i'_{m-2})^{d_{m-1}}}\\
& \quad = 
  \sum_{{\tiny \substack{(\sub{\mu}',\sub{\alpha}') \in \hat{\P}_{m-1}(\delta_k) 
    \\ \sub{\mu} \in [\![ \delta_{r+1}, \delta_k]\!] \\ \he(\sub{\mu}') \in  \I^\#}}}
  \sum_{{\tiny \substack{1\le i  \leqslant  r +1-k \\- i'_1 - \cdots  - i'_{p-1} }}}
  (i'_1)^{d_1} \cdots (i'_{p-2})^{d_{p-2}}
   \left( \sum_{j=0}^{d_{p-1}} \binom{d_{p-1}}{j}
  (i'_{p-1})^{d_{p-1}- j} i^j \right)
  (-i)^{d_p} \cdots (i'_{m-2})^{d_{m-1}}\\
  & \quad = 
  \sum_{j=0}^{d_{p-1}} \sum_{l=0}^{d_p+j} \frac{1}{d_p+j+1} \binom{d_{p-1}}{j} \binom{d_p+j+1}{l} \tilde{B}_l 
    \sum_{{\tiny \substack{\sub{q} \in \N^{p} 
    \\ |\sub{q}|=d_p+j+1-l}}} \frac{(d_p+j+1-l)!}{q_1!\ldots q_p!}  (-1)^{2d_p+j+1-l-q_p} (r+1-k)^{q_p}
 \tilde{T}_{\sub{d}',{\I}^{\#}}(k), 
\end{align*}}
\smallskip
\noindent
where   
$\sub{d}'=(d_1+q_1, \ldots, d_{p-1}+q_{p-1}-j, d_{p+1}, \ldots, d_{m-1})$,
with $|\sub{d}'| 
= d+1-l-q_p \leqslant d+1-q_p$. 
 
By the induction hypothesis applied to $m-1$ and ${\I}^{\#}$, 
there exists a polynomial 
${T}_{\sub{d}',{\I}^{\#}}$ of degree $\leqslant  d+1-q_p+ (m-1)- \deg A_{{\I}^{\#}} = d-q_p+m- \deg A_{\I}$, since $\deg A_{{\I}^{\#}} = \deg A_{{\I}}$ by Lemma \ref{Lem:weight1}, 
such that ${T}_{\sub{d}',{\I}^{\#}}(k)= \tilde{T}_{\sub{d}',{\I}^{\#}}(k)$
for all $k$ 
such that $\{ (\sub{\mu}',\sub{\alpha}') \in \hat{\P}_{m-1}(\delta_k), \ | \  
\sub{\mu}' \in [\![ \delta_{r+1}, \delta_k]\!],
\he(\sub{\mu}')\in \I^{\#} \}$ 
is nonempty. 
Setting 
\begin{align*} 
{T}_{\sub{d},{\I}}(X) &= 
\sum_{j=0}^{d_{p-1}} \sum_{l=0}^{d_p+j} \frac{1}{d_p+j+1} \binom{d_{p-1}}{j} \binom{d_p+j+1}{l} \tilde{B}_l \\
& \quad 
\times     \sum_{{\tiny \substack{\sub{q} \in \N^{p+1} 
    \\ |\sub{q}|=d_p+j+1-l}}} \frac{(d_p+j+1-l)!}{q_1!\ldots q_p!}  (-1)^{2d_p+j+1-l-q_p} (r+1-k)^{q_p}
 {T}_{\sub{d}',{\I}^{\#}}(X), 
\end{align*}
it is clear by the induction hypothesis applied to $m-1$ and ${\I}^{\#}$ 
that 
$T_{\sub{d}, {\I}}(k) = \tilde{T}_{\sub{d},{\I}}(k)$ and that 
$T_{\sub{d}, {\I}}$ is a polynomial of degree 
$\leqslant d+m- \deg A_{{\I}}$ for all $k$ 
such that $\{ (\sub{\mu},\sub{\alpha}) \in \hat{\P}_m(\delta_k), \ | \  
\sub{\mu} \in [\![ \delta_{r+1}, \delta_k]\!],
\he(\sub{\mu})\in \I \}$ 
is nonempty.

It remains to verify that $T_{\sub{d}, \I} (k) =0$  
when the set
\begin{align} \label{eq:setT}
\{ (\sub{\mu},\sub{\alpha}) \in \hat{\P}_{m}(\delta_k) \; | \;  
\sub{\mu} \in [\![ \delta_{r+1}, \delta_k]\!] \textrm{ and } 
\he(\sub{\mu}) \in \I \}
\end{align}
is empty.
In this case, $\tilde{T}_{\sub{d},\I}(k) =0$ by the definition.
The set~\eqref{eq:setT} is empty if  
 $\{ (\sub{\mu}',\sub{\alpha}') \in \hat{\P}_{m-1}(\delta_k) \; | \;  
\sub{\mu}' \in [\![ \delta_{r+1}, \delta_k]\!] \textrm{ and } 
\he(\sub{\mu}') \in \I^\#  \}=\varnothing$.   
But our induction hypothesis 
says that, in this case, $$T_{\sub{d}, \I} (k) =T_{\sub{d}', \I^\# }(k) = 0,$$
for any $\sub{d}' \in \Z_{\geqslant 0}^{p-1}$.  
Otherwise, this means that the set 
 $\{ (\sub{\mu}',\sub{\alpha}') \in \hat{\P}_{m-1}(\delta_k) \; | \;  
\sub{\mu}' \in [\![ \delta_{r+1}, \delta_k]\!] \textrm{ and } 
\he(\sub{\mu}')~\in~\I^\# ~\}$ is nonempty and so for any $\sub{i}' \in  \I^\# $ 
and 
for any $(\sub{\mu}',\sub{\alpha}') \in \hat{\P}_{m-1}(\delta_k)$ 
such that $\he(\sub{\mu}') = \sub{i}'$, we have 
$i'_1 + \cdots  + i'_{p-1} = r +1-k$. 
In that event, 
$T_{\sub{d}, \I} (k)= 0$ by the construction.

$\ast$ 
Assume that $i_{p-1}+i_p <0$. 
Let $\sub{\mu} \in [\![\delta_{r+1},\delta_k ]\!]$ 
such that $\sub{i}:=\he(\sub{\mu}) \in \I$. 
Set $\sub{i}' := \he(\sub{\mu}^{\# p})$, 
\begin{align*}
 \sub{i}' 
 &= (i_1', \ldots, i'_{p-2}, i'_{p-1}, i'_p, \ldots, i'_{m-1})
\end{align*}
where $i'_{p-1} = i_{p-1}+i_p <0$. 
We have $i_{p} < i'_{p-1} <0$, and so $ i_{p} = i'_{p-1} - i$
with $i$ runs through $\{1,\ldots,r +1 - k -i'_1 - \cdots - i'_{p-1}\}$.
Hence,
the heights  $\sub{i}:=\he(\sub{\mu}) \in \I$
can be expressed in term of $\sub{i}'$ as follows:
\begin{align*}
i_1 = i'_1,\quad  \ldots, \quad i_{p-2} = i'_{p-2},  
\, i_{p-1} = i,\quad  i_{p} = i'_{p-1} - i,  
\, i_{p+1} = i'_{p},\quad  \ldots, \quad  i_{m-1} = i'_{m-2}, 
\end{align*}
where $i$ runs through $\{1,\ldots,r +1 - k -i'_1 - \cdots - i'_{p-1}\}$ 
(see Figure \ref{fig:keyA2} 
 for an illustration).
\begin{figure}[h]
\begin{center}
\begin{tikzpicture}[scale=0.85]
    \fill (0,-1) circle (0.05)node(xline)[above] {$\delta_k$}; 
    \fill (8,-1) circle (0.05)node(xline)[above] {$\mu^{(p)}$};
    \fill (9.5,-1) circle (0.05)node(xline)[above] {$\delta_{r+1}$}; 
\fill[blue] (0,-1) circle (0.05);
 \draw[thick, blue, decoration={markings, mark=at position 0.5 with {\arrow{>}}},
        postaction={decorate}](0,-1) -- (1,-1);
  \node[text=blue] at (0.5,-0.8) {{\scriptsize $i_1$}} ;
\fill[blue] (1,-1) circle (0.05);
 \draw[thick, blue, decoration={markings, mark=at position 0.5 with {\arrow{>}}},
        postaction={decorate}](1,-1) -- (2,-1);
  \node[text=blue] at (1.5,-0.8) {{\scriptsize $i_2$}} ;
 \draw[thick, dotted,blue] (2,-1) -- (5,-1);
  \draw[thick, blue, decoration={markings, mark=at position 0.5 with {\arrow{>}}},
        postaction={decorate}](5,-1) -- (6,-1);
  \node[text=blue] at (5.5,-0.8) {{\scriptsize $i_{p-2}$}} ;
\fill[blue] (6,-1) circle (0.05);
  \draw[thick, blue, decoration={markings, mark=at position 0.5 with {\arrow{>}}},
        postaction={decorate}](6,-1) -- (8,-1);
  \node[text=blue] at (7,-0.8) {{\scriptsize $i_{p-1}$}} ;
\fill[blue] (8,-1) circle (0.05);
  \draw[thick, blue, decoration={markings, mark=at position 0.5 with {\arrow{>}}},
        postaction={decorate}](8,-1.15) -- (3,-1.15);
  \node[text=blue] at (5.5,-1.35) {{\scriptsize $i_{p}$}} ;
  \draw[thick, blue, decoration={markings, mark=at position 0.5 with {\arrow{>}}},
        postaction={decorate}](3,-1.15) -- (2,-1.15);
  \node[text=blue] at (2.5,-1.35) {{\scriptsize $i_{p+1}$}} ;
\fill[blue] (3,-1.15) circle (0.05);
 \draw[thick, dotted,blue] (2,-1.15) -- (1,-1.15);
%
\fill[red] (0,-2.1) circle (0.05);
 \draw[thick, red, decoration={markings, mark=at position 0.5 with {\arrow{>}}},
        postaction={decorate}](0,-2.1) -- (1,-2.1);
  \node[text=red] at (0.5,-1.85) {{\scriptsize $i'_1$}} ;
\fill[red] (1,-2.1) circle (0.05);
 \draw[thick, red, decoration={markings, mark=at position 0.5 with {\arrow{>}}},
        postaction={decorate}](1,-2.1) -- (2,-2.1);
  \node[text=red] at (1.5,-1.85) {{\scriptsize $i'_2$}} ;
 \draw[thick, dotted,red] (2,-2.1) -- (5,-2.1);
  \draw[thick, red, decoration={markings, mark=at position 0.5 with {\arrow{>}}},
        postaction={decorate}](5,-2.1) -- (6,-2.1);
  \node[text=red] at (5.5,-1.85) {{\scriptsize $i'_{p-2}$}} ;
\fill[red] (6,-2.1) circle (0.05);
  \draw[thick, dashed, red, decoration={markings, mark=at position 0.5 with {\arrow{>}}},
        postaction={decorate}](6,-2.1) -- (9.5,-2.1);
\node[text=red] at (8,-1.85) {{\scriptsize $i$}} ;
\draw[thick,dashed, red, decoration={markings, mark=at position 0.5 with {\arrow{>}}},
        postaction={decorate}](9.5,-2.25) -- (6,-2.25);
  \node[text=red] at (8,-2.5) {{\scriptsize $-i$}} ;
\fill[red] (6,-2.25) circle (0.05);
  \draw[thick, red, decoration={markings, mark=at position 0.5 with {\arrow{>}}},
        postaction={decorate}](6,-2.25) -- (3,-2.25);
  \node[text=red] at (4.5,-2.5) {{\scriptsize $i'_{p-1}=i_{p-1}+i_p$}} ;
  \fill[red] (3,-2.25) circle (0.05);
  \draw[thick, red, decoration={markings, mark=at position 0.5 with {\arrow{>}}},
        postaction={decorate}](3,-2.25) -- (2,-2.25);
  \node[text=red] at (2.5,-2.5) {{\scriptsize $i'_p = i_{p+1}$}} ;
fill[red] (3,-2.25) circle (0.05);
 \draw[thick, dotted,red] (2,-2.25) -- (1,-2.25);
 \end{tikzpicture}
 \caption{\label{fig:keyA2}
$\sub{i}= \he(\sub{\mu})$ and $\sub{i}' := \he(\sub{\mu}^{\# p})$ 
for the case $i_{p-1}+i_p <0$} 
\end{center}
\end{figure}

Hence, 
\begin{align*}
\tilde{T}_{\sub{d},{\I}}(k)
=\sum_{{\tiny \substack{(\sub{\mu}',\sub{\alpha}') \in \hat{\P}_{m-1}(\delta_k) 
    \\ \sub{\mu} \in [\![ \delta_{r+1}, \delta_k]\!] \\ \he(\sub{\mu}') \in  \I^\#}}}
  \sum_{{\tiny \substack{1\le i  \leqslant  r +1-k \\- i'_1 - \cdots  - i'_{p-1} }}}
  {\footnotesize (i'_{1})^{d_1} \cdots (i'_{p-2})^{d_{p-2}}i^{d_{p-1}} 
 (i'_{p-1}-i)^{d_{p}}  (i'_{p})^{d_{p+1}}\cdots (i'_{m-2})^{d_{m-1}}}.
\end{align*}
Then we conclude exactly as in the first case. 
To verify that $T_{\sub{d}, \I} (k) =0$  
when the set~\eqref{eq:setT} is empty,
the arguments are the same as for the case $i_{p-1}+i_p >0$ so we omit details. 

$\ast$ Assume that $i_{p-1}+i_p = 0$. 
Since $\I$ has no zero, we necessarily have $m \geqslant 4$. 
Let $\sub{\mu} \in [\![\delta_{r+1},\delta_k ]\!]$ 
such that $\sub{i}:=\he(\sub{\mu}) \in \I$. 
Set $\sub{i}' := \he(\sub{\mu}^{\# p})$, 
\begin{align*}
 \sub{i}' 
 &= (i_1', \ldots, i'_{p-2}, i'_{p-1}, i'_p, \ldots, i'_{m-2})
\end{align*}
where $i'_{p-1} = i_{p+1}$. Set $i_{p-1} = i$ then $ i_{p} = - i$
with $i$ runs through $\{1,\ldots,r +1 - k -i'_1 - \cdots - i'_{p-2}\}$.
Hence,
the heights  $\sub{i}:=\he(\sub{\mu}) \in \I$
can be expressed in term of $\sub{i}'$ as follows:
\begin{align*}
i_1 = i'_1, \quad  \ldots,  \quad i_{p-2} = i'_{p-2}, \, 
i_{p-1} = i, \quad  i_{p} =  - i, \, 
i_{p+1} = i'_{p-1},\quad  \ldots,\quad  i_{m-1} = i'_{m-3}
\end{align*}
where $i$ runs through $\{1,\ldots,r +1 - k -i'_1 - \cdots - i'_{p-2}\}$ 
(see Figure \ref{fig:keyA3} 
 for an illustration).
\begin{figure}[h]
\begin{center}
\begin{tikzpicture} [scale=0.85]
    \fill (0,-1) circle (0.05)node(xline)[below] {$\delta_k$};
        \fill (6,-1) circle (0.05)node(xline)[below] {$\mu^{(p)}$};
    \fill (7.5,-1) circle (0.05)node(xline)[below] {$\delta_{r+1}$}; 
\fill[blue] (0,-1) circle (0.05);
 \draw[thick, blue, decoration={markings, mark=at position 0.5 with {\arrow{>}}},
        postaction={decorate}](0,-1) -- (1,-1);
  \node[text=blue] at (0.5,-0.8) {{\scriptsize $i_1$}} ;
 \draw[thick, dotted,blue] (1,-1) -- (2,-1);
  \draw[thick, blue, decoration={markings, mark=at position 0.5 with {\arrow{>}}},
        postaction={decorate}](2,-1) -- (3,-1);
  \node[text=blue] at (2.5,-0.8) {{\scriptsize $i_{p-2}$}} ;
\fill[blue] (3,-1) circle (0.05);
  \draw[thick, blue, decoration={markings, mark=at position 0.5 with {\arrow{>}}},
        postaction={decorate}](3,-1) -- (6,-1);
  \node[text=blue] at (4.5,-0.8) {{\scriptsize $i_{p-1}$}} ;
\fill[blue] (6,-1) circle (0.05);
  \draw[thick, blue, decoration={markings, mark=at position 0.5 with {\arrow{>}}},
        postaction={decorate}](6,-1.15) -- (3,-1.15);
  \node[text=blue] at (4.5,-1.35) {{\scriptsize $i_{p}$}} ;
  \draw[thick, blue, decoration={markings, mark=at position 0.5 with {\arrow{>}}},
        postaction={decorate}](3,-1.15) -- (2,-1.15);
  \node[text=blue] at (2.5,-1.35) {{\scriptsize $i_{p+1}$}} ;
\fill[blue] (3,-1.15) circle (0.05);
 \draw[thick, dotted,blue] (2,-1.15) -- (1,-1.15);
\fill[red] (0,-2.1) circle (0.05);
 \draw[thick, red, decoration={markings, mark=at position 0.5 with {\arrow{>}}},
        postaction={decorate}](0,-2.1) -- (1,-2.1);
  \node[text=red] at (0.5,-1.85) {{\scriptsize $i'_1$}} ;
 \draw[thick, dotted,red] (1,-2.1) -- (2,-2.1);
  \draw[thick, red, decoration={markings, mark=at position 0.5 with {\arrow{>}}},
        postaction={decorate}](2,-2.1) -- (3,-2.1);
  \node[text=red] at (2.5,-1.85) {{\scriptsize $i'_{p-2}$}} ;
\fill[red] (3,-2.1) circle (0.05);
\draw[thick,dashed, red, decoration={markings, mark=at position 0.5 with {\arrow{>}}},
        postaction={decorate}](3,-2.1) -- (7.5,-2.1);
  \node[text=red] at (5.25,-1.85) {{\scriptsize $i$}} ;
  \draw[thick,dashed, red, decoration={markings, mark=at position 0.5 with {\arrow{>}}},
        postaction={decorate}](7.5,-2.25) -- (3,-2.25);
  \node[text=red] at (5.25,-2.5) {{\scriptsize $-i$}} ;
  \draw[thick, red, decoration={markings, mark=at position 0.5 with {\arrow{>}}},
        postaction={decorate}](3,-2.25) -- (2,-2.25);
  \node[text=red] at (2.5,-2.5) {{\scriptsize $i'_{p-1} = i_{p+1}$}} ;
\fill[red] (3,-2.25) circle (0.05);
 \draw[thick, dotted,red] (2,-2.25) -- (1,-2.25);
 \end{tikzpicture}
 \caption{\label{fig:keyA3}
$\sub{i}= \he(\sub{\mu})$ and $\sub{i}' := \he(\sub{\mu}^{\# p})$ 
for the case $i_{p-1}+i_p =0$} 
\end{center}
\end{figure}

By Lemma \ref{Lem:Newton} we get,   
\begin{align*}
\tilde{T}_{\sub{d},\I}(k) 
&  =\sum_{{\tiny \substack{(\sub{\mu}',\sub{\alpha}') \in \hat{\P}_{m-1}(\delta_k) 
    \\ \sub{\mu} \in [\![ \delta_{r+1}, \delta_k]\!] \\ \he(\sub{\mu}') \in  \I^\#}}}
  \sum_{{\tiny \substack{1\le i  \leqslant  r +1-k \\- i'_1 - \cdots  - i'_{p-2} }}}
  {\footnotesize (i'_{1})^{d_1} \cdots (i'_{p-2})^{d_{p-2}} (i)^{d_{p-1}}
  (-i)^{d_p} (i'_{p})^{d_{p+1}} \cdots (i'_{m-3})^{d_{m-1}}}\\
& = 
\frac{1}{d_{p-1}+d_p+1}\sum_{j=0}^{d_{p-1}+d_p} \binom{d_{p-1}+d_p+1}{j} \tilde{B}_j 
\sum_{{\tiny \substack{\sub{q} \in \N^{p-1} 
    \\ |\sub{q}|=d_{p-1}+d_p+1-j}}} \frac{(d_{p-1}+d_p+1-j)!}{q_1!\ldots q_{p-1}!} \\
& \hspace{0.5cm} \times 
    (-1)^{d_{p-1}+2d_p+1-j-q_{p-1}}
    (r+1-k)^{q_{p-1}}
 \tilde{T}_{\sub{d}',{\I}^{\#}}(k), 
\end{align*}
where 
$\sub{d}'=(d_1+q_1, \ldots, d_{p-2}+q_{p-2}, d_{p+1}, \ldots, d_{m-1})$,
with $|\sub{d}'|= d-d_{p-1} - d_p+ d_{p-1} + d_p
+1-j-q_{p-1} \leqslant d+1-q_{p-1}$. 

Setting 
\begin{align*} 
{T}_{\sub{d},{\I}}(X) &= 
\frac{1}{d_{p-1}+d_p+1}\sum_{j=0}^{d_{p-1}+d_p} \binom{d_{p-1}+d_p+1}{j} \tilde{B}_j 
\sum_{{\tiny \substack{\sub{q} \in \N^{p-1} 
    \\ |\sub{q}|=d_{p-1}+d_p+1-j}}} \frac{(d_{p-1}+d_p+1-j)!}{q_1!\ldots q_{p-1}!} \\
& \quad 
   \times (-1)^{d_{p-1}+2d_p+1-j-q_{p-1}}
    (r+1-k)^{q_{p-1}}
 {T}_{\sub{d}',{\I}^{\#}}(X), 
\end{align*}
it is clear by the induction hypothesis applied to $m-2$ and ${\I}^{\#}$ 
that 
$T_{\sub{d}, {\I}}(k) = \tilde{T}_{\sub{d},{\I}}(k)$ 
and that $T_{\sub{d}, {\I}}$ is a polynomial of degree 
$\leqslant  d+m- \deg A_{{\I}}$ for all $k$ 
such that $\{ (\sub{\mu},\sub{\alpha}) \in \hat{\P}_m(\delta_k), \ | \  
\sub{\mu} \in [\![ \delta_{r+1}, \delta_k]\!],
\he(\sub{\mu})\in \I \}$ 
is nonempty.

It remains to verify that $T_{\sub{d},\I} (k) =0$  
when the Set~\eqref{eq:setT} is empty. 
In this case, $\tilde{T}_{\sub{d},\I} (k)= 0$ by definition.
The set~\eqref{eq:setT} is empty if  
 $\{ (\sub{\mu}',\sub{\alpha}') \in \hat{\P}_{m-1}(\delta_k) \; | \;  
\sub{\mu}' \in [\![ \delta_{r+1}, \delta_k]\!] \textrm{ and } 
\he(\sub{\mu}') \in \I^\#  \}=\varnothing$.   
But our induction hypothesis 
says that, in this case, $$T_{\sub{d},\I} (k) =T_{\sub{d}',
\I^\# }(k) = 0$$
for any $\sub{d'} \in \Z_{\geqslant 0}^{p-1}$. 
Otherwise, this means that for any $\sub{i}' \in  \I^\# $ 
and 
for any $(\sub{\mu}',\sub{\alpha}') \in \hat{\P}_{m-2}(\delta_k)$ 
such that $\he(\sub{\mu}') = \sub{i}'$, we have 
$i'_1 + \cdots  + i'_{p-2} = r +1-k$. 
In that event,  
$T_{\sub{d},\I} (k)= 0$ by the construction. \qed
\end{proof}

We now consider the elements of ${\E}_m$ with zeroes.  
Let $(m,n) \in (\Z_{> 0})^2$, with $n \in\{0,\ldots,m\}$, and $\I \in \E_m$ 
with $n$ zeroes in positions $j_1 <\cdots <j_n$. 
This means that $\mu^{(j_{l})}=\mu^{(j_{l}+1)}$ for $l=1,\ldots,n$. 
Let $\sub{i} \in \I$,  
and let $\tilde{\sub{\imath}}$ be the sequence of ${\Z}^{m-n}$ obtained from $\sub{i}$ 
by removing all zeroes.  
Denote by $\tilde{\I}$ the equivalence class of $\sub{\tilde{i}}$ in 
${\Z}^{m-n}$. 
This class only depends on $\I$ and has no zero. 
 
Let $(\tilde{\sub{\mu}},\tilde{\sub{\alpha}}) \in \hat{\P}_{m-n}$ 
with $\he(\tilde{\sub{\mu}}) \in \tilde{\I}$. 
Thus $(\tilde{\sub{\mu}},\tilde{\sub{\alpha}}) $ has no loop. 
Define weighted paths whose height is in $\I$ 
from $(\tilde{\sub{\mu}},\tilde{\sub{\alpha}})$ as follows. 
Set $\sub{j}:= (j_1,\ldots,j_n)$ 
and let $\sub{\beta}:=(\beta^{(1)},\ldots,\beta^{(n)})$ be in 
$\Pi_{\sub{\mu}^{(\sub{j})}}:=
\Pi_{\mu^{(j_1)}} \times \cdots \times \Pi_{\mu^{(j_n)}}.$ 
Define $(\tilde{\sub{\mu}},\tilde{\sub{\alpha}})_{\sub{j};\sub{\beta}}$ 
to be the weighted path of length $m$ obtained from 
$(\tilde{\sub{\mu}},\tilde{\sub{\alpha}})$ by 
^^ ^^ gluing the loop'' labeled by $\beta^{(l)}$ 
at the vertex $\mu^{(j_l)}$ for $l=1,\ldots,n$. 
Thus for such a $\sub{\beta} \in \Pi_{\sub{\mu}^{(\sub{j})}}$ the height 
of the weighted path $(\tilde{\sub{\mu}},\tilde{\sub{\alpha}})_{\sub{j};\sub{\beta}}$ 
is in $\I$. Moreover, all $(\sub{\mu},\sub{\alpha}) \in \hat{\P}_m$ 
with $\he(\sub{\mu}) \in \I$ are of this form.

\begin{example} \label{Ex:loopA} 
Assume that $r>4$. 
Let $\I \in \E_5$ be the class 
$[1,0,1,0,-2]$. 
Then $\I$ has 2 zeros in positions 2 and 4 and $\tilde{\I} = [1,1,-2]$. 
We represent below the weighted path 
$$(\tilde{\sub{\mu}},\tilde{\sub{\alpha}}) 
=((\delta_2,\delta_4,\delta_5,\delta_2),(\beta_{2}+\beta_3, 
\beta_4,-\beta_2-\beta_3-\beta_4)) 
\in \hat{\P}_{3}(\delta_2)$$  
whose height $(2,1,-3)$ is in $\tilde{\I}$, and 
the four weighted paths 
$(\tilde{\sub{\mu}},\tilde{\sub{\alpha}})_{(2,4);(\beta_4,\beta_5)}$, 
$(\tilde{\sub{\mu}},\tilde{\sub{\alpha}})_{(2,4);(\beta_4,\beta_4)}$, 
$(\tilde{\sub{\mu}},\tilde{\sub{\alpha}})_{(2,4);(\beta_3,\beta_5)}$ 
and $(\tilde{\sub{\mu}},\tilde{\sub{\alpha}})_{(2,4);(\beta_3,\beta_4)}$ 
whose height is in $\I$ 
obtained from it: 

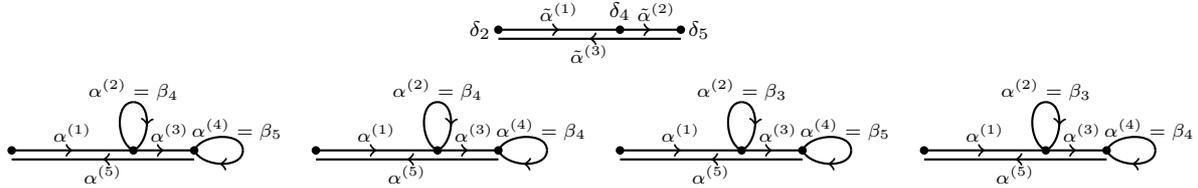
\begin{center} 
\begin{tikzpicture} [scale=0.8]
\fill (8,2) circle (0.07);\node at (7.7,2) {{\small $\delta_2$}}; 
\draw[thick,
        decoration={markings, mark=at position 0.5 with {\arrow{>}}},
        postaction={decorate}
        ]
        (8,2) -- (10,2);\node at (9,2.3) {{\scriptsize $\tilde{\alpha}^{(1)}$}};
\fill (10,2) circle (0.07);\node at (10,2.3) {{\small $\delta_4$}}; 
\draw[thick,
        decoration={markings, mark=at position 0.5 with {\arrow{>}}},
        postaction={decorate}
        ]
        (10,2) -- (11,2);\node at (10.6,2.3) {{\scriptsize $\tilde{\alpha}^{(2)}$}};
\fill (11,2) circle (0.07);\node at (11.3,2) {{\small $\delta_5$}}; 
\draw[thick,
        decoration={markings, mark=at position 0.5 with {\arrow{>}}},
        postaction={decorate}
        ]
        (11,1.85) -- (8,1.85);\node at (9.5,1.6) {{\scriptsize $\tilde{\alpha}^{(3)}$}};
\fill (0,0) circle (0.07);
\draw[thick,
        decoration={markings, mark=at position 0.5 with {\arrow{>}}},
        postaction={decorate}
        ]
        (0,0) -- (2,0);\node at (1,0.3) {{\scriptsize ${\alpha}^{(1)}$}};
\fill (2,0) circle (0.07);
\draw[thick] (2,0) to [out=155,in=180] (2,0.8);
 \draw[thick, decoration={markings, mark=at position 0.5 with {\arrow{>}}},
         postaction={decorate}] (2,0.8) to [out=360,in=25] (2,0);        
\node at (2,1) {{\scriptsize ${\alpha}^{(2)}=\beta_4$}};
\draw[thick,
        decoration={markings, mark=at position 0.5 with {\arrow{>}}},
        postaction={decorate}
        ]
        (2,0) -- (3,0);\node at (2.6,0.3) {{\scriptsize ${\alpha}^{(3)}$}};
\fill (3,0) circle (0.07);
\draw[thick] (3,0) to [out=65,in=90] (3.8,0);
 \draw[thick, decoration={markings, mark=at position 0.5 with {\arrow{>}}},
         postaction={decorate}] (3.8,0) to [out=270,in=295] (3,0);        
\node at (3.7,0.35) {{\scriptsize ${\alpha}^{(4)}=\beta_5$}};
\draw[thick,
        decoration={markings, mark=at position 0.5 with {\arrow{>}}},
        postaction={decorate}
        ]
        (3,-0.15) -- (0,-0.15);\node at (1.5,-0.4) {{\scriptsize ${\alpha}^{(5)}$}};
\fill (5,0) circle (0.07);
\draw[thick,
        decoration={markings, mark=at position 0.5 with {\arrow{>}}},
        postaction={decorate}
        ]
        (5,0) -- (7,0);\node at (6,0.3) {{\scriptsize ${\alpha}^{(1)}$}};
\fill (7,0) circle (0.07);
\draw[thick] (7,0) to [out=155,in=180] (7,0.8);
 \draw[thick, decoration={markings, mark=at position 0.5 with {\arrow{>}}},
         postaction={decorate}] (7,0.8) to [out=360,in=25] (7,0);        
\node at (7,1) {{\scriptsize ${\alpha}^{(2)}=\beta_4$}};
\draw[thick,
        decoration={markings, mark=at position 0.5 with {\arrow{>}}},
        postaction={decorate}
        ]
        (7,0) -- (8,0);\node at (7.6,0.3) {{\scriptsize ${\alpha}^{(3)}$}};
\fill (8,0) circle (0.07);
\draw[thick] (8,0) to [out=65,in=90] (8.8,0);
 \draw[thick, decoration={markings, mark=at position 0.5 with {\arrow{>}}},
         postaction={decorate}] (8.8,0) to [out=270,in=295] (8,0);        
\node at (8.7,0.35) {{\scriptsize ${\alpha}^{(4)}=\beta_4$}};
\draw[thick,
        decoration={markings, mark=at position 0.5 with {\arrow{>}}},
        postaction={decorate}
        ]
        (8,-0.15) -- (5,-0.15);\node at (6.5,-0.4) {{\scriptsize ${\alpha}^{(5)}$}};
\fill (10,0) circle (0.07);
\draw[thick,
        decoration={markings, mark=at position 0.5 with {\arrow{>}}},
        postaction={decorate}
        ]
        (10,0) -- (12,0);\node at (11,0.3) {{\scriptsize ${\alpha}^{(1)}$}};
\fill (12,0) circle (0.07);
\draw[thick] (12,0) to [out=155,in=180] (12,0.8);
 \draw[thick, decoration={markings, mark=at position 0.5 with {\arrow{>}}},
         postaction={decorate}] (12,0.8) to [out=360,in=25] (12,0);        
\node at (12,1) {{\scriptsize ${\alpha}^{(2)}=\beta_3$}};
\draw[thick,
        decoration={markings, mark=at position 0.5 with {\arrow{>}}},
        postaction={decorate}
        ]
        (12,0) -- (13,0);\node at (12.6,0.3) {{\scriptsize ${\alpha}^{(3)}$}};
\fill (13,0) circle (0.07);
\draw[thick] (13,0) to [out=65,in=90] (13.8,0);
 \draw[thick, decoration={markings, mark=at position 0.5 with {\arrow{>}}},
         postaction={decorate}] (13.8,0) to [out=270,in=295] (13,0);        
\node at (13.7,0.35) {{\scriptsize ${\alpha}^{(4)}=\beta_5$}};
\draw[thick,
        decoration={markings, mark=at position 0.5 with {\arrow{>}}},
        postaction={decorate}
        ]
        (13,-0.15) -- (10,-0.15);\node at (11.5,-0.4) {{\scriptsize ${\alpha}^{(5)}$}};
\fill (15,0) circle (0.07);
\draw[thick,
        decoration={markings, mark=at position 0.5 with {\arrow{>}}},
        postaction={decorate}
        ]
        (15,0) -- (17,0);\node at (16,0.3) {{\scriptsize ${\alpha}^{(1)}$}};
\fill (17,0) circle (0.07);
\draw[thick] (17,0) to [out=155,in=180] (17,0.8);
 \draw[thick, decoration={markings, mark=at position 0.5 with {\arrow{>}}},
         postaction={decorate}] (17,0.8) to [out=360,in=25] (17,0);        
\node at (17,1) {{\scriptsize ${\alpha}^{(2)}=\beta_3$}};
\draw[thick,
        decoration={markings, mark=at position 0.5 with {\arrow{>}}},
        postaction={decorate}
        ]
        (17,0) -- (18,0);\node at (17.6,0.3) {{\scriptsize ${\alpha}^{(3)}$}};
\fill (18,0) circle (0.07);
\draw[thick] (18,0) to [out=65,in=90] (18.8,0);
 \draw[thick, decoration={markings, mark=at position 0.5 with {\arrow{>}}},
         postaction={decorate}] (18.8,0) to [out=270,in=295] (18,0);        
\node at (18.7,0.35) {{\scriptsize ${\alpha}^{(4)}=\beta_4$}};
\draw[thick,
        decoration={markings, mark=at position 0.5 with {\arrow{>}}},
        postaction={decorate}
        ]
        (18,-0.15) -- (15,-0.15);\node at (16.5,-0.4) {{\scriptsize ${\alpha}^{(5)}$}};
 \end{tikzpicture}
\end{center} 
\end{example}

\begin{lemma}       \label{Lem:weight2} 
Let $m \in \Z_{> 0} $ 
and $\I \in \E_m$ with $n$ zeroes  
in positions $j_1,\ldots,j_n$ ($n\le m$). 
There is a polynomial $B_{\I}$ 
of degree $n$ such that 
for all $(\tilde{\sub{\mu}},\tilde{\sub{\alpha}}) \in \hat{\P}_{m-n}(\delta_k)$,  
$k \in \{1,\ldots,r+1\}$, such that $\he(\tilde{\sub{\mu}})\in \tilde{\I}$,  
$$\sum_{\sub{\beta} \in \Pi_{\sub{\mu}^{(\sub{j})}} } 
\wt\left((\tilde{\sub{\mu}},\tilde{\sub{\alpha}})_{\sub{j};\sub{\beta}}  \right) 
= B_{\I} (k) A_{\tilde{\I}}(\tilde{\i}_1, \ldots, \tilde{\i}_{m-n-1}),$$ 
where $\he(\tilde{\sub{\mu}})=(\tilde{\i}_1, \ldots, \tilde{\i}_{m-n-1})$.  
Moreover, we have 
$$B_{\I}(X) = \sum_{j=0}^n C_{\I}^{(n-j)}
(\tilde{\i}_1, \ldots, \tilde{\i}_{m-n-1}) 
X^j$$ 
where $C_{\I}^{(0)}=(-1)^n$ 
and $C_{\I}^{(j)}\in \C[X_1,\ldots,X_{m-n-1}]$ has total degree 
$< j$ for $j=1,\ldots,n$. 
In particular, if $n=0$, we have $B_{\I}(X)=1$. 
\end{lemma}

\begin{proof}
First of all, observe that if $n=m$ then the result is 
known by Lemma \ref{Lem2:main1}. At the extreme opposite, 
if $n=0$, then the result is known by Lemma \ref{Lem:weight1}. 
We prove the lemma by induction on $m$. 
By the above observation, the lemma is true for $m=1$ and $m=2$. 
Let $m\ge 3$ and assume the lemma true for any $m' \in \{1,\ldots,m-1\}$.  
Set $p:=p(\I)$ and $q:=q(\I)$. 

$\ast$ First case: $i_p=0$. Then $p=j_l$ for some $l \in \{1,\ldots,n\}$. 
Assume that $\mu^{(1)} = \delta_k$, then 
by Lemma~\ref{Lem0:cut} (3)
and Lemma~\ref{Lem1:main1} (2), we have
\begin{align*}
 \sum_{\sub{\beta} \in \Pi_{\sub{\mu}^{(\sub{j})}} } 
\wt\left((\tilde{\sub{\mu}},\tilde{\sub{\alpha}})_{\sub{j};\sub{\beta}}\right)  = \sum_{\sub{\beta} \in \Pi_{\sub{\mu}^{(\sub{j}')}} } 
\left(\frac{r}{2} - k- \sum_{t=1}^{p-1} i_t +2\right) 
\wt\left(((\tilde{\sub{\mu}},\tilde{\sub{\alpha}})_{\sub{j};\sub{\beta}})^{\# p}  
\right),
\end{align*}
where $\sub{j}':=(j_1,\ldots,j_{l-1},j_{l+1}, 
\ldots,j_n)$.
So, the induction hypothesis 
applied to $\I^{\#}$ gives, 
\begin{align*}
\sum_{\sub{\beta} \in \Pi_{\sub{\mu}^{(\sub{j})}} } 
\wt\left((\tilde{\sub{\mu}},\tilde{\sub{\alpha}})_{\sub{j};\sub{\beta}}  \right) 
&=\left(\frac{r}{2} -k- \sum_{t=1}^{p-1} i_t  +2\right)  
B_{\I^{\#}}(k)A_{\widetilde{\I^{\# }}}
(\tilde{\i}_1, \ldots, \tilde{\i}_{m-n-1}).
\end{align*}
Note that in this case ${\widetilde{\I^{\# }}}=\tilde{\I}$ and   
$\sum_{t=1}^{p-1} i_t =\sum_{t=1}^{\tilde{p}-1} \tilde{\i}_t.$ 
Thus we get,
\begin{align*}
\sum_{\sub{\beta} \in \Pi_{\sub{\mu}^{(\sub{j})}} } 
\wt\left((\tilde{\sub{\mu}},\tilde{\sub{\alpha}})_{\sub{j};\sub{\beta}}  \right) 
&= \left(\frac{r}{2} -k - \sum_{t=1}^{\tilde{p}-1} \tilde{\i}_t +2\right) 
B_{\I^{\#}}(k)A_{\tilde{\I}}
(\tilde{\i}_1, \ldots, \tilde{\i}_{m-n-1}). 
\end{align*}
Set
$$B_{\I}(X) :=\left(\frac{r}{2} -X- \sum_{t=1}^{\tilde{p}-1} \tilde{\i}_t  
+2\right)  B_{\I^{\#}}(X).$$ 
By our induction hypothesis applied to $\I^{\#}$, 
$B_{\I}$ has degree $n$, its leading term 
is $(-1)^n X^n$ since $\I^{\#}$ has $n-1$ zeros,   
and the coefficient of $B_{\I}(X)$ in $X^j$, $j\le n$, 
is a polynomial  
in the variable $\tilde{\i}_1, \ldots, \tilde{\i}_{m-n-1}$ 
of total degree $\leqslant n-j$. 
This proves the statement in this case. 

$\ast$ Second case: $p=q$ and  $i_{p-1}+i_p\not =0$. Then necessarily, $n<m$. 
By Lemma~\ref{Lem0:cut}(1), 
we get 
 \begin{align*}
\sum_{\sub{\beta} \in \Pi_{\sub{\mu}^{(\sub{j})}} } 
\wt\left((\tilde{\sub{\mu}},\tilde{\sub{\alpha}})_{\sub{j};\sub{\beta}}  \right)  
& =  \sum_{\sub{\beta} \in \Pi_{\sub{\mu}^{(\sub{j})}} } 
 \wt\left(((\tilde{\sub{\mu}},\tilde{\sub{\alpha}})_{\sub{j};\sub{\beta}})^{\# p}  \right).&
 \end{align*}
 Let $(\sub{\mu},\sub{\alpha}) \in \hat{\P}_m(\delta_k)$ such that 
 $\he(\sub{\mu}) \in \I$. 
 Observe that the class of $\he(\widetilde{\sub{\mu}^{\# p}})$ does not 
 depend on such a $(\sub{\mu},\sub{\alpha})$. 
 Moreover, for any $\sub{\beta} \in \Pi_{\sub{\mu}^{(\sub{j})}}$, 
 $$((\tilde{\sub{\mu}},\tilde{\sub{\alpha}})_{\sub{j};\sub{\beta}})^{\# p}
 = 
 (\widetilde{\sub{\mu}^{\# p}},\widetilde{\sub{\alpha}^{\# p}})_{\sub{j}';\sub{\beta}}$$
  where 
 $\sub{j}'=(j'_1,\ldots,j'_n)$ is the sequence of positions of zeroes of $
 \he({\sub{\mu}}^{\# p})$. 
Therefore by our induction hypothesis 
and Lemma \ref{Lem:weight1}, we get 
 \begin{align*}
& \sum_{\sub{\beta} \in \Pi_{\sub{\mu}^{(\sub{j})}} } 
 \wt\left(((\tilde{\sub{\mu}},\tilde{\sub{\alpha}})_{\sub{j};\sub{\beta}})^{\# p}  \right) 
 = 
 \sum_{\sub{\beta} \in \Pi_{\sub{\mu}^{(\sub{j}')}} } 
 \wt\left((\widetilde{\sub{\mu}^{\# p}},\widetilde{\sub{\alpha}^{\# p}})_{\sub{j}';\sub{\beta}}  \right) & \\
&\qquad = B_{\I^{\#}}(k) A_{\widetilde{\I^{\#}}}(\tilde{\i}_1,\ldots,
 \tilde{\i}_{p(\I^{\#})-1}+\tilde{\i}_{p(\I^\#)},\tilde{\i}_{m-1-n} )
 = B_{\I^{\#}} (k)A_{\tilde{\I}}(\tilde{\i}_1,\ldots,\tilde{\i}_{m-n-1} ). 
 \end{align*}
Since $\I^{\#}$ and $\I$ have the same number $n$ of zeroes,   
by setting 
 $B_{\I} := B_{\I^{\#}},$ 
 we get the statement by induction hypothesis. 
 
$\ast$ Third case: $p=q$  
and $i_{p-1}+i_p =0$. 
First assume that $i_1=\cdots=i_{p-2}=0$ or that $p=2$. 
By Lemma~\ref{Lem0:cut}(2)(a), 
we get 
 \begin{align*}
\sum_{\sub{\beta} \in \Pi_{\sub{\mu}^{(\sub{j})}} } 
\wt\left((\tilde{\sub{\mu}},\tilde{\sub{\alpha}})_{\sub{j};\sub{\beta}}  \right)  
& =  \sum_{\sub{\beta} \in \Pi_{\sub{\mu}^{(\sub{j})}} } 
 i_{p-1} 
 \wt\left(((\tilde{\sub{\mu}},\tilde{\sub{\alpha}})_{\sub{j};\sub{\beta}})^{\# p}  \right). 
 & 
\end{align*}
where $\sub{j}'$ is the sequence of 
positions of the zeroes of $\he({\sub{\mu}}^{\# p})$. 
So, by induction hypothesis 
Lemma \ref{Lem:weight1}, 
we obtain, arguing as in the second case, that 
 \begin{align*}
 &\sum_{\sub{\beta} \in \Pi_{\sub{\mu}^{(\sub{j})}} } 
 \wt\left(((\tilde{\sub{\mu}},\tilde{\sub{\alpha}})_{\sub{j};\sub{\beta}})^{\# p}  \right)
= i_{p-1} B_{\I^{\#}}(k) A_{\widetilde{\I^{\#}}}(\tilde{\i}_1,\ldots,
  \tilde{\i}_{p(\I^{\#})-2},\tilde{\i}_{p(\I^\#)+1},\tilde{\i}_{m-1-n} )
= B_{\I^{\#}} (k)A_{\tilde{\I}}(\tilde{\i}_1,\ldots,\tilde{\i}_{m-n-1} ).
 \end{align*} 
Since $\I^{\#}$ and $\I$ have the same number $n$ of zeroes, 
by setting 
 $B_{\I} := B_{\I^{\#}},$ 
 we get the statement by induction hypothesis. 

If we are not in one of the above situations, then by 
Lemma~\ref{Lem0:cut} (2)(b) and 
Lemma~\ref{Lem:weight1}   
and the induction hypothesis 
we conclude similarly. 
Namely, here we get that 
 \begin{align*}
 &\sum_{\sub{\beta} \in \Pi_{\sub{\mu}^{(\sub{j})}} } 
 \wt\left(((\tilde{\sub{\mu}},\tilde{\sub{\alpha}})_{\sub{j};\sub{\beta}})^{\# p}  \right) 
= B_{\I^{\#}} (k)A_{\tilde{\I}}(\tilde{\i}_1,\ldots,\tilde{\i}_{m-n-1} )
 \end{align*} 
 and we set $B_{\I} := B_{\I^{\#}}$. Since $\I$ and $\I^{\#}$ have 
 the same number $n$ of zeroes, we get the statement by induction hypothesis.\qed
\end{proof}

\begin{corollary}  \label{corollary:T_m} 
Let $m\in \Z_{> 0}$ and $n\in\{0,\ldots,m\}$. 
Let $\I \in \E_m$ with $n \leqslant m$ zeroes in positions 
$j_1,\ldots,j_n$,    
and $\tilde{\I}$ as in Lemma~\ref{Lem:weight2}. 
Then for some polynomial $T_{\I} \in \C[X]$ of degree 
at most $m$, 
for all $k\in\{1,\ldots,r+1\}$, 
$$
\sum_{\tilde{\sub{i}}  \in \tilde{\I}}  
\sum_{\sub{\beta} \in \Pi_{\sub{\mu}^{\sub{j}}}}  
\sum_{(\tilde{\sub{\mu}},\tilde{\sub{\alpha}}) \in \hat{\P}_{m-n}(\delta_k),   
\atop \tilde{\sub{\mu}} \in [\![ \delta_{r+1}, \delta_k]\!], 
\,\he(\tilde{\sub{\mu}}) = \tilde{\sub{i}}} 
\wt \left( (\tilde{\sub{\mu}},\tilde{\sub{\alpha}})_{\sub{j};\sub{\beta}} \right)
= T_{\I}(k).
$$
\end{corollary}

If $n=0$ or if $\I$ has no zero, then $T_{\I}$ is the polynomial provided 
by Lemma \ref{Lem:T_m}. If $n=m$, then $\I=[\sub{0}]$ 
and $T_{\I}=T_{m}$ 
is the polynomial provided by Lemma \ref{Lem2:main1}. 
So, in these two cases, the statement is known. 
\begin{proof}
Let $k \in \{1,\ldots,r+1\}$. 
By Lemma \ref{Lem:T_m} and Lemma \ref{Lem:weight2}, 
we have, 
\begin{align*}
&\sum_{\tilde{\sub{i}}  \in \tilde{\I}}  
\sum_{\sub{\beta} \in \Pi_{\sub{\mu}^{\sub{j}}}}  
\sum_{(\tilde{\sub{\mu}},\tilde{\sub{\alpha}}) \in \hat{\P}_{m-n}(\delta_k),   
\atop \tilde{\sub{\mu}} \in [\![ \delta_{r+1}, \delta_k]\!], 
\,\he(\tilde{\sub{\mu}}) = \tilde{\sub{i}}} 
\wt \left( (\tilde{\sub{\mu}},\tilde{\sub{\alpha}})_{\sub{j};\sub{\beta}} \right)  
= 
\sum_{(\tilde{\sub{\mu}},\tilde{\sub{\alpha}}) \in \hat{\P}_{m-n}(\delta_k),   
\atop \tilde{\sub{\mu}} \in [\![ \delta_{r+1}, \delta_k]\!], 
\,\he(\tilde{\sub{\mu}}) \in \tilde{\I}} 
B_{\I} (k) A_{\tilde{\I}}(\tilde{\i}_1, \ldots, \tilde{\i}_{m-n-1})\\
&\qquad 
= 
\sum_{(\tilde{\sub{\mu}},\tilde{\sub{\alpha}}) \in \hat{\P}_{m-n}(\delta_k),   
\atop \tilde{\sub{\mu}} \in [\![ \delta_{r+1}, \delta_k]\!], 
\,\he(\tilde{\sub{\mu}}) \in \tilde{\I}}
\sum_{j=0}^n 
\sum_{\sub{d}_j=(d_1, \ldots, d_{m-n-1}),
\atop {d_1 +\cdots + d_{m-n-1} \leqslant  n-j}}
C_{\sub{d},j}
\tilde{\i}_1^{d_1} \cdots \tilde{\i}_{m-n-1}^{d_{m-n-1}} 
k^j A_{\tilde{\I}}(\tilde{\i}_1, \ldots, \tilde{\i}_{m-n-1}). 
\end{align*}
Set 
$$\tilde{T}_{\sub{d}_j, \tilde{\I}}= 
\sum_{(\tilde{\sub{\mu}},\tilde{\sub{\alpha}}) \in \hat{\P}_{m-n}(\delta_k),   
\atop \tilde{\sub{\mu}} \in [\![ \delta_{r+1}, \delta_k]\!], 
\,\he(\tilde{\sub{\mu}}) \in \tilde{\I}}
\tilde{\i}_1^{d_1} \cdots \tilde{\i}_{m-n-1}^{d_{m-n-1}} 
A_{\tilde{\I}}(\tilde{\i}_1, \ldots, \tilde{\i}_{m-n-1}).$$
Then by Lemma \ref{Lem:T_m}, 
there are some polynomials 
${T}_{\sub{d}_j, \tilde{\I}}$ of degree at most $(n-j)+(m-n)=m-j$,
such that 
$$ \sum_{\tilde{\sub{i}}  \in \tilde{\I}}  
\sum_{\sub{\beta} \in \Pi_{\sub{\mu}^{\sub{j}}}}  
\sum_{(\tilde{\sub{\mu}},\tilde{\sub{\alpha}}) \in \hat{\P}_{m-n}(\delta_k),   
\atop \tilde{\sub{\mu}} \in [\![ \delta_{r+1}, \delta_k]\!], 
\,\he(\tilde{\sub{\mu}}) = \tilde{\sub{i}}} 
\wt \left( (\tilde{\sub{\mu}},\tilde{\sub{\alpha}})_{\sub{j};\sub{\beta}} \right)
= \sum_{j=0}^n 
\sum_{\sub{d}_j=(d_1, \ldots, d_{m-n-1}),
\atop {d_1 +\cdots + d_{m-n-1} \leqslant  n-j}}
 C_{\sub{d},j} k^j {T}_{\sub{d}_j, \tilde{\I}} (k).$$ 
Moreover, if $j <n$, then ${T}_{\sub{d}_j, \tilde{\I}}$ 
has degree $< m-j$. 
By setting 
$$T_{\I} (X):= \sum_{j=0}^n \sum_{\sub{d}_j=(d_1, \ldots, d_{m-n-1}),
\atop {d_1 +\cdots + d_{m-n-1} \leqslant  n-j}}
 C_{\sub{d},j} X^j {T}_{\sub{d}_j, \tilde{\I}} (X),$$
we have that $T_{\I}$ is a polynomial of 
degree at most $m-j+j = m$. \qed
\end{proof}

The following result is a direct consequence of Corollary \ref{corollary:T_m}. 
\begin{lemma}\label{Lem:T}
Let $m\in \Z_{>0}$, 
There is a polynomial $\hat{T}_{m}$ in $\C[X]$ of degree 
 at most $m$ 
such that for all $k\in\{1,\ldots,r+1\}$, 
$$\sum_{(\sub{\mu},\sub{\alpha}) \in \hat{\P}_m(\delta_{k})  
\atop \sub{\mu}\in  [\![\delta_{r+1},\delta_{k} ]\!]} \, 
\wt(\sub{\mu},\sub{\alpha}) = \hat{T}_m(k).$$
\end{lemma}

We are now in a position to prove Theorem \ref{theorem:main2}. 

\begin{proof}[Proof of Theorem \ref{theorem:main2}] 
By Lemma \ref{Lem:formulas}, we have
$$ \ev_\rho( \widehat{\d p}_{m,k}) =  \sum_{\mu \in P(\delta)_k} \, \sum_{(\sub{\mu},\sub{\alpha}) \in \hat{\P}_m(\mu)}  
\wt(\sub{\mu},\sub{\alpha}) \langle \mu, \c{\beta}_k \rangle
=\sum_{(\sub{\mu},\sub{\alpha}) \in \hat{\P}_m(\delta_k)}  \wt(\sub{\mu},\sub{\alpha}) - 
\sum_{(\sub{\mu},\sub{\alpha}) \in \hat{\P}_m(\delta_{k+1})}  \wt(\sub{\mu},\sub{\alpha}),$$
where $\sub{\mu}$ is entirely contained in $[\![ \delta_{r+1}, \delta_k]\!]$.
Let $\hat{T}_m$ be as in Lemma \ref{Lem:T} 
and set
\begin{align} 
\label{eq:hatQi-sl}
\hat{Q}_m:=\hat{T}_m (X) - \hat{T}_m(X+1).
\end{align}
Then $\hat{Q}_m$ is a polynomial of degree at most $m-1$, and we have 
\begin{align*}
 \ev_\rho (\widehat{{\rm d} p}_{m}  ) 
 &=  \ev_\rho \left(\frac{1}{m!}  \sum_{k=1}^{r}  
\widehat{\d p}_{m,k}\otimes \s{\varpi_k} \right) 
= \frac{1}{m!}  \sum_{k=1}^{r}  \ev_\rho \left(
\widehat{\d p}_{m,k} \right)\s{\varpi_k} 
 = \frac{1}{m!}  \sum_{k=1}^{r}  \hat{Q}_m\s{\varpi_k}.
\end{align*}
Moreover, $\hat{Q}_1 = 1$ \qed
\end{proof}

\section{The proofs for type $C$}\label{sec:main2-C}

The purpose of this section is to prove Theorem~\ref{theorem:main1} and Theorem~\ref{theorem:main2} for 
$\sp_{2r}$, $r \geqslant  2$. 
We follow the general strategy of the  $\sl_{r+1}$ case.
However, since the simple Lie algebra $\sp_{2r}$ is non simply-laced, it induces new phenomenon and the proof is much more technical 
and new tools are needed. 
Throughout this section, it is assumed 
that $\g=\sp_{2r}, r \geqslant 2$, and $\delta=\varpi_1$. 
We retain all relative notations from previous sections 
and Appendix \ref{sec:rootTypeC}. 
In particular, 
$P(\delta) = \{\delta_1, \cdots,\delta_r, \overline{\delta}_1, \ldots, \overline{\delta}_r \}$,
$P(\delta)_k=\{\delta_{k},\delta_{k+1}, \overline{\delta}_k, \overline{\delta}_{k+1}\}$ 
for $k=1,\ldots,r-1$,  
$P(\delta)_r=\{\delta_{r}, \overline{\delta}_{r}\},$ and  
$\Pi_{\delta_k}=\{\beta_{k-1},\beta_k\}$, $\Pi_{\overline{\delta}_k} 
=\{\beta_{k-1},\beta_k\}$ for $k=2,\ldots,r$, 
$\Pi_{\delta_1}=\{\beta_1\}$, $\Pi_{\overline{\delta}_1} =\{\beta_1\}.$  
 
According to Lemma \ref{Lem:formulas} and \eqref{eq:formulas}, we get 
for $k \in \{1,\ldots,r-1\}$, 
\begin{align} \label{eq:main3} \nonumber
\ev_\rho(\overline{\d p}_{m,k} ) &= \sum_{ \sub{\alpha} \in (\Pi_{\delta_k})^m}  
\prod_{i=1}^m \langle \delta_k, \c{\alpha}^{(i)} \rangle 
\langle \rho,\s{\varpi}_{\alpha^{(i)}} \rangle 
-  \sum_{(\sub{\alpha} \in (\Pi_{\delta_{k+1}})^m}  
\prod_{i=1}^m \langle \delta_{k+1}, \c{\alpha}^{(i)} \rangle 
\langle \rho, \s{\varpi}_{\alpha^{(i)}}  \rangle\nonumber\\
&\qquad + \sum_{ \sub{\alpha} \in (\Pi_{\overline{\delta}_{k+1}})^m}  
\prod_{i=1}^m \langle \overline{\delta}_{k+1}, \c{\alpha}^{(i)} \rangle 
\langle \rho,\s{\varpi}_{\alpha^{(i)}} \rangle 
-  \sum_{(\sub{\alpha} \in (\Pi_{\overline{\delta}_k})^m}  
\prod_{i=1}^m \langle \overline{\delta}_k, \c{\alpha}^{(i)} \rangle 
\langle \rho, \s{\varpi}_{\alpha^{(i)}}  \rangle. 
\end{align}
For $k =r$, 
\begin{align} \label{eq:main4}
\ev_\rho(\overline{\d p}_{m,r} ) = \sum_{ \sub{\alpha} \in (\Pi_{\delta_r})^m}  
\prod_{i=1}^m \langle \delta_r, \c{\alpha}^{(i)} \rangle 
\langle \rho,\s{\varpi}_{\alpha^{(i)}} \rangle  
- \sum_{(\sub{\alpha} \in (\Pi_{\overline{\delta}_r})^m}  
\prod_{i=1}^m \langle \overline{\delta}_r, \c{\alpha}^{(i)} \rangle 
\langle \rho, \s{\varpi}_{\alpha^{(i)}}  \rangle.
\end{align}

\begin{lemma} \noindent \label{lem:rho}
\begin{enumerate}
\item 
For any $j \in\{1,\ldots,r\}$, 
$\langle \rho, \eps_j \rangle = r - j+1.$ 
\item For $k \in \{1,\ldots,r\}$, 
$\langle \rho, \s{\varpi}_k \rangle = \frac{k}{2}(2r-k+1)$ and $   
\langle \rho, \s{\varpi}_k -\s{\varpi}_{k-1} \rangle 
= r -k +1, $
where by convention $\s{\varpi}_0=0$. 
\end{enumerate}
\end{lemma}

The proof of the lemma is easy and is left to the reader. 

\begin{lemma}\label{lem:symC} 
For some polynomial $T_m \in \C[X]$ of degree $m$,  we have 
$$
\sum_{ \sub{\alpha} \in (\Pi_{\delta_k})^m}  
\prod_{i=1}^m \langle \delta_k, \c{\alpha}^{(i)} \rangle 
\langle \rho,\s{\varpi}_{\alpha^{(i)}} \rangle =T_m(k)$$
and
$$
\sum_{ \sub{\alpha} \in (\Pi_{\overline{\delta}_k})^m}  
\prod_{i=1}^m 
\langle \overline{\delta}_k, \c{\alpha}^{(i)} \rangle 
\langle \rho,\s{\varpi}_{\alpha^{(i)}} \rangle = (-1)^m T_m(k)$$
for all $k \in\{ 1,\ldots,r\}$.
\end{lemma}

\begin{proof}
Assume first that $k\in \{2,\ldots,r\}$. 
Then by Lemma \ref{lem:rho}, 
\begin{align*}
\sum_{ \sub{\alpha} \in (\Pi_{\delta_k})^m}  
\prod_{i=1}^m \langle \delta_k, \c{\alpha}^{(i)} \rangle 
\langle \rho,\s{\varpi}_{\alpha^{(i)}} \rangle 
&= \sum_{i=0}^m \begin{pmatrix} m \\
i 
\end{pmatrix} (-1)^{i} 
\langle \rho,\s{\varpi}_{k-1} \rangle^{i} 
\langle \rho,\s{\varpi}_{k} \rangle^{m-i} 
 = (r -k +1)^m ,&
\end{align*}
and
\begin{align*}
\sum_{ \sub{\alpha} \in (\Pi_{\overline{\delta}_k})^m}  
\prod_{i=1}^m \langle \overline{\delta}_k, \c{\alpha}^{(i)} \rangle 
\langle \rho,\s{\varpi}_{\alpha^{(i)}} \rangle 
= \sum_{i=0}^m  \begin{pmatrix} m \\i 
\end{pmatrix} (-1)^{i} 
\langle \rho,\s{\varpi}_{k} \rangle^{i} 
\langle \rho,\s{\varpi}_{k-1} \rangle^{m-i} 
= (-1)^m(r-k+1)^m.&
\end{align*}
If $k=1$, then by Lemma \ref{lem:rho}, 
 \begin{align*}
\sum_{ \sub{\alpha} \in (\Pi_{\delta_1})^m}  
\prod_{i=1}^m \langle \delta_1, \c{\alpha}^{(i)} \rangle 
\langle \rho,\c{\varpi}_{\alpha^{(i)}} \rangle 
&= \langle \rho,\c{\varpi}_{1} \rangle^{m} 
=(r)^m=
 (r-k +1)^m,
\end{align*}
and
\begin{align*}
\sum_{ \sub{\alpha} \in (\Pi_{\overline{\delta}_1})^m}  
\prod_{i=1}^m \langle \delta_1, \c{\alpha}^{(i)} \rangle 
\langle \rho,\c{\varpi}_{\alpha^{(i)}} \rangle 
&= (- \langle \rho,\c{\varpi}_{1} \rangle)^{m} 
=(- r)^m=
 (-1)^m(r-k +1)^m.
\end{align*}
Hence, setting $T_m(X) := (r -X +1)^m$ we get the statement.\qed
\end{proof} 

We are now in a position to prove Theorem \ref{theorem:main1}  for $\g=\sp_{2r}$. 
\begin{proof}
[Proof of Theorem \ref{theorem:main1} for $\g=\sp_{2r}$] 
It easily seen that $p_m = 0$ if $m$ is even, then  
there is no loss of generality assuming that 
$m \in \{1,3, \ldots,2r-1\}$.
By Lemma~\ref{lem:symC} and \eqref{eq:main3} and \eqref{eq:main4}, 
we have for any $k \in \{1,\ldots,r-1\}$,
\begin{align*}
\ev_\rho(\overline{\d p}_{m,k} )&= T_m(k) - T_m(k+1)+ (-1)^m T_m(k+1) - (-1)^m T_m(k)
= 2\left((r-k+1)^m - (r-k)^m \right) \\
&= 2 \left( m(-k)^{m-1} + \sum_{i=0}^{m-2} \begin{pmatrix} m \\i 
\end{pmatrix} (r+1)^{m-i} (-k)^{i}
- \sum_{i=0}^{m-2} \begin{pmatrix} m \\ i \end{pmatrix} (r)^{m-i} (-k)^{i} \right)
\end{align*}
since $m$ is odd. 

For $k =r$,
$\ev_\rho(\overline{\d p}_{m,r} ) = T_m(r) - (-1)^m T_m(r) = (1-(-1)^m)T_m(r)= 2 (1)^m = 2,$ 
again since $m$ is odd. On the other hand,
$$ \ev_\rho(\overline{\d p}_{m,r} ) = 2 = 2(T_m(r)- T_m(r+1)).$$ 
Hence, by setting
\begin{align} 
\label{eq:Qi-sp}
&\bar{Q}_m(X):=
2 \left( m(-X)^{m-1} + \sum_{i=0}^{m-2} \begin{pmatrix} m \\i 
\end{pmatrix} (r+1)^{m-i} (-X)^{i}
- \sum_{i=0}^{m-2} \begin{pmatrix} m \\ i \end{pmatrix} (r)^{m-i} (-X)^{i} \right)
\end{align}
we get 
\begin{align*}
\ev_\rho (\overline{\d p}_{m}  ) 
=\frac{1}{m!}\sum_{k=1}^{r}\ev_\rho (\overline{\d p}_{m,k}) \s{\varpi}_k 
=\frac{1}{m!} \sum_{k=1}^{r}  \bar{Q}_m(k) 
\s{\varpi}_k. 
\end{align*}
Moreover, $\bar{Q}_1 = 2$ and $\bar{Q}_m$ is a polynomial of degree $m-1$. \qed
\end{proof}

Following the approach adopted for the $\sl_{r+1}$ case, 
our next step is to prove Theorem~\ref{corollary:cut-type-C}. 
To this end, we first intend to prove the following proposition 
whose proof will be a consequence of 
Lemma~\ref{Lem:nostar}, Lemma~\ref{Lem:star+}, Lemma~\ref{Lem:star0},
Lemma~\ref{lem:p<q} and Lemma~\ref{lem:loops}. 

\begin{proposition} \label{Pro:Conc1}
Let $m \in \Z_{> 1}$, $\mu \in P(\delta)$,  
$\sub{i} \in {\Z}^m_{\succ \sub{0}}$ and 
 $(\sub{\mu},\sub{\alpha}) \in \hat{\P}_m(\mu)_{\sub{i}}$.
Set $p:=p(\sub{i})$ and $q:=q(\sub{i}).$

\begin{enumerate}
\item Assume $\alpha^{(p-1)} +\alpha^{(p)} \not =0$.
\begin{enumerate}
 \item If for all $s \in \{1,\ldots,p-2\}$, either $\alpha^{(s)} +\alpha^{(p)} \in -\Delta_+$ or $\alpha^{(s)} +\alpha^{(p)} \notin \Delta \cup \{0\}$ then 
 there is a constant $K_{(\sub{\mu},\sub{\alpha})}^{\# p}$ and a scalar $\hat{K}$ such that
  $$\wt(\sub{\mu},\sub{\alpha}) = \hat{K} K_{(\sub{\mu},\sub{\alpha})}^{\# p} \wt(\sub{\mu},\sub{\alpha})^{\# p} .$$
  \item If for some $s \in \{1,\ldots,p-2\}$, $\alpha^{(s)} +\alpha^{(p)} \in \Delta_+$ then there are
  some constants $K_{(\sub{\mu},\sub{\alpha})}^{\# p}$, $K_{(\sub{\mu},\sub{\alpha})}^{\star a}$ and scalar $\hat{K}$ such that 
   $$\wt(\sub{\mu},\sub{\alpha}) = \hat{K} \left( K_{(\sub{\mu},\sub{\alpha})}^{\# p} \wt(\sub{\mu},\sub{\alpha})^{\# p} 
   + \sum_{a=1}^{N}  K_{(\sub{\mu},\sub{\alpha})}^{\star a} \wt(\sub{\mu},\sub{\alpha})^{\star a} \right),$$
   with $(\sub{\mu},\sub{\alpha})^{\star a}$ is a weighted path of length strictly smaller than 
   $m$.
  \item If for some $s \in \{1,\ldots,p-2\}$, $\alpha^{(s)} = -\alpha^{(p)}$ then there are
  some constants $K_{(\sub{\mu},\sub{\alpha})}^{\# p}$, $K_{(\sub{\mu},\sub{\alpha})}^{\star a}$ and a scalar $\hat{K}$ such that
  $$\wt(\sub{\mu},\sub{\alpha}) = \hat{K} \left( K_{(\sub{\mu},\sub{\alpha})}^{\# p} \wt(\sub{\mu},\sub{\alpha})^{\# p} 
  + (\he(\c{\alpha}^{(s)}) -c_s) \sum_{a=1}^{N}  K_{(\sub{\mu},\sub{\alpha})}^{\star a} \wt(\sub{\mu},\sub{\alpha})^{\star a} \right),$$
  with $c_{s}:= \langle \alpha^{(s-1)} , \c{\alpha}^{(s)} \rangle$.
\end{enumerate}
\item Assume $\alpha^{(p-1)} + \alpha^{(p)} = 0.$
\begin{enumerate} 
\item If $ i_1=\cdots=i_{p-2}=0$, or if $p=2$, then
$$\wt(\sub{\mu},\sub{\alpha})= 
(c_{\alpha^{(p-1)}})^2 \ \he(\c{\alpha}^{(p-1)}) 
\,\wt(\sub{\mu},\sub{\alpha})^{\# p}.$$
\item Otherwise, 
 \begin{enumerate}
  \item If for all $s \in \{1,\ldots,p-2\}$,  either $\alpha^{(s)} +\alpha^{(p)} \in -\Delta_+$ or $\alpha^{(s)} +\alpha^{(p)} \notin \Delta \cup \{0\}$ then 
 there is a scalar $\hat{K}$ such that
   \begin{align*}
   \wt(\sub{\mu},\sub{\alpha}) = \hat{K} (c_{\alpha^{(p-1)}})^2 \
   (\he(\c{\alpha}^{(p-1)})- c_{p-1})\wt(\sub{\mu},\sub{\alpha})^{\# p} ,
   \end{align*}
   with $c_{p-1}$ is an integer as in Lemma~\ref{Lem:nostar}.
  \item If for some $s \in \{1,\ldots,p-2\}$, $\alpha^{(s)} +\alpha^{(p)} \in \Delta_+$ then there are
  some constants $K_{(\sub{\mu},\sub{\alpha})}^{\star a}$ and a scalar $\hat{K}$ such that 
  \begin{align*}
\wt(\sub{\mu},\sub{\alpha}) = \hat{K} \left(
(\he(\c{\alpha}^{(p-1)})- c_{p-1}) \wt(\sub{\mu},\sub{\alpha})^{\# p} 
+ \sum_{a=1}^{N}  K_{(\sub{\mu},\sub{\alpha})}^{\star a} \wt(\sub{\mu},\sub{\alpha})^{\star a} \right).
\end{align*}
\item If for some $s \in \{1,\ldots,p-2\}$, $\alpha^{(s)} = -\alpha^{(p)}$ then  there are
  some constants $K_{(\sub{\mu},\sub{\alpha})}^{\star a}$ and a
  scalar $\hat{K}$ such that 
\begin{align*}
 \wt(\sub{\mu},\sub{\alpha}) &= \hat{K}
(\he(\c{\alpha}^{(p-1)})- c_{p-1}) \wt(\sub{\mu},\sub{\alpha})^{\# p} 
+ \hat{K} (\he(\c{\alpha}^{(s)}) - c_{s}) \sum_{a=1}^{N}  
K_{(\sub{\mu},\sub{\alpha})}^{\star a} \wt(\sub{\mu},\sub{\alpha})^{\star a}.
\end{align*}
 \end{enumerate}
\end{enumerate}
\item Assume $p < q$, then 
$$\wt(\sub{\mu},\sub{\alpha})=   \langle \mu^{(p)}, \c{\alpha}^{(p)} \rangle 
(\langle \rho , \s{\varpi}_{\alpha^{(p)}} \rangle - c_{p-1}) 
\wt (\sub{\mu},\sub{\alpha})^{\# p}.$$
\end{enumerate}
Note that
$(\sub{\mu},\sub{\alpha})^{\star a}$ is a weighted path as in Lemma~\ref{Lem:star+} or Lemma~\ref{Lem:star0}, depending on different cases,
and $N$ the number of possible paths of ${(\tilde{\sub{\mu}}^{\star a},\tilde{\sub{\alpha}}^{\star a})}$ as in Lemma~\ref{Lem:star+} or Lemma~\ref{Lem:star0}.
Note also that what we mean by ``constant'' is a complex number which only depends on constant structures $(n_{\alpha, \beta}, a_{\lambda, \mu}^{(b)}$),
but not on the height of $\alpha$. Moreover, the scalar $\hat{K}$
is described in Lemma~\ref{lem:loops}.
\end{proposition}

Fix $\mu \in P(\delta)$, $m \in \Z_{> 1}$,  
$\sub{i} \in {\Z}^m_{\succ \sub{0}}$ and 
$(\sub{\mu},\sub{\alpha}) \in \hat{\P}_m(\mu)_{\sub{i}}$, 
and set as usual 
$p:=p(\sub{i})$, $q:=q(\sub{i})$.   
The proof of the following lemma is similar 
to that of Lemma~\ref{Lem0:cut} (1) and (2). So we omit the details. 

\begin{lemma} \label{Lem:nostar} 
Assume that the weighted path $(\sub{\mu},\sub{\alpha}) \in \hat{\P}_m(\mu)_{\sub{i}}$ has no loop 
and that $p=q$. 
If for all $s \in \{1, \ldots, p-2 \}$ either $\alpha^{(s)} + \alpha^{(p)} \notin \Delta \cup \{0\}$ or $\alpha^{(s)} + \alpha^{(p)} \in -\Delta_+$,
then there is a scalar $\bar{K}^{\# p}$ 
such that: $$ \wt (\sub{\mu},\sub{\alpha}) = \bar{K}^{\# p} \wt(\sub{\mu},\sub{\alpha})^{\# p},$$
where $\bar{K}^{\# p}$ is described as follows:
\begin{enumerate}
\item If $\alpha^{(p-1)} + \alpha^{(p)} \not= 0$,
then $$ \bar{K}^{\# p} = {K}_{(\sub{\mu},\sub{\alpha})}^{\# p},$$
where 
${K}_{(\sub{\mu},\sub{\alpha})}^{\# p}$ is a constant which only depends on constant structures (\S\ref{sec:rootTypeC}).
\item Assume $\alpha^{(p-1)} + \alpha^{(p)} = 0$
\begin{enumerate} 
\item If $p=2$, then 
$$ \bar{K}^{\# p} = (c_{\alpha^{(1)}})^2 \ \he(\c{\alpha}^{(1)}).$$
\item Otherwise, 
$$ \bar{K}^{\# p} = (c_{\alpha^{(p-1)}})^2 \ (\he(\c{\alpha}^{(p-1)})- c_{p-1}), $$
\end{enumerate}
where 
\begin{align*} 
 c_{p-1} := \sum_{\alpha^{(j_k)} \in \sub{\alpha}} \langle \alpha^{(j_k)} , \c{\alpha}^{(p-1)} \rangle.
\end{align*}
\end{enumerate}
\end{lemma}

Contrary to the $\sl_{r+1}$ case (cf.~Lemma~\ref{Lem2:cut-bis}), 
it may happen that there exists, for $s \in \{1, \ldots, p-2 \}$, 
a positive root $\alpha^{(s)} \in \sub{\alpha}$ such that 
either $\alpha^{(s)} + \alpha^{(p)} \in \Delta_+$ or $\alpha^{(s)}= - \alpha^{(p)}$. In this situation, 
the cutting vertex operation induces several new paths. This phenomenon will appear in the next 
two lemmas. 

\begin{lemma} \label{Lem:star+} 
Assume that the weighted path $(\sub{\mu},\sub{\alpha}) \in \hat{\P}_m(\mu)_{\sub{i}}$ has no loop 
and that $p=q$.  
Assume that for some $s \in \{1, \ldots, p-2 \}$, $\alpha^{(s)} + \alpha^{(p)} \in \Delta_+$.
In this case,
\begin{align*}
\wt (\sub{\mu},\sub{\alpha}) 
&= \bar{K}^{\# p} \wt(\sub{\mu},\sub{\alpha})^{\# p} + \sum_{a=1}^{N}  
{K}_{(\sub{\mu},\sub{\alpha})}^{\star a}\wt(\sub{\mu}, \sub{\alpha})^{\star a},
\end{align*}
where $\bar{K}^{\# p}$ is a scalar as in Lemma~\ref{Lem:nostar}, 
$K^{\star a}$ are some constants which only depend on constant structures, 
and $(\sub{\mu}, \sub{\alpha})^{\star a}:= (\sub{\mu}^{\star a},\sub{\alpha}^{\star a})$ 
is a concatenation of paths defined as follows. 
\begin{enumerate}
\item If $\alpha^{(p)} = \overline{\delta}_j - \delta_i$ and $\alpha^{(s)} =\delta_k - \overline{\delta}_i$, with $k <j<i$, 
then 
\begin{align*}
(\sub{\mu}, \sub{\alpha})^{\star a} = (\sub{\mu}',\sub{\alpha}')\star 
\big( (\delta_k,\delta_j), (\alpha^{(s)}+\alpha^{(p)})\big) 
\star (\tilde{\sub{\mu}}^{\star a},\tilde{\sub{\alpha}}^{\star a})
\star(\sub{\mu}'',\sub{\alpha}''),
\end{align*}
where 
$(\tilde{\sub{\mu}}^{\star a},\tilde{\sub{\alpha}}^{\star a})$ 
is a path of length $< p-s$ between 
$\delta_j$ and $\delta_i$ whose roots $\tilde{\sub{\alpha}}^{\star a}$ are 
sums among $(\alpha^{(p-1)}, \ldots, \alpha^{(s+1)})$.

\item If $\alpha^{(p)} = \overline{\delta}_j - \delta_i$ and $\alpha^{(s)} =\delta_j - \overline{\delta}_l$, with $j < l <i$, then
\begin{align*}
(\sub{\mu}, \sub{\alpha})^{\star a} = (\sub{\mu}',\sub{\alpha}')
\star (\tilde{\sub{\mu}}^{\star a},\tilde{\sub{\alpha}}^{\star a}) 
\star(\sub{\mu}'',\sub{\alpha}''),
\end{align*}
where
$(\tilde{\sub{\mu}}^{\star a},\tilde{\sub{\alpha}}^{\star a})$ 
is a path of length $< p-s-1$ between 
$\delta_j$ and $\delta_i$ whose roots $\tilde{\sub{\alpha}}^{\star a}$ are 
sums among $(\alpha^{(p-1)}, \ldots, \alpha^{(s+1)}, \alpha^{(s)}+\alpha^{(p)} ).$

\item If $\alpha^{(p)} = \overline{\delta}_i - \overline{\delta}_j$ and $\alpha^{(s)} =\delta_i - \overline{\delta}_l$, with $i<j$and $i<l$, then 
 \begin{align*}
(\sub{\mu}, \sub{\alpha})^{\star a} = (\sub{\mu}',\sub{\alpha}')
\star (\tilde{\sub{\mu}}^{\star a},\tilde{\sub{\alpha}}^{\star a})
\star(\sub{\mu}'',\sub{\alpha}''),
\end{align*}
where 
$(\tilde{\sub{\mu}}^{\star a},\tilde{\sub{\alpha}}^{\star a})$ 
is a path of length $< p-s-1$ between 
$\delta_i$ and $\overline{\delta}_j$ whose roots $\tilde{\sub{\alpha}}^{\star a}$ are 
sums among $(\alpha^{(p-1)}, \ldots, \alpha^{(s+1)}, \alpha^{(s)}+\alpha^{(p)} )$.

\item If $\alpha^{(p)} = \overline{\delta}_i - \overline{\delta}_j$ and $\alpha^{(s)} =\delta_i - \delta_l$, with $i < j <l$, then
 \begin{align*}
(\sub{\mu}, \sub{\alpha})^{\star a}=(\sub{\mu}',\sub{\alpha}')
\star (\tilde{\sub{\mu}}^{\star a},\tilde{\sub{\alpha}}^{\star a})
\star(\sub{\mu}'',\sub{\alpha}''),
\end{align*}
where 
$(\tilde{\sub{\mu}}^{\star a},\tilde{\sub{\alpha}}^{\star a})$ 
is a path of length $< p-s-1$ between 
$\delta_i$ and $\overline{\delta}_j$ whose roots $\tilde{\sub{\alpha}}^{\star a}$ are 
sums among $(\alpha^{(p-1)}, \ldots, \alpha^{(s+1)}, \alpha^{(s)}+\alpha^{(p)} )$.

\item If $\alpha^{(p)} = \overline{\delta}_i - \overline{\delta}_j$ and $\alpha^{(s)} =\delta_k - \delta_j$, with $k <i<j$, then  
 \begin{align*}
(\sub{\mu}, \sub{\alpha})^{\star a} = (\sub{\mu}',\sub{\alpha}')\star 
\big( (\delta_k,\delta_i), (\alpha^{(s)}+\alpha^{(p)})\big) 
\star (\tilde{\sub{\mu}}^{\star a},\tilde{\sub{\alpha}}^{\star a})
\star(\sub{\mu}'',\sub{\alpha}''),
\end{align*}
where 
$(\tilde{\sub{\mu}}^{\star a},\tilde{\sub{\alpha}}^{\star a})$ 
is a path of length $< p-s$ between 
$\delta_i$ and $\overline{\delta}_j$ whose roots whose roots $\tilde{\sub{\alpha}}^{\star a}$ are 
sums among $(\alpha^{(p-1)}, \ldots, \alpha^{(s+1)})$.
\end{enumerate}
In all those cases,
$$
(\sub{\mu}',\sub{\alpha}') =\big( (\mu^{(1)},\ldots,\mu^{(s-1)}, 
\mu^{(s)}),
(\alpha^{(1)},\ldots,\alpha^{(s-1)})\big),
$$
$$
(\sub{\mu}'',\sub{\alpha}'') =\big((\mu^{(p+1)},\ldots,\mu^{(m+1)}),
(\alpha^{(p+1)},\ldots,\alpha^{(m)})\big), 
$$
and ${N}$ is the number of possible paths of $(\tilde{\sub{\mu}}^{\star a},\tilde{\sub{\alpha}}^{\star a})$.
\end{lemma}

\begin{proof}
Let $-\alpha^{(p)}, \alpha^{(s)} \in \Delta_+$ such that $\alpha^{(s)} 
+ \alpha^{(p)}  \in \Delta_+$. By Lemma~\ref{lem:types} 
the only possibilities for $\alpha^{(p)}$ and $\alpha^{(s)}$ are:
\begin{itemize} 
\item 
$\alpha^{(p)} = \overline{\delta}_j - \delta_i$ and $\alpha^{(s)}=\delta_k - \overline{\delta}_i$, with $k < j < i$,  
\item  
$\alpha^{(p)} = \overline{\delta}_j - \delta_i$ and $\alpha^{(s)}=\delta_j - \overline{\delta}_l$, with $j < l < i$,  
\item 
$\alpha^{(p)} = \overline{\delta}_i - \overline{\delta}_j$ and $\alpha^{(s)} =\delta_i - \overline{\delta}_l$, with $i<l$,
\item
$\alpha^{(p)} = \overline{\delta}_i - \overline{\delta}_j$ and $\alpha^{(s)} =\delta_i - \delta_l$, with $i < j <l$,
\item
$\alpha^{(p)} = \overline{\delta}_i - \overline{\delta}_j$ and $\alpha^{(s)} =\delta_k - \delta_j$, with $k <i<j$.
\end{itemize}

\noindent
(1) Assume that $\alpha^{(p)} = \overline{\delta}_j - \delta_i$ 
and $\alpha^{(s)}=\delta_k - \overline{\delta_i}$, with $k < j < i$. 
Write 
 \begin{align*}
  (\sub{\mu},\sub{\alpha}) = &(\sub{\mu}',\sub{\alpha}') \star 
\big(\delta_k,\overline{\delta}_i, \ldots, 
\mu^{(p-1)},\overline{\delta}_j, \delta_i),
(\alpha^{(s)}, \alpha^{(s+1)}, \ldots ,\alpha^{(p-1)},\alpha^{(p)})\big) 
\star (\sub{\mu}'',\sub{\alpha}''),
 \end{align*}
where 
$$(\sub{\mu}',\sub{\alpha}')=\big((\mu^{(1)},\ldots,\mu^{(s-1)}, 
\mu^{(s)}=\delta_{k}), (\alpha^{(1)},\ldots,\alpha^{(s-1)})\big),$$
$$(\sub{\mu}'',\sub{\alpha}'')=\big((\mu^{(p+1)}=\delta_i,\ldots,\mu^{(m+1)}),
(\alpha^{(p+1)},\ldots,\alpha^{(m)})\big)$$
have length $s-1$ and $m - p$, respectively. 
In this case $\mu^{(p+1)} \succ \mu^{(s+1)} \succcurlyeq \mu^{(p-1)}$ 
and so $\alpha^{(p-1)}+ \alpha^{(p)} < 0$.
Note that 
$\alpha^{(p)}$ commutes with all roots 
$\alpha^{(t)}$ for $t< p$, except with $\alpha^{(p-1)}$ 
and $\alpha^{(s)}$. 
Set 
\begin{align}\label{eq:a}
a := a_{\mu^{(p+1)},\mu^{(p)}}^{(c_{\alpha^{(p)}} e_{-\alpha^{(p)}})}
a_{\mu^{(p)},\mu^{(p-1)}}^{(c_{\alpha^{(p-1)}} e_{-\alpha^{(p-1)}})}
\cdots 
a_{\mu^{(s+2)},\mu^{(s+1)}}^{(c_{\alpha^{(s+1)}}e_{-\alpha^{(s+1)}})}
a_{\mu^{(s+1)},\mu^{(s)}}^{(c_{\alpha^{(s)}}e_{-\alpha^{(s)}})}. 
\end{align} 
We have,
\begin{align}\label{eq:hc-type-II} \nonumber
\hc(b_{(\sub{\mu},\sub{\alpha})}^*)
 &= a  
 \hc(b_{(\sub{\mu}',\sub{\alpha}')}^* 
 e_{\alpha^{(s)}}
 e_{\alpha^{(s+1)}} \cdots 
 e_{\alpha^{(p-2)}}
 e_{\alpha^{(p-1)}} e_{\alpha^{(p)}} 
 b_{(\sub{\mu}'',\sub{\alpha}'')}^*) \\ 
&= {K}_{(\sub{\mu},\sub{\alpha})}^{\# p} \hc (b_{(\sub{\mu},\sub{\alpha})^{\# p}}^*)
+ a n_{\alpha^{(s)},\alpha^{(p)}}
\hc(b_{(\sub{\mu}',\sub{\alpha}')}^* 
 e_{\alpha^{(s)}+\alpha^{(p)}}
  e_{\alpha^{(s+1)}} \cdots 
 e_{\alpha^{(p-2)}}
 e_{\alpha^{(p-1)}} 
 b_{(\sub{\mu}'',\sub{\alpha}'')}^*) 
\end{align}
since $(\sub{\mu}',\sub{\alpha}')$ 
 and the positive root $-\alpha^{(p)}$ verify the conditions of 
 Lemma~\ref{Lem2:cut-bis} and 
${K}_{(\sub{\mu},\sub{\alpha})}^{\# p}$ is as in Lemma~\ref{Lem:nostar}~(1).
$\alpha^{(s)}+\alpha^{(p)} =\eps_k  - \eps_j = \delta_k -\delta_j$.
Because of the configuration, for all $t \in \{1,\ldots,s-1\}$, 
$\alpha^{(t)}=\delta_{j_t}-\delta_{j_{t+1}}
= \eps_{j_t}-\eps_{j_{t+1}}$, with $j_t <j_{t+1}\leqslant k<i$, 
and 
for all $t \in \{s+1,\ldots,p-1\}$, 
$\alpha^{(t)}=\overline{\delta}_{j_t}-\overline{\delta}_{j_{t+1}}
= \eps_{j_{t+1}}-\eps_{j_{t}}$, with $j \leqslant j_{t+1} <j_{t}\leqslant i$. 
Hence 
$\big((\delta_j,\delta_{p-1},\delta_{p-2},
\ldots,\delta_{s+2},\delta_{i}), \ (\alpha^{(p-1)},\ldots,\alpha^{(s+1)})\big)$ 
is a path from $\delta_j$ to $\delta_i$ (see Figure \ref{fig:typeII1} for an illustration).

\begin{figure}[h]
\begin{center}
\begin{tikzpicture}[scale=0.85]
  \draw[ thick,->](2,0) -- (2.3,0);
  \draw[ thick,decoration={markings, mark=at position 0.5 with {\arrow{>}}},
        postaction={decorate}](2.7,0) -- (3,0);
  \draw[ thick, ->](3,0) -- (3.8,0);
  \draw[ thick,decoration={markings, mark=at position 0.5 with {\arrow{>}}},
        postaction={decorate}](4.2,0) -- (5,0);
  \draw[ thick,->](5,0) -- (5.8,0);
  \draw[ thick, decoration={markings, mark=at position 0.5 with {\arrow{>}}},
        postaction={decorate}](6.2,0) -- (7,0);
  \draw[ thick,->](7,0) -- (7.8,0);
  \draw[ thick,decoration={markings, mark=at position 0.5 with {\arrow{>}}},
        postaction={decorate}](8.2,0) -- (9,0);
 \draw[ thick,decoration={markings, mark=at position 0.5 with {\arrow{>}}},
        postaction={decorate}](9,0) -- (10,0);
 \draw[ thick, ->](10,0) -- (10.8,0);
 \draw[ thick,decoration={markings, mark=at position 0.5 with {\arrow{>}}},
        postaction={decorate}](11.2,0) -- (12,0);
 \draw[ thick,->](12,0) -- (12.8,0);
 \draw[ thick,decoration={markings, mark=at position 0.5 with {\arrow{>}}},
        postaction={decorate}](13.2,0) -- (14,0);
 \draw[ ->, thick](14,0) -- (14.5,0);
     \draw[thick, dotted] (2.3,0) -- (2.7,0);
      \draw[thick, dotted] (3.8,0) -- (4.2,0);
      \draw[thick, dotted] (5.8,0) -- (6.2,0);
      \draw[thick, dotted] (7.8,0) -- (8.2,0);
      \draw[thick, dotted] (10.8,0) -- (11.2,0);
      \draw[thick, dotted] (12.8,0) -- (13.2,0);
      \draw[thick, dotted] (14.5,0) -- (15,0);
    \fill (2,0) circle (0.07)node(xline)[below] {{\small $\delta_1$}}; 
    \fill (3,0) circle (0.07)node(xline)[below] {{\small $\delta_k$}};
    \fill (5,0) circle (0.07)node(xline)[below] {{\small $\delta_j$}};
    \fill (7,0) circle (0.07)node(xline)[below] {{\small $\delta_{i}$}}; 
    \fill (9,0) circle (0.07)node(xline)[below] {{\small $\delta_{r}$}};
    \fill (10,0) circle (0.07)node(xline)[below] {{\small $\bar{\delta}_{r}$}};
    \fill (12,0) circle (0.07)node(xline)[below] {{\small $\bar{\delta}_{i}$}}; 
    \fill (14,0) circle (0.07)node(xline)[below] {{\small $\bar{\delta}_{j}$}};
\fill[blue] (3,-1.2) circle (0.05);
 \draw[dashed, blue, decoration={markings, mark=at position 0.5 with {\arrow{>}}},
        postaction={decorate}](2,-1.2) -- (3,-1.2);
  \node[text=blue] at (2.5,-0.9) {{\footnotesize $\sub{\alpha}'$}} ;
  \draw[thick, blue, decoration={markings, mark=at position 0.5 with {\arrow{>}}},
        postaction={decorate}](3,-1.2) -- (12,-1.2);
  \node[text=blue] at (7.5,-0.9) {{\footnotesize $\alpha^{(s)}$}} ;
\fill[blue] (12,-1.2) circle (0.05);
  \draw[thick, blue, ->](12,-1.2) -- (12.8,-1.2);
  \node[text=blue] at (12.6,-0.9) {{\footnotesize $\alpha^{(s+1)}$}} ;
\draw[thick, dotted,blue] (12.8,-1.2) -- (13.2,-1.2);
  \draw[thick, blue, decoration={markings, mark=at position 0.5 with {\arrow{>}}},
        postaction={decorate}](13.2,-1.2) -- (14,-1.2);
  \node[text=blue] at (13.8,-0.9) {{\footnotesize $\alpha^{(p-1)}$}} ;
  \fill[blue] (14,-1.2) circle (0.05);
  \draw[thick, blue, decoration={markings, mark=at position 0.5 with {\arrow{>}}},
        postaction={decorate}](14,-1.35) -- (7,-1.35);
  \node[text=blue] at (10.5,-1.65) {{\footnotesize $\alpha^{(p)}$}} ;
\fill[blue] (7,-1.35) circle (0.05);
  \draw[dashed, blue, decoration={markings, mark=at position 0.5 with {\arrow{>}}},
        postaction={decorate}](7,-1.35) -- (5,-1.35);
  \node[text=blue] at (6,-1.65) {{\footnotesize $\sub{\alpha}''$}} ;
\fill[red] (3,-2.5) circle (0.05);
\draw[dashed, red, decoration={markings, mark=at position 0.5 with {\arrow{>}}},
        postaction={decorate}](2,-2.5) -- (3,-2.5);
  \node[text=red] at (2.5,-2.2) {{\footnotesize $\sub{\alpha}'$}} ;
 \draw[thick, red,decoration={markings, mark=at position 0.5 with {\arrow{>}}},
        postaction={decorate}](3,-2.5) -- (5,-2.5);
  \node[text=red] at (4,-2.2) {{\footnotesize $\alpha^{(s)}+\alpha^{(p)}$}} ;
\fill[red] (5,-2.5) circle (0.05);
  \draw[thick, red, ->](5,-2.5) -- (5.8,-2.5);
  \node[text=red] at (5.6,-2.2) {{\footnotesize $\alpha^{(p-1)}$}} ;
 \draw[thick, dotted,red] (5.8,-2.5) -- (6.2,-2.5);
 \draw[thick, red, , decoration={markings, mark=at position 0.5 with {\arrow{>}}},
        postaction={decorate}](6.2,-2.5) -- (7,-2.5);
  \node[text=red] at (6.8,-2.2) {{\footnotesize $\alpha^{(s+1)}$}} ;
\fill[red] (7,-2.5) circle (0.05);
 \draw[dashed, red, decoration={markings, mark=at position 0.5 with {\arrow{>}}},
        postaction={decorate}](7,-2.65) -- (5,-2.65);
  \node[text=red] at (6,-2.95) {{\footnotesize $\sub{\alpha}''$}} ;
\end{tikzpicture}
 \caption{The case for $\alpha^{(p)} = \overline{\delta}_j - \delta_i$ and $\alpha^{(s)} = \delta_k - \overline{\delta}_i$.}
  \label{fig:typeII1}
\end{center}
\end{figure}
We have to ^^ ^^ reverse the order'' of 
$ e_{\alpha^{(s+1)}} \cdots 
 e_{\alpha^{(p-1)}}$ in \eqref{eq:hc-type-II}, 
in order to get $e_{\alpha^{(p-1)}} \cdots e_{\alpha^{(s+1)}}$.
Doing this, it induces several new paths. 
Let $(\tilde{\sub{\mu}}^{\star a},\tilde{\sub{\alpha}}^{\star a})$, for $a = 1, \ldots, N$, denote these 
new paths from
$\delta_j$ to $\delta_i$ whose roots are sums among the roots $(\alpha^{(p-1)},\ldots,\alpha^{(s+1)})$.
More precisely, for each $a \in \{1, \ldots, N \}$ there exists a partition 
$(P_1, \ldots, P_{n_a})$ of the set $\{p-1, \ldots, s+1 \}$ such that  
$$(\tilde{\alpha}^{\star a})^{(s+j)} = \sum_{t \in P_j} \alpha^{(t)}, $$
for $j = 1, \ldots, n_a$. 
Furthermore, by setting 
$(\sub{\mu}, \sub{\alpha})^{\star a} := (\sub{\mu}',\sub{\alpha}')\star 
\big( (\delta_k,\delta_j), (\alpha^{(s)}+\alpha^{(p)})\big) 
\star (\tilde{\sub{\mu}}^{\star a},\tilde{\sub{\alpha}}^{\star a})
\star(\sub{\mu}'',\sub{\alpha}''), $
we get that 
$$
\wt (\sub{\mu},\sub{\alpha})= {K}_{(\sub{\mu},\sub{\alpha})}^{\# p} \wt(\sub{\mu},\sub{\alpha})^{\# p} + \sum_{a=1}^{N}    
{K}_{(\sub{\mu},\sub{\alpha})}^{\star a}\wt(\sub{\mu}, \sub{\alpha})^{\star a},$$
where ${K}_{(\sub{\mu},\sub{\alpha})}^{\star a}$ are some constants.

\noindent
(2) Assume that $\alpha^{(p)} = \overline{\delta}_j - \delta_i$ and
$\alpha^{(s)}=\delta_j - \overline{\delta_l}$, with $j < l < i$.  
Write 
 \begin{align*}
  (\sub{\mu},\sub{\alpha}) = &(\sub{\mu}',\sub{\alpha}') \star 
\big(\delta_j,\overline{\delta}_l, \ldots, 
\mu^{(p-1)},\overline{\delta}_j,\delta_i),
(\alpha^{(s)}, \alpha^{(s+1)}, \ldots ,\alpha^{(p-1)},\alpha^{(p)})\big) 
\star (\sub{\mu}'',\sub{\alpha}''),
 \end{align*}
where
$$
(\sub{\mu}',\sub{\alpha}') =\big((\mu^{(1)},\ldots,\mu^{(s-1)}, 
\mu^{(s)}=\delta_{j}), \  
(\alpha^{(1)},\ldots,\alpha^{(s-1)})\big), 
$$
$$
(\sub{\mu}'',\sub{\alpha}'') =\big((\mu^{(p+1)}=\delta_i,\ldots,\mu^{(m+1)}), \
(\alpha^{(p+1)},\ldots,\alpha^{(m)})\big)
$$
have length $s-1$ and $m - p$, respectively. 
Observe that in this 
 case $ \alpha^{(p-1)} + \alpha^{(p)} < 0$ and
$\alpha^{(p)}$ commutes with all roots 
$\alpha^{(t)}, t< p$, except with $\alpha^{(p-1)}$ 
and $\alpha^{(s)}$. 
With  $a$ as in \eqref{eq:a}, 
we obtain that 
 \begin{align} \label{eq:hc-type-II2}
\hc( b_{(\sub{\mu},\sub{\alpha})}^*)
= {K}_{(\sub{\mu},\sub{\alpha})}^{\# p} \hc (b_{(\sub{\mu},\sub{\alpha})^{\# p}}^*)
+ a n_{\alpha^{(s)},\alpha^{(p)}}
\hc(b_{(\sub{\mu}',\sub{\alpha}')}^* 
 e_{\alpha^{(s)}+\alpha^{(p)}}
 e_{\alpha^{(s+1)}} \cdots 
 e_{\alpha^{(p-2)}}
 e_{\alpha^{(p-1)}} 
 b_{(\sub{\mu}'',\sub{\alpha}'')}^*),  
 \end{align}
where ${K}_{(\sub{\mu},\sub{\alpha})}^{\# p}$ is as in Lemma~\ref{Lem:nostar}(1).

We easily verify that 
$\big((\delta_j,\delta_{p-1},\delta_{p-2},
\ldots,\delta_{s+2},\delta_{l}, \delta_i), (\alpha^{(p-1)},\ldots,\alpha^{(s+1)}, \alpha^{(s)}+\alpha^{(p)})\big)$ 
is a path from $\delta_j$ to $\delta_i$ (see Figure~\ref{fig:typeII2} for an illustration).

\begin{figure}[h]
\begin{center}
\begin{tikzpicture} [scale=0.8]
\draw[ thick,decoration={markings, mark=at position 0.5 with {\arrow{>}}},
        postaction={decorate}](1.2,0) -- (2,0);
\draw[ thick, ->](2,0) -- (2.8,0);
\draw[ thick,decoration={markings, mark=at position 0.5 with {\arrow{>}}},
        postaction={decorate}](3.2,0) -- (4,0);
\draw[ thick,->](4,0) -- (4.8,0);
\draw[ thick,decoration={markings, mark=at position 0.5 with {\arrow{>}}},
        postaction={decorate}](5.2,0) -- (6,0);
\draw[ thick,->](6,0) -- (6.8,0);
\draw[ thick,decoration={markings, mark=at position 0.5 with {\arrow{>}}},
        postaction={decorate}](7.2,0) -- (8,0);
\draw[ thick,decoration={markings, mark=at position 0.5 with {\arrow{>}}},
        postaction={decorate}](8,0) -- (9,0);
\draw[ thick, ->](9,0) -- (9.8,0); 
\draw[ thick,decoration={markings, mark=at position 0.5 with {\arrow{>}}},
        postaction={decorate}](10.2,0) -- (11,0);
\draw[ thick,->](11,0) -- (11.8,0);
\draw[ thick,decoration={markings, mark=at position 0.5 with {\arrow{>}}},
        postaction={decorate}](12.2,0) -- (13,0);
\draw[ ->, thick](13,0) -- (13.8,0);
\draw[ thick,decoration={markings, mark=at position 0.5 with {\arrow{>}}},
        postaction={decorate}](14.2,0) -- (15,0);
     \draw[thick, dotted] (0.8,0) -- (1.2,0);
      \draw[thick, dotted] (2.8,0) -- (3.2,0);
      \draw[thick, dotted] (4.8,0) -- (5.2,0);
      \draw[thick, dotted] (6.8,0) -- (7.2,0);
      \draw[thick, dotted] (9.8,0) -- (10.2,0);
      \draw[thick, dotted] (11.8,0) -- (12.2,0);
      \draw[thick, dotted] (13.8,0) -- (14.2,0);
      \draw[thick, dotted] (15,0) -- (15.3,0);
    \fill (0.8,0) circle (0.07)node(xline)[below] {{\small $\delta_1$}}; 
    \fill (2,0) circle (0.07)node(xline)[below] {{\small $\delta_j$}};
    \fill (4,0) circle (0.07)node(xline)[below] {{\small $\delta_l$}};
    \fill (6,0) circle (0.07)node(xline)[below] {{\small $\delta_{i}$}}; 
    \fill (8,0) circle (0.07)node(xline)[below] {{\small $\delta_{r}$}};
    \fill (9,0) circle (0.07)node(xline)[below] {{\small $\bar{\delta}_{r}$}};
    \fill (11,0) circle (0.07)node(xline)[below] {{\small $\bar{\delta}_{i}$}}; 
    \fill (13,0) circle (0.07)node(xline)[below] {{\small $\bar{\delta}_{l}$}};
     \fill (15,0) circle (0.07)node(xline)[below] {{\small $\bar{\delta}_{j}$}};
\fill[blue] (2,-1.2) circle (0.05);
 \draw[dashed, blue, decoration={markings, mark=at position 0.5 with {\arrow{>}}},
        postaction={decorate}](1,-1.2) -- (2,-1.2);
  \node[text=blue] at (1.5,-0.9) {{\footnotesize $\sub{\alpha}'$}} ;
  \draw[thick, blue, decoration={markings, mark=at position 0.5 with {\arrow{>}}},
        postaction={decorate}](2,-1.2) -- (13,-1.2);
  \node[text=blue] at (7.5,-0.9) {{\footnotesize $\alpha^{(s)}$}} ;
\fill[blue] (13,-1.2) circle (0.05);
  \draw[thick, blue, ->](13,-1.2) -- (13.8,-1.2);
  \node[text=blue] at (13.6,-0.9) {{\footnotesize $\alpha^{(s+1)}$}} ;
\draw[thick, dotted,blue] (13.8,-1.2) -- (14.2,-1.2);
  \draw[thick, blue, decoration={markings, mark=at position 0.5 with {\arrow{>}}},
        postaction={decorate}](14.2,-1.2) -- (15,-1.2);
  \node[text=blue] at (14.8,-0.9) {{\footnotesize $\alpha^{(p-1)}$}} ;
  \fill[blue] (15,-1.2) circle (0.05);
  \draw[thick, blue, decoration={markings, mark=at position 0.5 with {\arrow{>}}},
        postaction={decorate}](15,-1.35) -- (6,-1.35);
  \node[text=blue] at (10.5,-1.7) {{\footnotesize $\alpha^{(p)}$}} ;
\fill[blue] (6,-1.35) circle (0.05);
  \draw[dashed, blue, decoration={markings, mark=at position 0.5 with {\arrow{>}}},
        postaction={decorate}](6,-1.35) -- (4,-1.35);
  \node[text=blue] at (5,-1.65) {{\footnotesize $\sub{\alpha}''$}} ;
\fill[red] (2,-2.5) circle (0.05);
\draw[dashed, red, decoration={markings, mark=at position 0.5 with {\arrow{>}}},
        postaction={decorate}](1,-2.5) -- (2,-2.5);
  \node[text=red] at (1.5,-2.2) {{\footnotesize $\sub{\alpha}'$}} ;
  \draw[thick, red, ->](2,-2.5) -- (2.8,-2.5);
  \node[text=red] at (2.6,-2.2) {{\footnotesize $\alpha^{(p-1)}$}} ;
 \draw[thick, dotted,red] (2.8,-2.5) -- (3.2,-2.5);
 \draw[thick, red, , decoration={markings, mark=at position 0.5 with {\arrow{>}}},
        postaction={decorate}](3.2,-2.5) -- (4,-2.5);
  \node[text=red] at (3.8,-2.2) {{\footnotesize $\alpha^{(s+1)}$}} ;
\fill[red] (4,-2.5) circle (0.05);
  \draw[thick, red,decoration={markings, mark=at position 0.5 with {\arrow{>}}},
        postaction={decorate}](4,-2.5) -- (6,-2.5);
  \node[text=red] at (5.2,-2.2) {{\footnotesize $\alpha^{(s)}+\alpha^{(p)}$}} ;
\fill[red] (6,-2.5) circle (0.05);
 \draw[dashed, red, decoration={markings, mark=at position 0.5 with {\arrow{>}}},
        postaction={decorate}](6,-2.7) -- (4,-2.7);
  \node[text=red] at (5,-3) {{\footnotesize $\sub{\alpha}''$}} ;
\end{tikzpicture}
 \caption{The case for $\alpha^{(p)} = \overline{\delta}_j - \delta_i$ and $\alpha^{(s)} = \delta_j - \overline{\delta}_l$.}
  \label{fig:typeII2}
\end{center}
\end{figure}

\noindent 
As in case (1), reversing the order of 
$e_{\alpha^{(s)}+\alpha^{(p)}} e_{\alpha^{(s+1)}} \cdots e_{\alpha^{(p-1)}}$ in \eqref{eq:hc-type-II2} 
induces several new paths.
Let $(\tilde{\sub{\mu}}^{\star a},\tilde{\sub{\alpha}}^{\star a})$, for $a = 1, \ldots, N$, denote these 
new paths from
$\delta_j$ to $\delta_i$ whose roots are sums among the roots $(\alpha^{(p-1)},\ldots,\alpha^{(s+1)}, \alpha^{(s)}+\alpha^{(p)})$.
More precisely, for each $a \in \{1, \ldots, N \}$ there exists a partition 
$(P_1, \ldots, P_{n_a})$ of the set $\{p-1, \ldots, s+1, s, p \}$, where $s$ and $p$ are always in the same partition, such that  
$$(\tilde{\alpha}^{\star a})^{(s+j-1)} = \sum_{t \in P_j} \alpha^{(t)}, $$
for $j = 1, \ldots, n_a$. 
Furthermore, by setting
$(\sub{\mu}, \sub{\alpha})^{\star a} := (\sub{\mu}',\sub{\alpha}')
\star (\tilde{\sub{\mu}}^{\star a},\tilde{\sub{\alpha}}^{\star a}) 
\star(\sub{\mu}'',\sub{\alpha}''),$
we get
$$
\wt (\sub{\mu},\sub{\alpha})= {K}_{(\sub{\mu},\sub{\alpha})}^{\# p} \wt(\sub{\mu},\sub{\alpha})^{\# p} + \sum_{a=1}^{N}  
{K}_{(\sub{\mu},\sub{\alpha})}^{\star a}\wt(\sub{\mu}, \sub{\alpha})^{\star a},
$$
where ${K}_{(\sub{\mu},\sub{\alpha})}^{\star a}$ are some constants.

\smallskip
\noindent
(3) Assume that $\alpha^{(p)}= \overline{\delta}_i - \overline{\delta}_j$ 
and $\alpha^{(s)}=\delta_i - \overline{\delta_l}$, with $i < l $. 
 Write 
 \begin{align*}
  (\sub{\mu},\sub{\alpha}) = &(\sub{\mu}',\sub{\alpha}') \star 
\big((\delta_i,\overline{\delta}_l, \ldots, 
\mu^{(p-1)},\overline{\delta}_i, \overline{\delta}_j),
(\alpha^{(s)}, \alpha^{(s+1)}, \ldots ,\alpha^{(p-1)},\alpha^{(p)})\big) 
\star (\sub{\mu}'',\sub{\alpha}''),
 \end{align*}
where
$$
(\sub{\mu}',\sub{\alpha}') =\big((\mu^{(1)},\ldots,\mu^{(s-1)}, 
\mu^{(s)}={\delta}_i) ; 
(\alpha^{(1)},\ldots,\alpha^{(s-1)})\big), 
$$
$$
(\sub{\mu}'',\sub{\alpha}'') =\big((\mu^{(p+1)}=\overline{\delta}_j,\ldots,\mu^{(m+1)}) ;
(\alpha^{(p+1)},\ldots,\alpha^{(m)})\big) 
$$
have length $s-1$ and $m - p$, respectively. 
With $a$ as in \eqref{eq:a}, 
we have
\begin{align*}
 \hc(b_{(\sub{\mu},\sub{\alpha})}^*)
= \bar{K}^{\# p} \hc(b_{(\sub{\mu},\sub{\alpha})^{\# p}}^*)
+ a \ \hc(b_{(\sub{\mu}',\sub{\alpha}')}^* 
 e_{\alpha^{(s)}} 
  \cdots e_{\alpha^{(p)}} e_{\alpha^{(p-1)}} 
 b_{(\sub{\mu}'',\sub{\alpha}'')}^*),
\end{align*}
where $\bar{K}^{\# p}$ is as in Lemma~\ref{Lem:nostar} depending on the different cases for $\alpha^{(p-1)}$
and $\alpha^{(p)}$. 
Assume that there exists a positive root $\alpha^{(t)}$, with $t<p-1$ and $\alpha^{(t)} \neq \alpha^{(s)}$,
such that $\alpha^{(t)} + \alpha^{(p)} \in \Delta$. By Lemma~\ref{lem:types} we observe that there is at most one root
$\alpha^{(t)}$ that satisfies such condition (see Figure~\ref{fig:type3a}).
In this case,
$\alpha^{(t)} + \alpha^{(p)} \in -\Delta_+$ and $ s < t < p-1$. 
Hence,
 \begin{align} \label{eq:hc-type-III2a}
\hc( b_{(\sub{\mu},\sub{\alpha})}^*)
= \bar{K}^{\# p} \hc (b_{(\sub{\mu},\sub{\alpha})^{\# p}}^*)
 + a n_{\alpha^{(s)},\alpha^{(p)}}
\hc(b_{(\sub{\mu}',\sub{\alpha}')}^* 
 e_{\alpha^{(s)}+\alpha^{(p)}}
e_{\alpha^{(s+1)}}\cdots
 e_{\alpha^{(p-1)}} 
 b_{(\sub{\mu}'',\sub{\alpha}'')}^*).  
 \end{align}

We verify easily that 
$\big((\delta_i, \overline{\mu}^{(p-1)}, 
\ldots, \overline{\mu}^{(s+2)}, \delta_{l}, \overline{\delta}_j)
; (\alpha^{(p-1)}, \ldots, \alpha^{(s+1)}, \alpha^{(s)}+\alpha^{(p)}) \big)$ 
is a path from $\delta_i$ to $\overline{\delta}_j$ 
and that 
$\big((\delta_i, \overline{\mu}^{(p-1)}, 
\ldots, \overline{\delta}_j)
; (\alpha^{(p-1)}, \alpha^{(s)}+\alpha^{(p)}, \ldots, \alpha^{(p-2)}) \big)$ 
is a path from $\delta_i$ to $\overline{\delta}_j$ as well
(see Figure~\ref{fig:type3a} for an illustration).
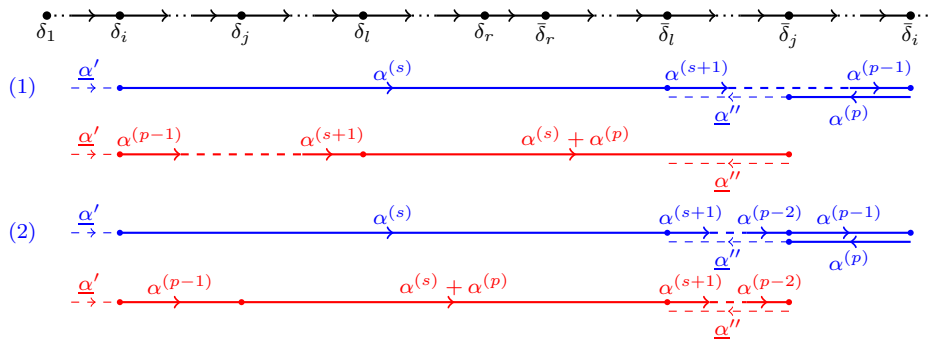
\begin{figure}[h]
\begin{center}
\begin{tikzpicture} [scale=0.8]
\draw[ thick,decoration={markings, mark=at position 0.5 with {\arrow{>}}},
        postaction={decorate}](1.2,0) -- (2,0);
\draw[ thick, ->](2,0) -- (2.8,0);
\draw[ thick,decoration={markings, mark=at position 0.5 with {\arrow{>}}},
        postaction={decorate}](3.2,0) -- (4,0);
\draw[ thick,->](4,0) -- (4.8,0);
\draw[ thick,decoration={markings, mark=at position 0.5 with {\arrow{>}}},
        postaction={decorate}](5.2,0) -- (6,0);
\draw[ thick,->](6,0) -- (6.8,0);
\draw[ thick,decoration={markings, mark=at position 0.5 with {\arrow{>}}},
        postaction={decorate}](7.2,0) -- (8,0);
\draw[ thick,decoration={markings, mark=at position 0.5 with {\arrow{>}}},
        postaction={decorate}](8,0) -- (9,0);
\draw[ thick, ->](9,0) -- (9.8,0);
\draw[ thick,decoration={markings, mark=at position 0.5 with {\arrow{>}}},
        postaction={decorate}](10.2,0) -- (11,0);
\draw[ thick,->](11,0) -- (11.8,0);
\draw[ thick,decoration={markings, mark=at position 0.5 with {\arrow{>}}},
        postaction={decorate}](12.2,0) -- (13,0);
\draw[ ->, thick](13,0) -- (13.8,0);
\draw[ thick,decoration={markings, mark=at position 0.5 with {\arrow{>}}},
        postaction={decorate}](14.2,0) -- (15,0);
     \draw[thick, dotted] (0.8,0) -- (1.2,0);
      \draw[thick, dotted] (2.8,0) -- (3.2,0);
      \draw[thick, dotted] (4.8,0) -- (5.2,0);
      \draw[thick, dotted] (6.8,0) -- (7.2,0);
      \draw[thick, dotted] (9.8,0) -- (10.2,0);
      \draw[thick, dotted] (11.8,0) -- (12.2,0);
      \draw[thick, dotted] (13.8,0) -- (14.2,0);
      \draw[thick, dotted] (15,0) -- (15.3,0);
    \fill (0.8,0) circle (0.07)node(xline)[below] {{\small $\delta_1$}}; 
    \fill (2,0) circle (0.07)node(xline)[below] {{\small $\delta_i$}};
    \fill (4,0) circle (0.07)node(xline)[below] {{\small $\delta_j$}};
    \fill (6,0) circle (0.07)node(xline)[below] {{\small $\delta_{l}$}}; 
    \fill (8,0) circle (0.07)node(xline)[below] {{\small $\delta_{r}$}};
    \fill (9,0) circle (0.07)node(xline)[below] {{\small $\bar{\delta}_{r}$}};
    \fill (11,0) circle (0.07)node(xline)[below] {{\small $\bar{\delta}_{l}$}}; 
    \fill (13,0) circle (0.07)node(xline)[below] {{\small $\bar{\delta}_{j}$}};
     \fill (15,0) circle (0.07)node(xline)[below] {{\small $\bar{\delta}_{i}$}};
\node[text=blue] at (0.4,-1.2) {(1)};
\fill[blue] (2,-1.2) circle (0.05);
 \draw[dashed, blue, decoration={markings, mark=at position 0.5 with {\arrow{>}}},
        postaction={decorate}](1.2,-1.2) -- (2,-1.2);
  \node[text=blue] at (1.5,-0.9) {{\footnotesize $\sub{\alpha}'$}} ;
  \draw[thick, blue, decoration={markings, mark=at position 0.5 with {\arrow{>}}},
        postaction={decorate}](2,-1.2) -- (11,-1.2);
  \node[text=blue] at (6.5,-0.9) {{\footnotesize $\alpha^{(s)}$}} ;
\fill[blue] (11,-1.2) circle (0.05);
  \draw[thick, blue, ->](11,-1.2) -- (12,-1.2);
  \node[text=blue] at (11.5,-0.9) {{\footnotesize $\alpha^{(s+1)}$}} ;
\draw[thick, dashed,blue] (12,-1.2) -- (14,-1.2);
  \draw[thick, blue, decoration={markings, mark=at position 0.5 with {\arrow{>}}},
        postaction={decorate}](14,-1.2) -- (15,-1.2);
  \node[text=blue] at (14.5,-0.9) {{\footnotesize $\alpha^{(p-1)}$}} ;
  \fill[blue] (15,-1.2) circle (0.05);
  \draw[thick, blue, decoration={markings, mark=at position 0.5 with {\arrow{>}}},
        postaction={decorate}](15,-1.35) -- (13,-1.35);
  \node[text=blue] at (14,-1.65) {{\footnotesize $\alpha^{(p)}$}} ;
\fill[blue] (13,-1.35) circle (0.05);
  \draw[dashed, blue, decoration={markings, mark=at position 0.5 with {\arrow{>}}},
        postaction={decorate}](13,-1.35) -- (11,-1.35);
  \node[text=blue] at (12,-1.65) {{\footnotesize $\sub{\alpha}''$}} ;
 
\fill[red] (2,-2.3) circle (0.05);
\draw[dashed, red, decoration={markings, mark=at position 0.5 with {\arrow{>}}},
        postaction={decorate}](1.2,-2.3) -- (2,-2.3);
  \node[text=red] at (1.5,-2) {{\footnotesize $\sub{\alpha}'$}} ;
  \draw[thick, red, ->](2,-2.3) -- (3,-2.3);
  \node[text=red] at (2.5,-2) {{\footnotesize $\alpha^{(p-1)}$}} ;
 \draw[thick, dashed,red] (3,-2.3) -- (5,-2.3);
 \draw[thick, red, , decoration={markings, mark=at position 0.5 with {\arrow{>}}},
        postaction={decorate}](5,-2.3) -- (6,-2.3);
  \node[text=red] at (5.5,-2) {{\footnotesize $\alpha^{(s+1)}$}} ;
\fill[red] (6,-2.3) circle (0.05);
  \draw[thick, red,decoration={markings, mark=at position 0.5 with {\arrow{>}}},
        postaction={decorate}](6,-2.3) -- (13,-2.3);
  \node[text=red] at (9.5,-2) {{\footnotesize $\alpha^{(s)}+\alpha^{(p)}$}} ;
\fill[red] (13,-2.3) circle (0.05);
 \draw[dashed, red, decoration={markings, mark=at position 0.5 with {\arrow{>}}},
        postaction={decorate}](13,-2.45) -- (11,-2.45);
  \node[text=red] at (12,-2.75) {{\footnotesize $\sub{\alpha}''$}} ;
  
\node[text=blue] at (0.4,-3.6) {(2)};
\fill[blue] (2,-3.6) circle (0.05);
 \draw[dashed, blue, decoration={markings, mark=at position 0.5 with {\arrow{>}}},
        postaction={decorate}](1.2,-3.6) -- (2,-3.6);
  \node[text=blue] at (1.5,-3.3) {{\footnotesize $\sub{\alpha}'$}} ;
  \draw[thick, blue, decoration={markings, mark=at position 0.5 with {\arrow{>}}},
        postaction={decorate}](2,-3.6) -- (11,-3.6);
  \node[text=blue] at (6.5,-3.3) {{\footnotesize $\alpha^{(s)}$}} ;
\fill[blue] (11,-3.6) circle (0.05);
  \draw[thick, blue, ->](11,-3.6) -- (11.7,-3.6);
  \node[text=blue] at (11.4,-3.3) {{\footnotesize $\alpha^{(s+1)}$}} ;
\draw[thick, dashed,blue] (11.7,-3.6) -- (12.3,-3.6);
  \draw[thick, blue, decoration={markings, mark=at position 0.5 with {\arrow{>}}},
        postaction={decorate}](12.3,-3.6) -- (13,-3.6);
  \node[text=blue] at (12.7,-3.3) {{\footnotesize $\alpha^{(p-2)}$}} ; 
    \fill[blue] (13,-3.6) circle (0.05);
  \draw[thick, blue, decoration={markings, mark=at position 0.5 with {\arrow{>}}},
        postaction={decorate}](13,-3.6) -- (15,-3.6);
  \node[text=blue] at (14,-3.3) {{\footnotesize $\alpha^{(p-1)}$}} ;
  \fill[blue] (15,-3.6) circle (0.05);
  \draw[thick, blue, decoration={markings, mark=at position 0.5 with {\arrow{>}}},
        postaction={decorate}](15,-3.75) -- (13,-3.75);
  \node[text=blue] at (14,-4.05) {{\footnotesize $\alpha^{(p)}$}} ;
\fill[blue] (13,-3.75) circle (0.05);
  \draw[dashed, blue, decoration={markings, mark=at position 0.5 with {\arrow{>}}},
        postaction={decorate}](13,-3.75) -- (11,-3.75);
  \node[text=blue] at (12,-4.05) {{\footnotesize $\sub{\alpha}''$}} ;
 
\fill[red] (2,-4.75) circle (0.05);
\draw[dashed, red, decoration={markings, mark=at position 0.5 with {\arrow{>}}},
        postaction={decorate}](1.2,-4.75) -- (2,-4.75);
  \node[text=red] at (1.5,-4.45) {{\footnotesize $\sub{\alpha}'$}} ;
  \draw[thick, red, decoration={markings, mark=at position 0.5 with {\arrow{>}}},
        postaction={decorate}](2,-4.75) -- (4,-4.75);
  \node[text=red] at (3,-4.45) {{\footnotesize $\alpha^{(p-1)}$}} ;
  \fill[red] (4,-4.75) circle (0.05);
  \draw[thick, red,decoration={markings, mark=at position 0.5 with {\arrow{>}}},
        postaction={decorate}](4,-4.75) -- (11,-4.75);
  \node[text=red] at (7.5,-4.45) {{\footnotesize $\alpha^{(s)}+\alpha^{(p)}$}} ;
    \fill[red] (11,-4.75) circle (0.05);
\draw[thick, red, ->](11,-4.75) -- (11.7,-4.75);
  \node[text=red] at (11.4,-4.45) {{\footnotesize $\alpha^{(s+1)}$}} ;
 \draw[thick, dashed,red] (11.7,-4.75) -- (12.3,-4.75);
 \draw[thick, red, , decoration={markings, mark=at position 0.5 with {\arrow{>}}},
        postaction={decorate}](12.3,-4.75) -- (13,-4.75);
  \node[text=red] at (12.7,-4.45) {{\footnotesize $\alpha^{(p-2)}$}} ;
\fill[red] (13,-4.75) circle (0.05);
 \draw[dashed, red, decoration={markings, mark=at position 0.5 with {\arrow{>}}},
        postaction={decorate}](13,-4.9) -- (11,-4.9);
  \node[text=red] at (12,-5.2) {{\footnotesize $\sub{\alpha}''$}} ;
\end{tikzpicture}
 \caption{The case for $\alpha^{(p)}= \overline{\delta}_i - \overline{\delta}_j$ and
$\alpha^{(s)} = \delta_i - \overline{\delta}_l$.}
  \label{fig:type3a}
\end{center}
\end{figure}
To conclude, we copy word for word the end of case (2). 

\smallskip
\noindent
(4) Assume that $\alpha^{(p)}= \overline{\delta}_i - \overline{\delta}_j$ 
and $\alpha^{(s)}=\delta_i - \delta_l$, with $i <j< l $. 
Write 
 \begin{align*}
  (\sub{\mu},\sub{\alpha}) = &(\sub{\mu}',\sub{\alpha}') \star 
\big((\delta_i,\delta_l, \ldots, 
\mu^{(p-1)},\overline{\delta}_i, \overline{\delta}_j),
(\alpha^{(s)}, \alpha^{(s+1)}, \ldots ,\alpha^{(p-1)},\alpha^{(p)})\big) 
\star (\sub{\mu}'',\sub{\alpha}''),
 \end{align*}
where
$$
(\sub{\mu}',\sub{\alpha}') =\big((\mu^{(1)},\ldots,\mu^{(s-1)}, 
\mu^{(s)}=\delta_i) ; 
(\alpha^{(1)},\ldots,\alpha^{(s-1)})\big), 
$$
$$
(\sub{\mu}'',\sub{\alpha}'') =\big((\mu^{(p+1)}=\overline{\delta}_j,\ldots,\mu^{(m+1)}) ;
(\alpha^{(p+1)},\ldots,\alpha^{(m)})\big), 
$$
have length $s-1$ and $m - p$ respectively. 
Set $a$ as in \eqref{eq:a}. 
In the same manner as in case (3), we get
 \begin{align} \label{eq:hc-type-III2b}
\hc( b_{(\sub{\mu},\sub{\alpha})}^*)
\bar{K}^{\# p} \hc (b_{(\sub{\mu},\sub{\alpha})^{\# p}}^*)
+ a n_{\alpha^{(s)},\alpha^{(p)}}
\hc(b_{(\sub{\mu}',\sub{\alpha}')}^* 
 e_{\alpha^{(s)}+\alpha^{(p)}}
 e_{\alpha^{(s+1)}}\cdots
 e_{\alpha^{(p-1)}} 
 b_{(\sub{\mu}'',\sub{\alpha}'')}^*),
 \end{align}
 where $\bar{K}^{\# p}$ is as in Lemma~\ref{Lem:nostar} depending on the different cases for $\alpha^{(p-1)}$
and $\alpha^{(p)}$.
We verify easily that 
$\big((\delta_i,\overline{\mu}^{(p-1)},\overline{\mu}^{(p-2)},
\ldots,\overline{\mu}^{(s+2)},\overline{\delta}_{l}, \overline{\delta}_j );(\alpha^{(p-1)},\ldots,\alpha^{(s+1)},
\alpha^{(s)}+\alpha^{(p)})\big)$ 
is a path from $\delta_i$ to $\overline{\delta}_j$. 
Observe that this case is similar as the case (3) and so there exist
different sequences of roots in a path from $\delta_i$ to $\overline{\delta}_j$ as well
(see Figure \ref{fig:type3b} for an illustration).
\begin{figure}[h]
\begin{center}
\begin{tikzpicture}[scale=0.8]
\draw[ thick,decoration={markings, mark=at position 0.5 with {\arrow{>}}},
        postaction={decorate}](1.2,0) -- (2,0);
\draw[ thick, ->](2,0) -- (2.8,0);
\draw[ thick,decoration={markings, mark=at position 0.5 with {\arrow{>}}},
        postaction={decorate}](3.2,0) -- (4,0);
\draw[ thick,->](4,0) -- (4.8,0);
\draw[ thick,decoration={markings, mark=at position 0.5 with {\arrow{>}}},
        postaction={decorate}](5.2,0) -- (6,0);
\draw[ thick,->](6,0) -- (6.8,0);
\draw[ thick,decoration={markings, mark=at position 0.5 with {\arrow{>}}},
        postaction={decorate}](7.2,0) -- (8,0);
\draw[ thick,decoration={markings, mark=at position 0.5 with {\arrow{>}}},
        postaction={decorate}](8,0) -- (9,0);
\draw[ thick, ->](9,0) -- (9.8,0);
\draw[ thick,decoration={markings, mark=at position 0.5 with {\arrow{>}}},
        postaction={decorate}](10.2,0) -- (11,0);
\draw[ thick,->](11,0) -- (11.8,0);
\draw[ thick,decoration={markings, mark=at position 0.5 with {\arrow{>}}},
        postaction={decorate}](12.2,0) -- (13,0);
\draw[ ->, thick](13,0) -- (13.8,0);
\draw[ thick,decoration={markings, mark=at position 0.5 with {\arrow{>}}},
        postaction={decorate}](14.2,0) -- (15,0);
      \draw[thick, dotted] (0.8,0) -- (1.2,0);
      \draw[thick, dotted] (2.8,0) -- (3.2,0);
      \draw[thick, dotted] (4.8,0) -- (5.2,0);
      \draw[thick, dotted] (6.8,0) -- (7.2,0);
      \draw[thick, dotted] (9.8,0) -- (10.2,0);
      \draw[thick, dotted] (11.8,0) -- (12.2,0);
      \draw[thick, dotted] (13.8,0) -- (14.2,0);
      \draw[thick, dotted] (15,0) -- (15.3,0);
    \fill (0.8,0) circle (0.07)node(xline)[below] {{\small $\delta_1$}}; 
    \fill (2,0) circle (0.07)node(xline)[below] {{\small $\delta_i$}};
    \fill (4,0) circle (0.07)node(xline)[below] {{\small $\delta_j$}};
    \fill (6,0) circle (0.07)node(xline)[below] {{\small $\delta_{l}$}}; 
    \fill (8,0) circle (0.07)node(xline)[below] {{\small $\delta_{r}$}};
    \fill (9,0) circle (0.07)node(xline)[below] {{\small $\bar{\delta}_{r}$}};
    \fill (11,0) circle (0.07)node(xline)[below] {{\small $\bar{\delta}_{l}$}}; 
    \fill (13,0) circle (0.07)node(xline)[below] {{\small $\bar{\delta}_{j}$}};
     \fill (15,0) circle (0.07)node(xline)[below] {{\small $\bar{\delta}_{i}$}};
\node[text=blue] at (0.4,-1.2) {(1)};
\fill[blue] (2,-1.2) circle (0.05);
 \draw[dashed, blue, decoration={markings, mark=at position 0.5 with {\arrow{>}}},
        postaction={decorate}](1.2,-1.2) -- (2,-1.2);
  \node[text=blue] at (1.5,-0.9) {{\footnotesize $\sub{\alpha}'$}} ;
  \draw[thick, blue, decoration={markings, mark=at position 0.5 with {\arrow{>}}},
        postaction={decorate}](2,-1.2) -- (6,-1.2);
  \node[text=blue] at (4,-0.9) {{\footnotesize $\alpha^{(s)}$}} ;
\fill[blue] (6,-1.2) circle (0.05);
  \draw[thick, blue, ->](6,-1.2) -- (7,-1.2);
  \node[text=blue] at (6.5,-0.9) {{\footnotesize $\alpha^{(s+1)}$}} ;
\draw[thick, dashed,blue] (7,-1.2) -- (14,-1.2);
  \draw[thick, blue, decoration={markings, mark=at position 0.5 with {\arrow{>}}},
        postaction={decorate}](14,-1.2) -- (15,-1.2);
  \node[text=blue] at (14.5,-0.9) {{\footnotesize $\alpha^{(p-1)}$}} ;
  \fill[blue] (15,-1.2) circle (0.05);
  \draw[thick, blue, decoration={markings, mark=at position 0.5 with {\arrow{>}}},
        postaction={decorate}](15,-1.35) -- (13,-1.35);
  \node[text=blue] at (14,-1.65) {{\footnotesize $\alpha^{(p)}$}} ;
\fill[blue] (13,-1.35) circle (0.05);
  \draw[dashed, blue, decoration={markings, mark=at position 0.5 with {\arrow{>}}},
        postaction={decorate}](13,-1.35) -- (11,-1.35);
  \node[text=blue] at (12,-1.65) {{\footnotesize $\sub{\alpha}''$}} ;
\fill[red] (2,-2.3) circle (0.05);
\draw[dashed, red, decoration={markings, mark=at position 0.5 with {\arrow{>}}},
        postaction={decorate}](1.2,-2.3) -- (2,-2.3);
  \node[text=red] at (1.5,-2) {{\footnotesize $\sub{\alpha}'$}} ;
  \draw[thick, red, ->](2,-2.3) -- (3,-2.3);
  \node[text=red] at (2.5,-2) {{\footnotesize $\alpha^{(p-1)}$}} ;
 \draw[thick, dashed,red] (3,-2.3) -- (10,-2.3);
 \draw[thick, red, , decoration={markings, mark=at position 0.5 with {\arrow{>}}},
        postaction={decorate}](10,-2.3) -- (11,-2.3);
  \node[text=red] at (10.5,-2) {{\footnotesize $\alpha^{(s+1)}$}} ;
\fill[red] (11,-2.3) circle (0.05);
  \draw[thick, red,decoration={markings, mark=at position 0.5 with {\arrow{>}}},
        postaction={decorate}](11,-2.3) -- (13,-2.3);
  \node[text=red] at (12,-2) {{\footnotesize $\alpha^{(s)}+\alpha^{(p)}$}} ;
\fill[red] (13,-2.3) circle (0.05);
 \draw[dashed, red, decoration={markings, mark=at position 0.5 with {\arrow{>}}},
        postaction={decorate}](13,-2.45) -- (11,-2.45);
  \node[text=red] at (12,-2.75) {{\footnotesize $\sub{\alpha}''$}} ;
\node[text=blue] at (0.4,-3.6) {(2)};
\fill[blue] (2,-3.6) circle (0.05);
 \draw[dashed, blue, decoration={markings, mark=at position 0.5 with {\arrow{>}}},
        postaction={decorate}](1.2,-3.6) -- (2,-3.6);
  \node[text=blue] at (1.5,-3.3) {{\footnotesize $\sub{\alpha}'$}} ;
  \draw[thick, blue, decoration={markings, mark=at position 0.5 with {\arrow{>}}},
        postaction={decorate}](2,-3.6) -- (6,-3.6);
  \node[text=blue] at (4,-3.3) {{\footnotesize $\alpha^{(s)}$}} ;
\fill[blue] (6,-3.6) circle (0.05);
  \draw[thick, blue, ->](6,-3.6) -- (7,-3.6);
  \node[text=blue] at (6.5,-3.3) {{\footnotesize $\alpha^{(s+1)}$}} ;
\draw[thick, dashed,blue] (7,-3.6) -- (12,-3.6);
  \draw[thick, blue, decoration={markings, mark=at position 0.5 with {\arrow{>}}},
        postaction={decorate}](12,-3.6) -- (13,-3.6);
  \node[text=blue] at (12.5,-3.3) {{\footnotesize $\alpha^{(p-2)}$}} ; 
    \fill[blue] (13,-3.6) circle (0.05);
  \draw[thick, blue, decoration={markings, mark=at position 0.5 with {\arrow{>}}},
        postaction={decorate}](13,-3.6) -- (15,-3.6);
  \node[text=blue] at (14,-3.3) {{\footnotesize $\alpha^{(p-1)}$}} ;
  \fill[blue] (15,-3.6) circle (0.05);
  \draw[thick, blue, decoration={markings, mark=at position 0.5 with {\arrow{>}}},
        postaction={decorate}](15,-3.75) -- (13,-3.75);
  \node[text=blue] at (14,-4.05) {{\footnotesize $\alpha^{(p)}$}} ;
\fill[blue] (13,-3.75) circle (0.05);
  \draw[dashed, blue, decoration={markings, mark=at position 0.5 with {\arrow{>}}},
        postaction={decorate}](13,-3.75) -- (11,-3.75);
  \node[text=blue] at (12,-4.05) {{\footnotesize $\sub{\alpha}''$}} ;
\fill[red] (2,-4.75) circle (0.05);
\draw[dashed, red, decoration={markings, mark=at position 0.5 with {\arrow{>}}},
        postaction={decorate}](1.2,-4.75) -- (2,-4.75);
  \node[text=red] at (1.5,-4.45) {{\footnotesize $\sub{\alpha}'$}} ;
  \draw[thick, red, decoration={markings, mark=at position 0.5 with {\arrow{>}}},
        postaction={decorate}](2,-4.75) -- (4,-4.75);
  \node[text=red] at (3,-4.45) {{\footnotesize $\alpha^{(p-1)}$}} ;
  \fill[red] (4,-4.75) circle (0.05);
  \draw[thick, red,decoration={markings, mark=at position 0.5 with {\arrow{>}}},
        postaction={decorate}](4,-4.75) -- (6,-4.75);
  \node[text=red] at (5,-4.45) {{\footnotesize $\alpha^{(s)}+\alpha^{(p)}$}} ;
    \fill[red] (6,-4.75) circle (0.05);
\draw[thick, red, ->](6,-4.75) -- (7,-4.75);
  \node[text=red] at (6.5,-4.45) {{\footnotesize $\alpha^{(s+1)}$}} ;
 \draw[thick, dashed,red] (7,-4.75) -- (12,-4.75);
 \draw[thick, red, , decoration={markings, mark=at position 0.5 with {\arrow{>}}},
        postaction={decorate}](12,-4.75) -- (13,-4.75);
  \node[text=red] at (12.5,-4.45) {{\footnotesize $\alpha^{(p-2)}$}} ;

\fill[red] (13,-4.75) circle (0.05);
 \draw[dashed, red, decoration={markings, mark=at position 0.5 with {\arrow{>}}},
        postaction={decorate}](13,-4.9) -- (11,-4.9);
  \node[text=red] at (12,-5.2) {{\footnotesize $\sub{\alpha}''$}} ;
\end{tikzpicture}
 \caption{The case for $\alpha^{(p)}= \overline{\delta}_i - \overline{\delta}_j$ and
$\alpha^{(s)} = \delta_j - {\delta}_l$.}
  \label{fig:type3b}
\end{center}
\end{figure}
We conclude exactly as in case (3). 

\smallskip
\noindent
(5) Assume that $\alpha^{(p)}= \overline{\delta}_i - \overline{\delta}_j$ 
and $\alpha^{(s)}=\delta_k - \delta_j$, with $k <i< j $.
 Write  
 \begin{align*}
  (\sub{\mu},\sub{\alpha}) = &(\sub{\mu}',\sub{\alpha}') \star 
\big((\delta_k,\delta_j, \ldots, 
\mu^{(p-1)},\overline{\delta}_i, \overline{\delta}_j),
(\alpha^{(s)}, \alpha^{(s+1)}, \ldots ,\alpha^{(p-1)},\alpha^{(p)})\big) 
\star (\sub{\mu}'',\sub{\alpha}''),
 \end{align*}
where
$$
(\sub{\mu}',\sub{\alpha}') =\big((\mu^{(1)},\ldots,\mu^{(s-1)}, 
\mu^{(s)}=\delta_k) ; 
(\alpha^{(1)},\ldots,\alpha^{(s-1)})\big), 
$$
$$
(\sub{\mu}'',\sub{\alpha}'') =\big((\mu^{(p+1)}=\overline{\delta}_j,\ldots,\mu^{(m+1)}) ;
(\alpha^{(p+1)},\ldots,\alpha^{(m)})\big), 
$$
have length $s-1$ and $m - p$, respectively.
With $a$ as in \eqref{eq:a}, 
we obtain that 
 \begin{align} \label{eq:hc-type-III2c}
\hc( b_{(\sub{\mu},\sub{\alpha})}^*)
\bar{K}^{\# p} \hc (b_{(\sub{\mu},\sub{\alpha})^{\# p}}^*) 
 + a n_{\alpha^{(s)},\alpha^{(p)}}
\hc(b_{(\sub{\mu}',\sub{\alpha}')}^* 
 e_{\alpha^{(s)}+\alpha^{(p)}}
 e_{\alpha^{(s+1)}} \cdots 
 e_{\alpha^{(p-2)}}
 e_{\alpha^{(p-1)}} 
 b_{(\sub{\mu}'',\sub{\alpha}'')}^*), 
 \end{align}
 where $\bar{K}^{\# p}$ is as in Lemma~\ref{Lem:nostar} depending on the different cases for $\alpha^{(p-1)}$
and $\alpha^{(p)}$.
We verify easily that 
$\big((\delta_i,\overline{\mu}^{(p-1)},\overline{\mu}^{(p-2)},
\ldots,\overline{\mu}^{(s+2)},\overline{\delta}_{j});(\alpha^{(p-1)},\ldots,\alpha^{(s+1)})\big)$ 
is a path from $\delta_i$ to $\overline{\delta}_j$ (see Figure \ref{fig:type3c} for an illustration).
\begin{figure}[h]
\begin{center}
\begin{tikzpicture}[scale=0.8]
\draw[ thick,->](2,0) -- (2.3,0);
\draw[ thick,decoration={markings, mark=at position 0.5 with {\arrow{>}}},
        postaction={decorate}](2.7,0) -- (3,0);
\draw[ thick, ->](3,0) -- (3.8,0);
\draw[ thick,decoration={markings, mark=at position 0.5 with {\arrow{>}}},
        postaction={decorate}](4.2,0) -- (5,0);
\draw[ thick,->](5,0) -- (5.8,0);
\draw[ thick,decoration={markings, mark=at position 0.5 with {\arrow{>}}},
        postaction={decorate}](6.2,0) -- (7,0);
\draw[ thick,->](7,0) -- (7.8,0);
\draw[ thick,decoration={markings, mark=at position 0.5 with {\arrow{>}}},
        postaction={decorate}](8.2,0) -- (9,0);
\draw[ thick,decoration={markings, mark=at position 0.5 with {\arrow{>}}},
        postaction={decorate}](9,0) -- (10,0);
\draw[ thick, ->](10,0) -- (10.8,0);
\draw[ thick,decoration={markings, mark=at position 0.5 with {\arrow{>}}},
        postaction={decorate}](11.2,0) -- (12,0);
\draw[ thick,->](12,0) -- (12.8,0);
\draw[ thick,decoration={markings, mark=at position 0.5 with {\arrow{>}}},
        postaction={decorate}](13.2,0) -- (14,0);
\draw[ ->, thick](14,0) -- (14.5,0);
     \draw[thick, dotted] (2.3,0) -- (2.7,0);
      \draw[thick, dotted] (3.8,0) -- (4.2,0);
      \draw[thick, dotted] (5.8,0) -- (6.2,0);
      \draw[thick, dotted] (7.8,0) -- (8.2,0);
      \draw[thick, dotted] (10.8,0) -- (11.2,0);
      \draw[thick, dotted] (12.8,0) -- (13.2,0);
      \draw[thick, dotted] (14.5,0) -- (15,0);
    \fill (2,0) circle (0.07)node(xline)[below] {{\small $\delta_1$}}; 
    \fill (3,0) circle (0.07)node(xline)[below] {{\small $\delta_k$}};
    \fill (5,0) circle (0.07)node(xline)[below] {{\small $\delta_i$}};
    \fill (7,0) circle (0.07)node(xline)[below] {{\small $\delta_{j}$}}; 
    \fill (9,0) circle (0.07)node(xline)[below] {{\small $\delta_{r}$}};
    \fill (10,0) circle (0.07)node(xline)[below] {{\small $\bar{\delta}_{r}$}};
    \fill (12,0) circle (0.07)node(xline)[below] {{\small $\bar{\delta}_{j}$}}; 
    \fill (14,0) circle (0.07)node(xline)[below] {{\small $\bar{\delta}_{i}$}};
\fill[blue] (3,-1.2) circle (0.05);
 \draw[dashed, blue, decoration={markings, mark=at position 0.5 with {\arrow{>}}},
        postaction={decorate}](2,-1.2) -- (3,-1.2);
  \node[text=blue] at (2.5,-0.9) {{\footnotesize $\sub{\alpha}'$}} ;
  \draw[thick, blue, decoration={markings, mark=at position 0.5 with {\arrow{>}}},
        postaction={decorate}](3,-1.2) -- (7,-1.2);
  \node[text=blue] at (5.5,-0.9) {{\footnotesize $\alpha^{(s)}$}} ;
\fill[blue] (7,-1.2) circle (0.05);
  \draw[thick, blue, ->](7,-1.2) -- (8,-1.2);
  \node[text=blue] at (7.5,-0.9) {{\footnotesize $\alpha^{(s+1)}$}} ;
\draw[thick, dashed,blue] (8,-1.2) -- (13,-1.2);
  \draw[thick, blue, decoration={markings, mark=at position 0.5 with {\arrow{>}}},
        postaction={decorate}](13,-1.2) -- (14,-1.2);
  \node[text=blue] at (13.5,-0.9) {{\footnotesize $\alpha^{(p-1)}$}} ;
  \fill[blue] (14,-1.2) circle (0.05);
  \draw[thick, blue, decoration={markings, mark=at position 0.5 with {\arrow{>}}},
        postaction={decorate}](14,-1.4) -- (12,-1.4);
  \node[text=blue] at (13,-1.7) {{\footnotesize $\alpha^{(p)}$}} ;
\fill[blue] (12,-1.4) circle (0.05);
  \draw[dashed, blue, decoration={markings, mark=at position 0.5 with {\arrow{>}}},
        postaction={decorate}](12,-1.4) -- (10,-1.4);
  \node[text=blue] at (11,-1.7) {{\footnotesize $\sub{\alpha}''$}} ; 
\fill[red] (3,-2.5) circle (0.05);
\draw[dashed, red, decoration={markings, mark=at position 0.5 with {\arrow{>}}},
        postaction={decorate}](2,-2.5) -- (3,-2.5);
  \node[text=red] at (2.5,-2.2) {{\footnotesize $\sub{\alpha}'$}} ;
 \draw[thick, red,decoration={markings, mark=at position 0.5 with {\arrow{>}}},
        postaction={decorate}](3,-2.5) -- (5,-2.5);
  \node[text=red] at (4,-2.2) {{\footnotesize $\alpha^{(s)}+\alpha^{(p)}$}} ;
\fill[red] (5,-2.5) circle (0.05);
  \draw[thick, red, ->](5,-2.5) -- (6,-2.5);
  \node[text=red] at (5.5,-2.2) {{\footnotesize $\alpha^{(p-1)}$}} ;
 \draw[thick, dashed,red] (6,-2.5) -- (11,-2.5);
 \draw[thick, red, , decoration={markings, mark=at position 0.5 with {\arrow{>}}},
        postaction={decorate}](11,-2.5) -- (12,-2.5);
  \node[text=red] at (11.5,-2.2) {{\footnotesize $\alpha^{(s+1)}$}} ;
\fill[red] (12,-2.5) circle (0.05);
 \draw[dashed, red, decoration={markings, mark=at position 0.5 with {\arrow{>}}},
        postaction={decorate}](12,-2.7) -- (10,-2.7);
  \node[text=red] at (11,-3) {{\footnotesize $\sub{\alpha}''$}} ;
\end{tikzpicture}
 \caption{The case for $\alpha^{(p)}= \overline{\delta}_i - \overline{\delta}_j$ and
$\alpha^{(s)} = \delta_k - {\delta}_j$.}
  \label{fig:type3c}
\end{center}
\end{figure}
\noindent 
To conclude, we copy word for word the end of case (2). \qed
 \end{proof}

\begin{lemma}
\label{Lem:star0}
Assume that the weighted path $(\sub{\mu},\sub{\alpha}) \in \hat{\P}_m(\mu)_{\sub{i}}$ has no loop 
and that $p=q$.  
If for some $s \in \{1, \ldots, p-2 \}$, $\alpha^{(s)} + \alpha^{(p)} =0$, then 
\begin{align*}
\wt (\sub{\mu},\sub{\alpha}) 
&= \bar{K}^{\# p} \wt(\sub{\mu},\sub{\alpha})^{\# p} + 
(\he(\c{\alpha}^{(s)})- \langle \alpha^{(s-1)}, \c{\alpha}^{(s)} \rangle)\ 
\sum_{a=1}^{N} {K}_{(\sub{\mu},\sub{\alpha})}^{\star a} \wt(\sub{\mu}, \sub{\alpha})^{\star a},
\end{align*}
where $\bar{K}^{\# p}$ is a scalar as in Lemma~\ref{Lem:nostar}, 
${K}_{(\sub{\mu},\sub{\alpha})}^{\star a}$ are some constants which only depend on constant structures,
and $(\sub{\mu}, \sub{\alpha})^{\star a}:= (\sub{\mu}^{\star a},\sub{\alpha}^{\star a})$
is a concatenation of paths defined as follows. 
\begin{enumerate}
\item If $\alpha^{(p)} = \overline{\delta}_j - \delta_i$ 
and $\alpha^{(s)} =\delta_j - \overline{\delta}_i$, with $i>j$, 
then 
\begin{align*}
(\sub{\mu}, \sub{\alpha})^{\star a} = (\sub{\mu}',\sub{\alpha}')
\star (\tilde{\sub{\mu}}^{\star a},\tilde{\sub{\alpha}}^{\star a})
\star(\sub{\mu}'',\sub{\alpha}''),
\end{align*}
where 
$(\tilde{\sub{\mu}}^{\star a},\tilde{\sub{\alpha}}^{\star a})$ 
is a path of length $< p-s$ between 
$\delta_j$ and $\delta_i$ whose roots $\tilde{\sub{\alpha}}^{\star a}$ are 
sums among $(\alpha^{(p-1)}, \ldots, \alpha^{(s+1)})$.

\item If $\alpha^{(p)} = \overline{\delta}_i - \overline{\delta}_j$ 
and $\alpha^{(s)} =\delta_i - \delta_j = \gamma$, with $i < j$, 
then 
\begin{align*}
(\sub{\mu}, \sub{\alpha})^{\star a} = (\sub{\mu}',\sub{\alpha}')
\star (\tilde{\sub{\mu}}^{\star a},\tilde{\sub{\alpha}}^{\star a})
\star(\sub{\mu}'',\sub{\alpha}''),
\end{align*}
where 
$(\tilde{\sub{\mu}}^{\star a},\tilde{\sub{\alpha}}^{\star a})$ 
is a path of length $< p-s$ between 
$\delta_i$ and $\overline{\delta}_j$ whose roots $\tilde{\sub{\alpha}}^{\star a}$ are 
sums among $(\alpha^{(p-1)}, \ldots, \alpha^{(s+1)})$.
\end{enumerate}
In all those cases,
$$
(\sub{\mu}',\sub{\alpha}') =\big( (\mu^{(1)},\ldots,\mu^{(s-1)}, 
\mu^{(s)}),
(\alpha^{(1)},\ldots,\alpha^{(s-1)})\big),
$$
$$
(\sub{\mu}'',\sub{\alpha}'') =\big((\mu^{(p+1)},\ldots,\mu^{(m+1)}),
(\alpha^{(p+1)},\ldots,\alpha^{(m)})\big), 
$$
and ${N}$ is the number of possible paths of $(\tilde{\sub{\mu}}^{\star a},\tilde{\sub{\alpha}}^{\star a})$.
\end{lemma}

\begin{proof} 
Let $-\alpha^{(p)}, \alpha^{(s)} \in \Delta_+$ such that $\alpha^{(s)} 
+ \alpha^{(p)} =0$. By Lemma~\ref{lem:types} 
the only possibilities for $\alpha^{(p)}$ and $\alpha^{(s)}$ are:
\begin{itemize} 
\item 
$\alpha^{(p)}= \overline{\delta}_j - {\delta}_i$ 
and $\alpha^{(s)}=\delta_j - \overline{\delta}_i$, with $j <i $. 
\item  
$\alpha^{(p)}= \overline{\delta}_i - \overline{\delta}_j$
and $\alpha^{(s)} =\delta_i -\delta_j$, with $i <j$.
\end{itemize} 

\noindent
(1) Assume that $\alpha^{(s)}=\delta_j - \overline{\delta_i}= -\alpha^{(p)}$, with $ j < i$. 
 Write 
 \begin{align*}
  (\sub{\mu},\sub{\alpha}) = &(\sub{\mu}',\sub{\alpha}') \star 
\big((\delta_j,\overline{\delta}_i, \ldots, 
\mu^{(p-1)},\overline{\delta}_j, \delta_i)),
(\alpha^{(s)}, \alpha^{(s+1)}, \ldots ,\alpha^{(p-1)},\alpha^{(p)})\big) 
\star (\sub{\mu}'',\sub{\alpha}''),
 \end{align*}
where
$$
(\sub{\mu}',\sub{\alpha}') =\big((\mu^{(1)},\ldots,\mu^{(s-1)}, 
\mu^{(s)}=\delta_{j}) ; 
(\alpha^{(1)},\ldots,\alpha^{(s-1)})\big), 
$$
$$
(\sub{\mu}'',\sub{\alpha}'') =\big((\mu^{(p+1)}=\delta_i,\ldots,\mu^{(m+1)}) ;
(\alpha^{(p+1)},\ldots,\alpha^{(m)})\big) 
$$
have length $s-1$ and $m - p$, respectively. 
Observe that in this case $ \alpha^{(p-1)} + \alpha^{(p)} < 0$ and  
$\alpha^{(p)}$ commutes with all roots 
$\alpha^{(t)}$ for $t< p$, except with $\alpha^{(p-1)}$ 
and $\alpha^{(s)}$.

\begin{figure}[h]
\begin{center}
\begin{tikzpicture}
\draw[ thick,->](0.5,0) -- (1,0);
\draw[ thick,decoration={markings, mark=at position 0.5 with {\arrow{>}}},
        postaction={decorate}](1.5,0) -- (2,0);
\draw[ thick, ->](2,0) -- (2.8,0);
\draw[ thick,decoration={markings, mark=at position 0.5 with {\arrow{>}}},
        postaction={decorate}](3.2,0) -- (4,0);
\draw[ thick,->](4,0) -- (4.8,0);
\draw[ thick,decoration={markings, mark=at position 0.5 with {\arrow{>}}},
        postaction={decorate}](5.2,0) -- (6,0);
\draw[ thick,decoration={markings, mark=at position 0.5 with {\arrow{>}}},
        postaction={decorate}](6,0) -- (7,0);
\draw[ thick, ->](7,0) -- (7.8,0);
\draw[ thick,decoration={markings, mark=at position 0.5 with {\arrow{>}}},
        postaction={decorate}](8.2,0) -- (9,0);
\draw[ thick,->](9,0) -- (9.8,0);
\draw[ thick,decoration={markings, mark=at position 0.5 with {\arrow{>}}},
        postaction={decorate}](10.2,0) -- (11,0);
      \draw[thick, dotted] (1,0) -- (1.5,0);
      \draw[thick, dotted] (2.8,0) -- (3.2,0);
      \draw[thick, dotted] (4.8,0) -- (5.2,0);
      \draw[thick, dotted] (7.8,0) -- (8.2,0);
      \draw[thick, dotted] (9.8,0) -- (10.2,0);
      \draw[thick, dotted] (11,0) -- (11.5,0);
    \fill (0.5,0) circle (0.07)node(xline)[below] {{\small $\delta_1$}}; 
    \fill (2,0) circle (0.07)node(xline)[below] {{\small $\delta_j$}};
    \fill (4,0) circle (0.07)node(xline)[below] {{\small $\delta_i$}};
    \fill (6,0) circle (0.07)node(xline)[below] {{\small $\delta_{r}$}};
    \fill (7,0) circle (0.07)node(xline)[below] {{\small $\bar{\delta}_{r}$}};
    \fill (9,0) circle (0.07)node(xline)[below] {{\small $\bar{\delta}_{i}$}}; 
    \fill (11,0) circle (0.07)node(xline)[below] {{\small $\bar{\delta}_{j}$}};
\fill[blue] (2,-1.2) circle (0.05);
 \draw[dashed, blue, decoration={markings, mark=at position 0.5 with {\arrow{>}}},
        postaction={decorate}](1,-1.2) -- (2,-1.2);
  \node[text=blue] at (1.5,-0.9) {{\footnotesize $\sub{\alpha}'$}} ;
  \draw[thick, blue, decoration={markings, mark=at position 0.5 with {\arrow{>}}},
        postaction={decorate}](2,-1.2) -- (9,-1.2);
  \node[text=blue] at (5.5,-0.9) {{\footnotesize $\alpha^{(s)}$}} ;
\fill[blue] (9,-1.2) circle (0.05);
  \draw[thick, blue, ->](9,-1.2) -- (9.8,-1.2);
  \node[text=blue] at (9.6,-0.9) {{\footnotesize $\alpha^{(s+1)}$}} ;
\draw[thick, dotted,blue] (9.8,-1.2) -- (10.2,-1.2);
  \draw[thick, blue, decoration={markings, mark=at position 0.5 with {\arrow{>}}},
        postaction={decorate}](10.2,-1.2) -- (11,-1.2);
  \node[text=blue] at (10.8,-0.9) {{\footnotesize $\alpha^{(p-1)}$}} ;
  \fill[blue] (11,-1.2) circle (0.05);
  \draw[thick, blue, decoration={markings, mark=at position 0.5 with {\arrow{>}}},
        postaction={decorate}](11,-1.4) -- (4,-1.4);
  \node[text=blue] at (7.5,-1.7) {{\footnotesize $\alpha^{(p)}$}} ;
\fill[blue] (4,-1.4) circle (0.05);
  \draw[dashed, blue, decoration={markings, mark=at position 0.5 with {\arrow{>}}},
        postaction={decorate}](4,-1.4) -- (2,-1.4);
  \node[text=blue] at (3,-1.7) {{\footnotesize $\sub{\alpha}''$}} ;
\fill[red] (2,-2.5) circle (0.05);
\draw[dashed, red, decoration={markings, mark=at position 0.5 with {\arrow{>}}},
        postaction={decorate}](1,-2.5) -- (2,-2.5);
  \node[text=red] at (1.5,-2.2) {{\footnotesize $\sub{\alpha}'$}} ;
  \draw[thick, red, ->](2,-2.5) -- (2.8,-2.5);
  \node[text=red] at (2.6,-2.2) {{\footnotesize $\alpha^{(p-1)}$}} ;
 \draw[thick, dotted,red] (2.8,-2.5) -- (3.2,-2.5);
 \draw[thick, red, , decoration={markings, mark=at position 0.5 with {\arrow{>}}},
        postaction={decorate}](3.2,-2.5) -- (4,-2.5);
  \node[text=red] at (3.8,-2.2) {{\footnotesize $\alpha^{(s+1)}$}} ;
\fill[red] (4,-2.5) circle (0.05);
 \draw[dashed, red, decoration={markings, mark=at position 0.5 with {\arrow{>}}},
        postaction={decorate}](4,-2.7) -- (2,-2.7);
  \node[text=red] at (3,-3) {{\footnotesize $\sub{\alpha}''$}} ;
\end{tikzpicture}
 \caption{The case for $\alpha^{(p)} = \overline{\delta}_j - \delta_i$ and $\alpha^{(s)} = \delta_j - \overline{\delta}_i$.}
  \label{fig:typeII3}
\end{center}
\end{figure}

Note that $\c{\alpha}^{(s)}$ commutes with all roots in $\sub{\alpha}',$ except with $\alpha^{(s-1)}$.
With $a$ as in \eqref{eq:a}, we obtain that 
 \begin{align} \label{eq:hc-type-II3}
\hc(b_{(\sub{\mu},\sub{\alpha})}^*) 
= {K}_{(\sub{\mu},\sub{\alpha})}^{\# p} \hc (b_{(\sub{\mu},\sub{\alpha})^{\# p}}^*) 
+ a \  (\c{\alpha}^{(s)} - \langle \alpha^{(s-1)} ,\c{\alpha}^{(s)} \rangle)
\hc(b_{(\sub{\mu}',\sub{\alpha}')}^* 
 e_{\alpha^{(s+1)}} \cdots 
 e_{\alpha^{(p-1)}} 
 b_{(\sub{\mu}'',\sub{\alpha}'')}^*),  
 \end{align}
where ${K}_{(\sub{\mu},\sub{\alpha})}^{\# p}$ is as in Lemma~\ref{Lem:nostar} (1). 
 
We easily verify that 
$\big((\delta_j,\delta_{p-1},\delta_{p-2},
\ldots,\delta_{s+2},\delta_{i}), (\alpha^{(p-1)},\ldots,\alpha^{(s+1)})\big)$ 
is a path from $\delta_j$ to $\delta_i$ (see Figure \ref{fig:typeII3} for an illustration).
By reversing the order in the roots $\alpha^{(s+1)},\ldots, \alpha^{(p-1)}$ 
we obtain a path from $\delta_j$ to $\delta_i$. The operations of permuting all roots 
induce
several new paths. 
Let $(\tilde{\sub{\mu}}^{\star a},\tilde{\sub{\alpha}}^{\star a})$, for $a = 1, \ldots, N$, 
denote these new paths from
$\delta_j$ to ${\delta}_i$ whose roots are sums among the roots 
$(\alpha^{(p-1)},\ldots,\alpha^{(s+1)})$.
More precisely, for each $a \in \{1, \ldots, N \}$ there exists a partition 
$(P_1, \ldots, P_{n_a})$ of the set $\{p-1, \ldots, s+1\}$ such that  
$$(\tilde{\alpha}^{\star a})^{(s+j-1)} = \sum_{t \in P_j} \alpha^{(t)}, $$
for $j = 1, \ldots, n_a$. Furthermore, we set
$ (\sub{\mu},\sub{\alpha})^{\star a} :=
(\sub{\mu}',\sub{\alpha}') 
\star (\tilde{\sub{\mu}}^{\star a},\tilde{\sub{\alpha}}^{\star a})
\star(\sub{\mu}'',\sub{\alpha}'').$
Hence, 
$$\wt (\sub{\mu},\sub{\alpha}) = {K}_{(\sub{\mu},\sub{\alpha})}^{\# p} \wt((\sub{\mu},\sub{\alpha})^{\# p}) 
+ (\he(\c{\alpha}^{(s)})- \langle \alpha^{(s-1)} ,\c{\alpha}^{(s)} \rangle) 
\sum_{a=1}^{N} {K}_{(\sub{\mu},\sub{\alpha})}^{\star a} \wt (\sub{\mu},\sub{\alpha})^{\star a},
$$
where ${K}_{(\sub{\mu},\sub{\alpha})}^{\star a}$ are some constants.

(2) Assume that $\alpha^{(p)}= \overline{\delta}_i - \overline{\delta}_j$ and
$\alpha^{(s)}=\delta_i - \delta_j$, with $ i < j$.  
 Write 
 \begin{align*}
  (\sub{\mu},\sub{\alpha}) = &(\sub{\mu}',\sub{\alpha}') \star 
\big((\delta_i,\delta_j, \ldots, 
\mu^{(p-1)},\overline{\delta}_i, \overline{\delta}_j)),
(\alpha^{(s)}, \alpha^{(s+1)}, \ldots ,\alpha^{(p-1)},\alpha^{(p)})\big) 
\star (\sub{\mu}'',\sub{\alpha}''),
 \end{align*}
where
$$
(\sub{\mu}',\sub{\alpha}') =\big((\mu^{(1)},\ldots,\mu^{(s-1)}, 
\mu^{(s)}=\delta_{i}) ; 
(\alpha^{(1)},\ldots,\alpha^{(s-1)})\big), 
$$
$$
(\sub{\mu}'',\sub{\alpha}'') =\big((\mu^{(p+1)}=\overline{\delta}_j,\ldots,\mu^{(m+1)}) ;
(\alpha^{(p+1)},\ldots,\alpha^{(m)})\big) 
$$
have length $s-1$ and $m - p$, respectively. 
Assume that there exists a positive root $\alpha^{(t)}$, with $t<p-1$ and $\alpha^{(t)} \neq \alpha^{(s)}$,
such that $\alpha^{(t)} + \alpha^{(p)} \in \Delta$. By Lemma~\ref{lem:types} we observe that there is at most one root
$\alpha^{(t)}$ that satisfies such condition (see Figure \ref{fig:type3d}).
\begin{figure}[h]
\begin{center}
\begin{tikzpicture} [scale=0.8]
\draw[ thick,->](0.5,0) -- (1,0);
\draw[ thick,decoration={markings, mark=at position 0.5 with {\arrow{>}}},
        postaction={decorate}](1.5,0) -- (2,0);
\draw[ thick, ->](2,0) -- (2.8,0);
\draw[ thick,decoration={markings, mark=at position 0.5 with {\arrow{>}}},
        postaction={decorate}](3.2,0) -- (4,0);
\draw[ thick,->](4,0) -- (4.8,0);
\draw[ thick,decoration={markings, mark=at position 0.5 with {\arrow{>}}},
        postaction={decorate}](5.2,0) -- (6,0);
\draw[ thick,decoration={markings, mark=at position 0.5 with {\arrow{>}}},
        postaction={decorate}](6,0) -- (7,0);
\draw[ thick, ->](7,0) -- (7.8,0);
\draw[ thick,decoration={markings, mark=at position 0.5 with {\arrow{>}}},
        postaction={decorate}](8.2,0) -- (9,0);
\draw[ thick,->](9,0) -- (9.8,0);
\draw[ thick,decoration={markings, mark=at position 0.5 with {\arrow{>}}},
        postaction={decorate}](10.2,0) -- (11,0);
     \draw[thick, dotted] (1,0) -- (1.5,0);
      \draw[thick, dotted] (2.8,0) -- (3.2,0);
      \draw[thick, dotted] (4.8,0) -- (5.2,0);
      \draw[thick, dotted] (7.8,0) -- (8.2,0);
      \draw[thick, dotted] (9.8,0) -- (10.2,0);
      \draw[thick, dotted] (11,0) -- (11.5,0);
    \fill (0.5,0) circle (0.07)node(xline)[below] {{\small $\delta_1$}}; 
    \fill (2,0) circle (0.07)node(xline)[below] {{\small $\delta_i$}};
    \fill (4,0) circle (0.07)node(xline)[below] {{\small $\delta_j$}};
    \fill (6,0) circle (0.07)node(xline)[below] {{\small $\delta_{r}$}};
    \fill (7,0) circle (0.07)node(xline)[below] {{\small $\bar{\delta}_{r}$}};
    \fill (9,0) circle (0.07)node(xline)[below] {{\small $\bar{\delta}_{j}$}}; 
    \fill (11,0) circle (0.07)node(xline)[below] {{\small $\bar{\delta}_{i}$}};
\fill[blue] (2,-1.2) circle (0.05);
 \draw[dashed, blue, decoration={markings, mark=at position 0.5 with {\arrow{>}}},
        postaction={decorate}](1,-1.2) -- (2,-1.2);
  \node[text=blue] at (1.5,-0.9) {{\footnotesize $\sub{\alpha}'$}} ;
  \draw[thick, blue, decoration={markings, mark=at position 0.5 with {\arrow{>}}},
        postaction={decorate}](2,-1.2) -- (4,-1.2);
  \node[text=blue] at (3,-0.9) {{\footnotesize $\alpha^{(s)}$}} ;
\fill[blue] (4,-1.2) circle (0.05);
  \draw[thick, blue, ->](4,-1.2) -- (5,-1.2);
  \node[text=blue] at (4.5,-0.9) {{\footnotesize $\alpha^{(s+1)}$}} ;
\draw[thick, dashed,blue] (5,-1.2) -- (10,-1.2);
  \draw[thick, blue, decoration={markings, mark=at position 0.5 with {\arrow{>}}},
        postaction={decorate}](10,-1.2) -- (11,-1.2);
  \node[text=blue] at (10.5,-0.9) {{\footnotesize $\alpha^{(p-1)}$}} ;
  \fill[blue] (11,-1.2) circle (0.05);
  \draw[thick, blue, decoration={markings, mark=at position 0.5 with {\arrow{>}}},
        postaction={decorate}](11,-1.4) -- (9,-1.4);
  \node[text=blue] at (10,-1.7) {{\footnotesize $\alpha^{(p)}$}} ;
\fill[blue] (9,-1.4) circle (0.05);
  \draw[dashed, blue, decoration={markings, mark=at position 0.5 with {\arrow{>}}},
        postaction={decorate}](9,-1.4) -- (7,-1.4);
  \node[text=blue] at (8,-1.7) {{\footnotesize $\sub{\alpha}''$}} ;
\fill[red] (2,-2.5) circle (0.05);
\draw[dashed, red, decoration={markings, mark=at position 0.5 with {\arrow{>}}},
        postaction={decorate}](1,-2.5) -- (2,-2.5);
  \node[text=red] at (1.5,-2.2) {{\footnotesize $\sub{\alpha}'$}} ;
  \draw[thick, red, ->](2,-2.5) -- (3,-2.5);
  \node[text=red] at (2.5,-2.2) {{\footnotesize $\alpha^{(p-1)}$}} ;
 \draw[thick, dashed,red] (3,-2.5) -- (8,-2.5);
 \draw[thick, red, , decoration={markings, mark=at position 0.5 with {\arrow{>}}},
        postaction={decorate}](8,-2.5) -- (9,-2.5);
  \node[text=red] at (8.5,-2.2) {{\footnotesize $\alpha^{(s+1)}$}} ;
\fill[red] (9,-2.5) circle (0.05);
 \draw[dashed, red, decoration={markings, mark=at position 0.5 with {\arrow{>}}},
        postaction={decorate}](9,-2.7) -- (7,-2.7);
  \node[text=red] at (8,-3) {{\footnotesize $\sub{\alpha}''$}} ;
\end{tikzpicture}
 \caption{The case for $\alpha^{(p)}= \overline{\delta}_i - \overline{\delta}_j$ and
$\alpha^{(s)}=\delta_i - \delta_j$.}
  \label{fig:type3d}
\end{center}
\end{figure}
In this case,
$\alpha^{(t)} + \alpha^{(p)} \in -\Delta_+$ and $ s < t < p-1$. 

With $a$ as in \eqref{eq:a}, we get
 \begin{align} \label{eq:hc-type-III4}
\hc( b_{(\sub{\mu},\sub{\alpha})}^*) 
 \bar{K}^{\# p} \hc (b_{(\sub{\mu},\sub{\alpha})^{\# p}}^*)
 + a  \ (\c{\alpha}^{(s)} -\langle \alpha^{(s-1)} , \c{\alpha}^{(s)} \rangle)
\hc \big( b_{(\sub{\mu}',\sub{\alpha}')}^* 
 e_{\alpha^{(s+1)}} \cdots 
 e_{\alpha^{(p-1)}} 
 b_{(\sub{\mu}'',\sub{\alpha}'')}^* \big), 
 \end{align}
where $\bar{K}^{\# p}$ as in Lemma~\ref{Lem:nostar} depending on the different cases for $\alpha^{(p-1)}$
and $\alpha^{(p)}$. 
We easily verify that 
$\big((\delta_i,\overline{\mu}^{(p-1)},\overline{\mu}^{(p-2)},
\ldots,\overline{\mu}^{(s+2)},\overline{\delta}_{j});(\alpha^{(p-1)},\ldots,\alpha^{(s+1)})\big)$ 
is a path from $\delta_i$ to $\overline{\delta}_j$ (see Figure \ref{fig:type3d} for an illustration). 
We conclude exactly as in case (1). \qed
\end{proof}

We now turn to weighted path with loops. 
Recall that for $\alpha \in \Delta$, the {\em support} of $\alpha$  is the set ${\rm supp} (\alpha)$ of 
$\beta \in \Pi$ such that $\langle \alpha, \c{\varpi}_\beta \rangle \not= 0$.

\begin{lemma} \label{lem:loops} 
Fix $(\sub{\mu},\sub{\alpha}) \in \hat{\P}_m(\mu)_{\sub{i}}$ a weighted path 
with $n$ loops in the positions $j_1, \ldots , j_n,$, with $0 \leqslant n \leqslant m$, and 
assume that $p=q$. 
Let $j'_{l,1}, \ldots, j'_{l,n'}$ be integers of $\{1, \ldots, j_l-1 \}$, for $l = 1, \ldots, n$,
such that ${\rm supp}(\alpha^{(j'_{l,t})})$ contains the simple root $\alpha^{(j_l)}$.
Then, 
$$ \wt(\sub{\mu},\sub{\alpha}) = \hat{K} \wt(\tilde{\sub{\mu}},\tilde{\sub{\alpha}}),$$
where $$ \hat{K} := \prod_{j_l= 1}^{n} \langle \mu^{(j_l)}, \c{\alpha}^{(j_l)} \rangle \Big(\langle \rho, \s{\varpi}_{\alpha^{(j_l)}} \rangle
 - \sum_{{\tiny \substack{j'_{l,t} \in \{1, \ldots, j_l-1\}, \\  \alpha^{(j_l)} \in {\rm supp}(\alpha^{(j'_{l,t})})}}} 
\langle \alpha^{(j'_{l,t})}, \s{\varpi}_{\alpha^{(j_l)}} \rangle \Big) $$
and 
$(\tilde{\sub{\mu}},\tilde{\sub{\alpha}})$ is the weighted path of length $m-n$ obtained from $(\sub{\mu},\sub{\alpha})$
by ``removing all loops'' from the path. In particular, if $n=0$, we have 
$\hat{K} = 1$.
\end{lemma}

\begin{proof}
First of all, if $n=0$, then the results are known by Lemma~\ref{Lem:nostar}, Lemma~\ref{Lem:star+} and Lemma~\ref{Lem:star0}, 
and if $n=m$, then the result is known by Lemma~\ref{lem:symC}. 
So there is no loss of generality by assuming that $0 < n < m$. 

 Write
$$(\sub{\mu},\sub{\alpha}) = \left((\mu^{(1)}, \ldots, \mu^{(j_1)}, \mu^{(j_1+1)} \ldots, \mu^{(j_n+1)}),( \alpha^{(1)}, \ldots, \alpha^{(j_1)}, \alpha^{(j_1 +1)}, \ldots, \alpha^{(j_n)}) \right) \star 
(\sub{\mu}'',\sub{\alpha}''),$$
where $$
(\sub{\mu}'',\sub{\alpha}'') =\left( (\mu^{(j_n+1)},\ldots,\mu^{(m+1)}), 
(\alpha^{(j_n+1)},\ldots,\alpha^{(m)})\right)),
$$
and $(\sub{\mu}'',\sub{\alpha}'')$ 
has length $m - j_n$.
We have $i_{j_l} = 0$ and $\alpha^{(j_l)} \in \Pi$. 
For each $j_l, l \in {1, \ldots, n}$, denote by $j'_{l,1}, \ldots, j'_{l,n'}$ the integers of $\{1, \ldots, j_l-1 \}$ 
such that simple root $\alpha^{(j_l)}$ appears in the support of $(\alpha^{(j'_{l,t})})$, 
 thus $[e_{\alpha^{(j'_{l,t})}}, \s{\varpi}_{\alpha^{(j_l)}}] =
  -\langle \alpha^{(j'_{l,t})}, \s{\varpi}_{\alpha^{(j_l)}} \rangle e_{\alpha^{(j'_{l,t})}} \not=0$ 
  and, 
hence,  $e_{\alpha^{(j'_{l,t})}}\s{\varpi}_{\alpha^{(j_l)}} \not=  \s{\varpi}_{\alpha^{(j_l)}}e_{\alpha^{(j'_{l,t})}}$.
In other words, $\s{\varpi}_{\alpha^{(j_l)}}$
commutes with $e_{\alpha^{(s)}}$, where $s \neq j'_{l,t}$ for $l \in \{1, \ldots, n \}, t \in \{1, \ldots, n' \}$. 
With $$ a := a_{\mu^{(j_n+1)},\mu^{(j_n)}}^{(b_{(\sub{\mu},\sub{\alpha}),j_n})}
a_{\mu^{(j_n)},\mu^{(j_n-1)}}^{(b_{(\sub{\mu},\sub{\alpha}),j_n-1})}
\cdots 
a_{\mu^{(j_1+1)},\mu^{(j_1)}}^{(b_{(\sub{\mu},\sub{\alpha}),j_1})}
\cdots
a_{\mu^{(2)},\mu^{(1)}}^{(b_{(\sub{\mu},\sub{\alpha}),1})},$$ 
we get that 
\begin{align*}
 b_{(\sub{\mu},\sub{\alpha})}^*
& = 
\prod_{j_k=1}^{n} \langle \mu^{(j_k)}, \c{\alpha}^{(j_k)} \rangle
 \Big( \s{\varpi}_{\alpha^{(j_k)}} - 
 \sum_{{\tiny \substack{j'_{k,t} \in \{1, \ldots, j_k-1\}, \\ \alpha^{(j_k)} \in {\rm supp}(\alpha^{(j'_{k,t})})}}} 
\langle \alpha^{(j'_{k,t})}, \s{\varpi}_{\alpha^{(j_k)}} \rangle \Big) b_{(\tilde{\sub{\mu}},\tilde{\sub{\alpha}})}^*,
\end{align*}
whence we get the statement. \qed
\end{proof}

\begin{lemma} \noindent \label{lem:p<q} 
Fix $(\sub{\mu},\sub{\alpha}) \in \hat{\P}_m(\mu)_{\sub{i}}$ 
and assume that $p<q$. 
\begin{enumerate} 
 \item Assume $\mu^{(p)} \in \{ \delta_1, \ldots, \delta_r \}.$
  \begin{enumerate}
 \item If 
$\langle \mu^{(p)},\c{\alpha}^{(p)} \rangle =1,$ 
then 
$\wt(\sub{\mu},\sub{\alpha})=   
\langle \rho, \s{\varpi}_{\alpha^{(p)}} \rangle 
\wt(\sub{\mu},\sub{\alpha})^{\# p}.$
\item If $\langle \mu^{(p)},\c{\alpha}^{(p)} \rangle =-1,$
then  
$\wt(\sub{\mu},\sub{\alpha})=  
(-\langle \rho, \s{\varpi}_{\alpha^{(p)}}  \rangle 
+ 1)
\wt(\sub{\mu},\sub{\alpha})^{\# p}.$
\end{enumerate}
 \item Assume $\mu^{(p)} \in \{ \overline{\delta}_1, \ldots, \overline{\delta}_r \}.$
 \begin{enumerate}
 \item If 
$\langle \mu^{(p)},\c{\alpha}^{(p)} \rangle =1,$
then 
$\wt(\sub{\mu},\sub{\alpha})=   
(\langle \rho, \c{\varpi}_{\alpha^{(p)}} \rangle -1)
\wt(\sub{\mu},\sub{\alpha})^{\# p}.$
\item If $\langle \mu^{(p)},\c{\alpha}^{(p)} \rangle =-1,$
then  
$\wt(\sub{\mu},\sub{\alpha})=  
(2 - \langle \rho, \s{\varpi}_{\alpha^{(p)}}  \rangle)
\wt(\sub{\mu},\sub{\alpha})^{\# p}.$
\end{enumerate}
\end{enumerate}
\end{lemma}

Note that the hypothesis $p< q$ implies that $i_p=0$ and so $\alpha^{(p)} \in \Pi_{\mu^{(p)}}
=\{\beta \in \Pi \; | \; \langle \mu^{(p)},\c{\beta} \rangle\not=0\}$. 

\begin{proof}
Write 
$$(\sub{\mu},\sub{\alpha}) = (\sub{\mu}',\sub{\alpha}') 
 \star ((\mu^{(p-1)},\mu^{(p)}),\alpha^{(p-1)})  \star 
((\mu^{(p)},\mu^{(p+1)}),\alpha^{(p)}) 
\star (\sub{\mu}'',\sub{\alpha}''),$$
where $(\sub{\mu}',\sub{\alpha}')$  
and $(\sub{\mu}'',\sub{\alpha}'')$ 
have length $p - 2$ and $m - p$, respectively. 
Since $p < q$, then $i_p=0$ and $\alpha^{(p)} \in \Pi_{\mu^{(p)}}
=\{\beta \in \Pi \; | \; \langle \mu^{(p)},\c{\beta} \rangle\not=0\}$.

\noindent
(1)  We have $\mu^{(p)} = \delta_j$ for some $j = 1, \ldots, r$ and 
 $\alpha^{(p)} = \beta_j$ or $\beta_{j-1}$.
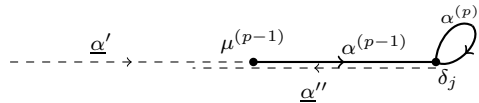
\begin{center} 
\begin{tikzpicture} [scale=0.8]
 \draw[dashed,
        decoration={markings, mark=at position 0.5 with {\arrow{>}}},
        postaction={decorate}
        ]
        (0,0) -- (4,0);\node at (1.5,0.3) {{\footnotesize $\sub{\alpha}'$}}; 
  \draw[ thick,
        decoration={markings, mark=at position 0.5 with {\arrow{>}}},
        postaction={decorate}
        ]
        (4,0) -- (7,0); \node at (6,0.3) {{\footnotesize ${\alpha}^{(p-1)}$}};
\draw[thick] (7,0) to [out=90,in=150] (7.6,0.6);
 \draw[thick, decoration={markings, mark=at position 0.5 with {\arrow{>}}},
         postaction={decorate}] (7.6,0.6) to [out=310,in=330] (7,0);        
 \node at (7.4,0.8) {{\scriptsize ${\alpha}^{(p)}$}};
   \draw[ dashed,
        decoration={markings, mark=at position 0.5 with {\arrow{>}}},
        postaction={decorate}
        ]
        (7,-0.1) -- (3,-0.1);\node at (5,-0.5) {{\footnotesize $\sub{\alpha}''$}};
 \fill (4,0) circle (0.07)node(xline)[above] {{\small $\mu^{(p-1)}$}}; 
    \fill (7,0) circle (0.07); \node at (7.2,-0.3) {{\small $\delta_j$}}; 
\end{tikzpicture}      
\end{center} 

This case is similar as the $\sl_{r+1}$ case (Lemma~\ref{Lem0:cut} (3)), 
so we omit the end of the arguments. 

(2)  We have $\mu^{(p)} = \overline{\delta}_j$ for some $j = 1, \ldots, r$ and 
 $\alpha^{(p)} = \beta_j$ or $\beta_{j-1}$. 
Let $s$ be an integer in $\{ 1, \ldots, p-2 \}$ such that the support of $\alpha^{(s)}$ contains the simple root $\alpha^{(p)}$,
then either $i_s = 0$ or $i_s \not= 0$.
\begin{figure}[h]
\centering
\begin{tikzpicture}[scale=0.8]

\draw[dashed,decoration={markings, mark=at position 0.5 with {\arrow{>}}},
        postaction={decorate}]
        (-1,0) -- (1,0);\node at (-0.1,0.3) {{\scriptsize $\sub{\alpha}'$}}; 
\draw[ thick,->](1,0) to (2.5,0);\node at (2,0.3) {{\scriptsize ${\alpha}^{(s)}$}}; 
\draw[dotted,decoration={markings, mark=at position 0.5 with {\arrow{>}}},
        postaction={decorate}](2,0) -- (3,0);
\draw[ thick,decoration={markings, mark=at position 0.5 with {\arrow{>}}},
        postaction={decorate}](3,0) -- (4,0); \node at (3.5,0.3) {{\scriptsize ${\alpha}^{(p-1)}$}};
 \draw[thick] (4,0) to [out=90,in=150] (4.6,0.6);
 \draw[thick, decoration={markings, mark=at position 0.5 with {\arrow{>}}},
         postaction={decorate}] (4.6,0.6) to [out=310,in=330] (4,0);        
 \node at (4.4,0.8) {{\scriptsize ${\alpha}^{(p)}= \beta_j$}};
 \draw[dashed,decoration={markings, mark=at position 0.5 with {\arrow{>}}},
        postaction={decorate}]
        (3.9,-0.2) -- (2.5,-0.2);\node at (3,-0.5) {{\scriptsize $\sub{\alpha}''$}};
 \fill (1,0) circle (0.07)node(xline)[below] {{\footnotesize $\delta_j$}}; 
\fill (4,0) circle (0.07); \node at (4.2,-0.3) {{\footnotesize $\overline{\delta}_j$}};
 \draw[thick] (1,0) to [out=180,in=250] (0.2,0.4);
 \draw[thick, decoration={markings, mark=at position 0.625 with {\arrow{>}}},
         postaction={decorate}] (0.2,0.4) to [out=50,in=105] (1,0);
 \draw[thick] (1,0) to [out=100,in=165] (1.4,0.8);
 \draw[thick, decoration={markings, mark=at position 0.5 with {\arrow{>}}},
         postaction={decorate}] (1.4,0.8) to [out=335,in=20] (1,0);        
 \node at (1.4,1) {{\scriptsize ${\alpha}^{(s-1)}$}};
\draw[dotted, thick] (0.7,0.5) to (1,0.6);
\draw[dashed,decoration={markings, mark=at position 0.5 with {\arrow{>}}},
        postaction={decorate}]
        (6,0) -- (8,0);\node at (6.9,0.3) {{\scriptsize $\sub{\alpha}'$}};
\draw[ thick,->](8,0) to (9.5,0);\node at (9,0.3) {{\scriptsize ${\alpha}^{(s)}$}}; 
\draw[dotted](9.5,0) -- (10,0);
\draw[ thick,decoration={markings, mark=at position 0.5 with {\arrow{>}}},
        postaction={decorate}]
        (10,0) -- (11,0); \node at (10.5,0.3) {{\scriptsize ${\alpha}^{(p-1)}$}};
 \draw[thick] (11,0) to [out=90,in=150] (11.6,0.6);
 \draw[thick, decoration={markings, mark=at position 0.5 with {\arrow{>}}},
         postaction={decorate}] (11.6,0.6) to [out=310,in=330] (11,0);        
 \node at (11.6,0.8) {{\scriptsize ${\alpha}^{(p)}= \beta_{j-1}$}};
   \draw[dashed,decoration={markings, mark=at position 0.5 with {\arrow{>}}},
        postaction={decorate}]
        (10.9,-0.2) -- (9.5,-0.2);\node at (10,-0.5) {{\scriptsize $\sub{\alpha}''$}};
 \fill (8,0) circle (0.07)node(xline)[below] {{\footnotesize $\delta_{j-1}$}}; 
\fill (11,0) circle (0.07); \node at (11.2,-0.3) {{\footnotesize $\overline{\delta}_{j}$}};
\draw[thick] (8,0) to [out=180,in=250] (7.2,0.4);
 \draw[thick, decoration={markings, mark=at position 0.625 with {\arrow{>}}},
         postaction={decorate}] (7.2,0.4) to [out=50,in=105] (8,0);
 \draw[thick] (8,0) to [out=100,in=165] (8.4,0.8);
 \draw[thick, decoration={markings, mark=at position 0.5 with {\arrow{>}}},
         postaction={decorate}] (8.4,0.8) to [out=335,in=20] (8,0);        
 \node at (8.4,1) {{\scriptsize ${\alpha}^{(s-1)}$}};
\draw[dotted, thick] (7.7,0.5) to (8,0.6);
\end{tikzpicture}
  \caption{Path in case $p < q$ and $\mu^{(p)} = \overline{\delta}_j$ }
  \label{fig:p<q-b}
\end{figure}
If $i_s = 0$, then $b_{(\sub{\mu},\sub{\alpha}),s}^* = \s{\varpi}_{\alpha^{(s)}} \in U(\h)$, and so
$\s{\varpi}_{\alpha^{(p)}}$ commutes with $b_{(\sub{\mu},\sub{\alpha}),s}^*$.
Otherwise,
there is at most one root $\alpha^{(s)}$ with $i_s \not= 0$ such that $\alpha^{(p)} \in {\rm supp} (\alpha^{(s)})$ 
(see Figure \ref{fig:p<q-b}), and so $\s{\varpi}_{\alpha^{(p)}} e_{\alpha^{(s)}} \neq e_{\alpha^{(s)}} \s{\varpi}_{\alpha^{(p)}}$.  
Hence $\s{\varpi}_{\alpha^{(p)}}$ commutes 
with $b_{(\sub{\mu}',\sub{\alpha}')}^*$ except with ${\alpha^{(s)}}$. 
Writing $(\sub{\mu}',\sub{\alpha}') = (\sub{\mu}'_1,\sub{\alpha}'_1) 
\star ((\mu^{(s)},\mu^{(s+1)}),\alpha^{(s)}) 
\star (\sub{\mu}'_2,\sub{\alpha}'_2),$
we get that 
\begin{align*}
b_{(\sub{\mu},\sub{\alpha})}^*  
&= \langle \mu^{(p)}, \c{\alpha}^{(p)} \rangle 
\left( \s{\varpi}_{\alpha^{(p)}} - \langle \alpha^{(s)} , \s{\varpi}_{\alpha^{(p)}} \rangle
- \langle \alpha^{(p-1)} , \s{\varpi}_{\alpha^{(p)}} \rangle \right) b_{(\sub{\mu},\sub{\alpha})^{\#}}^*,
\end{align*}
since $\s{\varpi}_{\alpha^{(p)}}$ commutes with all roots of $\sub{\alpha}'_1$ and $\sub{\alpha}'_2$
and $$ a_{\mu^{(p)},\mu^{(p-1)}}^{(c_{\alpha^{(p-1)}}e_{-\alpha^{(p-1)}})} 
a_{\mu^{(s+1)},\mu^{(s)}}^{(c_{\alpha^{(s)}}e_{-\alpha^{(s)}})} 
b_{(\sub{\mu}'_1,\sub{\alpha}'_1)}^*
 e_{\alpha^{(s)}}b_{(\sub{\mu}'_2,\sub{\alpha}'_2)}^* e_{\alpha^{(p-1)}} b_{(\sub{\mu}'',\sub{\alpha}'')}^*
 = b_{(\sub{\mu},\sub{\alpha})^{\#}}^*.$$ 
Hence,
\begin{align*}
\wt(\sub{\mu},\sub{\alpha})  
&=
\langle \mu^{(p)}, \c{\alpha}^{(p)} \rangle
\left( 
\langle \rho, \s{\varpi}_{\alpha^{(p)}} \rangle - \langle \alpha^{(s)} , \s{\varpi}_{\alpha^{(p)}} \rangle
- \langle \alpha^{(p-1)} , \s{\varpi}_{\alpha^{(p)}} \rangle
\right) \wt(\sub{\mu},\sub{\alpha})^{\# p}.
\end{align*}

(a)   If $\langle \mu^{(p)}, \c{\alpha}^{(p)} \rangle =1$ then $\alpha^{(p)} = \beta_{j-1}$,
   for some $j= 2, \ldots, r$.
   Observe that if $\langle \alpha^{(p-1)} , \s{\varpi}_{\alpha^{(p)}} \rangle =1$ then 
   $\langle \alpha^{(s)} , \s{\varpi}_{\alpha^{(p)}} \rangle = 0$, and vice versa, 
   whence the expected equality. 
   
(b)
   If $\langle \mu^{(p)}, \c{\alpha}^{(p)} \rangle =-1$ then $\alpha^{(p)} = \beta_{j}$,
   for some $ i= 1, \ldots, r$.   
   Observe that $\langle \alpha^{(p-1)} , \s{\varpi}_{\alpha^{(p)}} \rangle = 1 \text{ or } 2$. 
   If $\langle \alpha^{(p-1)} , \s{\varpi}_{\alpha^{(p)}} \rangle =2$, it means that
   $\alpha^{(p-1)} = \eps_k + \eps_j, k = 1, \ldots, j$, whence  
   $\langle \alpha^{(s)} , \s{\varpi}_{\alpha^{(p)}} \rangle = 0$.
Otherwise  $\langle \alpha^{(p-1)} , \s{\varpi}_{\alpha^{(p)}} \rangle =1$, and so
   $\alpha^{(p-1)} = \eps_k + \eps_j, k > j$, whence  
   $\langle \alpha^{(s)} , \s{\varpi}_{\alpha^{(p)}} \rangle = 1$, 
   whence the expected equality. \qed
\end{proof}

Next theorem is the analog of Theorem \ref{corollary:cut}. 

\begin{theorem}   \label{corollary:cut-type-C}
Let $m \in \Z_{> 0}$, $\mu \in P(\delta)$ 
and $(\sub{\mu},\sub{\alpha}) \in \hat{\P}_m(\mu)$. 
Assume that for some $i \in \{ 1,\ldots, m\}$, 
$\mu^{(i)} \succ  \mu$. 
Then $\wt(\sub{\mu},\sub{\alpha}) = 0$. 
\end{theorem}

\begin{proof} 
First of all, we observe that for all $\sub{i} \in {\Z}^m_{\prec \sub{0}}$ there exists $ 1 \leqslant  p \leqslant  m$ such that
$i_1 = i_2 = \ldots = i_{p-1} = 0, i_p <0$.
So 
$$b_{(\sub{\mu},\sub{\alpha})}^* = a_{\sub{\mu},\sub{\alpha}} 
b_{(\sub{\mu},\sub{\alpha}),1}^* \ldots b_{(\sub{\mu},\sub{\alpha}),p-1}^* b_{(\sub{\mu},\sub{\alpha}),p}^* 
\ldots b_{(\sub{\mu},\sub{\alpha}),m}^* \in n_- U(\g).$$
Hence $\hc({b}_{\sub{\mu},\sub{\alpha}}^*)=0$ and so 
the theorem is clear for $\sub{i} \in {\Z}^m_{\prec \sub{0}}$. 
We prove the statement by induction on $m$. 
Necessarily, $m \geqslant 2$.

$\ast$ If $m=2$, then the hypothesis 
implies that $\sub{i} \in {\Z}^m_{\prec \sub{0}}$ and so the statement 
is true.

$\ast$ Assume $m \geqslant 3$ and 
that for all weighted paths $(\sub{\mu'},\sub{\alpha'}) \in \hat{\P}_{m'}(\mu)$, 
with $m' < m$, such that for some $i' \in \{ 1,\ldots, m'\}$, 
$\mu'^{(i')} \geqslant  \mu$, we have 
$\wt(\sub{\mu'},\sub{\alpha'}) = 0$. 
If $\sub{i} \in {\Z}^m_{\prec \sub{0}}$ 
the statement is true. 
So we can assume that  
$\sub{i} \in {\Z}^m_{\succcurlyeq \sub{0}}$,
and by the assumption, necessarily, $\sub{i} \in {\Z}^m_{\succ \sub{0}}$.
By Proposition~\ref{Pro:Conc1}, 
there are some scalars $K^{\# p}$ and $K^{\star a}$ such that 
$$\wt(\sub{\mu},\sub{\alpha}) = K^{\# p} \wt(\sub{\mu},\sub{\alpha})^{\# p}
\quad \text{ or } \quad 
\wt(\sub{\mu},\sub{\alpha}) = K^{\# p} \wt(\sub{\mu},\sub{\alpha})^{\# p}
+ \sum_{a=1}^{N}  K^{\star a} \wt(\sub{\mu},\sub{\alpha})^{\star a}.$$

If $\wt(\sub{\mu},\sub{\alpha}) = K^{\# p} \wt(\sub{\mu},\sub{\alpha})^{\# p}$, 
then the weighted path $(\sub{\mu},\sub{\alpha})^{\# p}$ 
satisfies the hypothesis of the theorem and it is not empty.  
Hence by our induction hypothesis and Proposition~\ref{Pro:Conc1} 
we get the statement. 

If $$\wt(\sub{\mu},\sub{\alpha}) = K^{\# p} \wt(\sub{\mu},\sub{\alpha})^{\# p} 
+ \sum_{a=1}^{N}  K^{\star a} \wt(\sub{\mu},\sub{\alpha})^{\star a},$$
then are several cases to consider. 
Let us first consider 
the path $(\sub{\mu},\sub{\alpha})^{\star a}$ as in Lemma~\ref{Lem:star+} (1)
as follows :
$$(\sub{\mu},\sub{\alpha})^{\star a} :=(\sub{\mu}',\sub{\alpha}')\star 
\big( (\delta_k,\delta_j) ; (\alpha^{(s)}+\alpha^{(p)})\big) 
\star (\tilde{\sub{\mu}}^{\star a},\tilde{\sub{\alpha}}^{\star a})
\star(\sub{\mu}'',\sub{\alpha}''),$$
where 
$$
(\sub{\mu}',\sub{\alpha}') =\big( (\mu^{(1)},\ldots,\mu^{(s-1)}, 
\mu^{(s)}=\delta_{k}) ; 
(\alpha^{(1)},\ldots,\alpha^{(s-1)})\big),
$$
$$
(\sub{\mu}'',\sub{\alpha}'') =\big((\delta_i=\mu^{(p+1)},\ldots,\mu^{(m+1)}) ;
(\alpha^{(p+1)},\ldots,\alpha^{(m)})\big), 
$$
and $(\tilde{\sub{\mu}}^{\star a},\tilde{\sub{\alpha}}^{\star a})$ 
is a path of length $< p-s$ between 
$\delta_j$ and $\delta_i$ whose roots $(\tilde{\alpha}^{\star a})^{(l)}$ 
have height 0 or strictly positive height. We argue similarly for the other cases. 

Let $t$ be the smallest integer such that $\mu^{(t)} > \mu$. 
Since the root $\alpha^{(s)} + \alpha^{(p)}$ and
all roots in path $(\sub{\mu'},\sub{\alpha'}), (\tilde{\sub{\mu}}^{\star a},\tilde{\sub{\alpha}}^{\star a})$ have height 0 or strictly positive, then
$t> p$ and
$\alpha^{(t)}$ is belongs to $\sub{\alpha}''$. Observe that 
the weighted paths $(\sub{\mu},\sub{\alpha})^{\# p}$  and $(\sub{\mu}'',\sub{\alpha}'')$
satisfy the hypothesis of the theorem and it is not empty. Note that the path $(\sub{\mu}'',\sub{\alpha}'')$ have length $m-p$ for each case.
Hence by our induction hypothesis we have 
$$\wt(\sub{\mu},\sub{\alpha})^{\# p} = 0 \quad \text { and } \quad \wt(\sub{\mu}'',\sub{\alpha}'') =0. $$
By Proposition \ref{Pro:Conc1} we get the statement. \qed
\end{proof}

Our next target is to introduce and study an equivalence relation on the set of weighted paths. 

Let $\alpha \in \Delta_+$ and $\alpha = \mu - \nu$. We say that $\alpha$ {\em has type I} 
(respectively, {\em has type II}, {\em type III}) 
if the admissible triple $(\alpha, \mu, \nu)$ (see Appendix \ref{app:Admissible}) 
has type I (respectively, {has type II}, {type III (a) or III (b)}). 

Recall that for $\lambda,\mu \in P(\delta)$, 
$[\![\lambda,\mu ]\!]$ denote the set of 
$\nu \in P(\delta)$ such that $\lambda \leqslant \nu \leqslant \mu$.

This definition is a generalization of Definition \ref{definition:equivalenceA}.  
Since we cannot here argue only on the heights of roots for the $\sp_{2r}$ case, 
as for the $\sl_{r+1}$ case,  
we introduce an equivalence relation directly on such paths as follows. 

\begin{definition} [equivalence relation on the paths for $\sp_{2r}$]
\label{definition:equivalent-class-typeC}
We define an equivalence relation $\sim$  
on  $\hat{\P}_m$ by induction on $m$ as follows. 
\begin{enumerate}
\item If $m=1$, there is only one equivalence 
class represented by the trivial path of length 0.
\item  If $m=2$, then two paths $(\sub{\mu},\sub{\alpha}), 
(\sub{\mu'},\sub{\alpha'})$ in $\hat{\P}_m$ 
are {\em equivalent}
if the following condition holds:
\begin{enumerate}
 \item there is $(\eps_1, \eps_2) \in \{ 0,1 \}^2$ such that 
$\he(\sub{\mu})\in \eps_1\Z_{\succcurlyeq \sub{0}} \times \eps_2\Z^m_{\succcurlyeq \sub{0}} $ and 
$\he(\sub{\mu}') \in \eps_1\Z^m_{\succcurlyeq \sub{0}} \times \eps_2\Z^m_{\succcurlyeq \sub{0}} $
\item the roots $\alpha^{(1)} \in \sub{\alpha}$ and $\alpha'^{(1)}  \in \sub{\alpha}'$ have same types.  
\end{enumerate}
\item  If $m >2$, then two paths $(\sub{\mu},\sub{\alpha}), (\sub{\mu}',\sub{\alpha}') $ in 
$\hat{\P}_m$ are equivalent if the following conditions hold, 
in the notations of Proposition~\ref{Pro:Conc1}: 
\begin{enumerate}
\item for all $i \in \{1,\ldots,m\}$, there is $(\eps_1,\ldots,\eps_m) \in \{ -1,0,1\}^m$ 
such that $\he({\sub{\mu}})_i \in \prod_{i=1}^m \eps_i \Z^m_{\succcurlyeq \sub{0}}$ and $\he({\sub{\mu'}})_i \in \prod_{i=1}^m \eps_i \Z^m_{\succcurlyeq \sub{0}}$;  
\item we have $p(\sub{i})=p(\sub{i'})=:p$, with 
$\sub{i}=\he(\sub{\mu})$, $\sub{i'}=\he(\sub{\mu'})$, and 
the weighted paths $(\sub{\mu},\sub{\alpha})^{\# p}$ and 
$(\sub{\mu}',\sub{\alpha}')^{\# p}$ are equivalent,
\item if there is an $s \in \{1,\ldots,p-2\}$ such that $\alpha^{(s)}+\alpha^{(q)} 
\in \Delta_+ \cup \{0\}$, with $q:=q(\sub{i})=q(\sub{i}')$, then $s$ is the unique 
integer of $\{1,\ldots,p-2\}$ such that $\alpha'^{(s)}+\alpha'^{(q)} 
\in \Delta_+ \cup \{0\}$. 
Moreover, 
all paths $(\sub{\mu},\sub{\alpha})^{\star a}$ and  $(\sub{\mu}',\sub{\alpha}')^{\star a}$  
are equivalent,
\item we have $K_{(\sub{\mu},\sub{\alpha})}^{\# p} = 
K_{(\sub{\mu}',\sub{\alpha}')}^{\# p}$, 
and if
an $s$ as in (c) exists, then 
for all the $N$ possible paths $(\sub{\mu},\sub{\alpha})^{\star a}$ , $K_{(\sub{\mu},\sub{\alpha})}^{\star a} = 
K_{(\sub{\mu}',\sub{\alpha}')}^{\star a}$. 
\end{enumerate}
\end{enumerate}
\end{definition}

\begin{remark}
 If $(\sub{\mu},\sub{\alpha}) \in \hat{\P}_m$ is a weighted path starting at $\delta_k$, 
 for some $k\in\{1,\ldots,r\}$, and contained in $[\![ \delta_r, \delta_k ]\!]$
 or starting at $\overline{\delta}_k$ and contained in $[\![ \overline{\delta}_1, \overline{\delta}_k ]\!]$, then Definition~\ref{definition:equivalenceA} and Definition~\ref{definition:equivalent-class-typeC} are equivalent.
\end{remark}
By the above remark, the paths as above can be dealt as in $\sl_{r+1}$.
Thus, one can use the results for the weights of the paths as for $\sl_{r+1}$.
We denote by 
$[(\sub{\mu},\sub{\alpha})]$ 
the class of a weighted 
paths of length $m$, 
by ${\E}_m$ the set of equivalence classes $[(\sub{\mu},\sub{\alpha})]$ 
and by $\bar{\E}_m$ the set of elements of $\E_m$ whose representative are not contained in 
$[\![ \delta_r, \delta_1 ]\!]$. 
We observe that an equivalent class in $\E_m$ can simply be described 
by the sequence $\sub{\mu}$.
Hence, we will often write $[\mu^{(1)}, \ldots, \mu^{(m)}]$ or simply by $\m$ 
for the class $[(\sub{\mu},\sub{\alpha})]$. 

\begin{example} \noindent \label{Ex:eqclassC}
For $m =2$, there are four equivalence 
classes represented by: $[\delta_k, \delta_k]$ for some $k$, $[\delta_k, \overline{\delta}_k]$ 
for some $k$, $[\delta_k, \delta_j]$ for $k < j$, and $[\delta_k, \overline{\delta}_j]$ for $k \neq j$. 
There are two elements of $\bar{\E}_m$. 

For $m=3$, there are 16 elements of $\bar{\E}_m$, with 6 elements with loops as follows: for $k\not= j$
\begin{align*}
[\delta_k, \delta_k, \overline{\delta}_k]; \quad [\delta_k, \delta_k, \overline{\delta}_j] ;
 \quad [\delta_k, \overline{\delta}_k, \delta_k]; \quad
 [\delta_k, \overline{\delta}_j,\delta_k]; \quad [\delta_k, \overline{\delta}_k, \overline{\delta}_k]; 
 \quad [\delta_k, \overline{\delta}_j, \overline{\delta}_j],
 \end{align*}
 and 10 elements without loops as follows:
\begin{align*}
&[\delta_k, \overline{\delta}_k, \delta_j],k<j ; 
\quad [\delta_k, \overline{\delta}_k, \overline{\delta}_j], k<j ; 
\quad [\delta_k, \overline{\delta}_j, \overline{\delta}_k], j < k; 
\quad [\delta_k, \overline{\delta}_j, \overline{\delta}_l], j \not=k \not=l, j < l; \\
&[\delta_k, \overline{\delta}_j, {\delta}_l], j \not=k \not=l, k < l;
\quad [\delta_k, \overline{\delta}_j, \overline{\delta}_l], j \not=k\not=l, l < j;
\quad [\delta_k, \overline{\delta}_j, \overline{\delta}_l], j \not= k \not=l, k < j; \\
&[\delta_k, {\delta}_j, \overline{\delta}_k], k < j;
\quad [\delta_k, \overline{\delta}_j, \overline{\delta}_k], k < j; 
\quad [\delta_k, \overline{\delta}_5, \overline{\delta}_j], k < j.
\end{align*}
\end{example}

Let $\m := [(\sub{\mu},\sub{\alpha})]\in  \E_m$.  
The number $n$ of zero values of $\sub{i}:=\he(\sub{\mu})$ 
does not depend on 
$(\sub{\mu'},\sub{\alpha'})$ in $\m$.  
We adopt the terminology of paths with zeroes and without zero as in $\sl_{r+1}$ case.
By definition, the position $p(\sub{i})$ of the first returning 
back does not depend on $(\sub{\mu},\sub{\alpha}) \in \m$.  
Similarly, the integers $q(\sub{i})$ and $s \in \{1,\ldots,p-2\}$ (if such an $s$ exists) such that 
$\alpha^{(s)}+\alpha^{(q)} 
\in \Delta_+ \cup \{0\}$ 
do not depend on $(\sub{\mu},\sub{\alpha}) \in \m$. 
Furthermore, the class of $(\sub{\mu},\sub{\alpha})^{\# p}$ and
the class of $(\sub{\mu},\sub{\alpha})^{\star a}$ only 
depend on $\m$. 
We denote by $\m^{\# }$ and $\m^{\star a}$ these equivalence classes, respectively.
Moreover, we denote by $K_{\m}^{\#}$ the scalar $K_{(\sub{\mu},\sub{\alpha})}^{\# p}$ and, if
an $s$ as in (3) (c) exists, we denote by $K_{\m}^{\star a}$ the scalar $K_{(\sub{\mu},\sub{\alpha})}^{\star a}$. 
We denote by $\ell(\m'):=m'$ the length of $\m'$ for some equivalence class 
$\m\in \E_{m'}$, $m' \in \Z_{>0}$. 
We have $\ell(\m)=m$, $\ell(\m^\#)=m-1$ if $i_p+i_{p-1}\not=0$, 
$\ell(\m^\#)=m-2$ if $i_p+i_{p-1}=0$, and $\ell(\m^{\star a}) < m$ for all $a \in \{1, \ldots, N \}$.
At last, note that if $\m$ has no zero then 
$\m^{\#}$ and $\m^{\star a}$ has no zero, too.

\begin{lemma} \label{Lem:weightC}
 Let $\m \in {\E}_m$ without zero, and set $p := p(\m)$.
\begin{enumerate}
\item There is a polynomial $A_{\m} \in  \C[X_1,\ldots,X_{m-1}]$ 
of total degree $\leqslant \lfloor \frac{m}{2}\rfloor $ 
such that 
for all weighted paths $(\sub{\mu},\sub{\alpha})$ such that 
$[(\sub{\mu},\sub{\alpha})] = \m$, 
$$\wt(\sub{\mu},\sub{\alpha}) = A_{\m} (i_1, \ldots,i_{m-1}).$$
Here, the integer $i_j$ denote the heights of $\alpha^{(j)}$.
Moreover, $A_{\m}$ is a sum of monomials of the 
form $X_{j_1}\cdots X_{j_l}$, with $1 \leqslant j_1 < \cdots < j_{l} < m$. 
\item The polynomial $A_{\m}$ is defined by 
induction as follows.  

{\rm (a)} Assume m = 2. Then 
$$A_{[\delta_k, \overline{\delta}_k]}(X_1) = 2(X_1 + 1), \quad 
A_{[\delta_k, \delta_j], k < j}(X_1) = X_1, \quad  
A_{[\delta_k, \overline{\delta}_j]} (X_1) = X_1 + 1.$$

{\rm (b)} Assume $m \geqslant 3.$  Set $P_\alpha (X)$ a polynomial of degree 1 by :
 \begin{align} \label{eq:p_alpha}
 & P_\alpha (X) :=
  \begin{cases} \frac{X + 1}{2} & \text{if} \quad \alpha \text { has type I},\\ 
 X+1 & \text{if} \quad \alpha \text{ has type II},\\
 X & \text{if} \quad \alpha \text { has type III},
 \end{cases}
 \end{align}
Let 
$m^{\star a}:= \ell (\m^{\star a})$ and $\sub{\alpha}^{\star a}$ be a sequence of roots as in Lemma~\ref{Lem:star+} and Lemma~\ref{Lem:star0}. 
We will denote by $\sub{X}^{\star a}$ 
the sequence of variables 
associated with $({\alpha}^{\star a})^{(j)} \in \sub{\alpha}^{\star a}$, $j = 1, \ldots, m^{\star a}-1$, where $({\alpha}^{\star a})^{(j)}$ is replaced by $X_{j}$.
Let $N$ denote the number of possible paths $\m^{\star a}$.
\begin{enumerate}
 \item[{\rm (i)}] Assume that $\alpha^{(p)} + \alpha^{(p-1)} \not = 0$.

{$\ast$} If for all $s\in \{1, \ldots, p-2 \}$, either $\alpha^{(s)} + \alpha^{(p)} \in -\Delta_+$ or $\alpha^{(s)} + \alpha^{(p)} \notin \Delta \cup \{0\}$, then 
 $$ A_{\m} (X_1, \ldots, X_{m-1}) = K_{\m}^{\#} A_{\m}(X_1, \ldots, X_{p-2}, X_{p-1}+X_{p}, \ldots, X_{m-1}).$$
 
 {$\ast$} If for some $s \in \{1, \ldots, p-2 \}$, $\alpha^{(s)} + \alpha^{(p)} \in \Delta_+$, then
 \begin{align*}
 A_{\m} (X_1, \ldots, X_{m-1})&= K_{\m}^{\#} A_{\m^\#} (X_1, \ldots, X_{p-2}, X_{p-1}+X_{p}, \ldots, X_{m-1})\\
  &\quad + \sum_{a=1}^{N}  K_{\m}^{\star a} A_{\m^{\star a}} (\sub{X}^{\star a}).
 \end{align*}
 
{$\ast$} If for some $s\in \{1, \ldots, p-2 \}$, $\alpha^{(s)} = -\alpha^{(p)} \in \Delta_+$, then
\begin{align*}
A_{\m} (X_1, \ldots, X_{m-1})&= K_{\m}^{\#} A_{\m^\#} (X_1, \ldots, X_{p-2}, X_{p-1}+X_{p}, \ldots, X_{m-1})\\
& \quad + (P_{\alpha^{(s)}} (X_s) - c_s) \sum_{a=1}^{N}  K_{\m}^{\star a} 
 A_{\m^{\star a}} (\sub{X}^{\star a}).
\end{align*} 
\item[{\rm (ii)}] Assume $\alpha^{(p)} + \alpha^{(p-1)} = 0$. 
 
 $\ast$ If for all $s\in \{1, \ldots, p-2 \}$, either $\alpha^{(s)} + \alpha^{(p)} \in -\Delta_+$ or $\alpha^{(s)} + \alpha^{(p)} \notin \Delta \cup \{0\}$, then 
\begin{align*}
A_{\m} (X_1, \ldots, X_{m-1})
= (c_{\alpha^{(p-1)}})^2 (P_{\alpha^{(p-1)}} (X_{p-1})- c_{p-1}) 
A_{\m^\#} (X_1, \ldots, X_{p-2}, X_{p+1}, \ldots, X_{m-1}).
\end{align*}
 
 $\ast$ If for some $s \in \{1, \ldots, p-2 \}$, $\alpha^{(s)} + \alpha^{(p)} \in \Delta_+$, then
\begin{align*}
A_{\m} (X_1, \ldots, X_{m-1})  & =  (P_{\alpha^{(p-1)}} (X_{p-1})- c_{p-1}) A_{\m^\#} (X_1, \ldots, X_{p-2}, X_{p+1}, \ldots, X_{m-1})\\
& \quad + \sum_{a=1}^{N}  K_{\m}^{\star a} A_{\m^{\star a}} (\sub{X}^{\star a}).
\end{align*}

$\ast$ If for some $s \in \{1, \ldots, p-2 \}$, $\alpha^{(s)} = -\alpha^{(p)} \in \Delta_+$, then
\begin{align*}
A_{\m} (X_1, \ldots, X_{m-1}) & = (P_{\alpha^{(p-1)}} (X_{p-1})- c_{p-1}) A_{\m^\#} (X_1, \ldots, X_{p-2}, X_{p+1}, \ldots, X_{m-1})\\
& \quad + (P_{\alpha^{(s)}} (X_s) - c_s) \sum_{a=1}^{N} K_{\m}^{\star a} A_{\m^{\star a}} (\sub{X}^{\star a}),
\end{align*}
 where  $K_{\m}^{\#}$, $K_{\m}^{\star a}$ are some constants, $c_{\alpha^{(p-1)}}$ is a constant,
$c_{p-1}$ is an integer as in Lemma~\ref{Lem:nostar} and $c_{s}:= \langle \alpha^{(s-1)} , \c{\alpha}^{(s)} \rangle$. 
\end{enumerate}
\end{enumerate}
\end{lemma}
\begin{proof}
Let $\alpha \in \Delta$. First, notice that if $\alpha = 2\eps_i$ has type I, 
then $\he(\alpha) = 2(r+1 - i) - 1$ and $\he(\c{\alpha}) = \he((2\eps_i)^\vee) = \he(\eps_i) = r+1-i = \frac{\he(\alpha) + 1}{2}$;  
if $\alpha = \eps_i+{\eps}_j$ has type II, 
then $\he(\alpha) = 2r-i-j+1$ and $\he(\c{\alpha}) = 2r - i - j + 2 = \he(\alpha) + 1$;  
if $\alpha = \eps_i - {\eps}_j$ has type III, then $\he(\alpha) = j-i$ and $\he(\c{\alpha}) = j-i = \he(\alpha).$ 

We prove the statements by induction on $m$.

(a) Assume $m=2$. 
By Example \ref{Ex:eqclassC}, there are three equivalence classes without zero. 
For $\m = [\delta_k, \overline{\delta}_k]$, 
 $$\hc(b_{\sub{\mu},\sub{\alpha}}^*) = a_{\eps_k,-\eps_k}^{(2e_{2\eps_k})} a_{-\eps_k,\eps_k}^{2(e_{-2\eps_k})} e_{2\eps_k}e_{-2\eps_k}
 = 4 e_{-2\eps_k}e_{2\eps_k} + 4 ((2\eps_k)^\vee) = 4 ((2\eps_k)^\vee). $$
 Hence $\wt(\sub{\mu},\sub{\alpha}) = 4 \he ((2\eps_k)^\vee) = 4(\frac{{i}_1 +1}{2}) = 2({i}_1 +1)$,
 and so $A_{[\delta_k, \overline{\delta}_k]}(X_1) = 2(X_1 + 1).$
For $\m = [\delta_k, {\delta}_j], k<j$, since $\m$ is entirely contained in $[\![ \delta_r, \delta_k ]\!]$, then by $\sl_{r+1}$ case
 we have $A_{[\delta_k, {\delta}_j], k<j}(X_1) = X_1.$
For $\m = [\delta_k, \overline{\delta}_j]$, 
 $$\hc(b_{\sub{\mu},\sub{\alpha}}^*) = a_{\eps_k,-\eps_j}^{(e_{\eps_k+\eps_j})} a_{-\eps_j,\eps_k}^{(e_{-\eps_k-\eps_j})} e_{\eps_k+\eps_j}e_{-\eps_k-\eps_j}
 = ((\eps_k+\eps_j)^\vee). $$
 Hence $\wt(\sub{\mu},\sub{\alpha}) = \he ((\eps_k+\eps_j)^\vee) = \he (\eps_k+\eps_j) +1 = {i}_1 +1$,
 and so $A_{[\delta_k, \overline{\delta}_j]}(X_1) = X_1 + 1.$
 Thus for all $\m \in \E_2$ without zero, $A_{\m}$ is polynomial of degree 1 in $\C[X_1]$.

(b) Let $m \geqslant 3$ and assume the proposition true for any 
$m' \in \{ 2,\ldots,m-1\}$ and any $\m' \in {\E}_{m'}$. 
Let $\m \in {\E}_m$ and 
$(\sub{\mu},\sub{\alpha})  \in \m$.
For $\alpha^{(j)} \in \sub{\alpha}$, let $P_{\alpha^{(j)}}$
be a polynomial as in \eqref{eq:p_alpha}.
Observe that $P_{\alpha^{(j)}}$ is a polynomial of degree 1.

Let $\hat{\sub{\i}}^{\star a}$ denote the sequence of heights 
$((\alpha^{\star a})^{(1)}, \ldots, (\alpha^{\star a})^{(m^{\star a}-1)})$.

(i) 

{$\ast$} If for all $s \in \{1, \ldots, p-2 \}$, either $\alpha^{(s)} + \alpha^{(p)} \in -\Delta_+$ or $\alpha^{(s)} + \alpha^{(p)} \notin \Delta \cup \{0\}$, then 
 by Proposition \ref{Pro:Conc1} 
there is a constant $K_{(\sub{\mu},\sub{\alpha})}^{\# p}$ such that
$$\wt(\sub{\mu},\sub{\alpha}) = K^{\# p} \wt(\sub{\mu},\sub{\alpha})^{\# p} .$$
We have $(\sub{\mu},\sub{\alpha})^{\# p}  \in \m^{\#}$ and by our induction hypothesis, 
there exists a polynomial $A_{\m^\# } \in  \C[Y_1,\ldots,Y_{m-2}]$
of total degree $\leqslant \lfloor \frac{m-1}{2}\rfloor$
such that  $$\wt(\sub{\mu},\sub{\alpha}) = 
K_{(\sub{\mu},\sub{\alpha})}^{\# p} \wt(\sub{\mu},\sub{\alpha})^{\# p}=
K_{(\sub{\mu},\sub{\alpha})}^{\# p} 
A_{\m^\# }(i_1,\ldots,i_{p-1} + i_{p},\ldots,i_{m-1}).$$
Hence the polynomial
$$A_{\m} (X_1,\ldots,X_{m-1}) 
:= K_{\m}^{\#} A_{\m^\# }(X_1, \ldots, X_{p-2}, X_{p-1}+X_p, \ldots, X_{m-1})$$
satisfies the conditions of the lemma.
\\{$\ast$} If for some $s \in \{1, \ldots, p-2 \}$, $\alpha^{(s)} + \alpha^{(p)} \in \Delta_+$, then by Proposition \ref{Pro:Conc1}
there are some constants $K_{(\sub{\mu},\sub{\alpha})}^{\# p}$ and $K_{(\sub{\mu},\sub{\alpha})}^{\star a}$ such that
$$\wt(\sub{\mu},\sub{\alpha}) = K_{(\sub{\mu},\sub{\alpha})}^{\# p} \wt(\sub{\mu},\sub{\alpha})^{\# p} 
+ \sum_{a=1}^{N} K_{(\sub{\mu},\sub{\alpha})}^{\star a} \wt(\sub{\mu},\sub{\alpha})^{\star a}.$$
We have $(\sub{\mu},\sub{\alpha})^{\# p}  \in \m^{\#}$ and $(\sub{\mu},\sub{\alpha})^{\star a}  \in \m^{\star a}$.
By our induction hypothesis, 
there are polynomials $A_{\m^\# } \in  \C[Y_1,\ldots,Y_{m-2}]$, and $A_{\m^{\star a} } \in  \C[Y_1,\ldots,Y_{m-2}]$
of total degree $\leqslant \lfloor \frac{m-1}{2}\rfloor$
such that  
\begin{align*}
\wt(\sub{\mu},\sub{\alpha}) &= K_{(\sub{\mu},\sub{\alpha})}^{\# p} 
A_{\m^\# }(i_1,\ldots,i_{p-1} + i_{p},\ldots,i_{m-1}) + \sum_{a=1}^{N} K_{(\sub{\mu},\sub{\alpha})}^{\star a}
A_{\m^{\star a}}(\hat{\sub{\i}}^{\star a}), 
\end{align*}
where $K_{(\sub{\mu},\sub{\alpha})}^{\# p}$ and $K_{(\sub{\mu},\sub{\alpha})}^{\star a}$ are some constants.
Set 
\begin{align*}
A_{\m} (X_1,\ldots,X_{m-1})&:= K_{\m}^{\#} A_{\m^\# }(X_1, \ldots, X_{p-2}, X_{p-1}+X_p, \ldots, X_{m-1})
+ \sum_{a=1}^{N} K_{\m}^{\star a} A_{\m^{\star a}}(\sub{X}^{\star a}).
\end{align*}
By our computation above, the polynomial $A_{\m}$ 
satisfies the condition of the lemma.

{$\ast$} If for some $s \in \{1, \ldots, p-2 \}$, $\alpha^{(s)} = -\alpha^{(p)}$ then 
$$\wt(\sub{\mu},\sub{\alpha}) = K_{(\sub{\mu},\sub{\alpha})}^{\# p} \wt(\sub{\mu},\sub{\alpha})^{\# p} 
+ (\he(\c{\alpha}^{(s)}) -c_s) \sum_{a=1}^{N} K_{(\sub{\mu},\sub{\alpha})}^{\star a} \wt(\sub{\mu},\sub{\alpha})^{\star a}.$$
We have $(\sub{\mu},\sub{\alpha})^{\# p}  \in \m^{\#}$ and $(\sub{\mu},\sub{\alpha})^{\star a}  \in \m^{\star a}$.
By our induction hypothesis, 
there are polynomials $A_{\m^\# } \in  \C[Y_1,\ldots,Y_{m-2}]$ of total degree $\leqslant \lfloor \frac{m-1}{2}\rfloor$
and $A_{\m^{\star a} } \in  \C[Y_1,\ldots,Y_{m-3}]$ of total
degree $\leqslant \lfloor \frac{m-2}{2}\rfloor$
such that  
\begin{align*}
\wt(\sub{\mu},\sub{\alpha}) &= K_{(\sub{\mu},\sub{\alpha})}^{\# p} 
A_{\m^\# }(i_1,\ldots,i_{p-1} + i_{p},\ldots,i_{m-1}) 
 + (P_{\alpha^{(s)}} (i_s) - c_s) \sum_{a=1}^{N} K_{(\sub{\mu},\sub{\alpha})}^{\star a}
A_{\m^{\star a} }(\hat{\sub{\i}}^{\star a}),  
\end{align*}
where $K_{(\sub{\mu},\sub{\alpha})}^{\# p}$, $K_{(\sub{\mu},\sub{\alpha})}^{\star a}$, and $c_s$ are some constants, whereas $P_{\alpha^{(s)}}$ is a polynomial of degree 1.

Set 
\begin{align*}
A_{\m} (X_1,\ldots,X_{m-1})
&:= K_{\m}^{\#} A_{\m^\# }(X_1, \ldots, X_{p-2}, X_{p-1}+X_p, \ldots, X_{m-1})
+ (P_{\alpha^{(s)}} (X_s) - c_s) \sum_{a=1}^{N} K_{\m}^{\star a}
A_{\m^{\star a}}(\sub{X}^{\star a}).
\end{align*}
Thus $A_{\m}$ has total degree $\leqslant 1 + \lfloor \frac{m-2}{2}\rfloor = \lfloor \frac{m}{2} \rfloor$. 
Hence this polynomial satisfies the conditions of the lemma.

(ii) 

{$\ast$} If for all $s\in \{1, \ldots, p-2 \}$, either $\alpha^{(s)} + \alpha^{(p)} \in -\Delta_+$ or $\alpha^{(s)} + \alpha^{(p)} \notin \Delta \cup \{0\}$, then 
by Proposition~\ref{Pro:Conc1} we have
$$\wt(\sub{\mu},\sub{\alpha}) = (c_{\alpha^{(p-1)}})^2 \
(\he(\c{\alpha}^{(p-1)})- c_{p-1})\wt(\sub{\mu},\sub{\alpha})^{\# p}.$$
We have $(\sub{\mu},\sub{\alpha})^{\# p}  \in \m^{\#}$ and by our induction hypothesis, 
there is a polynomial $A_{\m^\# } \in  \C[Y_1,\ldots,Y_{m-3}]$ of total degree $\leqslant \lfloor \frac{m-2}{2}\rfloor$ 
such that   
$$\wt(\sub{\mu},\sub{\alpha}) = 
(c_{\alpha^{(p-1)}})^2 \
(P_{\alpha^{(p-1)}} (i_{p-1})- c_{p-1}) 
A_{\m^\# }(i_1,\ldots,i_{p-2}, i_{p+1},\ldots,i_{m-1}),$$  
where $c_{p-1}$ is a constant and $P_{\alpha^{(p-1)}}$ is a polynomial of degree 1.
Set
$$A_{\m} (X_1,\ldots,X_{m-1}) 
:= (P_{\alpha^{(p-1)}} (X_{p-1})- c_{p-1}) A_{\m^\# }(X_1,\ldots,X_{p-2}, X_{p+1},\ldots,X_{m-1}).$$
Thus $A_{\m}$ has total degree $\leqslant 1 + \lfloor \frac{m-2}{2} \rfloor = \lfloor \frac{m}{2} \rfloor$.
Hence the polynomial $A_{\m}$ satisfies the condition of the lemma. 

{$\ast$} If for some $s\in \{1, \ldots, p-2 \}$, $\alpha^{(s)} + \alpha^{(p)} \in \Delta_+$, then by Proposition~\ref{Pro:Conc1}
there are some constants $K_{(\sub{\mu},\sub{\alpha})}^{\# p}$ and $K_{(\sub{\mu},\sub{\alpha})}^{\star a}$ such that
$$\wt(\sub{\mu},\sub{\alpha}) =(\he(\c{\alpha}^{(p-1)})- c_{p-1}) \wt(\sub{\mu},\sub{\alpha})^{\# p} 
+ \sum_{a=1}^{N}  K_{(\sub{\mu},\sub{\alpha})}^{\star a} \wt(\sub{\mu},\sub{\alpha})^{\star a}.$$
We have $(\sub{\mu},\sub{\alpha})^{\# p}  \in \m^{\#}$ and $(\sub{\mu},\sub{\alpha})^{\star a}  \in \m^{\star a}$.
By our induction hypothesis, 
there are polynomials $A_{\m^\# } \in  \C[Y_1,\ldots,Y_{m-3}]$  of total degree $\leqslant \lfloor \frac{m-2}{2}\rfloor$
and $A_{\m^{\star a} } \in  \C[Y_1,\ldots,Y_{m-2}]$
of total degree $\leqslant \lfloor \frac{m-1}{2}\rfloor$ 
such that  
\begin{align*}
\wt(\sub{\mu},\sub{\alpha}) &= (P_{\alpha^{(p-1)}} (i_{p-1})- c_{p-1})
A_{\m^\# }(i_1,\ldots,i_{p-2}, i_{p+1},\ldots,i_{m-1})
+ \sum_{a=1}^{N} K_{(\sub{\mu},\sub{\alpha})}^{\star a}
A_{\m^{\star a}}(\hat{\sub{\i}}^{\star a}), 
\end{align*}
where $K_{(\sub{\mu},\sub{\alpha})}^{\star a}$, $c_{p-1}$ are some constants and $P_{\alpha^{(p-1)}}$
is a polynomial of degree 1.
Set 
\begin{align*}
A_{\m} (X_1,\ldots,X_{m-1}) 
&= (P_{\alpha^{(p-1)}} (X_{p-1})- c_{p-1}) A_{\m^\# }(X_1,\ldots,X_{p-2}, X_{p+1},\ldots,X_{m-1})
+ \sum_{a=1}^{N} K_{\m}^{\star a}
A_{\m^{\star a}}(\sub{X}^{\star a}).
\end{align*}
By our computation above, $A_{\m}$ has total degree $\leqslant 1 + \lfloor \frac{m-2}{2} \rfloor = \lfloor \frac{m}{2} \rfloor$,
hence it satisfies the conditions of the lemma.

{$\ast$} If for some $s \in \{1, \ldots, p-2 \}$, $\alpha^{(s)} = -\alpha^{(p)}$ then 
$$ \wt(\sub{\mu},\sub{\alpha})=
(\he(\c{\alpha}^{(p-1)})- c_{p-1}) \wt(\sub{\mu},\sub{\alpha})^{\# p}
+ (\he(\c{\alpha}^{(s)}) - c_{s}) \sum_{a=1}^{N}  
K_{(\sub{\mu},\sub{\alpha})}^{\star a} \wt(\sub{\mu},\sub{\alpha})^{\star a}.$$
By our induction hypothesis, 
there are polynomials $A_{\m^\# } \in  \C[Y_1,\ldots,Y_{m-3}]$
and $A_{\m^{\star a} } \in  \C[Y_1,\ldots,Y_{m-3}]$
of total degree $\leqslant \lfloor \frac{m-2}{2}\rfloor$ 
such that  
\begin{align*}
\wt(\sub{\mu},\sub{\alpha})& = (P_{\alpha^{(p-1)}} (i_{p-1})- c_{p-1}) A_{\m^\# }(i_1,\ldots,i_{p-1} + i_{p},\ldots,i_{m-1})
+ (P_{\alpha^{(s)}} (i_s) - c_s) \sum_{a=1}^{N} K_{(\sub{\mu},\sub{\alpha})}^{\star a}
A_{\m^{\star a} }(\hat{\sub{\i}}^{\star a}),  
\end{align*}
where $c_{p-1}$ and $c_s$  are some constants, whereas $P_{\alpha^{(p-1)}}$ and $P_{\alpha^{(s)}}$ 
are polynomials of degree 1.
Hence the polynomial 
\begin{align*}
A_{\m} (X_1,\ldots,X_{m-1})
&:= (P_{\alpha^{(p-1)}} (X_{p-1})- c_{p-1})
A_{\m^\# } (X_1,\ldots,X_{p-2}, X_{p+1},\ldots,X_{m-1}) \\
& \quad + (P_{\alpha^{(s)}} (X_s) - c_s) \sum_{a=1}^{N} K_{\m}^{\star a}
A_{\m^{\star a}}(\sub{X}^{\star a})
\end{align*}
satisfies the conditions of the lemma.\qed
\end{proof}

\begin{example} \label{ex:polynomial}
Let $\m \in {\E}_m$ without zero.
\begin{enumerate}
 \item Assume $m=4$, and 
$\m = [\delta_k, \overline{\delta}_j, \overline{\delta}_l, \delta_j]$, with $k <l<j$.

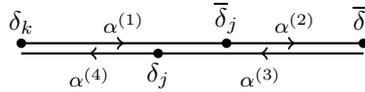
\begin{center}
 \begin{tikzpicture}[scale=0.9]
\fill (0,0) circle (0.07)node(xline)[above] {$\delta_k$}; 
\draw[thick,
        decoration={markings, mark=at position 0.5 with {\arrow{>}}},
        postaction={decorate}
        ]
        (0,0) -- (3,0);\node at (1.5,0.3) {{\small $\alpha^{(1)}$}};
\fill (3,0) circle (0.07)node(xline)[above] {$\overline{\delta}_j$};
 \draw[ thick,
        decoration={markings, mark=at position 0.5 with {\arrow{>}}},
        postaction={decorate}
        ]
        (3,0) -- (5,0); \node at (4,0.3) {{\small $\alpha^{(2)}$}};
\fill (5,0) circle (0.07)node(xline)[above] {$\overline{\delta}_l$};
 \draw[ thick,
        decoration={markings, mark=at position 0.5 with {\arrow{>}}},
        postaction={decorate}
        ]
        (5,-0.15) -- (2,-0.15); \node at (3.5,-0.5) {{\small $\alpha^{(3)}$}};
\fill (2,-0.15) circle (0.07)node(xline)[below] {${\delta}_j$};
 \draw[ thick,
        decoration={markings, mark=at position 0.5 with {\arrow{>}}},
        postaction={decorate}
        ]
        (2,-0.15) -- (0,-0.15); \node at (1,-0.5) {{\small $\alpha^{(4)}$}};
\end{tikzpicture}
\end{center}
 Thus $p=3$ and there is an integer $s = 1$ such that $\alpha^{(p)} + \alpha^{(s)} \in \Delta_+$.
 By Lemma~\ref{Lem:weightC} (2) (c) (i) we get
 \begin{align*}
  A_{\m}(X_1, X_2, X_3) &= K_{\m}^{\#} A_{\mu^{\#}}(X_1, X_2+X_3) + K_{\m}^{\star a} A_{\mu^{\star a}}(X_1+X_3, X_2)
  = X_1 + 2X_2 + X_3.
 \end{align*}
\item Assume $m=6$ and 
$\m = [\delta_k, \delta_i, \overline{\delta}_l, \overline{\delta}_j, \overline{\delta}_i, \overline{\delta}_j]$, 
with $k<i<j<l$.
\\Thus $p=5$ and there is an integer $s = 2$ such that $\alpha^{s} + \alpha^{p} \in \Delta_+$. 
The root $\alpha^{(p-1)}$ has type III. So $P_{\alpha^{(p-1)}}(X_{p-1}) = X_4$ and
$c_{p-1} = \langle \alpha^{(3)}, \c{\alpha}^{(4)} \rangle + \langle \alpha^{(2)}, \c{\alpha}^{(4)} \rangle 
           + \langle \alpha^{(1)}, \c{\alpha}^{(4)} \rangle = -1$.
\\By Lemma~\ref{Lem:weightC}, we have
\begin{align*}
A_{\m} (X_1, \ldots, X_5)
 &= (c_{\alpha^{(p-1)}})^2 (P_{\alpha^{(p-1)}} (X_{p-1})- c_{p-1}) A_{\m^{\#}} (X_1,X_2, X_3)
            + \sum_{a=1}^{5}  K_{\m}^{\star a} A_{\m^{\star a}} (\sub{X}^{\star a})\\
 &= X_1X_4 + 3X_1.
 \end{align*}
\end{enumerate}
\end{example}

\begin{remark}\label{rem:wtC}
 Let $m \in \Z_{>0}$ and $\m \in \E_m$ without zero. There are some classes ${\m}_1,\ldots,{\m}_N$ without zero 
of length $\ell_h:=\ell({\m}_h)< m$
such that 
$$
\wt(\sub{\mu},\sub{\alpha})
= \sum_{h=1}^N K_{\m_h} A_{{\m}_h}(i'_1, \ldots, i'_{\ell_h-1})
.$$ 
where $\he({\sub{\mu}_h})=(i'_1, \ldots, i'_{\ell_h})$. 
\end{remark}

The following useful result will be needed to prove Lemma~\ref{Lem:T_mC}.

\begin{lemma} \label{lem:aij}
 Let $S \in \Z_{\geqslant 1}$, $\ell_1, \ldots, \ell_S$ be a family of stricty positive integers.
For each $1 \leqslant j \leqslant \ell_i$, let $d_{i,j} \in \Z_{\ge0}$. Set $\sub{\ell} = (\ell_1, \ldots, \ell_S)$ and $\sub{d} = (d_{ij})$. 
 Then there is a polynomial $Q_{S, \sub{\ell}, \sub{d}} \in \C[X_1, \ldots, X_S]$ of degree  
$\leqslant \sum_{i,j} d_{ij} 
+ \sum_{i=1}^{S} (\ell_i-1)$ such that 
$$ \sum_{{\tiny \substack{ 
a_{ij} \in \Z{>0} \\ 
1 \leqslant i \leqslant S, \, 1 \leqslant j \leqslant \ell_i, \\ 
\sum_{j=1}^{\ell_i} a_{ij} = n_i}}}
\prod_{i,j} \, (a_{ij})^{d_{ij}} = Q_{S, \sub{\ell}, \sub{d}} (n_1, \ldots, n_S).$$
 \end{lemma}
\begin{proof}
First of all we consider the case where $S=1$. We prove the lemma by induction on $\ell$. For $\ell = 1$, the result is clear.
We assume that the lemma is true for some $\ell \geqslant 1$ and we prove the statement 
for $\ell+1$.
By the induction hypothesis there exists a polynomial $Q_{1, \ell, \sub{d}} \in \C[X]$ of degree $\leqslant d_1+ \cdots + \d_{\ell} + \ell -1$ such that 
$$ Q_{1, \ell, \sub{d}} (n) = \sum_{{\tiny \substack{a_i \in \Z_{>0} \\  \sum_{i=1}^{\ell}a_i = n}}} (a_1)^{d_1} \ldots (a_{\ell})^{d_{\ell}}. $$ 
Write $Q_{1, \ell, \sub{d}} (X) = \sum_{j=0}^{{\tiny d_1+\cdots + d_\ell + \ell-1}} C_j (X)^{j}$. Thus,
\begin{align*}
&\sum_{{\tiny \substack{a_i \in \Z_{>0} \\  \sum_{i=1}^{\ell+1}a_i = n}}} (a_1)^{d_1} \ldots (a_{\ell})^{d_{\ell}}(a_{\ell+1})^{d_{\ell+1}} 
= \sum_{k=1}^{n-\ell} \sum_{{\tiny \substack{a_i \in \Z_{>0} \\  \sum_{i=1}^{l}a_i = n-\i}}} (k)^{d_{\ell+1}} (a_1)^{d_1} \ldots (a_{\ell})^{d_{\ell}} \\
& \qquad = \sum_{k=1}^{n-\ell} (k)^{d_{\ell+1}} \sum_{j=0}^{{\tiny d_1+\cdots + d_\ell + \ell-1}} C_j (n - k)^{j}
=  \sum_{j=0}^{{\tiny d_1+\cdots + d_\ell + \ell-1}} C_j 
\sum_{t=0}^{j} \binom{j}{t} (n)^{j-t} (-1)^t S_{d_{\ell+1} + t}(n-\ell).
\end{align*}
$S_{d_{\ell+1} + t}$ is a polynomial of degree
$d_{\ell+1}+t+1$ with leading term $ (d_{\ell+1}+t+1)^{-1} \ n^{d_{\ell+1}+t+1}$. 
Set $\sub{d}'=(d_1, \ldots, d_{\ell+1})$,
thus there exists a polynomial $Q_{1, \ell+1, \sub{d}'} \in \C[X]$ of degree 
$\leqslant d_1+\ldots+d_\ell+\ell-1+d_{\ell+1}+1 = d_1+\ldots+d_\ell+d_{\ell+1}+(\ell+1)-1$ 
such that  $$Q_{1, \ell+1, \sub{d}'}(n) = \sum_{{\tiny \substack{a_i \in \Z_{>0} \\  \sum_{i=1}^{\ell+1}a_i = n}}} (a_1)^{d_1} \ldots (a_{\ell})^{d_{\ell}}(a_{\ell+1})^{d_{\ell+1}}.$$
And so the lemma is true for $\ell +1$. By induction, it is true for all $\ell \geqslant 2$

Now we consider the case when $S \geqslant 2$. We get 
\begin{align*}
& \sum_{{\tiny \substack{ 
a_{ij} \in \Z{>0} \\ 
1 \leqslant i \leqslant S, \, 1 \leqslant j \leqslant \ell_i, \\ 
\sum_{j=1}^{\ell_i} a_{ij} = n_i}}}
\prod_{i,j} \, (a_{ij})^{d_{ij}} 
 = \sum_{{\tiny \substack{ 
a_{1j} \in \Z{>0} \\ 
1 \leqslant j \leqslant \ell_1, \\ 
\sum_{j=1}^{\ell_1} a_{1j} = n_1}}}
(a_{11})^{d_{11}} \cdots (a_{1\ell_1})^{d_{1\ell_1}} \cdots 
\sum_{{\tiny \substack{ 
a_{Sj} \in \Z{>0} \\ 
1 \leqslant j \leqslant \ell_S, \\ 
\sum_{j=1}^{\ell_S} a_{Sj} = n_S}}}
(a_{S1})^{d_{S1}} \cdots (a_{S\ell_S})^{d_{S\ell_S}}\\
& \qquad  = Q_{1, \ell_1, \sub{d}_1}(n_1) \, Q_{1, \ell_2, \sub{d}_2}(n_2) \ \cdots \ Q_{1, \ell_S, \sub{d}_S}(n_S) = Q_{S, \sub{\ell}, \sub{d}}(n_1, n_2, \ldots , n_S),
\end{align*}
where $\sub{\ell} = (\ell_1, \ldots, \ell_S )$ and $\sub{d}=(d_{ij})$.
Hence,  $Q_{S, \sub{\ell}, \sub{d}}$ is a polynomial 
of degree $\leqslant \sum_{i,j} d_{ij} 
+ \sum_{i=1}^{S} (\ell_i-1)$. \qed
\end{proof}

\begin{remark}
The polynomial $Q_{S, \sub{\ell}, \sub{d}}$ has degree $\sum_{i,j} d_{ij} 
+ \sum_{i=1}^{S} (\ell_i-1)$, for $1 \leqslant i \leqslant S, \, 1 \leqslant j \leqslant \ell_i$, with leading term
$$ \frac{\prod_{ij} d_{ij}!}{\prod_{i=1}^{S} \big( \sum_{j=1}^{\ell_i} d_{ij} + \ell_i -1 \big)!} \, X^{\sum_{i,j} d_{ij} 
+ \sum_{i=1}^{S} (\ell_i-1)}. $$
\end{remark}

The following lemma is the analog to Lemma \ref{Lem:T_m}. 

\begin{lemma}    \label{Lem:T_mC} 
Let $d\in\Z_{\geqslant 0}$, and 
$\m \in {\E}_m$ without zero. 
Let $\sub{d}=(d_1,\ldots,d_{m-1})$ with $d_1+\cdots+d_{m-1}=d$. 
Then for some polynomial $T_{\sub{d},\m} \in \C[X]$ of degree  
$\leqslant d+m- \deg A_{\m}$, we have 
$\widetilde{T}_{\sub{d},{\m}}(k) = T_{\sub{d},\m}(k)$ 
for all $k \in \{1, \ldots, r\}$, where 
$$\widetilde{T}_{\sub{d},{\m}}(k) :=
\sum_{(\sub{\mu},\sub{\alpha}) \in \hat{\P}_m(\delta_k), 
\atop \sub{\mu} \in [\![ \overline{\delta}_{1}, \delta_k ]\!], 
\, (\sub{\mu},\sub{\alpha}) \in \m} i_1^{d_1}\ldots i_{m-1}^{d_{m-1}} ,$$ 
if the integer $i_j$ denote the height 
 of ${\alpha}^{(j)}$ for $j=1,\ldots,m-1$. 
In particular, if for some $k \in \{1,\ldots,r\}$, the set  
$\{ (\sub{\mu},\sub{\alpha}) \in \hat{\P}_m(\delta_k) \; | \; 
\sub{\mu} \in [\![ \overline{\delta}_{1}, \delta_k ]\!] ,\, (\sub{\mu},\sub{\alpha}) \in \m \}$ 
is empty, we have $T_{\sub{d},\m}(k) =0$. 
\end{lemma}

The lemma implies 
that for all $k \in \{1, \ldots, r \}$,
$$ \sum_{{\tiny \substack{(\sub{\mu},\sub{\alpha}) \in \hat{\P}_m(\delta_k), 
\\ \sub{\mu} \in [\![ \overline{\delta}_{1}, \delta_k ]\!],  (\sub{\mu},\sub{\alpha}) \in \m}}} i_1^{d_1}\ldots i_{m-1}^{d_{m-1}} {A}_{\m}(i_1, \ldots, i_{m-1})$$
is a polynomial of degree $d+m$.
\begin{proof}
 We prove the lemma by induction on $m$. 
More precisely, we prove by induction on $m$ the following: 

\smallskip

{\em For all 
$\m \in {\E}_m$ without zero and all  
$\sub{d}=(d_1,\ldots,d_{m-1})$  
with $d_1+\cdots+d_{p-1}=d$, $d \in \Z_{\geqslant 0}$,
Then there exists a polynomial $T_{\sub{d},\m} \in \C[X]$ of degree 
$\leqslant d+m-\deg A_{\m}$ such that for all  $k \in \{1, \ldots, r\},$
$T_{\sub{d},\m}(k)=\widetilde{T}_{\sub{d},{\m}}(k).$ 
In particular, if for some $k \in \{1,\ldots,r\}$, the set  
$\{ (\sub{\mu},\sub{\alpha}) \in \hat{\P}_m(\delta_k) \; | \; 
\sub{\mu} \in [\![ \overline{\delta}_{1}, \delta_k ]\!] ,\, (\sub{\mu},\sub{\alpha}) \in \m \}$ 
is empty, we have $T_{\sub{d},\m}(k) =0$.}

The case $m=1$ is empty. 
Assume $m = 2$, 
and let $\m \in {\E}_2$ without zero. 
The equivalence classes in ${\E}_2$ without zero 
are $[\delta_k, \overline{\delta}_k], [\delta_k, \overline{\delta}_j], [\delta_k, {\delta}_j]$ with $ j \neq k$.
Let $d=d_1 \in \Z_{\geqslant 0}$.

$\bullet$ We first compute the case $\m = [\delta_k, \overline{\delta}_k]$.
\\Let $k \in \{1,\ldots,r\}$. Then the set $\{ (\sub{\mu},\sub{\alpha}) \in \hat{\P}_m(\delta_k) \; | \; 
\sub{\mu} \in [\![ \overline{\delta}_{1}, \delta_k ]\!] ,\, (\sub{\mu},\sub{\alpha}) \in \m \}$ is nonempty.
Observe that there is only one path in this class. Thus, we get
$$\sum_{(\sub{\mu},\sub{\alpha}) \in \hat{\P}_m(\delta_k), 
\atop \sub{\mu} \in [\![ \overline{\delta}_{1}, \delta_k ]\!], 
\, (\sub{\mu},\sub{\alpha}) \in \m} i_1^{d_1} = i_1^{d_1} = (2(r-k) +1)^{d_1}.$$
Hence, the polynomial  $T_{\sub{d},\m}(k) := (2r +1 -2X)^{d_1}$ has degree $d_1 = d \leqslant d+2 - \deg A_{\m} = d+ 1$, 
since $A_{\m}(X_1) = 2(X_1)+1$ has degree 1 by Lemma~\ref{Lem:weightC}.

$\bullet$ Assume $\m = [\delta_k, \overline{\delta}_j], j \neq k$.
For $k \in \{1,\ldots,r\}$, the set $\{ (\sub{\mu},\sub{\alpha}) \in \hat{\P}_m(\delta_k) \; | \; 
\sub{\mu} \in [\![ \overline{\delta}_{1}, \delta_k ]\!] ,\, (\sub{\mu},\sub{\alpha}) \in \m \}$ is nonempty.
Observe that the height of $\alpha^{(1)}$ run through $\{ r-k+1, \ldots 2r-k \}$. Thus, we get
\begin{align*}
\sum_{{\tiny \substack{(\sub{\mu},\sub{\alpha}) \in \hat{\P}_m(\delta_k), 
\\ \sub{\mu} \in [\![ \overline{\delta}_{1}, \delta_k ]\!],\\ 
(\sub{\mu},\sub{\alpha}) \in \m}}} i_1^{d_1} 
= S_{d_1}(2r-k) - S_{d_1}(r-k+1) - (2(r -k) +1)^{d_1}. 
\end{align*}
By Lemma~\ref{Lem:Newton} (1), the polynomial  $T_{\sub{d},\m}(k) := S_{d_1}(2r-k) - S_{d_1}(r-k+1) - (2(r -k) +1)^{d_1}$ has degree $d_1 = d \leqslant d+2 - \deg A_{\m} = d+ 1$, 
since $A_{\m}(X_1) = X_1+1$ has degree 1 by Lemma~\ref{Lem:weightC}.

$\bullet$ Assume $\m = [\delta_k, {\delta}_j], k <j$.
Let $k \in \{1,\ldots,r-1\}$. Then the set $\{ (\sub{\mu},\sub{\alpha}) \in \hat{\P}_m(\delta_k) \; | \; 
\sub{\mu} \in [\![ \overline{\delta}_{1}, \delta_k ]\!] ,\, (\sub{\mu},\sub{\alpha}) \in \m \}$ is nonempty. 
It is empty for $k=r$.
We get  
$$ \sum_{(\sub{\mu},\sub{\alpha}) \in \hat{\P}_m(\delta_k), 
\atop \sub{\mu} \in [\![ \overline{\delta}_{1}, \delta_k ]\!], 
\, (\sub{\mu},\sub{\alpha}) \in \m} i_1^{d_1}
=  \sum_{i_1 =1}^{r +1- k}   i_1^{d_1}  = S_{d_1}(r+1 -  k). 
$$
By Lemma~\ref{Lem:Newton} (1), the polynomial 
$$T_{d,\m}(X) := S_{d_1}(r+1 -  X),$$ 
has degree $d_1+1=d+2- \deg A_{\m}$ 
since $A_{\m}(X_1)=X_1$ has degree 1 by Lemma~\ref{Lem:weightC}. 
For $k = r$, the equality still holds since 
$T_{d,\m} (r) = S_{d}(0) =0$.   
This proves the claim for $m=2$.

Assume $m \geqslant 3$ and the formula holds for any $m' \in \{ 1,\ldots,m-1\}.$
Let $\m \in {\E}_m$ without zero, set $p:=p(\m)$, and 
let $\sub{d}=(d_1,\ldots,d_{m-1})$ with 
$d_1 + \cdots+ d_{m-1} =d$. 
Let $k \in \{1,\ldots,r\}$ be such that 
the set $\{ (\sub{\mu},\sub{\alpha}) \in \hat{\P}_m(\delta_k) \; | \; 
\sub{\mu} \in [\![ \overline{\delta}_{1}, \delta_k ]\!] ,\, (\sub{\mu},\sub{\alpha}) \in \m \}$
is nonempty. Then the sets
$\{ (\sub{\mu}',\sub{\alpha}') \in \hat{\P}_{m-1}(\delta_k) \; | \; 
\sub{\mu}' \in [\![ \overline{\delta}_{1}, \delta_k ]\!] ,\, (\sub{\mu}',\sub{\alpha})' \in \m^{\#} \}$ 
and 
$\{ (\sub{\mu}',\sub{\alpha}') \in \hat{\P}_{m-1}(\delta_k) \; | \; 
\sub{\mu}' \in [\![ \overline{\delta}_{1}, \delta_k ]\!] ,\, (\sub{\mu}',\sub{\alpha})' \in \m^{\star a} \}$ 
are nonempty, too.

Let $\sub{\mu} \in [\![ \overline{\delta}_{1}, \delta_k ]\!]$ such that $(\sub{\mu},\sub{\alpha}) \in \m$ and let 
$\sub{i} := \he({\m})$. 

We first assume that $\mu^{(p)} \in [\![ {\delta}_{r}, \delta_k ]\!]$. By Lemma~\ref{lem:types} the 
admissible triple of $-\alpha^{(p)}$ has type III (a) and so this case can be dealt similarly as in
$\sl_{r+1}$ 
(cf.~Lemma~\ref{Lem:T_m}).

Assume $\mu^{(p)} \in [\![ \overline{\delta}_{1}, \overline{\delta}_r ]\!]$ and let $s \in \{ 1, \ldots, p-2 \}$
be a positive integer. 
If $p=2$ or if for all $s$, either $\alpha^{(s)} + \alpha^{(p)} \in -\Delta_+$ 
or $\alpha^{(s)} + \alpha^{(p)} \notin \Delta \cup \{0\}$, then by Lemma~\ref{Lem:weightC}, the
degree of polynomial $A_{\m}$ only depends on the degree of polynomial $A_{\m^{\#}}$.
Let $\sub{i}^{\#} := \he({\m}^{\#})$. Then one can argue as for the $\sl_{r+1}$ case (cf. Lemma~\ref{Lem:T_m}).

Hence it remains to verify the case when there is a positive integer  $s \in \{ 1, \ldots, p-2 \}$
such that either $\alpha^{(s)} + \alpha^{(p)} \in \Delta_+$ or $\alpha^{(s)} = - \alpha^{(p)}$.

\smallskip

$\ast$ Case $\alpha^{(s)} + \alpha^{(p)} \in \Delta_+$.
Lemma~\ref{Lem:weightC} shows that
\begin{align}\label{eq:key1}
 A_{\m} (X_1, \ldots, X_{m-1})&= \bar{K}_{\m}^{\#} A_{\m^\#} (X_1, \ldots, X_{p-2}, X_{p-1}+ X_p, X_{p+1}, \ldots, X_{m-1}) 
  + \sum_{a=1}^{N}  K_{\m}^{\star a} 
 A_{\m^{\star a}} (\sub{X}^{\star a}),
 \end{align}
where $\bar{K}_{\m}^{\#} =
\begin{cases}
{K}_{\m}^{\#} &\text{ if }  \alpha^{(p-1)} + \alpha^{(p)} \neq 0, \\
(P_{\alpha^{(p-1)}} (X_{p-1})- c_{p-1}) &\text{ if }  \alpha^{(p-1)} + \alpha^{(p)} = 0.
\end{cases}$

\noindent
If $\deg A_{\m} \leqslant \deg A_{\m^{\#}}$ (respectively, $\deg A_{\m} \leqslant \deg A_{\m^{\#}}+1$) 
for $ \alpha^{(p-1)} + \alpha^{(p)} \neq 0$ (respectively, $ \alpha^{(p-1)} + \alpha^{(p)} = 0$)
then using the same strategy as in Lemma~\ref{Lem:T_m} 
we get the statement. 
Therefore we can assume that $\deg A_{\m} > \deg A_{\m^{\#}}$ (respectively,
$\deg A_{\m} > \deg A_{\m^{\#}}+1$) if $ \alpha^{(p-1)} + \alpha^{(p)} \neq 0$
(respectively, $ \alpha^{(p-1)} + \alpha^{(p)} = 0$). 
This assumption means that for some $a \in \{1, \ldots, N \}$,
$\deg A_{\m^{\star a}} \geqslant \deg A_{\m}$.  

By Lemma~\ref{lem:types}, the possibilities for  $-\alpha^{(p)}$ and $\alpha^{(s)}$
 which satisfy the assumption $\alpha^{(s)} + \alpha^{(p)} \in \Delta_+$
 are the following:
\begin{itemize} 
\item[(1)]
$\alpha^{(p)} = \overline{\delta}_j - \delta_i$ and $\alpha^{(s)}=\delta_l - \overline{\delta}_i$, with $l < j < i$,  
\item[(2)]  
$\alpha^{(p)} = \overline{\delta}_j - \delta_i$ and $\alpha^{(s)}=\delta_j - \overline{\delta}_l$, with $j < l < i$,  
\item[(3)] 
$\alpha^{(p)} = \overline{\delta}_i - \overline{\delta}_j$ and $\alpha^{(s)} =\delta_i - \overline{\delta}_l$, with $i<j$ and $i<l$,
\item[(4)]
$\alpha^{(p)} = \overline{\delta}_i - \overline{\delta}_j$ and $\alpha^{(s)} =\delta_i - \delta_l$, with $i < j <l$,
\item[(5)]
$\alpha^{(p)} = \overline{\delta}_i - \overline{\delta}_j$ and $\alpha^{(s)} =\delta_l - \delta_j$, with $l <i<j$.
\end{itemize} 
Note that in all those cases, for all $a \in \{1, \ldots, N \}$, 
$m-p+s \leqslant \he(\m^{\star a}) \leqslant m-1$. Possibly changing the numbering of the equivalence class $\m^{\star a}$, one can assume throughout this proof that $\m^{\star 1} = [(\sub{\mu}^{\star 1},\sub{\alpha}^{\star 1})]$ is the equivalence class of the star paths with the longest length and set $\sub{i} := \he(\m)$ and $\sub{i}' := \he(\m^{\star 1})$. 

(1) 
Recall from Lemma~\ref{Lem:star+} that for  all $(\sub{\mu}, \sub{\alpha})^{\star a} \in \m^{\star a}$,
$a \in \{1, \ldots, N \},$
\begin{align*}
(\sub{\mu}, \sub{\alpha})^{\star a} = (\sub{\mu}',\sub{\alpha}')\star 
\big( (\delta_l,\delta_j), (\alpha^{(s)}+\alpha^{(p)})\big) 
\star (\tilde{\sub{\mu}}^{\star a},\tilde{\sub{\alpha}}^{\star a})
\star(\sub{\mu}'',\sub{\alpha}''),
\end{align*}
where 
$(\tilde{\sub{\mu}}^{\star a},\tilde{\sub{\alpha}}^{\star a})$ 
is a path of length $\leqslant p-s-1$ between 
$\delta_j$ and $\delta_i$ whose roots $\tilde{\sub{\alpha}}^{\star a}$ are 
sums among $(\alpha^{(p-1)}, \ldots, \alpha^{(s+1)})$, $\sub{\alpha}'= (\alpha^{(1)}, \ldots, \alpha^{(s-1)})$ and 
$\sub{\alpha}''= (\alpha^{(p+1)}, \ldots, \alpha^{(m)})$.
We have $\he	(\sub{\mu}^{\star 1}) = (i_1, \ldots, i_{s-1}, i_s+i_p, i_{p-1}, \ldots, i_{s+1}, i_{p+1}, \ldots, i_m)$. 

Set $\sub{i}' = (i'_1, \ldots, i'_{m-1})$, so $i'_s = i_s + i_p$.  Note that $\alpha^{(p)} = \overline{\delta}_j - \delta_i$ and $\alpha^{(s)}=\delta_l - \overline{\delta}_i$, with $l < j < i$. Thus,
$ l = k+ i'_1+ \cdots+ i'_{s-1}, j = l + i'_s  \text{ and } \ i = k+ i'_1+ \cdots+ i'_{p-1}.$
Hence, for all $(\sub{\mu}, \sub{\alpha}) \in \m$ the heights in $\sub{i}$  
can be expressed in term of $\sub{i}'$ as follows 
\begin{align*}
& i_1 = i'_1, \  \ldots, \  i_{s-1} = i'_{s-1}, \
 i_{s+1} = i'_{p-1}, \  \ldots,  \  i_{p-1} = i'_{s+1}, \
 i_{p+1} = i'_{p}, \  \ldots,  \  i_{m-1} = i'_{m-2}, & \\
& i_{s} = 2r - 2k -(i+l) + 1 = 2r +1 - 2k-2(i'_1+ \cdots+ i'_{s-1}) - (i'_{s}+ \cdots+ i'_{p-1}),  &\\
& i_{p} = i'_s - i_s = 2k - 2r -1 +2(i'_1+ \cdots+ i'_{s}) + (i'_{s+1}+ \cdots+ i'_{p-1}).
\end{align*}
We get,
\begin{align} 
\label{eq:tildeTmk} \nonumber
\widetilde{T}_{\sub{d},\m}(k) 
&= 
  \sum_{{\tiny \substack{\sub{q} \in \N^{p} 
    \\ |\sub{q}|=d_s}}}
    \sum_{{\tiny \substack{\sub{q'} \in \N^{p} 
    \\ |\sub{q'}|=d_p}}} \frac{(d_s)!}{q_1!\ldots q_{p}!}  \frac{(d_p)!}{q'_1!\ldots q'_{p}!}  
    (-1)^{d_s} (2)^{q_1 + \cdots +q_{s-1}+q'_1 + \cdots +q'_{s}}
    \ (2k-2r-1)^{q_p+q'_{p}} \\ 
 &  \qquad \times \sum_{{\tiny \substack{(\sub{\mu},\sub{\alpha})^{\star 1} \in \hat{\P}_{m-1}(\delta_k), 
\\ \sub{\mu}^{\star 1} \in [\![ \overline{\delta}_{1}, \delta_k ]\!], 
\\ (\sub{\mu},\sub{\alpha})^{\star 1} \in \m^{\star 1}}}} 
  (i'_1)^{d_1+q_1+q'_1} 
\cdots ({i'}_{s-1})^{d_{s-1}+q_{s-1}+q'_{s-1}}
 \ ({i'}_{s})^{q_{s}+q'_{s}} \\ 
 & \qquad \times  ({i'}_{s+1})^{d_{p-1}+q_{s+1}+q'_{s+1}} \cdots ({i'}_{p-1})^{d_{s+1}+q_{p-1}+q'_{p-1}}
 ({i'}_{p})^{d_{p+1}} \cdots ({i'}_{m-2})^{d_{m-1}}. \nonumber
\end{align}
Assume that for some $a' \in \{1, \ldots, N \}$ the equivalence class $\m^{\star a'} = [(\sub{\mu}^{\star a'},\sub{\alpha}^{\star a'})]$ satisfies 
$\deg A_{\m^{\star a'}} \geqslant \deg A_{\m}$. According to the proof of Lemma~\ref{Lem:star+}, there exists a partition $(t_1, \ldots, t_{n'}), t_i >0$,  of $p-s-1$ such that
\begin{align*} 
i^{\star}_{s+1} = i'_{s+1} +
 \cdots + i'_{s+t_1}, \quad 
i^{\star}_{s+2} = i'_{s+t_1+1} + \cdots + i'_{(s+t_1+t_2)}, 
\ldots,
i^{\star}_{s+n'} = i'_{s+t_1+\cdots + t_{n'-1}+1)} + \cdots + i'_{(p-1)}.
\end{align*}
Hence $\he(\sub{\mu}^{\star a'})= (i'_1, \ldots, i'_{s-1}, i'_{s}, i^{\star}_{s+1}, \ldots, i^{\star}_{s+n'}, i'_{p}, \ldots, i'_{m-1})$ with the length of $\m^{\star a'}$, denoted by $m^{\star}$, is equal to $m-p+s+n'$. Here $i^{\star}_{s+j}$ denote the height of the root $(\tilde{\sub{\alpha}}^{\star a'} )^{(s+j)}$.

By Lemma~\ref{lem:aij}, for each $\sub{d}'=(d'_{s+1}, \ldots, d'_{p-1})$ there is a polynomial 
$Q_{n',\sub{t}, \sub{d}'} \in \C[X_1, \ldots, X_{n'}]$, $\sub{t} = (t_1, \ldots, t_{n'})$, of degree 
$d'_{s+1}+ \cdots+ d'_{p-1}+(p-s-1) -n' $ such that 
$$  \sum_{{\tiny \substack{1 \leqslant j \leqslant n' 
\\ 1+t_1+\cdots+t_{j-1} \leqslant l \leqslant t_1+\cdots+t_j
\\ \sum_{l} i'_{s+l} = i^{\star}_{s+j}}}} (i'_{s+1})^{d'_{s+1}} \cdots (i'_{p-1})^{d'_{p-1}} =
Q_{n',\sub{t}, \sub{d}'} (i^{\star}_{s+1}, \ldots, i^{\star}_{s+n'}).$$
\noindent
Set $\sub{d}'_{\sub{q}, \sub{q}'} =(d_{p-1}+q_{s+1}+q'_{s+1}, \ldots, d_{s+1}+q_{p-1}+q'_{p-1})$, ${d}'_{\sub{q}, \sub{q}'}= 
|\sub{d}'_{\sub{q}, \sub{q}' }|$
and write
$$ Q_{n',\sub{t}, \sub{d}'_{\sub{q}, \sub{q}'}} (X_1, \ldots, X_{n'}) 
= \sum_{\tiny {\substack{\sub{j}=(j_1, \ldots, j_{n'})\\ 
 |\sub{j}| \leqslant {d}'_{\sub{q}, \sub{q}' }+p-s-1-n'}}} C_{\sub{j}} X^{j_1}_1 \cdots X^{j_{n'}}_{n'}.$$
It is a polynomial of degree $d_{p-1}+q_{s+1}+q'_{s+1}+ \cdots+ d_{s+1}+q_{p-1}+q'_{p-1}+ p-s-1 - n'$.
Hence,
\begin{align*}
\widetilde{T}_{\sub{d},\m}(k) 
&=   \sum_{{\tiny \substack{\sub{q} \in \N^{p} 
    \\ |\sub{q}|=d_s}}}
    \sum_{{\tiny \substack{\sub{q'} \in \N^{p} 
    \\ |\sub{q'}|=d_p}}} \frac{(d_s)!}{q_1!\ldots q_{p}!}  \frac{(d_p)!}{q'_1!\ldots q'_{p}!}  
    (-1)^{d_s} (2)^{q_1 + \cdots +q_{s-1}+q'_1 + \cdots +q'_{s}}
    \ (2k-2r-1)^{q_p+q'_{p}} \\ 
 & \quad  \times \sum_{{\tiny \substack{(\sub{\mu},\sub{\alpha})^{\star a'} \in \hat{\P}_{m^{\star}}(\delta_k), 
\\ \sub{\mu}^{\star a'} \in [\![ \overline{\delta}_{1}, \delta_k ]\!], 
\\ (\sub{\mu},\sub{\alpha})^{\star a'} \in \m^{\star a'}}}}  
  (i'_1)^{d_1+q_1+q'_1} 
\cdots ({i'}_{s-1})^{d_{s-1}+q_{s-1}+q'_{s-1}}
 \ ({i'}_{s})^{q_{s}+q'_{s}} \\ 
 & \quad \times  \sum_{\tiny {\substack{\sub{j}=(j_1, \ldots, j_{n'})\\ 
 |\sub{j}| \leqslant {d}'_{\sub{q}, \sub{q}' }+p-s-1-n'}}} C_{\sub{j}} \, (i^{\star}_{s+1})^{j_1} \cdots (i^{\star}_{s+n'})^{j_{n'}} \ 
 ({i'}_{p})^{d_{p+1}} \cdots ({i'}_{m-2})^{d_{m-1}}.
\end{align*}
Set  
$\sub{d}_{\sub{q}, \sub{q}'}=(d_1+q_1+q'_1, \ldots, d_{s-1}+q_{s-1}+q'_{s-1}, q_s+q'_s, j_{1}, \ldots, j_{n'}, d_{p+1}, \ldots, d_{m-1}),$
with $|\sub{d}_{\sub{q}, \sub{q}'}|= d+(p-s-1)-n'-q_{p} -q'_p$, and
\begin{align*}
\widetilde{T}_{\sub{d}_{\sub{q}, \sub{q}'},{\m}^{\star a'}}(k) &=
 \sum_{{\tiny \substack{(\sub{\mu},\sub{\alpha})^{\star a'} \in \hat{\P}_{m^{\star}}(\delta_k), 
         \\ \sub{\mu}^{\star a'} \in [\![ \overline{\delta}_{1}, \delta_k ]\!], 
          \\ (\sub{\mu},\sub{\alpha})^{\star a} \in \m^{\star a'}}}}
   (i'_1)^{d_1+q_1+q'_1} 
\cdots ({i'}_{s-1})^{d_{s-1}+q_{s-1}+q'_{s-1}}
 \ ({i'}_{s})^{q_{s}+q'_{s}} \\ 
 & \quad \times  (i^{\star}_{s+1})^{j_1} \cdots (i^{\star}_{s+n'})^{j_{n'}} \ 
 ({i'}_{p})^{d_{p+1}} \cdots ({i'}_{m-2})^{d_{m-1}}.
\end{align*}
Then 
\begin{align} \label{eq:TmK-C-1} \nonumber
\widetilde{T}_{\sub{d},\m}(k) 
&= 
\sum_{{\tiny \substack{\sub{q} \in \N^{p} 
    \\ |\sub{q}|=d_s}}}
    \sum_{{\tiny \substack{\sub{q'} \in \N^{p} 
    \\ |\sub{q'}|=d_p}}} 
   \sum_{\tiny {\substack{\sub{j}=(j_1, \ldots, j_{n'})\\ 
 |\sub{j}| \leqslant {d}'_{\sub{q}, \sub{q}' }+p-s-1-n'}}}    
    \frac{(d_s)!}{q_1!\ldots q_{p}!} 
  \frac{(d_p)!}{q'_1!\ldots q'_{p}!}  
   C_{\sub{j}} (-1)^{d_s} (2)^{q_1 + \cdots +q_{s-1}+q'_1 + \cdots +q'_{s}}
\\ & \quad \times     \ (2k-2r-1)^{q_p+q'_{p}}
 \widetilde{T}_{\sub{d}_{\sub{q}, \sub{q}'},{\m}^{\star a'}}(k).
\end{align}
 
By the induction hypothesis applied to $m^{\star}$ and ${\m}^{\star a'}$, 
we have $T_{\sub{d}, {\m}}(k) = \widetilde{T}_{\sub{d},{\m}}(k)$, where  
$T_{\sub{d}, {\m}}$ is a polynomial of degree 
$< d+m- \deg A_{{\m}}$ for all $k$ 
such that $\{ (\sub{\mu},\sub{\alpha}) \in \hat{\P}_{m}(\delta_k) \; | \; 
\sub{\mu} \in [\![ \overline{\delta}_{1}, \delta_k ]\!] ,\, (\sub{\mu},\sub{\alpha}) \in \m \}$  
is nonempty.

It remains to verify that $T_{\sub{d}, \m} (k) =0$  
 when
 $$\{ (\sub{\mu},\sub{\alpha}) \in \hat{\P}_{m}(\delta_k) \; | \;  
 \sub{\mu} \in [\![ \overline{\delta}_{1}, \delta_k]\!],  (\sub{\mu},\sub{\alpha}) \in \m \} =\varnothing.$$
Observe that this set is never empty if the set 
 $\{ (\sub{\mu},\sub{\alpha})^{\star 1} \in \hat{\P}_{m-1}(\delta_k) \; | \;  
 \sub{\mu}^{\star 1} \in [\![ \overline{\delta}_{1}, \delta_k]\!], (\sub{\mu},\sub{\alpha})^{\star 1} \in \m^{\star 1}  \}$ is nonempty.
For any $(\sub{\mu},\sub{\alpha})^{\star 1} \in \m^{\star 1}$ where $\he(\sub{\mu}^{\star 1}) = \sub{i}'$, we have 
 $i'_s= \delta_l - \delta_j = \eps_l-\eps_j$, the source of $i'_{s+1}$ is $\delta_j= \eps_j$ and the target of $i'_{p-1}$ is $\delta_i= \eps_i$.  Thus we can always reconstruct the initial path $(\sub{\mu},\sub{\alpha})$ by taking 
 $i_s = \delta_l - \overline{\delta}_i= \eps_l+\eps_i$ and $i_p = \overline{\delta}_j-\delta_i= -\eps_j-\eps_i$. This is possible since $(\alpha^{\star 1})^{(1)}, \ldots, (\alpha^{\star 1})^{(p-1)} $ is entirely contained in $[\![ \delta_r, \delta_k ]\!]$.  
Hence, the set 
$\{ (\sub{\mu},\sub{\alpha}) \in \hat{\P}_{m}(\delta_k) \; | \;  
 \sub{\mu} \in [\![ \overline{\delta}_{1}, \delta_k]\!],  (\sub{\mu},\sub{\alpha}) \in \m \}$ is empty if the set $\{ (\sub{\mu},\sub{\alpha})^{\star 1} \in \hat{\P}_{m-1}(\delta_k) \; | \;  
 \sub{\mu}^{\star 1} \in [\![ \overline{\delta}_{1}, \delta_k]\!], (\sub{\mu},\sub{\alpha})^{\star 1} \in \m^{\star 1}  \} = \varnothing$. Furthermore, the set  $\{ (\sub{\mu},\sub{\alpha})^{\star a'} \in \hat{\P}_{m^{\star}}(\delta_k) \; | \;  
 \sub{\mu}^{\star a'} \in [\![ \overline{\delta}_{1}, \delta_k]\!], (\sub{\mu},\sub{\alpha})^{\star a'} \in \m^{\star a'}  \}$ is empty too. 
 Our induction hypothesis 
 says that, in this case, $T_{\sub{d}, \m} (k)
 = {T}_{\sub{d}',{\m}^{\star a'}}(k) = 0$
 for any $\sub{d}' \in \Z_{\geqslant 0}^{m^{\star}}.$  
This proves the lemma for case I(1). 

Observe that our above strategy works for case (5) as well, since by Lemma~\ref{Lem:star+} we also have 
\begin{align*}
(\sub{\mu}, \sub{\alpha})^{\star a} = (\sub{\mu}',\sub{\alpha}')\star 
\big( (\delta_l,\delta_j), (\alpha^{(s)}+\alpha^{(p)})\big) 
\star (\tilde{\sub{\mu}}^{\star a},\tilde{\sub{\alpha}}^{\star a})
\star(\sub{\mu}'',\sub{\alpha}''),
\end{align*}
and so the arguments are quite similar. 
Since the calculations are similar, we omit the detail. Hence we conclude 
the case (5) as in the first case.

(2) Assume now that
$\alpha^{(p)} = \overline{\delta}_j - \delta_i$ and $\alpha^{(s)}=\delta_j - \overline{\delta}_l$, with $j < l < i$.
Recall from Lemma~\ref{Lem:star+}, for  all $(\sub{\mu}, \sub{\alpha})^{\star a} \in \m^{\star a}$,
$a \in \{1, \ldots, N \},$
\begin{align*}
(\sub{\mu}, \sub{\alpha})^{\star a} = (\sub{\mu}',\sub{\alpha}')\star 
 (\tilde{\sub{\mu}}^{\star a},\tilde{\sub{\alpha}}^{\star a})
\star(\sub{\mu}'',\sub{\alpha}''),
\end{align*}
where 
$(\tilde{\sub{\mu}}^{\star a},\tilde{\sub{\alpha}}^{\star a})$ 
is a path of length $\leqslant p-s$ between 
$\delta_j$ and $\delta_i$ whose roots $\tilde{\sub{\alpha}}^{\star a}$ are 
sums among $(\alpha^{(p-1)}, \ldots, \alpha^{(s+1)}, \alpha^{(s)}+\alpha^{(p)})$, 
$\sub{\alpha}'= (\alpha^{(1)}, \ldots, \alpha^{(s-1)})$ and 
$\sub{\alpha}''= (\alpha^{(p+1)}, \ldots, \alpha^{(m)})$.
We have $\he	(\sub{\mu}^{\star 1}) = (i_1, \ldots, i_{s-1}, i_{p-1}, \ldots, i_{s+1},  i_s+i_p, i_{p+1}, \ldots, i_m)$. 

Set $\sub{i}' = (i'_1, \ldots, i'_{m-1})$, so $i'_{p-1} = i_s + i_p$.
Note that $\alpha^{(p)} = \overline{\delta}_j - \delta_i$ and $\alpha^{(s)}=\delta_j - \overline{\delta}_l$, 
with $j < l < i$. Thus,
$ j = k+ i'_1+ \cdots+ i'_{s-1}, i = l + i'_{p-1}  \text{ and }  l = k+ i'_1+ \cdots+ i'_{p-2}.$
Hence, for all $(\sub{\mu}, \sub{\alpha}) \in \m$ the heights $\sub{i} := \he(\sub{\mu})$  
can be expressed in term of $\sub{i}'$ as follows:
\begin{align*}
& i_1 = i'_1, \, \ldots, \, i_{s-1} = i'_{s-1}, 
\ i_{s+1} = i'_{p-1}, \, \ldots, \,  i_{p-1} = i'_{s+1}, 
\ i_{p+1} = i'_{p}, \, \ldots, \, i_{m-1} = i'_{m-2}, & \\
& i_{s} = 2r - 2k -(j+l) + 1 = 2r +1 - 2k-2(i'_1+ \cdots+ i'_{s-1}) - (i'_{s}+ \cdots+ i'_{p-2}),  &\\
& i_{p} = i'_{p-1} - i_s = 2k - 2r -1 +2(i'_1+ \cdots+ i'_{s-1}) + (i'_{s}+ \cdots+ i'_{p-1}).
\end{align*}
In the same manner as in \eqref{eq:tildeTmk} we get,
\begin{align} \label{eq:tildeTmk_2} \nonumber
\widetilde{T}_{\sub{d},\m}(k) 
&=   \sum_{{\tiny \substack{\sub{q} \in \N^{p-1} 
    \\ |\sub{q}|=d_s}}}
    \sum_{{\tiny \substack{\sub{q'} \in \N^{p} 
    \\ |\sub{q'}|=d_p}}} C_{\sub{q},\sub{q}'}
    \ (2k-2r-1)^{q_{p-1}+q'_{p}} \sum_{{\tiny \substack{(\sub{\mu},\sub{\alpha})^{\star 1} \in \hat{\P}_{m-1}(\delta_k), 
\\ \sub{\mu}^{\star 1} \in [\![ \overline{\delta}_{1}, \delta_k ]\!], 
\\ (\sub{\mu},\sub{\alpha})^{\star 1} \in \m^{\star 1}}}} 
  (i'_1)^{d_1+q_1+q'_1} \\ 
 & \quad \times 
\cdots \times \ ({i'}_{s})^{d_{p-1}+q_{s}+q'_{s}}  \cdots ({i'}_{p-1})^{q'_{p-1}}
 ({i'}_{p})^{d_{p+1}} \cdots ({i'}_{m-2})^{d_{m-1}}.
\end{align}

Assume that the equivalence class $\m^{\star a'} = [(\sub{\mu}^{\star a'},\sub{\alpha}^{\star a'})]$ satisfies 
$\deg A_{\m^{\star a'}} \geqslant \deg A_{\m}$, for some $a' \in \{1, \ldots, N \}$. 
By Lemma~\ref{Lem:star+}, there exists a partition
$(t_1, \ldots, t_{n'} ), t_i >0$, of $p-s$ such that
\begin{align*} 
i^{\star}_{s} &= i'_{s} + \cdots + i'_{s+t_1}, \\ 
\vdots \\
i^{\star}_{s+n'-2} &= i'_{s+t_1+\cdots+t_{(n'-2)}+1} + \cdots + i'_{s+t_1+\ldots+t_{n'-1}}, \\
i^{\star}_{s+n'-1} &= i'_{s+t_1+\ldots+t_{n'-1}+1} + \cdots + i'_{p-1}.
\end{align*}
Hence $\he(\sub{\mu}^{\star a'})= (i'_1, \ldots, i'_{s-1}, i^{\star}_{s}, i^{\star}_{s+1}, \ldots, i^{\star}_{s+n'-1}, i'_{p}, \ldots, i'_{m-1})$ and the length of $\m^{\star a'}$, denoted by $m^{\star}$, is equal to $m-(p-s)+n'-2$. Here $i^{\star}_{s+j}$ denote the height of the root $(\tilde{\sub{\alpha}}^{\star a'} )^{(s+j)}$.

By Lemma~\ref{lem:aij} and \eqref{eq:tildeTmk_2} we get,
\begin{align*}
\widetilde{T}_{\sub{d},\m}(k) 
&=   \sum_{{\tiny \substack{\sub{q} \in \N^{p-1} 
    \\ |\sub{q}|=d_s}}}
    \sum_{{\tiny \substack{\sub{q'} \in \N^{p} 
    \\ |\sub{q'}|=d_p}}} C_{\sub{q},\sub{q}'}
    \ (2k-2r-1)^{q_{p-1}+q'_{p}} 
     \sum_{{\tiny \substack{(\sub{\mu},\sub{\alpha})^{\star a'} \in \hat{\P}_{m^{\star}}(\delta_k), 
\\ \sub{\mu}^{\star a'} \in [\![ \overline{\delta}_{1}, \delta_k ]\!], 
\\ (\sub{\mu},\sub{\alpha})^{\star a'} \in \m^{\star a'}}}} 
  (i'_1)^{d_1+q_1+q'_1} \cdots  \\ 
 & \quad \times (i'_{s-1})^{d_{s-1}+q_{s-1}+q'_{s-1}} 
\ Q_{n',\sub{t}, \sub{d}'_{\sub{q}, \sub{q}'}} (i^{\star}_{s}, \ldots, i^{\star}_{s+n'-1})
 ({i'}_{p})^{d_{p+1}} \cdots ({i'}_{m-2})^{d_{m-1}},
\end{align*}
where $Q_{n',\sub{t}, \sub{d}'_{\sub{q}, \sub{q}'}}$ is a polynomial of degree 
$d_{s+1}+\cdots+d_{p-1}+q_s+\cdots+q_{p-2}+q'_s+\cdots+q'_{p-1}+p-s-n'$.
Repeated steps as in (1) enables us to write
\begin{align*}
{T}_{\sub{d},\m}(k) 
&= 
\sum_{{\tiny \substack{\sub{q} \in \N^{p-1} 
    \\ |\sub{q}|=d_s}}}
    \sum_{{\tiny \substack{\sub{q'} \in \N^{p} 
    \\ |\sub{q'}|=d_p}}} \sum_{\tiny {\substack{\sub{j}=(j_1, \ldots, j_{n'})\\ 
 |\sub{j}| \leqslant {d}'_{\sub{q}, \sub{q}' }+p-s-1-n'}}} C_{\sub{j}} C_{\sub{q},\sub{q}'} \, (2k-2r-1)^{q_{p-1}+q'_{p}} 
{T}_{\sub{d}'_{\sub{q}, \sub{q}' },{\m}^{\star a'}}(k).
\end{align*}
As in the first case, 
we get that 
$T_{\sub{d}, {\m}}(k) = \widetilde{T}_{\sub{d},{\m}}(k)$ and 
$T_{\sub{d}, {\m}}$ is a polynomial of degree 
$< d+m- \deg A_{{\m}}$ for all $k$ 
such that $\{ (\sub{\mu},\sub{\alpha}) \in \hat{\P}_{m}(\delta_k) \; | \; 
\sub{\mu} \in [\![ \overline{\delta}_{1}, \delta_k ]\!] ,\, (\sub{\mu},\sub{\alpha}) \in \m \}$  
is nonempty.

It remains to verify that $T_{\sub{d}, \m} (k) =0$  
 when
 $$\{ (\sub{\mu},\sub{\alpha}) \in \hat{\P}_{m}(\delta_k) \; | \;  
 \sub{\mu} \in [\![ \overline{\delta}_{1}, \delta_k]\!],  (\sub{\mu},\sub{\alpha}) \in \m \} =\varnothing.$$
Observe that this set is never empty if 
 $\{ (\sub{\mu},\sub{\alpha})^{\star 1} \in \hat{\P}_{m-1}(\delta_k) \; | \;  
 \sub{\mu}^{\star 1} \in [\![ \overline{\delta}_{1}, \delta_k]\!],(\sub{\mu},\sub{\alpha})^{\star 1} \in \m^{\star 1}  \}$ is nonempty.
For any $(\sub{\mu},\sub{\alpha})^{\star 1} \in \m^{\star 1}$ where $\he(\sub{\mu}^{\star 1}) = \sub{i}'$, we have 
 $i'_{p-1}= \delta_l - \delta_j = \eps_l-\eps_j$, the source of $i'_{s}$ is $\delta_j= \eps_j$. Thus we can always reconstruct the initial path $(\sub{\mu},\sub{\alpha})$ by taking 
 $i_s = \delta_j - \overline{\delta}_l= \eps_j+\eps_l$ and $i_p = \overline{\delta}_j-\delta_i= -\eps_j-\eps_i$. This is possible since $(\alpha^{\star 1})^{(1)}, \ldots, (\alpha^{\star 1})^{(p-1)} $ is entirely contained in $[\![ \delta_r, \delta_k ]\!]$. 
Hence, the set 
$\{ (\sub{\mu},\sub{\alpha}) \in \hat{\P}_{m}(\delta_k) \; | \;  
 \sub{\mu} \in [\![ \overline{\delta}_{1}, \delta_k]\!],  (\sub{\mu},\sub{\alpha}) \in \m \}$ is empty if $\{ (\sub{\mu},\sub{\alpha})^{\star 1} \in \hat{\P}_{m-1}(\delta_k) \; | \;  
 \sub{\mu}^{\star 1} \in [\![ \overline{\delta}_{1}, \delta_k]\!], (\sub{\mu},\sub{\alpha})^{\star 1} \in \m^{\star 1}  \} = \varnothing$. Furthermore, the set  $\{ (\sub{\mu},\sub{\alpha})^{\star a'} \in \hat{\P}_{m^{\star}}(\delta_k) \; | \;  
 \sub{\mu}^{\star a'} \in [\![ \overline{\delta}_{1}, \delta_k]\!], (\sub{\mu},\sub{\alpha})^{\star a'} \in \m^{\star a'}  \}$ is empty too. 
 Our induction hypothesis 
 says that, in this case, $$T_{\sub{d}, \m} (k)
 = {T}_{\sub{d}',{\m}^{\star a'}}(k) = 0,$$
 for any $\sub{d}' \in \Z_{\geqslant 0}^{m^{\star}}.$  
This proves the lemma for case (2). 

(3) Assume that
$\alpha^{(p)} = \overline{\delta}_i - \overline{\delta}_j$ and $\alpha^{(s)} =\delta_i - \overline{\delta}_l$, with $i<j$ and $i<l$.
\\Recall from Lemma~\ref{Lem:star+}, for  all $(\sub{\mu}, \sub{\alpha})^{\star a} \in \m^{\star a}$,
$a \in \{1, \ldots, N \},$
\begin{align*}
(\sub{\mu}, \sub{\alpha})^{\star a} = (\sub{\mu}',\sub{\alpha}')\star 
 (\tilde{\sub{\mu}}^{\star a},\tilde{\sub{\alpha}}^{\star a})
\star(\sub{\mu}'',\sub{\alpha}''),
\end{align*}
where 
$(\tilde{\sub{\mu}}^{\star a},\tilde{\sub{\alpha}}^{\star a})$ 
is a path of length $\leqslant p-s$ between 
$\delta_j$ and $\delta_i$ whose roots $\tilde{\sub{\alpha}}^{\star a}$ are 
sums among $(\alpha^{(p-1)}, \ldots, \alpha^{(s+1)}, \alpha^{(s)}+\alpha^{(p)})$, 
$\sub{\alpha}'= (\alpha^{(1)}, \ldots, \alpha^{(s-1)})$ and 
$\sub{\alpha}''= (\alpha^{(p+1)}, \ldots, \alpha^{(m)})$.

Note that the order of roots in $\sub{\alpha}^{\star 1}$ in this case may differ than in the case (2), depending on the situation of $\alpha^{(p-1)}$ and $\alpha^{(p)}$ as described in the proof of Lemma~\ref{Lem:star+}(3) and (4).
But in all that events 
$\{ {\alpha}'_{s}, \ldots, {\alpha}'_{p-1} \} 
= \{ \alpha^{(p-1)}, \ldots, \alpha^{(s+1)}, \alpha^{(s)}+\alpha^{(p)} \}$. 
So we can always express the heights in $\sub{i}$ in term of $\sub{i}'$. With the same arguments as in  \eqref{eq:tildeTmk_2} we get similar equation and so we omit the detail.

Assume that for some $a' \in \{1, \ldots, N \}$ the equivalence class $\m^{\star a'} = [(\sub{\mu}^{\star a'},\sub{\alpha}^{\star a'})]$ satisfies 
$\deg A_{\m^{\star a'}} \geqslant \deg A_{\m}$. 
Note also that in this case the roots in $\tilde{\sub{\alpha}}^{\star a'}$ is a sum amongs $({\alpha}'_{s}, \ldots, {\alpha}'_{p-1} )$, but not necessary in the sequential order. 
Then we are doing the similar calculation and so we omit the detail. Hence we conclude as in (2). 

Observe that our method above will work for case (4) as well and so we omit the detail. 
Hence we conclude for this case as in (2).
 
It remains to verify that $T_{\sub{d}, \m} (k) =0$  
 when
 $$\{ (\sub{\mu},\sub{\alpha}) \in \hat{\P}_{m}(\delta_k) \; | \;  
 \sub{\mu} \in [\![ \overline{\delta}_{1}, \delta_k]\!],  (\sub{\mu},\sub{\alpha}) \in \m \} =\varnothing.$$
First, we consider the case where $\he(\mu^{\star 1}) = (i_1, \ldots, i_{s-1}, i_{p-1}, \ldots, i_{s+1}, i_{s}+i_{p}, i_{p+1}, \ldots, i_{m}).$ The other case will work similarly and so we omit the detail.
Observe that the set in the statement is never empty if the set 
 $\{ (\sub{\mu},\sub{\alpha})^{\star 1} \in \hat{\P}_{m-1}(\delta_k) \; | \;  
 \sub{\mu}^{\star 1} \in [\![ \overline{\delta}_{1}, \delta_k]\!], (\sub{\mu},\sub{\alpha})^{\star 1} \in \m^{\star 1}  \}$ is nonempty.
For any $(\sub{\mu},\sub{\alpha})^{\star 1} \in \m^{\star 1}$ where $\he(\sub{\mu}^{\star 1}) = \sub{i}'$, we have 
 $i'_{p-1}= \overline{\delta}_l - \overline{\delta}_j = \eps_j-\eps_l$, the source of $i'_{s}$ is $\delta_i= \eps_i$ and the target of $i'_{p-2}$ is $\overline{\delta}_l= -\eps_l$.  Thus we can always reconstruct the initial path $(\sub{\mu},\sub{\alpha})$ by taking 
 $i_s = \delta_i - \delta_l= \eps_i-\eps_l$ and $i_p = \overline{\delta}_i-\overline{\delta}_j= \eps_j-\eps_i$. 
Hence, the set 
$\{ (\sub{\mu},\sub{\alpha}) \in \hat{\P}_{m}(\delta_k) \; | \;  
 \sub{\mu} \in [\![ \overline{\delta}_{1}, \delta_k]\!],  (\sub{\mu},\sub{\alpha}) \in \m \}$ is empty if $\{ (\sub{\mu},\sub{\alpha})^{\star 1} \in \hat{\P}_{m-1}(\delta_k) \; | \;  
 \sub{\mu}^{\star 1} \in [\![ \overline{\delta}_{1}, \delta_k]\!], (\sub{\mu},\sub{\alpha})^{\star 1} \in \m^{\star 1}  \} = \varnothing$. Furthermore, the set  $\{ (\sub{\mu},\sub{\alpha})^{\star a'} \in \hat{\P}_{m^{\star}}(\delta_k) \; | \;  
 \sub{\mu}^{\star a'} \in [\![ \overline{\delta}_{1}, \delta_k]\!], (\sub{\mu},\sub{\alpha})^{\star a'} \in \m^{\star a'}  \}$ is empty too. 
 Our induction hypothesis 
 says that, in this case, $$T_{\sub{d}, \m} (k)
 = {T}_{\sub{d}',{\m}^{\star a'}}(k) = 0,$$
 for any $\sub{d}' \in \Z_{\geqslant 0}^{m^{\star}}.$

$\ast$ Case $\alpha^{(s)} =- \alpha^{(p)}$. 

 Lemma~\ref{Lem:weightC} shows that
{\small \begin{align} \label{eq:key3}
   A_{\m} (X_1, \ldots, X_{m-1}) = \bar{K}_{\m}^{\#} A_{\m^\#} (X_1, \ldots, X_{p-2}, X_{p-1}+ X_p, \ldots, X_{m-1}) 
 + (P_{\alpha^{(s)}} (X_s) - c_s) \sum_{a=1}^{N}  K_{\m}^{\star a} 
 A_{\m^{\star a}} (\sub{X}^{\star a}),
 \end{align}}
where $\bar{K}_{\m}^{\#} =
\begin{cases}
{K}_{\m}^{\#} &\text{ if }  \alpha^{(p-1)} + \alpha^{(p)} \neq 0, \\
(P_{\alpha^{(p-1)}} (X_{p-1})- c_{p-1}) &\text{ if }  \alpha^{(p-1)} + \alpha^{(p)} = 0.
\end{cases}$

If $\deg A_{\m} \leqslant \deg A_{\m^{\#}}$ (respectively, $\deg A_{\m} \leqslant \deg A_{\m^{\#}}+1$) 
for $ \alpha^{(p-1)} + \alpha^{(p)} \neq 0$ (respectively, $ \alpha^{(p-1)} + \alpha^{(p)} = 0$)
then using the same strategy as in Lemma~\ref{Lem:T_m}, which is by expressing 
the heights of $\sub{i}$ in term of $\sub{i}^{\#}$, we get the statement. 
Therefore we can assume that $\deg A_{\m} > \deg A_{\m^{\#}}$ (respectively,
$\deg A_{\m} > \deg A_{\m^{\#}}+1$) if $ \alpha^{(p-1)} + \alpha^{(p)} \neq 0$
(respectively, $ \alpha^{(p-1)} + \alpha^{(p)} = 0$). 
This assumption means that for some $a \in \{1, \ldots, N \}$,
$\deg A_{\m^{\star a}}+1 \geqslant \deg A_{\m}$. 
 By Lemma~\ref{lem:types}, the possibilities for $\alpha^{(p)}$ and $\alpha^{(s)}$
 which satisfying the assumptions $\alpha^{(s)} + \alpha^{(p)} =0$
 are the following:
 \begin{itemize}
  \item[(1)] $\alpha^{(p)} =  \overline{\delta}_j-\delta_i$ 
  and $\alpha^{(s)}= \delta_j-\overline{\delta}_i$, with $j<i$,
  \item[(2)] $\alpha^{(p)} =  \overline{\delta}_i-\overline{\delta}_j$ 
  and $\alpha^{(s)} = \delta_i -{\delta}_j$, with $i<j$.
 \end{itemize}
\noindent

Note that in all those cases, for all $a \in \{1, \ldots, N \}$, 
$m-p+s \leqslant \he(\m^{\star a}) \leqslant m-2$. Possibly changing the numbering of the equivalence class $\m^{\star a}$, one can assume throughout this proof that $\m^{\star 1} = [(\sub{\mu}^{\star 1},\sub{\alpha}^{\star 1})]$ is the equivalence class of the star paths with the longest length and set $\sub{i} := \he(\m)$ and $\sub{i}' := \he(\m^{\star 1})$. 

\emergencystretch 3em
(1) We first consider the case where $\alpha^{(p)} =  \overline{\delta}_j-\delta_i$ and $\alpha^{(s)}= \delta_j-\overline{\delta}_i$, with $j<i$. 
Recall from Lemma~\ref{Lem:star0}, for  all $(\sub{\mu}, \sub{\alpha})^{\star a} \in \m^{\star a}$,
$a \in \{1, \ldots, N \},$
\begin{align*}
(\sub{\mu}, \sub{\alpha})^{\star a} = (\sub{\mu}',\sub{\alpha}')\star
(\tilde{\sub{\mu}}^{\star a},\tilde{\sub{\alpha}}^{\star a})
\star(\sub{\mu}'',\sub{\alpha}''),
\end{align*}
where 
$(\tilde{\sub{\mu}}^{\star a},\tilde{\sub{\alpha}}^{\star a})$ 
is a path of length $\leqslant p-s-1$ between 
$\delta_j$ and $\delta_i$ whose roots $\tilde{\sub{\alpha}}^{\star a}$ are 
sums among $(\alpha^{(p-1)}, \ldots, \alpha^{(s+1)})$, $\sub{\alpha}'= (\alpha^{(1)}, \ldots, \alpha^{(s-1)})$ and 
$\sub{\alpha}''= (\alpha^{(p+1)}, \ldots, \alpha^{(m)})$.
We have $\he(\sub{\mu}^{\star 1}) = (i_1, \ldots, i_{s-1}, i_{p-1}, \ldots, i_{s+1}, i_{p+1}, \ldots, i_m)$. 

Set $\sub{i}' = (i'_1, \ldots, i'_{m-2})$.  Note that $\alpha^{(p)} = \overline{\delta}_j - \delta_i$ and $\alpha^{(s)}=\delta_j - \overline{\delta}_i$, with $j < i$. Thus,
$$ j = k+ i'_1+ \cdots+ i'_{s-1} \ \text{ and } \ i = k+ i'_1+ \cdots+ i'_{p-2}.$$
Hence, for all $(\sub{\mu}, \sub{\alpha}) \in \m$ the heights in $\sub{i}$  
can be expressed in term of $\sub{i}'$ as follows:
\begin{align*}
& i_1 = i'_1, \  \ldots, \  i_{s-1} = i'_{s-1}, \
 i_{s+1} = i'_{p-2}, \  \ldots,  \  i_{p-1} = i'_{s}, \
 i_{p+1} = i'_{p-1}, \  \ldots,  \  i_{m-1} = i'_{m-3}, & \\
& i_{s} = 2r +1 - 2k-2(i'_1+ \cdots+ i'_{s-1}) - (i'_{s}+ \cdots+ i'_{p-2}),  &\\
& i_{p} = - i_s = 2k - 2r -1 +2(i'_1+ \cdots+ i'_{s-1}) + (i'_{s+1}+ \cdots+ i'_{p-2}).
\end{align*}
With the same arguments as in \eqref{eq:tildeTmk} we get,
\begin{align} \label{eq:tildeTmk_3} \nonumber
\widetilde{T}_{\sub{d},\m}(k) 
&=   \sum_{{\tiny \substack{\sub{q} \in \N^{p-1} 
    \\ |\sub{q}|=d_s}}}
    \sum_{{\tiny \substack{\sub{q'} \in \N^{p-1} 
    \\ |\sub{q'}|=d_p}}} C_{\sub{q},\sub{q}'}
    \ (2k-2r-1)^{q_{p-1}+q'_{p-1}} \sum_{{\tiny \substack{(\sub{\mu},\sub{\alpha})^{\star 1} \in \hat{\P}_{m-2}(\delta_k), 
\\ \sub{\mu}^{\star 1} \in [\![ \overline{\delta}_{1}, \delta_k ]\!], 
\\ (\sub{\mu},\sub{\alpha})^{\star 1} \in \m^{\star 1}}}} 
  (i'_1)^{d_1+q_1+q'_1} \\ 
 & \quad \times 
\cdots \times \ ({i'}_{s})^{d_{p-1}+q_{s}+q'_{s}}  \cdots ({i'}_{p-2})^{d_{s+1}+q_{p-2}+q'_{p-2}}
 ({i'}_{p-1})^{d_{p+1}} \cdots ({i'}_{m-3})^{d_{m-1}}.
\end{align}
Assume that the equivalence class $\m^{\star a'} = [(\sub{\mu}^{\star a'},\sub{\alpha}^{\star a'})]$ satisfies 
$\deg A_{\m^{\star a'}} +1\ge \deg A_{\m}$, for some $a' \in \{1, \ldots, N \}$. 
According to the proof of Lemma~\ref{Lem:star0}, there exists a partition
$(t_1, \ldots, t_{n'} ), t_i >0$, of $p-s-1$ such that
\begin{align*} 
i^{\star}_{s} &= i'_{s} + \cdots + i'_{s+t_1}, \\ 
\vdots \\
i^{\star}_{s+n'-2} &= i'_{s+t_1+\cdots+t_{(n'-2)}+1} + \cdots + i'_{s+t_1+\ldots+t_{n'-1}}, \\
i^{\star}_{s+n'-1} &= i'_{s+t_1+\ldots+t_{n'-1}+1} + \cdots + i'_{p-2}.
\end{align*}
Hence $\he(\sub{\mu}^{\star a'})= (i'_1, \ldots, i'_{s-1}, i^{\star}_{s}, i^{\star}_{s+1}, \ldots, i^{\star}_{s+n'-1}, i'_{p-1}, \ldots, i'_{m-3})$ and the length of $\m^{\star a'}$, denoted by $m^{\star}$, is equal to $m-p+s+n'-1$. Here $i^{\star}_{s+j}$ denote the height of the root $(\tilde{\sub{\alpha}}^{\star a'} )^{(s+j)}$.

According to Lemma~\ref{lem:aij} and \eqref{eq:tildeTmk_3} and from the arguments of above cases it follows that
\begin{align*}
&\widetilde{T}_{\sub{d},\m}(k) 
=   \sum_{{\tiny \substack{\sub{q} \in \N^{p-1} 
    \\ |\sub{q}|=d_s}}}
    \sum_{{\tiny \substack{\sub{q'} \in \N^{p-1} 
    \\ |\sub{q'}|=d_p}}} C_{\sub{q},\sub{q}'}
    \ (2k-2r-1)^{q_{p-1}+q'_{p-1}} 
     \sum_{{\tiny \substack{(\sub{\mu},\sub{\alpha})^{\star a'} \in \hat{\P}_{m^{\star}}(\delta_k), 
\\ \sub{\mu}^{\star a'} \in [\![ \overline{\delta}_{1}, \delta_k ]\!], 
\\ (\sub{\mu},\sub{\alpha})^{\star a'} \in \m^{\star a'}}}} 
  (i'_1)^{d_1+q_1+q'_1} \cdots  \\ 
 & \quad \times (i'_{s-1})^{d_{s-1}+q_{s-1}+q'_{s-1}} 
\ \sum_{\tiny {\substack{\sub{j}=(j_1, \ldots, j_{n'})\\ 
 |\sub{j}| \leqslant {d}'_{\sub{q}, \sub{q}' }+p-s-1-n'}}} C_{\sub{j}} (i^{\star}_{s})^{j_{1}}\cdots (i^{\star}_{s+n'-1})^{j_{n'}}
 ({i'}_{p-1})^{d_{p+1}} \cdots ({i'}_{m-3})^{d_{m-1}},
\end{align*}
where ${d}'_{\sub{q}, \sub{q}' }= d_{s+1}+q_{p-2}+q'_{p-2}+\cdots+d_{p-1}+q_s+q'_{s}$.
Set 
$\sub{d}_{\sub{q}, \sub{q}'}=(d_1+q_1+q'_1, \ldots, d_{s-1}+q_{s-1}+q'_{s-1}, j_{1},\ldots, j_{n'}, d_{p+1}, \ldots, d_{m-1}),$
with $|\sub{d}_{\sub{q}, \sub{q}'}|= d+(p-s-1)-n'-q_{p-1} -q'_{p-1}$, and
\begin{align*}
&\widetilde{T}_{\sub{d}_{\sub{q}, \sub{q}'},{\m}^{\star a'}}(k)=
 \sum_{{\tiny \substack{(\sub{\mu},\sub{\alpha})^{\star a'} \in \hat{\P}_{m^{\star}}(\delta_k), 
\\ \sub{\mu}^{\star a'} \in [\![ \overline{\delta}_{1}, \delta_k ]\!], 
\\ (\sub{\mu},\sub{\alpha})^{\star a'} \in \m^{\star a'}}}} 
 (i'_1)^{d_1+q_1+q'_1} 
\cdots (i^{\star}_{s})^{j_{1}}\cdots (i^{\star}_{s+n'-1})^{j_{n'}}
 ({i'}_{p-1})^{d_{p+1}} \cdots ({i'}_{m-3})^{d_{m-1}}.
\end{align*}
Then 
\begin{align} \label{eq:TmK-C-3}
\widetilde{T}_{\sub{d},\m}(k) 
&= 
\sum_{{\tiny \substack{\sub{q} \in \N^{p-1} 
    \\ |\sub{q}|=d_s}}}
    \sum_{{\tiny \substack{\sub{q'} \in \N^{p-1} 
    \\ |\sub{q'}|=d_p}}} \sum_{\tiny {\substack{\sub{j}=(j_1, \ldots, j_{n'})\\ 
 |\sub{j}| \leqslant {d}'_{\sub{q}, \sub{q}' }+p-s-1-n'}}} C_{\sub{j}} C_{\sub{q},\sub{q}'} 
  (2k-2r-1)^{q_{p-1}+q'_{p-1}} 
 \widetilde{T}_{\sub{d}_{\sub{q}, \sub{q}'},{\m}^{\star a'}}(k).
\end{align}
By the induction hypothesis applied to $m^{\star}$ and ${\m}^{\star a'}$, 
we have
$T_{\sub{d}, {\m}}(k) = \widetilde{T}_{\sub{d},{\m}}(k)$ and 
$T_{\sub{d}, {\m}}$ is a polynomial of degree 
$< d+m- \deg A_{{\m}}$ for all $k$ 
such that $\{ (\sub{\mu},\sub{\alpha}) \in \hat{\P}_{m}(\delta_k) \; | \; 
\sub{\mu} \in [\![ \overline{\delta}_{1}, \delta_k ]\!] ,\, (\sub{\mu},\sub{\alpha}) \in \m \}$  
is nonempty.

If 
 $\{ (\sub{\mu},\sub{\alpha}) \in \hat{\P}_{m}(\delta_k) \; | \;  
 \sub{\mu} \in [\![ \overline{\delta}_{1}, \delta_k]\!],  (\sub{\mu},\sub{\alpha}) \in \m \} =\varnothing$ then 
 by the same kind of reasoning as before, we get 
 ${T}_{\sub{d}, \m} (k) =0$.
This proves the lemma for the case (1). 
Observe that our strategy above works for the case (2) as well. 
So we omit the detail. \qed
\end{proof}

\begin{lemma}       \label{Lem:wtloopC} 
Let $m \in \Z_{>0} $ 
and $\m \in {\E}_m$ with $n$ zeroes  
in positions $j_1,\ldots,j_n$ ($n\le m$). 
There are some classes $\tilde{\m}_1,\ldots,\tilde{\m}_N$ without zero 
of length $\ell_h:=\ell(\tilde{\m}_h)\leqslant m-n$ and a polynomial 
$B_{\m}$ of degree $\leqslant n$
such that 
$$\sum_{\sub{\beta} \in \Pi_{\sub{\mu}^{(\sub{j})}} } 
\wt\left((\tilde{\sub{\mu}},\tilde{\sub{\alpha}})_{\sub{j};\sub{\beta}}  \right) 
= B_{\m}(k) \sum_{h=1}^N K_{\tilde{\m}_h} A_{\tilde{\m}_h}(\tilde{\i}_1, \ldots, \tilde{\i}_{\ell_h-1})
,$$ 
where $\he(\tilde{\sub{\mu}_h})=(\tilde{i}_1, \ldots, \tilde{i}_{\ell_h})$. 
Moreover, we have 
$$B_{\m}(X) = \sum_{j=0}^n C_{\m}^{(n-j)}
(\tilde{\i}_1, \ldots, \tilde{\i}_{m-n-1})
X^j,$$ 
where 
$C_{\m}^{(0)}=(-1)^n$ 
and $C_{\m}^{(l)}\in \C[X_1,\ldots,X_{m-n-1}]$ has total degree 
$\leqslant l$ for $l=1,\ldots,n$.  
In particular, if $n=0$, we have $B_{\m}(X)=1$. 
\end{lemma}
\begin{proof}
We prove the statement by induction on the number of the loops $n$.
Set $\sub{j} = (j_1, \ldots, j_n)$
and let $\sub{\beta}:=(\beta^{(1)},\ldots,\beta^{(n)})$ be in 
$\Pi_{\sub{\mu}^{(\sub{j})}}.$

First of all, observe that 
if $n=0$, then the result is known by Lemma~\ref{Lem:weightC} and Remark~\ref{rem:wtC}.
\\If $n=m$ then either $(\sub{\mu}, \sub{\alpha})= \left( ({\delta_k}),  (\Pi_{\delta_k})^m \right)$ or 
$(\sub{\mu}, \sub{\alpha}) = \left( ({\overline{\delta}_k}), (\Pi_{\overline{\delta}_k})^m \right)$. 
By Lemma~\ref{lem:symC}, there is a polynomial $B_{\m}$ of degree $m$ such that $$\wt(\sub{\mu}, \sub{\alpha}) = B_{\m}(k),$$
and so the lemma is true for $n=m$. 
Hence it remains to prove for $n = 1, \ldots, m-1$.
\\For $n=1$,  we have  $\sub{j} = (j_1)$ for some $j_1 = 1, \ldots, m.$

$\ast$ Assume $j_1= 1$ then $\mu^{(j_1)} = \delta_k$.
 \\For some $k \in  \{2, \ldots, r \}$, we have 
 $ \Pi_{\sub{\mu}^{(\sub{j})}} = \{\beta_{k-1}, \beta_k\}. $
 Hence by Lemma~\ref{lem:loops}, Lemma~\ref{lem:rho} and Remark~\ref{rem:wtC}, 
 \begin{align*}
\sum_{\sub{\beta} \in \Pi_{\sub{\mu}^{(\sub{j})}} } 
\wt\left((\tilde{\sub{\mu}},\tilde{\sub{\alpha}})_{\sub{j};\sub{\beta}}\right)
  = \wt \left( (\tilde{\sub{\mu}},\tilde{\sub{\alpha}})_{(1);\{{\beta}_{k-1} \}} \right) 
 + \wt \left( (\tilde{\sub{\mu}},\tilde{\sub{\alpha}})_{(1);\{{\beta}_k \}}\right)
 = (r+1-k) \ \sum_{h=1}^N K_{\tilde{\m}_h} A_{\tilde{\m}_h}(\tilde{\i}_1, \ldots, \tilde{\i}_{\ell_h-1}).
 \end{align*}
 \\For $k = 1$, we have 
 $ \Pi_{\sub{\mu}^{(\sub{j})}} = \{\beta_{1}\}. $
 Hence by Lemma~\ref{lem:loops}, Lemma~\ref{lem:rho} and Remark~\ref{rem:wtC},
 \begin{align*}
 \sum_{\sub{\beta} \in \Pi_{\sub{\mu}^{(\sub{j})}} } 
  &= \wt \left((\tilde{\sub{\mu}},\tilde{\sub{\alpha}})_{(1);\{{\beta}_{1} \}} \right) 
  =(\langle \rho, \s{\varpi}_{1}  \rangle) \ \wt (\tilde{\sub{\mu}},\tilde{\sub{\alpha}})
 = (r+1-k) \ \sum_{h=1}^N K_{\tilde{\m}_h} A_{\tilde{\m}_h}(\tilde{\i}_1, \ldots, \tilde{\i}_{\ell_h-1}).
 \end{align*}
 Hence by setting $B_{\m}(X) : = r+1-X,$ we get the statement.

$\ast$ Assume $\mu^{(j_1)} = \delta_s$ for some $s \in  \{2, \ldots, r \}$ and $s \neq k$.
Note that $s = k + \sum_{t=1}^{j_1-1}i_t$ and $ \Pi_{\sub{\mu}^{(\sub{j})}} = \{\beta_{s-1}, \beta_s\}. $
By Lemma~\ref{lem:loops}, Lemma~\ref{lem:rho} and Remark~\ref{rem:wtC},
 \begin{align*}
&\sum_{\sub{\beta} \in \Pi_{\sub{\mu}^{(\sub{j})}} } 
\wt\left((\tilde{\sub{\mu}},\tilde{\sub{\alpha}})_{\sub{j};\sub{\beta}}\right)
 = \left(\langle \rho, \s{\varpi}_{s}-\s{\varpi}_{s-1}  \rangle +1 \right)
      \ \wt (\tilde{\sub{\mu}},\tilde{\sub{\alpha}})
 = \left(r-k-\sum_{t=1}^{j_1-1}i_t +2\right) \ \sum_{h=1}^N K_{\tilde{\m}_h} A_{\tilde{\m}_h}(\tilde{\i}_1, \ldots, \tilde{\i}_{\ell_h-1}).
 \end{align*}
 Note that $ \sum_{t=1}^{j_1-1}i_t = \sum_{t=1}^{\tilde{j}}\tilde{\i}_t,$
 where $\tilde{j} = j_1 -1$.
 Set $$B_{\m}(X) := r-X-\sum_{t=1}^{\tilde{j}}\tilde{\i}_t +2. $$
 Then we prove the statement in this case.

$\ast$ Assume $\mu^{(j_1)} = \overline{\delta}_s$. 
\\For $s \in  \{2, \ldots, r-1 \}$, 
note that $s = 2r - k -\sum_{t=1}^{j_1-1}i_t+1$ and $\Pi_{\sub{\mu}^{(\sub{j})}} = \{\beta_{s-1}, \beta_s\}.$
By Lemma~\ref{lem:loops}, Lemma~\ref{lem:rho} and Remark \ref{rem:wtC}, we have
 \begin{align*}
&\sum_{\sub{\beta} \in \Pi_{\sub{\mu}^{(\sub{j})}} } 
\wt\left((\tilde{\sub{\mu}},\tilde{\sub{\alpha}})_{\sub{j};\sub{\beta}}\right)
 = \left(\langle \rho, \s{\varpi}_{s-1}-\s{\varpi}_{s}  \rangle +1 \right)
      \ \wt (\tilde{\sub{\mu}},\tilde{\sub{\alpha}})
  = \left(r - k -\sum_{t=1}^{j_1-1}i_t+1\right) \ \sum_{h=1}^N K_{\tilde{\m}_h} A_{\tilde{\m}_h}(\tilde{\i}_1, \ldots, \tilde{\i}_{\ell_h-1}).
 \end{align*}
\\For $ s=1$, $\Pi_{\sub{\mu}^{(\sub{j})}} = \{\beta_{1}\}$, thus
\begin{align*}
 \sum_{\sub{\beta} \in \Pi_{\sub{\mu}^{(\sub{j})}} } 
\wt\left((\tilde{\sub{\mu}},\tilde{\sub{\alpha}})_{\sub{j};\sub{\beta}}\right)
 =(\langle -\rho, \s{\varpi}_{1}  \rangle) \ \wt (\tilde{\sub{\mu}},\tilde{\sub{\alpha}})
= \left(r - k -\sum_{t=1}^{j_1-1}i_t+1\right) \ \sum_{h=1}^N K_{\tilde{\m}_h} A_{\tilde{\m}_h}(\tilde{\i}_1, \ldots, \tilde{\i}_{\ell_h-1}).
 \end{align*}
 Note that $\sum_{t=1}^{j_1-1}i_t = \sum_{t=1}^{\tilde{j}}\tilde{\i}_t,$
 where $\tilde{j} = j_1 -1$.
 By setting $$B_{\m}(X) := r-X-\sum_{t=1}^{\tilde{j}}\tilde{\i}_t +1, $$
 we prove the statement in this case.

$\ast$ Assume $\mu^{(j_1)} = \overline{\delta}_r$. 
\\Note that $\Pi_{\sub{\mu}^{(\sub{j})}} = \{\beta_{r-1}, \beta_r\}. $
By Lemma~\ref{lem:loops}, Lemma~\ref{lem:rho} and Remark \ref{rem:wtC}, we have
 \begin{align*}
&\sum_{\sub{\beta} \in \Pi_{\sub{\mu}^{(\sub{j})}} } 
\wt\left((\tilde{\sub{\mu}},\tilde{\sub{\alpha}})_{\sub{j};\sub{\beta}}\right)
 = \left(\langle \rho, \s{\varpi}_{r-1}-\s{\varpi}_{r}  \rangle +1 \right)
      \ \wt (\tilde{\sub{\mu}},\tilde{\sub{\alpha}})
= 0.
 \end{align*}
 Since 0 is a polynomial of degree $ \leqslant n$ then we get the statement.
By the above observation, the lemma is true for $n=1$. 
Let $n \geqslant 2$ and assume the lemma is true for any
$n' \in \{ 1, \ldots, n-1 \}$.
Set $\sub{j}:= (j_1,\ldots,j_n)$ the position of loops and let $\sub{\beta}:=(\beta^{(1)},\ldots,\beta^{(n)})$ be in 
$\Pi_{\sub{\mu}^{(\sub{j})}}$.
We are doing the similar calculation as Lemma~\ref{lem:loops}, thus
\begin{align*}
 &\sum_{\sub{\beta} \in \Pi_{\sub{\mu}^{(\sub{j})}} } 
\wt\left((\tilde{\sub{\mu}},\tilde{\sub{\alpha}})_{\sub{j};\sub{\beta}}\right) = 
\left(\sum_{\alpha^{(j_1)} \in \Pi_{\mu^{(j_1)}}} \langle \mu^{j_1}, \alpha^{(j_1)} \rangle 
\left( \langle \rho, \s{\varpi}_{\alpha^{(j_1)}} \rangle - c_{\alpha^{(j_1)}} \right)  \right)
\sum_{{\tiny \sub{\beta}' \in \Pi_{\sub{\mu}^{(\sub{j}')}}} } 
\wt\left((\tilde{\sub{\mu}},\tilde{\sub{\alpha}})_{\sub{j}';\sub{\beta}'} 
\right),
\end{align*}
where $\sub{j}' = (j_2,\ldots,j_n)$, $c_{\alpha^{(j_1)}} := \sum_{i=1}^{j_1-1} 
\langle \alpha^{(i)}, \s{\varpi}_{\alpha^{(j_1)}} \rangle$ and $(\tilde{\sub{\mu}},\tilde{\sub{\alpha}})_{\sub{j}';\sub{\beta}'}$
is a weighted path with $n-1$ loops. Set $\m' := [(\tilde{\sub{\mu}},\tilde{\sub{\alpha}})_{\sub{j}';\sub{\beta}'}]$.
So, the induction hypothesis applied to $\m'$ gives,
\begin{align*}
\sum_{\sub{\beta} \in \Pi_{\sub{\mu}^{(\sub{j})}} } 
\wt\left((\tilde{\sub{\mu}},\tilde{\sub{\alpha}})_{\sub{j};\sub{\beta}}\right)
&= \left(\sum_{\alpha^{(j_1)} \in \Pi_{\mu^{(j_1)}}} \langle \mu^{j_1}, \alpha^{(j_1)} \rangle 
\left( \langle \rho, \s{\varpi}_{\alpha^{(j_1)}} \rangle - c_{\alpha^{(j_1)}} \right)  \right)
B_{\m'}(k) \sum_{h=1}^N K_{\tilde{\m}_h} A_{\tilde{\m}_h}(\tilde{\i}_1, \ldots, \tilde{\i}_{\ell_h-1})
\end{align*}
since $\tilde{\m}'_h = \tilde{\m}_h, h = 1, \ldots, N$.
Set $$ \tilde{B}_{\m}(k) := \left(\sum_{\alpha^{(j_1)} \in \Pi_{\mu^{(j_1)}}} \langle \mu^{j_1}, \alpha^{(j_1)} \rangle 
\left( \langle \rho, \s{\varpi}_{\alpha^{(j_1)}} \rangle - c_{\alpha^{(j_1)}} \right)  \right)B_{\m'}(k).$$
With the same arguments as for the case $n=1$ and the induction hypothesis applied to $\m'$, we see that there exists a polynomial
$B_{\m}$ of degree $\leqslant n$ with leading term 
is $(-1)^n X^n$,   
and the coefficient of $B_{\m}(X)$ in $X^j$, $j\le n$, 
is a polynomial  
in the variables $\tilde{\i}_1, \ldots, \tilde{\i}_{m-n-1}$, 
of total degree $\leqslant n-j$ such that $\tilde{B}_{\m}(k) = {B}_{\m}(k)$. \qed
\end{proof}

\begin{corollary}  \label{corollary:T_mC} 
Let $m\in \Z_{> 0}$ and $n\in\{0,\ldots,m\}$. 
Let $\m \in {\E}_m$ with $n \leqslant m$ zeroes in positions 
$j_1,\ldots,j_n$,    
and $\tilde{\m}$ as in Lemma~\ref{Lem:wtloopC}. 
Then for some polynomial $T_{\m} \in \C[X]$ of degree 
at most $m$, 
for all $k\in\{1,\ldots,r\}$, 
$$
\sum_{(\tilde{\sub{\mu}},\tilde{\sub{\alpha}})  \in \tilde{\m}}  
\sum_{\sub{\beta} \in \Pi_{\sub{\mu}^{\sub{j}}}}  
\sum_{(\tilde{\sub{\mu}},\tilde{\sub{\alpha}}) \in \hat{\P}_{m-n}(\delta_k),   
\atop \tilde{\sub{\mu}} \in [\![\bar{\delta}_{1},{\delta}_{k} ]\!], 
\,(\tilde{\sub{\mu}},\tilde{\sub{\alpha}}) \in \tilde{\m}} 
\wt \left( (\tilde{\sub{\mu}},\tilde{\sub{\alpha}})_{\sub{j};\sub{\beta}} \right)
= T_{\m}(k).
$$
\end{corollary}

If $n=0$ has no zero, then $T_{\m}$ is the polynomial provided 
by Lemma~\ref{Lem:T_mC}. If $n=m$, then $\he(\m)=(\sub{0})$ 
and $T_{\m}=T_{m}$ 
is the polynomial provided by Lemma~\ref{lem:symC}. 
So, in these two cases, the statement is known. 
\begin{proof}
Let $k \in \{1,\ldots,r\}$. 
By Lemma~\ref{Lem:weightC}, Remark~\ref{rem:wtC},  Lemma~\ref{Lem:T_mC} and Lemma~\ref{Lem:wtloopC}, 
we have, 
\begin{align*}
&\sum_{(\tilde{\sub{\mu}},\tilde{\sub{\alpha}})  \in \tilde{\m}}  
\sum_{\sub{\beta} \in \Pi_{\sub{\mu}^{\sub{j}}}}  
\sum_{(\tilde{\sub{\mu}},\tilde{\sub{\alpha}}) \in \hat{\P}_{m-n}(\delta_k),   
\atop \tilde{\sub{\mu}} \in [\![\bar{\delta}_{1},{\delta}_{k} ]\!], 
\,(\tilde{\sub{\mu}},\tilde{\sub{\alpha}}) \in \tilde{\m}} 
\wt \left( (\tilde{\sub{\mu}},\tilde{\sub{\alpha}})_{\sub{j};\sub{\beta}} \right) \\
&\qquad 
= 
\sum_{(\tilde{\sub{\mu}},\tilde{\sub{\alpha}}) \in \hat{\P}_{m-n}(\delta_k),   
\atop \tilde{\sub{\mu}} \in [\![\bar{\delta}_{1},{\delta}_{k} ]\!], 
\,(\tilde{\sub{\mu}},\tilde{\sub{\alpha}}) \in \tilde{\m}}
\sum_{j=0}^n 
\sum_{\sub{d}_j=(d_1, \ldots, d_{m-n-1}),
\atop {d_1 +\cdots + d_{m-n-1} \leqslant  n-j}}
C_{\sub{d},j}
\tilde{\i}_1^{d_1} \cdots \tilde{\i}_{m-n-1}^{d_{m-n-1}} 
k^j A_{\tilde{\m}}(\tilde{\i}_1, \ldots, \tilde{\i}_{m-n-1}). 
\end{align*}
Set 
$$\tilde{T}_{\sub{d}_j, \tilde{\m}}= 
\sum_{(\tilde{\sub{\mu}},\tilde{\sub{\alpha}}) \in \hat{\P}_{m-n}(\delta_k),   
\atop \tilde{\sub{\mu}} \in [\![\bar{\delta}_{1},{\delta}_{k} ]\!], 
\,(\tilde{\sub{\mu}},\tilde{\sub{\alpha}}) \in \tilde{\m}}
\tilde{\i}_1^{d_1} \cdots \tilde{\i}_{m-n-1}^{d_{m-n-1}} 
A_{\tilde{\m}}(\tilde{\i}_1, \ldots, \tilde{\i}_{m-n-1}).$$
Then by Lemma~\ref{Lem:T_mC}, 
there are some polynomials 
${T}_{\sub{d}_j, \tilde{\m}}$ of degree at most $(n-j)+(m-n)=m-j$,
such that 
$$ \sum_{(\tilde{\sub{\mu}},\tilde{\sub{\alpha}})  \in \tilde{\m}}  
\sum_{\sub{\beta} \in \Pi_{\sub{\mu}^{\sub{j}}}}  
\sum_{(\tilde{\sub{\mu}},\tilde{\sub{\alpha}}) \in \hat{\P}_{m-n}(\delta_k),   
\atop \tilde{\sub{\mu}} \in [\![\bar{\delta}_{1},{\delta}_{k} ]\!], 
\,(\tilde{\sub{\mu}},\tilde{\sub{\alpha}}) \in \tilde{\m}} 
\wt \left( (\tilde{\sub{\mu}},\tilde{\sub{\alpha}})_{\sub{j};\sub{\beta}} \right)
= \sum_{j=0}^n 
\sum_{\sub{d}_j=(d_1, \ldots, d_{m-n-1}),
\atop {d_1 +\cdots + d_{m-n-1} \leqslant  n-j}}
 C_{\sub{d},j} k^j {T}_{\sub{d}_j, \tilde{\m}} (k).$$ 
Moreover, if $j <n$, then ${T}_{\sub{d}_j, \tilde{\m}}$ 
has degree $< m-j$, and 
$$T_{\m} (X):= \sum_{j=0}^n \sum_{\sub{d}_j=(d_1, \ldots, d_{m-n-1}),
\atop {d_1 +\cdots + d_{m-n-1} \leqslant  n-j}}
 C_{\sub{d},j} X^j {T}_{\sub{d}_j, \tilde{\m}} (X),$$
is a polynomial of 
degree at most $m-j+j = m$. \qed
\end{proof}

As a consequence of Corollary \ref{corollary:T_mC}, we obtain the following 
crucial result. 

\begin{proposition} \label{Pro:sumC}
Let $m\in \Z_{> 0}$. 
There are polynomials $\hat{T}_{1,m}$ and $\hat{T}_{2,m}$ in $\C[X]$ of degree at most $m$ 
such that for all $k\in\{1,\ldots,r\}$, 
$$\sum_{(\sub{\mu},\sub{\alpha}) \in \hat{\P}_m(\delta_{k})  
\atop \sub{\mu}\in  [\![ \overline{\delta_{1}},\delta_{k} ]\!]} \, 
\wt(\sub{\mu},\sub{\alpha}) =\hat{T}_{1,m}(k),\quad \text{and}\quad 
\sum_{(\sub{\mu},\sub{\alpha}) \in \hat{\P}_m(\overline{\delta}_{k})  
\atop \sub{\mu}\in  [\![ \overline{\delta}_{1},\overline{\delta}_{k} ]\!]} \, 
\wt(\sub{\mu},\sub{\alpha}) =\hat{T}_{2,m}(k).$$
\end{proposition}

\begin{proof} 
According to Theorem \ref{corollary:cut-type-C}, the paths starting 
from $\overline{\delta}_{k}$ have weights entirely contained in 
$[\![ \overline{\delta}_{1},\overline{\delta}_{k} ]\!]$. 
So the sum 
$$\sum_{(\sub{\mu},\sub{\alpha}) \in \hat{\P}_m(\overline{\delta}_{k})  
\atop \sub{\mu}\in  [\![ \overline{\delta}_{1},\overline{\delta}_{k} ]\!]} \, 
\wt(\sub{\mu},\sub{\alpha})$$ can be computed 
exactly as in the $\sl_{r+1}$ case, and the result are known by Lemma~\ref{Lem:T}.
Hence there is a polynomial $\hat{T}_{2,m}$ of degree at most $m$ 
such that for all $k\in\{1,\ldots,r\}$, 
$$\sum_{(\sub{\mu},\sub{\alpha}) \in \hat{\P}_m(\bar{\delta}_k)  
\atop \sub{\mu}\in  [\![ \overline{\delta_{1}},\bar{\delta}_{k} ]\!]} \, 
\wt(\sub{\mu},\sub{\alpha}) =\hat{T}_{2,m}(k).$$
So it remains to consider the paths starting from $\delta_k$ 
and contained in $[\![\bar{\delta}_{1},{\delta}_{k} ]\!]$. 

By Corollary \ref{corollary:T_mC}, 
we have for all $k\in\{1,\ldots,r\}$, 
$$\sum_{(\sub{\mu},\sub{\alpha}) \in \hat{\P}_m(\delta_{k})  
\atop \sub{\mu}\in  [\![\bar{\delta}_{1},{\delta}_{k} ]\!]} \, 
\wt(\sub{\mu},\sub{\alpha})
=\sum_{\m \in {\E}_m} 
\sum_{(\sub{\mu},\sub{\alpha}) \in \hat{\P}_m(\delta_{k}),  
\atop \sub{\mu}\in  [\![\bar{\delta}_{1},{\delta}_{k} ]\!], \, (\sub{\mu},\sub{\alpha}) \in \m} \, 
\wt(\sub{\mu},\sub{\alpha}) 
=\sum_{\m \in {\E}_m}  T_{\m}(k).
$$
Set 
$$\hat{T}_{1,m} := \sum_{\m \in {\E}_m} T_{\m} \in \C[X].$$
Still by Corollary \ref{corollary:T_mC}, 
$T_{\m}$ has degree at most $m$. 
Therefore $\hat{T}_{1,m}$
has degree at most 
$m$ and satisfies the condition of the lemma. \qed
\end{proof}

We are now in a position to prove Theorem \ref{theorem:main2}

\begin{proof}[Proof of Theorem \ref{theorem:main2}] 
By Lemma~\ref{Lem:formulas}, we have
$$ \ev_\rho( \widehat{\d p}_{m,k}) =  \sum_{\mu \in P(\delta)_k} \, \sum_{(\sub{\mu},\sub{\alpha}) \in \hat{\P}_m(\mu)}  
\wt(\sub{\mu},\sub{\alpha}) \langle \mu, \c{\beta}_k \rangle.$$
Recall that 
$P(\delta)_k=\{\delta_{k},\delta_{k+1}, \overline{\delta}_k, \overline{\delta}_{k+1}\}, 
\, k=1,\ldots,r-1$ and $P(\delta)_r=\{\delta_{r}, \overline{\delta}_{r}\}.$
Hence, for $k=1,\ldots,r-1$
\begin{align*}
 \ev_\rho( \widehat{\d p}_{m,k}) &=  \sum_{{\tiny(\sub{\mu},\sub{\alpha}) \in \hat{\P}_m(\delta_k)}}  \wt(\sub{\mu},\sub{\alpha}) - 
\sum_{{\tiny(\sub{\mu},\sub{\alpha}) \in \hat{\P}_m(\delta_{k+1})}}  \wt(\sub{\mu},\sub{\alpha}) 
+ \sum_{{\tiny(\sub{\mu},\sub{\alpha}) \in \hat{\P}_m(\bar{\delta}_{k+1})}}  \wt(\sub{\mu},\sub{\alpha}) - 
\sum_{{\tiny(\sub{\mu},\sub{\alpha}) \in \hat{\P}_m(\bar{\delta}_{k})}}  \wt(\sub{\mu},\sub{\alpha}),
\end{align*}
and
\begin{align*}
 \ev_\rho( \widehat{\d p}_{m,r}) &=  \sum_{{\tiny(\sub{\mu},\sub{\alpha}) \in \hat{\P}_m(\delta_r)}}  \wt(\sub{\mu},\sub{\alpha}) -  
                              \sum_{{\tiny(\sub{\mu},\sub{\alpha}) \in \hat{\P}_m(\bar{\delta}_{r})}}  \wt(\sub{\mu},\sub{\alpha}).
\end{align*}
Let $\hat{T}_{1,m}$ and $\hat{T}_{2,m}$ be as in Proposition~\ref{Pro:sumC} 
and set 
\begin{align} 
\label{eq:hatQi-sp}
\hat{Q}_{m}:=\hat{T}_{1,m} (X) - \hat{T}_{1,m}(X+1)+\hat{T}_{2,m} (X+1) - \hat{T}_{2,m}(X)
\end{align}
(Note that $ \hat{T}_{1,m}(r+1) = \hat{T}_{2,m} (r+1)=0$.)
Then $\hat{Q}_{m}$ is a polynomial of degree at most $m-1$, and
we have 
\begin{align*}
 \ev_\rho (\widehat{{\rm d} p}_{m}  ) 
 &=  \ev_\rho \left(\frac{1}{m!}  \sum_{k=1}^{r}  
\widehat{\d p}_{m,k}\otimes \s{\varpi_k} \right) 
= \frac{1}{m!}  \sum_{k=1}^{r}  \ev_\rho \left(
\widehat{\d p}_{m,k} \right)\s{\varpi_k} 
 = \frac{1}{m!}  \sum_{k=1}^{r}  \hat{Q}_m\s{\varpi_k}.
\end{align*}
Moreover, $\hat{Q}_1 = 2$. \qed
\end{proof}

\begin{acknowledgements}
Both authors are greatly indebted to Abderrazak Bouaziz 
for his substantial support in this work, as well as  
to Cédric Lecouvey, Maxim Nazarov and Pierre Torasso for the valuable
remarks and comments.
The text is taken from the first named author's PhD thesis in the LMA of University of Poitiers.
The second named author wishes to express her thanks to Anton Alekseev  
for making her discovering this intriguing topic, and to Jean-Yves Charbonnel 
for useful discussions. 
\end{acknowledgements}

\bibliographystyle{abbrv}
\bibliography{Nilam.bib}

\begin{thebibliography}{10}

\bibitem{AlekMor12}
A.~Alekseev and A.~Moreau.
\newblock {On the Kostant conjecture for Clifford algebras}.
\newblock {\em Comptes Rendus Mathematique de l'Académie des Sciences Paris},
  350(1):13 -- 18, 2012.

\bibitem{Baz03}
Y.~Bazlov.
\newblock {\em {Exterior powers of the adjoint representation of a semisimple
  Lie algebra}}.
\newblock PhD thesis, The Weizmann Institute of Science, 2003.

\bibitem{Baz09}
Y.~{Bazlov}.
\newblock {The Harish-Chandra isomorphism for Clifford algebras}.
\newblock {\em arXiv e-prints}, page arXiv:0812.2059v2, April 2009.

\bibitem{Dixmier}
J.~Dixmier.
\newblock {\em Enveloping algebras}.
\newblock North-Holland Pub. Co., Amsterdam, 1977.

\bibitem{Jos11b}
A.~Joseph.
\newblock {Analogue Zhelobenko invariants, Bernstein--Gelfand--Gelfand
  operators and the Kostant Clifford algebra conjecture}.
\newblock {\em Transformation Groups}, 17(3):823--833, Sep 2012.

\bibitem{Jos11a}
A.~Joseph.
\newblock {Zhelobenko invariants, Bernstein-Gelfand-Gelfand operators and the
  analogue Kostant Clifford algebra conjecture}.
\newblock {\em Transformation Groups}, 17(3):781--821, Sep 2012.

\bibitem{Kas95}
M.~Kashiwara.
\newblock On crystal bases.
\newblock In {\em Proceedings of the 1994 Annual Seminar of the Canadian
  Mathematical Society}, volume~16, pages 155--197. Canadian Mathematical
  Society, 1995.

\bibitem{KNV11}
S.~Khoroshkin, M.~Nazarov, and E.~Vinberg.
\newblock {A generalized Harish-Chandra isomorphism}.
\newblock {\em Advances in Mathematics}, 226(2):1168 -- 1180, 2011.

\bibitem{Kos63}
B.~Kostant.
\newblock Lie group representations on polynomial rings.
\newblock {\em American Journal of Mathematics}, 85(3):327--404, 1963.

\bibitem{LLP10}
C.~Lecouvey, E.~Lesigne, and M.~Peigné.
\newblock {Random walks in Weyl chambers and crystals}.
\newblock {\em {Proceedings of the London Mathematical Society}}, 104(2):323 --
  358, 2012.

\bibitem{Mehta88}
M.~Mehta.
\newblock Basic sets of invariant polynomials for finite reflection groups.
\newblock {\em Communications in Algebra}, 16(5):1083--1098, 1988.

\bibitem{Rohr08}
R.~Rohr.
\newblock {Principal basis in Cartan subalgebra}.
\newblock {\em J. Lie Theory}, 20(5):673--687, 2010.

\bibitem{TauvelYu}
P.~Tauvel and R.~W.~T. Yu.
\newblock {\em Lie algebras and algebraic groups}.
\newblock Springer, Berlin, Heidelberg, New York, 2005.

\end{thebibliography}

\appendix

\section{Roots and weights in type $A$} 
\label{sec:rootTypeA}
Assume that $\g=\sl_{r+1}$. 
We may realize $\g$ as the set of $(r+1)$-size square traceless matrices. 
We choose as a Cartan subalgebra $\h$ the set of all diagonal matrices 
of $\g$.
Define a linear map $\eps_i \in \h^*$ 
by $\eps_i(h)= h_i$ if $h = {\rm diag}(h_1,\ldots,h_{r+1})$. 
Then the root system of $(\g,\h)$ 
is $\Delta=\{ \eps_i - \eps_j \mid i \not=j \}$.
We make the standard choice of 
$\Delta_+=\{ \eps_i - \eps_j \mid i <j \}$  
for the set of positive roots. The set of simple roots is 
$\Pi=\{\beta_1,\ldots,
\beta_{r}\},$ 
where  $\beta_1=\eps_{1}-\eps_{2} ,\ldots,
\beta_{r}=\eps_{r}-\eps_{r+1}$.
The fundamental weights 
are 
$$\varpi_i=(\eps_1+\cdots +\eps_i) 
- \frac{i}{r+1}(\eps_1+\cdots+\eps_{r+1}), 
\qquad i=1,\ldots,r.$$ 
The $(\eps_i - \eps_j)$-root space of $\g$ 
is spanned by the root vector $e_{\eps_i - \eps_j}:=E_{i,j}$, 
where $E_{i,j}$ denotes the 
elementary matrix associated with the coefficient $(i,j)$. 
Take for $B_\g$ the invariant non-degenerate bilinear form $(X,Y) \mapsto \tr(XY)$. 
Since $\tr(E_{i,j}E_{k,l}) = \delta_{j,k} \delta_{i,l}$, 
in the notation of Section \ref{sec:gen} (after Lemma \ref{lem:dec2}), we have
$c_{\eps_i - \eps_j} =1$, $i \not= j.$  
The nonzero weights of the standard representation $V(\varpi_1)$ are 
$\{\delta_1,\ldots,\delta_{r+1}\}$ with 
$$\delta_i = \eps_i - \frac{1}{r+1}(\eps_1+\cdots+\eps_{r+1}), 
\qquad i=1,\ldots,r+1.$$
Moreover,  $V_{\delta_i}= \C v_i$, for $i=1,\ldots, r+1$,  
where $(v_1,\ldots,v_{r+1})$ is the canonical basis of $\C^{r+1}$. 
We have 
$e_{\eps_i - \eps_j} v_k = E_{i,j} v_k = \delta_{j,k} v_i.$
In other words, 
$e_{\eps_i - \eps_j} v_k = a_{\eps_i,\eps_k}^{(e_{\eps_i - \eps_j})} 
v_i = \delta_{j,k} v_i.$ 
Hence 
$$a_{\eps_i,\eps_k}^{(e_{\eps_l - \eps_j})}  =\delta_{i,l}\delta_{j,k} 
\quad \text{ and so }\quad a_{\eps_i,\eps_k}^{(e_{\eps_i - \eps_k})}  =1.$$

\section{Roots and weights in type $C$} 
\label{sec:rootTypeC}
Assume now that 
$\g=\sp_{2r}$. 
We may realize $\g$ as 
the set of of $2r$-size square matrices$ A$ of $\mathscr{M}_{2r}(\C)$ such that 
$JA + A^tJ = 0$, where 
$J=\left(\begin{array}{cc}
0 & I_r \\
-I_r & 0 \end{array}
\right).$ 
We choose as the Cartan subalgebra $\h$ of $\g$ the set of all diagonal matrices 
of $\g$.
For any $i = 1, \ldots, r$, we define a linear
map $\eps_i \in \h^*$ by $\eps_i(h) = h_i$, 
if $h= {\rm diag}(h_1,\ldots, h_r, - h_1,\ldots, - h_r) \in \h$. 
The set of roots of $(\g,\h)$ 
is $\Delta=\{\pm \eps_i \pm \eps_j, \pm 2 \eps_k
\, |\,  1 \leqslant i < j \leqslant r,  \, 1\le k \leqslant r \}$. 
We make the standard choice of 
$\Delta_+=\{\eps_i \pm \eps_j, 2 \eps_k \mid 1 \leqslant i < j \leqslant r, \, 
1 \leqslant k \leqslant r \}$, with basis 
$\Pi = \{ \beta_1, \ldots, \beta_r\},$
where $\beta_i := \eps_i - \eps_{i+1}$ for $i = 1, \ldots, r-1$ 
and $\beta_r = 2\eps_r$. 
Denote by $\c{\beta}_i $ the coroot of the simple root $\beta_i$.
Take for $B_\g$ the non-degenerate invariant bilinear form $(X,Y) \mapsto \dfrac{1}{2}\tr(XY)$.
Using the non-degenerate invariant bilinear form $B_\g$, 
one may identify the Cartan subalgebra $\h$ of $\g$ with $\h^*$ and we get 
$$ \c{\beta}_1=\eps_1-\eps_2, \quad \ldots, \quad \c{\beta}_{r-1} 
 =\eps_{r-1}-\eps_r,  \quad \c{\beta}_r=\eps_r.$$
The fundamental weights are: 
$$\varpi_i=\eps_1+\cdots +\eps_i,  
\qquad i=1,\ldots,r.$$
The fundamental co-weights are:
$$\c{ \varpi}_i = \eps_1+\cdots +\eps_i, \quad \text{ for }\quad  i=1,\ldots,r-1$$
$$\text{and } \quad 
\c{\varpi}_r = \frac{1}{2}(\eps_1+\cdots +\eps_r).$$
Note that the half-sum of positive roots is 
$\rho = \varpi_1+\cdots+\varpi_r
 = \sum_{i=1}^r (\eps_1+\cdots+\eps_i) 
= \sum_{i=1}^r (r-i+1) \eps_i.$
For $\alpha \in \Delta$, the $\alpha$-root space of $\g$ 
is spanned by the root vector $e_\alpha$ as defined below:

\begin{tabular}{lll}
$e_{\eps_i - \eps_j} =E_{i,j} -E_{j+r,i+r}, \quad 1 \leqslant i \not= j \leqslant r,$ & 
&  $e_{2\eps_i} =E_{i,i+r} , \quad 1 \leqslant i \leqslant r,$ \\
 $e_{\eps_i + \eps_j} =E_{i,j+r} + E_{j,i+r}, \quad 1 \leqslant i < j \leqslant r,$  &
&$e_{- 2\eps_i} =E_{i+r,i} , \quad 1 \leqslant i \leqslant r.$\\
$e_{- \eps_i - \eps_j} =E_{i+r,j} +E_{j+r,i}, \quad 1 \leqslant i < j \leqslant r.$ & & \\
\end{tabular} 

\noindent
The constant structures are the following (we write only the nonzero ones): 
for $ i \not= j \not=k$,

\begin{tabular}{lll}
$[e_{\eps_i - \eps_j},e_{\eps_j - \eps_k}] = 
e_{\eps_i - \eps_k},$
&$[e_{\eps_i - \eps_j},e_{- \eps_i - \eps_j}] 
= - 2 e_{-2\eps_j},$
& $[e_{\eps_i + \eps_j},e_{- \eps_j - \eps_k}] = e_{\eps_i - \eps_k},$ \\
$[e_{\eps_i - \eps_j},e_{\eps_k - \eps_i}] = 
- e_{\eps_k - \eps_j},$
 & $[e_{\eps_i - \eps_j},e_{-\eps_i - \eps_k}] = 
 - e_{- \eps_j- \eps_k},$
& $[e_{\eps_i + \eps_j},e_{- 2\eps_j}] = e_{\eps_i-\eps_j},$\\
$[e_{\eps_i - \eps_j},e_{\eps_j + \eps_k}] = 
e_{\eps_i +\eps_k},$
& $[e_{\eps_i - \eps_j},e_{2 \eps_j }] =  e_{\eps_i+ \eps_j},$
& $[e_{2\eps_i},e_{-\eps_i - \eps_j}] = 
e_{\eps_i -\eps_j}.$\\
$[e_{\eps_i - \eps_j},e_{\eps_j + \eps_i}] 
= 2 e_{2\eps_i},$
 & $[e_{\eps_i - \eps_j},e_{- 2 \eps_i }] =  - e_{- \eps_i- \eps_j},$\\
\end{tabular} 

\noindent
and 
$$ c_{\eps_i - \eps_j} = 1, \qquad c_{\eps_i + \eps_j} = 1, \qquad 
c_{- \eps_i - \eps_j} =1, \qquad c_{2\eps_i } = 2, \qquad c_{-2\eps_i }=2.$$ 
The nonzero weights of the standard representation $V(\varpi_1)$ are 
$\{\delta_1,\ldots,\delta_{r},\overline{\delta}_1,\ldots,\overline{\delta}_r\}$, where 
$\delta_i = \eps_i$, $\overline{\delta}_i = -\eps_i$, $i = 1, \ldots, r.$
Moreover,
$V_{\eps_i}= \C v_i ,\quad i=1,\ldots, r,$ and $ 
V_{-\eps_i}= \C v_{i+r} ,\quad i=1,\ldots, r$,
where $(v_1,\ldots,v_{r},v_{1+r},\ldots,v_{2r})$ 
is the canonical basis of $\C^{2r}$.
We have that
\begin{align} \label{eq:coef-a}
 a_{\eps_i,\eps_k}^{(e_{\eps_i - \eps_k})}  =1, \ 
 a_{-\eps_j,-\eps_i}^{(e_{\eps_i - \eps_j})}  =-1, \ 
 a_{\eps_i,-\eps_j}^{(e_{\eps_i + \eps_j})}  = 1, \
 a_{-\eps_i,\eps_j}^{(e_{-\eps_i - \eps_j})}  = 1, \ 
 a_{\eps_i,-\eps_i}^{(e_{2\eps_i})}  =1, \ 
 a_{-\eps_i,\eps_i}^{(e_{-2\eps_i})}  =1.
\end{align}

\section{Admissible triples} 
\label{app:Admissible}
As noted in Remark \ref{Rem:difference}, the difference 
of two difference weights is always a root for $\g=\sl_{r+1}$ of $\g=\sp_{2r}$ 
with $\delta=\varpi_1$. 
For the type $C$ we need to distinguish the different {\em types} of possible configurations.  
A triple $(\alpha, \mu, \nu)$ is called {\em admissible} if 
 $(\mu, \nu) \in (P(\delta))^2$, $\alpha \in \Delta_+$ and $\alpha = \mu - \nu$. 
Such triples are classified by different types   
 as follows:
\begin{description}
\item[Type I]: for some $i \in \{1,\ldots,r\}$, $(\alpha, \mu, \nu) = (2 \eps_i, \delta_i, \bar{\delta}_i),$  
\item[Type II]: for some $i,j \in \{1,\ldots,r\}$, $i \not=j$, 
$ (\alpha, \mu, \nu) = (\eps_i +\eps_j,\delta_i, \overline{\delta}_j),$
\item[Type III a]: for some $i,j \in \{1,\ldots,r\}$, $i \not=j$, 
$ (\alpha, \mu, \nu) = (\eps_i -\eps_j,\delta_i ,\delta_j),$
\item[Type III b]: for some $i,j \in \{1,\ldots,r\}$, $i \not=j$, 
$ (\alpha, \mu, \nu) = (\eps_i -\eps_j, \overline{\delta}_j, \overline{\delta}_i).$
\end{description}
\begin{remark}
 If $\alpha$ is a long root then there is unique $\mu \in P(\delta)$ such that $\mu - \alpha \in P(\delta)$,
 and if $\alpha$ is a short root then there are exactly
 two weights $\mu \in P(\delta)$ such that $\mu - \alpha \in P(\delta)$. In other words, if $\alpha$ is short root then 
there are two admissible triples $(\alpha, \mu, \nu)$ containing $\alpha$.
\end{remark}

For $\gamma$ in the root lattice $Q$, 
we say that $\gamma$ has sign $+$ (respectively, $-$, $0$) if $\gamma \succ 0$
(respectively, $\gamma \prec 0$, $\gamma = 0$). The following lemma will be particularly useful in the proof of Theorem~\ref{corollary:cut-type-C}. 
Its proof of this lemma does not present major difficulties 
and is left to the reader. 
\begin{lemma} \label{lem:types}
Let $(\gamma, \mu, \nu)$ be an admissible triple. 
The admissible triple $(\alpha, \mu', \nu')$ 
such that either $\alpha -\gamma \in \Delta$ or $\alpha =\gamma$,  
and $\nu' \succ \nu$ are the following,
depending on the type of $(\gamma, \mu, \nu)$: 
\begin{enumerate}
\item (Type I)\; If $(\gamma, \mu, \nu) = (2\eps_i, \delta_i, \overline{\delta}_i)$ 
with $i \in \{1,\ldots,r\}$.
Then the triples $(\alpha, \mu', \nu')$ satisfying the conditions 
(i) and (ii) are: 
\begin{center}
\begin{tabular}{|l|l|c|c|}
\hline
$(\alpha, \mu', \nu')$ & condition & $\alpha - \gamma $ & sign of $\alpha - \gamma $ \\
\hline 
$(\eps_i+\eps_k, \delta_i, \overline{\delta}_k)$ & $i<k$ & $\eps_k - \eps_i$ & 
$-$ \\
$(\eps_i-\eps_k, \delta_i, {\delta_k})$ & $i<k$ & $-(\eps_i + \eps_k)$ & 
$-$ \\
 \hline
\end{tabular} 
\end{center}
\item (Type II)\; If 
$(\gamma, \mu, \nu) =( \eps_i+\eps_j, \delta_i, \overline{\delta}_j)$, 
with $i,j \in \{1,\ldots,r\}$, $i \not=j$. 
Then the triples $(\alpha, \mu', \nu')$ satisfying the conditions 
(i) and (ii) are: 
\begin{center}
\begin{tabular}{|l|l|c|c|}
\hline
$(\alpha, \mu', \nu')$ & condition & $\alpha - \gamma $ & sign of $\alpha - \gamma $ \\
\hline 
$(2\eps_i, \delta_i, \overline{\delta}_i)$ & $j<i$ & $\eps_i - \eps_j$ & 
$-$ \\
$(\eps_i+\eps_k, \delta_i, \overline{\delta}_k)$ & $j<k, k \neq i$ & $\eps_k - \eps_j$ & 
$-$ \\
$(\eps_j+\eps_k, \delta_j, \overline{\delta}_k)$ & $j<k<i$ & $\eps_k - \eps_i$ & 
$+$ \\
$(\eps_j+\eps_k, \delta_j, \overline{\delta}_k)$ & $i<k, j<k$ & $\eps_k - \eps_i$ & 
$-$ \\
$(\eps_j+\eps_i, \delta_j, \overline{\delta}_i)$ & $j<i$ & $0$ & 
$0$ \\
$(\eps_k + \eps_i, \delta_k, \overline{\delta}_i)$ & $k <j<i$ & $\eps_k - \eps_j$ & 
$+$ \\
$(\eps_k + \eps_i, \delta_k, \overline{\delta}_i)$ & $k \neq i, j<i, j<k$ & $\eps_k - \eps_j$ & 
$-$ \\
$(\eps_i-\eps_k, \delta_i, {\delta}_k)$ & $i<k$ & $-(\eps_k + \eps_j)$ & 
$-$ \\
$(\eps_j-\eps_k, \delta_j, {\delta}_k)$ & $j<k$ & $-(\eps_k + \eps_j)$ & 
$-$ \\
 \hline
\end{tabular} 
\end{center}
\item (Type III a)\; If 
$(\gamma, \mu, \nu) =( \eps_i - \eps_j, \delta_i, {\delta}_j)$, 
with $i,j \in \{1,\ldots,r\}$, $i < j$. 
Then the only  triple $(\alpha, \mu', \nu')$ satisfying the conditions 
(i) and (ii) is:
\begin{center}
\begin{tabular}{|l|l|c|c|}
\hline
$(\alpha, \mu', \nu')$ & condition & $\alpha - \gamma $ & sign of $\alpha - \gamma $ \\
\hline 
$(\eps_i-\eps_k,\delta_i,\delta_k)$ & $i < k < j$ & $\eps_j - \eps_k$ & 
$-$ \\
 \hline
\end{tabular} 
\end{center}
\item (Type III b)\; If 
$(\gamma, \mu, \nu) =( \eps_i - \eps_j, \overline{\delta}_j, \overline{\delta}_i)$, 
with $i,j \in \{1,\ldots,r\}$, $i < j$.
Then the triples $(\alpha, \mu', \nu')$ satisfying the conditions 
(i) and (ii) are:
\begin{center}
\begin{tabular}{|l|l|c|c|}
\hline
$(\alpha, \mu', \nu')$ & condition & $\alpha - \gamma $ & sign of $\alpha - \gamma $ \\
\hline 
$(\eps_i+\eps_k, \delta_i, \overline{\delta}_k)$ & $i<k$ & $\eps_k + \eps_j$ & 
$+$ \\
$(\eps_i-\eps_k, \delta_i, {\delta}_k)$ & $i<j<k$ & $\eps_j - \eps_k$ & 
$+$ \\
$(\eps_i-\eps_k, \delta_i, {\delta}_k)$ & $i<k<j$ & $\eps_j - \eps_k$ & 
$-$ \\
$(\eps_i-\eps_j, \delta_i, {\delta}_j)$ & $i<j$ & $0$ & 
$0$ \\
$(\eps_k-\eps_j, \delta_k, {\delta}_j)$ & $k<i<j$ & $\eps_k - \eps_i$ & 
$+$ \\
$(\eps_k-\eps_j, \delta_k, {\delta}_j)$ & $i<k<j$ & $\eps_k - \eps_i$ & 
$-$ \\
$(\eps_k - \eps_i, \overline{\delta}_j, \overline{\delta}_k)$ & $i<k<j$ & $\eps_k - \eps_i$ & 
$-$ \\
 \hline
\end{tabular} 
\end{center} 
\end{enumerate}
\end{lemma}

\end{document}